\documentclass[11pt]{article}
\usepackage{a4wide}  
\usepackage{amsmath,amssymb,bbm}
\usepackage{mathrsfs}
\usepackage{mathenv}
\usepackage{fancybox}
\usepackage{version}
\usepackage{authblk}
\usepackage{mathtools}

\usepackage{blindtext}
\usepackage{multicol}
\usepackage{xcolor}
\usepackage{comment}
\usepackage{verbatim}
\usepackage{wrapfig}
\usepackage{graphicx}
\usepackage[left=2.20cm, right=2.20cm, top=2.50cm, bottom=2.50cm]{geometry}
\usepackage{esint}

\usepackage{hyperref}
\hypersetup{colorlinks=true, linkcolor=blue, filecolor=magenta, urlcolor=cyan}
\usepackage{amsthm,cleveref}
\newcommand*\samethanks[1][\value{footnote}]{\footnotemark[#1]}

\newcommand{\R}{\mathbb{R}}

\newcommand{\Z}{\mathbb{Z}}

\newcommand{\br}{\mathbb{R}}

\newcommand{\imm}{\mathrm{i}}

\newcommand{\norme}[1]{\left\Vert #1\right\Vert}
\newcommand{\abs}[1]{\left\lvert #1\right\rvert}
\newcommand{\norm}[1]{\left\lVert #1\right\rVert}

\usepackage{scalerel,stackengine}
\stackMath
\newcommand\reallywidehat[1]{%
\savestack{\tmpbox}{\stretchto{%
  \scaleto{%
    \scalerel*[\widthof{\ensuremath{#1}}]{\kern-.6pt\bigwedge\kern-.6pt}%
    {\rule[-\textheight/2]{1ex}{\textheight}}
  }{\textheight}%
}{0.5ex}}%
\stackon[1.3pt]{#1}{\tmpbox}%
}

\newtheorem{theorem}{Theorem}[section]

\newtheorem{lemma}[theorem]{Lemma}

\newtheorem{remark}[theorem]{Remark}

\newtheorem{proposition}[theorem]{Proposition}

\numberwithin{equation}{section}

\title{From relativistic Vlasov-Maxwell to electron-MHD \\in the quasineutral regime}

\author{
Antoine Gagnebin\thanks{ETH Z\"urich, Department of Mathematics, R\"amistrasse 101, 8092 Z\"urich, Switzerland.\\
Email: \textsf{antoine.gagnebin@math.ethz.ch}, \textsf{mikaela.iacobelli@math.ethz.ch}, \textsf{alexandre.rege@math.ethz.ch}, 
 \textsf{stefano.rossi@math.ethz.ch}}
  \quad
Mikaela Iacobelli\samethanks
  \quad
Alexandre Rege\samethanks
  \quad
Stefano Rossi\samethanks
}

\date{\vspace{-8ex}}

\begin{document}


\maketitle

 \begin{abstract}
We study the quasineutral limit for the relativistic Vlasov--Maxwell system in the framework of analytic regularity. Following the high regularity approach introduced by Grenier \cite{Grenier96} for the Vlasov–Poisson system, we construct local-in-time solutions with analytic bounds uniform in the quasineutrality parameter $\varepsilon$. In contrast to the electrostatic case, the presence of a magnetic field and a solenoidal electric component leads to new oscillatory effects that require a refined decomposition of the electromagnetic fields and the introduction of dispersive correctors. We show that, after appropriate filtering, solutions converge strongly as $\varepsilon$ tends to zero to a limiting system describing kinetic electron magnetohydrodynamics (e-MHD). This is the first strong convergence result for the Vlasov--Maxwell system in the quasineutral limit under analytic regularity assumptions, providing a rigorous justification for the e-MHD reduction, widely used in modelling plasmas in tokamaks and stellarators.
\end{abstract}

\begin{center}
\textit{This paper is dedicated to Claude Bardos on the occasion of his 85\textsuperscript{th} birthday,\\ in recognition of his profound contributions to kinetic theory\\ and the lasting impact of his scientific legacy.}
\end{center}

\vspace{0.5cm}


\section{Introduction}
The relativistic Vlasov--Maxwell system plays a central role in plasma physics, providing a first-principles model for the evolution of collisionless charged particles interacting through self-consistent electromagnetic fields. It accurately describes key features of high-temperature and high-energy plasmas, where relativistic effects and magnetic phenomena are significant.

From a mathematical perspective, the Vlasov--Maxwell system is a nonlinear kinetic model characterized by strongly coupled multiscale behavior, and understanding its evolution in asymptotic regimes poses severe analytic challenges.

In this work, we consider a relativistic electromagnetic plasma of electrons, with ions having infinite mass and constituting a fixed background. The statistical evolution of this system is described by the relativistic Vlasov--Maxwell system:
{\begin{equation}
\label{VM_SI}
\left\{
\begin{aligned}
& \partial_t f(t,x,\xi) + \frac{c\xi}{\sqrt{(cm)^2+\xi^2}}\cdot \nabla_x f(t,x,\xi) -e\left(E+\frac{c\xi}{\sqrt{(cm)^2 +\xi^2}} \wedge B\right) \cdot \nabla_{\xi} f(t,x,\xi)= 0, \\
& \nabla_x \cdot E(t,x)=-\frac{e}{\epsilon_0}\left(\int_{\R^3}f(t,x,\xi)d\xi-\rho_{\text{ion}}\right),\\
& \nabla_x\cdot B(t,x)=0,\\
& \nabla_x \wedge  E(t,x)=-\partial_t B(t,x),\\
&\nabla_x \wedge B(t,x)=-e\mu_0\int_{\R^3}\frac{c\xi}{\sqrt{(cm)^2+\xi^2}}f(t,x,\xi)d\xi+\frac{1}{c^2}\partial_t E(t,x),
\end{aligned}
\right.
\end{equation}}

\noindent
where $-e$ is the elementary electron charge, $m$ the electron mass, and $\rho_{\text{ion}}$ the constant ion density.
Here $f(t,x,\xi)$ is the distribution function of electrons at time $t\geq 0$ with position $x \in \mathbb{T}_L^3\equiv(\R/(2\pi L\Z))^3$ for a given length scale $L>0$ and momentum $\xi \in \R^3$. The electric and magnetic fields $E(t,x)$ and $B(t,x)$ satisfy the Maxwell equations, $\epsilon_0$ and $\mu_0$ are the electric permittivity and magnetic permeability of vacuum, while
$c:= (\epsilon_0 \mu_0)^{-1/2}$ is the speed of light. 

Plasmas, being excellent conductors, are typically treated as quasineutral on macroscopic scales. However, this approximation breaks down at small spatial and temporal scales, where charge separation effects become significant. This behavior is characterized by the Debye length, a fundamental parameter that depends on the physical characteristics of the plasma, defined by
\begin{equation*}
\lambda_{\text{D}} := \sqrt{\frac{\epsilon_0 m v_{\text{th,e}}^2}{e^2 \rho_{\text{el}}}},
\end{equation*}
where $\rho_{\text{el}}$ denotes the average electron density, and $v_{\text{th,e}}$ is the electron thermal velocity given by
\[
v_{\text{th,e}} := \sqrt{\frac{k_{\text{B}} T_{\text{e}}}{m}},
\]
with $k_{\text{B}}$ the Boltzmann constant and $T_{\text{e}}$ the mean electron temperature. Due to the global neutrality condition, we have $\rho_{\text{el}}=\rho_{\text{ion}}$.

In present‐day tokamaks and stellarators the Debye length is {\em much} smaller than the machine size $L$, that is the typical scale of observation.  
This scale separation motivates the introduction of the dimensionless parameter
\begin{equation}
\label{sec0:defpar}
\varepsilon := \frac{\lambda_{\text{D}}}{L} \ll 1,
\end{equation}
which will play the role of a small parameter throughout our analysis.
For typical core
parameters, {$\rho_{\text{el}}\simeq10^{20}\,\mathrm{m^{-3}}$} and
$k_{\text{B}}T_{\text{e}}\simeq(10$--$20)\,\mathrm{keV}$, one finds
$$
  \lambda_{\text{D}}\approx7.5\times10^{-5}\,\mathrm{m}
$$
so that the single ordering parameter
$$
  \varepsilon=\lambda_{\text{D}}/L\approx10^{-5}.
$$
Our analysis therefore treats the same
$\varepsilon\ll 1$ regime that characterises reactor plasmas
\cite{ITER_Physics_Basis1999}.

Introducing the rescaled variables
\[
\widetilde{x} := \frac{x}{L} \in \mathbb{T}^3_1, \qquad
\widetilde{\xi} := \frac{\xi}{m \gamma(v_{\text{th,e}})}, \qquad
\tilde{t} := \frac{\gamma(v_{\text{th,e}}) t}{L}, \qquad
\widetilde{f}(\tilde{t}, \tilde{x}, \tilde{\xi}) := \frac{(m \gamma(v_{\text{th,e}}))^3}{\rho_{\text{ion}}} f(t,x,\xi),
\]
where $\gamma(s) := s / \sqrt{1 - s^2 / c^2}$ is the relativistic factor for $0\leq s < c$, we obtain the corresponding rescaled relativistic Vlasov–Maxwell system depending on $\varepsilon:$
\begin{equation*}
\left\{
\begin{aligned}
& \partial_{\tilde t} \tilde f + \frac{\tilde \xi}{\sqrt{1+\beta^2 {\tilde \xi}^2}}\cdot \nabla_{\tilde x} \tilde f +\Bigg(\tilde E+\frac{\alpha\,\tilde \xi}{\sqrt{1+\beta^2 {\tilde \xi}^2}}\wedge \tilde B\Bigg) \cdot \nabla_{\tilde \xi} \tilde f= 0, \\
& \varepsilon^2 \nabla_{\tilde x} \cdot \tilde E=\int_{\mathbb{R}^3} \tilde f(\tilde t, \tilde x, \tilde \xi) d \tilde \xi-1, \quad  \nabla_{\tilde x} \cdot \tilde B=0,\\
& \nabla_{\tilde x}\wedge \tilde E=-\alpha\partial_{\tilde t} \tilde B,\quad \nabla_{\tilde x}\wedge \tilde B= \alpha \varepsilon^2\partial_{\tilde t} \tilde E+\alpha \int_{\R^3} \frac{\tilde \xi}{\sqrt{1+\beta^2 \tilde \xi^2}}\tilde f(\tilde t, \tilde x, \tilde \xi) d\tilde\xi,
\end{aligned}
\right.
\end{equation*}
where \begin{equation}
    \label{sec0:scalassump}
\beta:= \frac{\gamma(v_{\text{th,e}})}{c}, \qquad \alpha:=\frac{\beta}{\varepsilon},
\end{equation}
with $\varepsilon$ defined in \eqref{sec0:defpar}, and
\[
\tilde E(\tilde t, \tilde x)=-\frac{eL}{m \gamma^2(v_{\text{th,e}})}E(t,x), \quad \tilde B(\tilde t, \tilde x)=-\frac{eL}{\alpha m \gamma(v_{\text{th,e}})}B(t,x).
\]

Our goal is to study this system in the quasineutral regime $\varepsilon \ll 1$ and to rigorously justify its limiting behavior as $\varepsilon$ tends to zero.

Looking at \eqref{sec0:scalassump}, we observe that the relativistic parameter $\beta$ is a priori independent of $\varepsilon$, so the asymptotic behavior of the system is governed by the scaling of the ratio $\alpha = \beta/\varepsilon$. 
In this work, we focus on the regime where magnetic effects remain significant in the limit. 
For this reason, as in \cite{BMP03, PSR04, HKN16}, we follow the classical scaling assumption that $\alpha$ remains of unit size, specifically assuming that $\beta = \varepsilon$. (We refer to \cite{BMP03, YW11} for further discussion of the case where $\alpha$ tends to zero, in which the incompressible Euler equations are obtained in the limit.)
From now on, we will use the term \emph{quasineutral limit} to indicate the scaling regime just specified.

Under this scaling assumption, we are thus led to study the following rescaled relativistic Vlasov–Maxwell system:
\begin{equation}
\label{sys:VMquasi}
(\text{VM}^\varepsilon):=
\left\{
\begin{aligned}
& \partial_t f^\varepsilon(t,x,\xi) + v(\xi)\cdot \nabla_x f^\varepsilon(t,x,\xi) + \left(E^\varepsilon+v(\xi) \wedge B^\varepsilon\right) \cdot \nabla_{\xi}f^\varepsilon(t,x,\xi)= 0, \\
& \varepsilon^2 \nabla_x \cdot E^\varepsilon(t,x)=\rho^\varepsilon(t,x)-1,\\
& \nabla_x \cdot B^\varepsilon(t,x)=0,\\
& \nabla_x \wedge E^\varepsilon(t,x)=-\partial_t B^\varepsilon(t,x),\\
& \nabla_x \wedge B^\varepsilon(t,x)= \varepsilon^2\partial_t E^\varepsilon(t,x)+j^\varepsilon(t,x),
\end{aligned}
\right.
\end{equation}
where $v(\xi)$ denotes the relativistic velocity, and $\rho^\varepsilon(t,x)$ and $j^\varepsilon(t,x)$ are the spatial and current densities, respectively, defined by
\begin{equation}
 \label{def:v&rho&j}
v(\xi)= \frac{\xi}{\sqrt{1+\varepsilon^2 \abs{\xi}^2}},\qquad \rho^\varepsilon(t,x)=\int_{\R^3} f^\varepsilon(t,x,\xi)d\xi, \qquad j^\varepsilon(t,x)=\int_{\R^3} v(\xi)f^\varepsilon(t,x,\xi)d\xi.
\end{equation}
Formally substituting $\varepsilon = 0$ in \eqref{sys:VMquasi}, we obtain the limiting system:
\begin{equation}
\label{sys:kinMHD}
(\text{kinetic e-MHD}):=
\left\{
\begin{aligned}
& \partial_t f(t,x,\xi) + \xi\cdot \nabla_x f(t,x,\xi) + \left(E+\xi \wedge B\right) \cdot \nabla_{\xi} f= 0, \\
& \rho(t,x)=1,\\
& \nabla_x \cdot B(t,x)=0,\\
& \nabla_x \wedge E(t,x)=-\partial_t B(t,x),\\
& \nabla_x \wedge B(t,x)= j(t,x),
\end{aligned}
\right.
\end{equation}
where $j(t,x) = \int_{\R^3} \xi f(t,x,\xi)\,d\xi$. This effective system describes the kinetic electron Magnetohydrodynamic (kinetic e-MHD) regime.

The terminology \emph{quasineutral limit} is justified by the fact that, in the limit, the electron density $\rho(t,x)$ becomes identically equal to the background ion density $\rho_{\text{ion}} = 1$.
The goal of this work is to rigorously justify this limit procedure by studying how, and under which assumptions, solutions of \eqref{sys:VMquasi} converge to solutions of \eqref{sys:kinMHD} as $\varepsilon$ tends to zero.

\subsection{Previous results} 

Since the 1990s, the quasineutral limit has been studied in relation to various Vlasov-type equations describing different types of interacting charged particles. We briefly review the main developments below, with particular attention to works on the electromagnetic Vlasov--Maxwell system.

\subsubsection*{Quasineutral limits for the Vlasov--Poisson system}

The first studies in this area focused on the quasineutral regime for electrons in the electrostatic approximation, described by the Vlasov--Poisson system. The convergence of measure-valued solutions to the incompressible Euler equations as $\varepsilon$ tends to zero was established by Brenier and Grenier \cite{BG94} for time-independent solutions, and later extended by Grenier \cite{G95} to general time-dependent solutions, both using a defect measure argument.

A different approach, particularly relevant to the present work, was introduced by Grenier in \cite{Grenier96}. For general data without structural conditions, he showed that the quasineutral limit holds for initial data with uniformly analytic spatial regularity. The key idea was to relate the Vlasov--Poisson system to a compressible Euler--Poisson model through a multifluid decomposition. Grenier constructed strong solutions with analytic regularity for the Euler--Poisson system that exist on a time interval independent of $\varepsilon$, and showed strong convergence to the incompressible Euler equations after filtering out velocity correctors of amplitude $O(1)$ and frequency $O(\varepsilon^{-1})$. The uniform-in-$\varepsilon$ existence of these solutions does not follow from classical results such as those in \cite{P92, LP91} (see also \cite{GPI21} for a recent review), and requires a dedicated analytic construction.

In the setting of weak solutions, Brenier introduced the modulated energy method in \cite{B00}, applying it to well-prepared initial data to prove convergence to the incompressible Euler equations. This result was later extended by Masmoudi \cite{M01} to more general data that are close to being monokinetic—that is, sharply concentrated around a single velocity profile—allowing for the appearance of electromagnetic oscillations, as also seen in \cite{Grenier96}.

The results discussed above address either weak solutions under very specific assumptions, such as initial data that are essentially monokinetic, or solutions with smooth density and velocity fields, notably those with uniformly analytic initial data. While these conditions might appear restrictive, they are, in fact, necessary. Without such assumptions, the quasineutral limit can fail when only a finite number of derivatives of the initial data are controlled, as demonstrated in \cite{HKH15, HKN16, Baradat20}. Specifically, unless perturbations are made around Penrose-stable homogeneous profiles \cite{HKR15}, instabilities such as the \emph{two-stream instability} can arise. This phenomenon occurs when the velocity distribution exhibits a double-bump structure, leading to linear instability in the Penrose sense \cite{Penrose}. Nevertheless, the quasineutral limit remains stable under rough perturbations when measured in the Wasserstein distance, as shown in \cite{Han-Kwan_Iacobelli, HKI15, Iacobelli22} (see also the survey \cite{GI21}).

\subsubsection*{Quasineutral limit for the Vlasov--Maxwell system}
The quasineutral limit for the non-relativistic Vlasov--Maxwell system was first studied in \cite{BMP03}. In that work, the authors rigorously derived the limit in two regimes: when both $\varepsilon$ and $\alpha$ tend to zero, the limit system is the incompressible Euler equations (see also \cite{YW11}); when $\alpha$ is of order one and $\varepsilon$ tends to zero, the limiting system is the kinetic (e-MHD) model \eqref{sys:kinMHD}, which is the regime considered in this paper. The results are proved for well-prepared initial data using a modulated energy method. Under this assumption, no electromagnetic oscillations appear in the limit (see also \cite{PW08, PW17}).

This analysis was extended in \cite{PSR04}, where the initial data are assumed to be close to monokinetic profiles. In this setting, oscillations of the electromagnetic field emerge. 
By performing a multiscale expansion and constructing suitable correctors for the leading oscillations, the authors prove convergence to the (e-MHD) system.
An instability result in the $L^2$ framework was later established in \cite{HKN2016}, showing that the quasineutral limit may fail for Penrose-unstable initial velocity profiles.

The construction of the quasineutral limit in the analytic framework for the Vlasov--Maxwell system, which corresponds to the electromagnetic counterpart of the electrostatic case derived by Grenier in \cite{Grenier96}, has remained open due to significant technical difficulties. This gap is filled by the present paper, which provides a detailed analysis of the quasineutral limit in this setting.

Unlike the electrostatic case, where the dynamics are generated solely by an irrotational electric field that can be written as a gradient, the Vlasov--Maxwell system involves both electric and magnetic fields. The presence of a magnetic field introduces additional oscillations whose amplitudes and frequencies depend on the quasineutrality parameter. Moreover, the electric field is no longer irrotational, and relativistic corrections must also be taken into account. These features make the analysis of the limit significantly more delicate.

We recall that the well-posedness theory for the Vlasov--Maxwell system remains a major open problem. Nonetheless, several important contributions have significantly advanced our understanding. These include the classical works \cite{GS86, KS02}, as well as more recent developments \cite{LS14, LS16}. In the context of the quasineutral limit, however, a tailored analytic construction with uniform-in-$\varepsilon$ bounds and time of existence is still required, similarly to the Vlasov--Poisson case.

We also refer to further works concerning the construction of particular global solutions and the asymptotic behavior of the Vlasov--Maxwell system, such as \cite{B20,B21,GPS10,BCP19,BAP22,CGII25}. Related results on the regularity of weak solutions in the relativistic setting can be found in \cite{BBech18, BBNG20}.

\subsubsection*{Other models and related singular limits}

The approach introduced in \cite{Grenier96} has proved useful beyond the quasineutral limit, applying to other equations and singular regimes. It has been adapted to the quasineutral limit of the Navier--Stokes--Poisson system in \cite{DM12}, and to the non-relativistic limit in the recent work \cite{BHK22}. Related contributions include earlier results in \cite{AU86, Degond86, Schaffer86}, as well as the stability estimates around Penrose equilibria in the non-relativistic limit presented in \cite{HKNR18}.

The quasineutral limit has also been studied for the Vlasov--Poisson system with massless electrons, also called the ionic Vlasov--Poisson system, which models ion dynamics in the electrostatic regime. The first rigorous result in this direction was established in \cite{CG00}, with further developments in \cite{HK11}. The stability of this limit under rough perturbations was studied in \cite{Han-Kwan_Iacobelli}, and later extended in \cite{IGP20, GPI24}. In the Sobolev setting, the validity of the quasineutral limit for Penrose-stable data was proved in \cite{HKR15}, while the case of domains with boundaries was investigated in \cite{GHKR_13, GVHKR14}.

In the screened regime, the quasineutral limit of the Vlasov--Poisson system with massless electrons leads to the Vlasov--Dirac--Benney equation, as studied in \cite{BB13, BB15}. Additional results for massless electron limits can be found in \cite{Herda16, GJP99}, while for the quasineutral limit in the context of ionic diffusion in the Nernst--Planck--Navier--Stokes system, we refer to \cite{CIL_21}.

Another important physical regime involves the presence of a strong external magnetic field. In the gyrokinetic limit, where the magnetic field strength tends to infinity, it was shown in \cite{GSR98, GSR99} that the particle density converges to a solution of the incompressible Euler equation.

We also mention works related to the (e-MHD) system, including studies on its well-posedness, both with and without viscosity and resistivity, as well as on steady states and asymptotic behavior. Without aiming to be exhaustive, we refer to \cite{BSS88, DZ17, HP17, DZh18, LBLP18, CDG21} for contributions in these directions.

The quasineutral limit can be connected to the study of long-time behavior of solutions to plasma equations: as discussed in \cite{HKD_HDR} (see also \cite{HKI17}), with the right choice of scaling, one can relate the quasineutral limit to how plasmas behave over long periods of time.
A well-known effect in this setting is Landau damping, which has been studied in many works such as \cite{CM_CV_LD, BMM_13,
GNR, HKNR, BMM_22_linear_VP, GNR_2, GI_23, BCGIR24, IRW24, IPWW, HKNR_VM}.

Finally, concerning numerical methods, we refer to \cite{DDS12, DDD17} for the construction of asymptotic-preserving schemes for the Vlasov--Maxwell and Euler--Maxwell systems that remain stable in the quasineutral limit. Related schemes for the Vlasov--Poisson system have been developed in \cite{CDV07, DLV08, BCDS09}.

\subsection{Notation}
To state our results, we begin by introducing the notation used throughout the paper.

Let $\mathbb{T} := \R/(2\pi \Z)$. For a function $g: \mathbb{T}^3_x \to \R$ (here the subscript $x$ is just to emphasize that the function $g$ depends on the spatial variable $x$), we define its Fourier transform by
\begin{equation*}
    \mathcal{F} g(k) \equiv \widehat{g}\,(k) := \int_{\mathbb{T}_x^3} g(x)\, e^{- \imm k\cdot x} \, dx, \quad k \in \Z^3.
\end{equation*}
Then, given a family $\{a_k\}_{k\in \mathbb Z^3}$ of complex numbers, we define 
$$
\mathcal{F}^{-1}\big(\{a_k\}_{k\in \mathbb Z^3}\big)(x):=\frac{1}{(2\pi)^3}\sum_{k \in \mathbb Z^3} a_k e^{\imm k\cdot x}.
$$
With this definition $\mathcal{F}^{-1}\big(\{\widehat{g}(k)\}_{k\in \mathbb Z^3}\big)=g$. In other words,  $\mathcal{F}^{-1}$ is the inverse of $\mathcal{F}$, which also justifies the notation.

Our main results will involve functions with Sobolev and analytic regularity, so we introduce the corresponding functional spaces.
For $s \ge 0$, the Sobolev norm of a function $g : \mathbb{T}^3_x \to \R$ is defined as
\begin{equation}
\label{sec1:sobolevnorm}
    \norm{g}_{H^s_x} := \left( \sum_{k \in \Z^3} (1 + |k|^2)^s\, |\widehat{g}(k)|^2 \right)^{1/2},
\end{equation}
and we denote by $H^s_x \equiv H^s(\mathbb{T}^3_x)$ the space of $L^2(\mathbb{T}^3_x)$ functions with finite norm \eqref{sec1:sobolevnorm}.

To quantify analytic regularity, consider a time-dependent function $g(t,x): [0,\eta] \times \mathbb{T}^3_x \to \R$, for a fixed $\eta > 0$. For $\delta > 1$, we define the analytic norm with parameter $\delta$ by
\begin{equation}
\label{sec1:analyticnorm}
    \left| g(t) \right|_\delta := \sum_{k \in \Z^3} \delta^{|k|}\, \left| \widehat{g}(t,k) \right|, \quad t \in [0,\eta],
\end{equation}
and denote by $B_\delta$ the Banach space of analytic functions with finite norm \eqref{sec1:analyticnorm}.

Fixing $\beta \in (0,1)$, we define the uniform-in-time analytic norm 
\begin{equation}
\label{sec1:analyticnormunif}
    \norm{g}_{\delta_0} := \sup_{\substack{(t,\delta):\, 1 < \delta \le \delta_0 \\ \delta_0 - \delta - \tfrac{t}{\eta} \ge 0}} \left( |g(t)|_\delta + \left( \delta_0 - \delta - \tfrac{t}{\eta} \right)^\beta\, |\nabla_x g(t)|_\delta \right),
\end{equation}
and denote by $B_{\delta_0}^\eta$ the Banach space of continuous functions $g(t,x): [0,\eta] \times \mathbb{T}^3_x \to \R$ with finite norm \eqref{sec1:analyticnormunif}.

We also use the same norm to quantify the analytic regularity of time-independent functions $g_0: \mathbb{T}^3_x \to \R$ such as initial data. In this case the supremum in \eqref{sec1:analyticnormunif} is attained for $t=0$, therefore
\begin{equation}
\label{sec1:analyticnorm_IC}
    \norm{g_0}_{\delta_0} =\sup_{1 < \delta \le \delta_0} \left( |g_0|_\delta + \left( \delta_0 - \delta\right)^\beta\, |\nabla_x g_0|_\delta \right),
\end{equation}
and denote by the symbol $\widetilde{B}_{\delta_0}$ the corresponding Banach space of analytic functions with finite norm \eqref{sec1:analyticnorm_IC}.

\subsection{Main results}
In this section, we state the main theorems of the paper. Before doing so, we reformulate the Vlasov--Maxwell system as a system of compressible Euler-type equations using a multifluid representation, as originally proposed in \cite{Grenier96} in the Vlasov-Poisson setting. 

This reformulation involves decomposing the distribution function into a superposition of monokinetic layers, each indexed by a parameter $\Theta \in M$, where $(M,\mu)$ is a given probability space. Each layer describes particles characterized by their own macroscopic density and momentum fields, which evolve under the influence of the global electromagnetic field. As a result, the kinetic Vlasov--Maxwell equation becomes a continuum of compressible Euler-type systems—one for each layer—coupled through the common electromagnetic field generated by the full ensemble of layers.

More precisely, we consider measure-valued distribution functions $f^\varepsilon(t, x, \xi)$ represented in momentum space as
\begin{equation}
\label{typesoluGrenier}
f^\varepsilon(t, x, \xi) = \int_M \rho^\varepsilon_\Theta(t, x)\, \delta(\xi - \xi^\varepsilon_\Theta(t, x))\, \mu(d\Theta),
\end{equation}
where $\rho^\varepsilon_\Theta(t, x)$ denotes the macroscopic density of the layer indexed by $\Theta$, and $\xi^\varepsilon_\Theta(t, x)$ its momentum field.

Given an initial distribution $f^\varepsilon(0,x,\xi)$ expressed through a multifluid decomposition with prescribed initial data $\rho^\varepsilon_\Theta(0, x)$ and $\xi^\varepsilon_\Theta(0, x)$, each layer evolves according to a compressible Euler--Maxwell system. This consists of a continuity equation and a momentum equation driven by the Lorentz force. For each $\Theta \in M$, the equations read:
\begin{equation}\label{sys:EM}
(\text{EM}^\varepsilon):=\left\{
\begin{aligned}
& \partial_t \rho_\Theta^\varepsilon + \nabla_x \cdot(\rho_\Theta^\varepsilon v(\xi_\Theta^\varepsilon)) = 0, \\
& \partial_t \xi_\Theta^\varepsilon + \left(v(\xi_\Theta^\varepsilon) \cdot \nabla_x \right)\xi_\Theta^\varepsilon = E^\varepsilon + v(\xi_\Theta^\varepsilon) \wedge B^\varepsilon, \\
& \varepsilon^2 \nabla_x \cdot E^\varepsilon = \int_M \rho_\Theta^\varepsilon\, \mu(d\Theta) - 1, \\
& \nabla_x \cdot B^\varepsilon = 0, \\
& \nabla_x \wedge E^\varepsilon = -\partial_t B^\varepsilon, \\
& \nabla_x \wedge B^\varepsilon = \varepsilon^2 \partial_t E^\varepsilon + \int_M \rho_\Theta^\varepsilon v(\xi_\Theta^\varepsilon)\, \mu(d\Theta).
\end{aligned}
\right.
\end{equation}
Once the system is evolved, the full kinetic distribution for the Vlasov--Maxwell system \eqref{VM_SI} can be reconstructed as a measure in momentum space. For any smooth test function $\varphi$, one has
\[
\int_{\R^3} \varphi(\xi)\, f^\varepsilon(t, x,\xi)\, d\xi = \int_M \varphi(\xi^\varepsilon_\Theta(t, x))\, \rho^\varepsilon_\Theta(t, x)\, \mu(d\Theta),
\]
which defines a solution to the Vlasov--Maxwell system that is strong in space and weak in momentum.

This framework can model various classes of initial distributions: For instance, if $ f^\varepsilon(0,x, \xi)$ is a smooth function, it can be expressed by choosing $M = \mathbb{R}^3 $, $ \mu(d\Theta) := \lambda \, \frac{d\Theta}{1 + |\Theta|^4}$, $\xi^\varepsilon_\Theta(0, x) := \Theta $, and $ \rho^\varepsilon_\Theta(0, x) = \lambda^{-1}(1 + |\Theta|^4) f^\varepsilon_0(x, \Theta) $, with a suitable normalization constant $\lambda>0$. 

Alternatively, if $f^\varepsilon(0,x,\xi)$ is a multiple-electron sheet, i.e., a finite sum of $n \in \mathbb{N}$ Dirac masses in momentum space,
\[
f^\varepsilon(0,x,\xi) = \sum_{j=1}^n \alpha_j\, \delta(\xi - \xi_j),
\]
for given momenta $\{\xi_1, \dots, \xi_n\}$ and positive weights $\{\alpha_j\}_{j=1}^n$, then one can take $M = \{1, \dots, n\}$ with $\mu(d\Theta) = \frac{1}{n}\sum_{j=1}^n \delta(\Theta - j)$, and set $\xi^\varepsilon_\Theta(0,x) \equiv \xi_\Theta$, $\rho^\varepsilon_\Theta(0,x) \equiv n \alpha_\Theta$ for each $\Theta \in \{1,\dots,n\}$.
\vskip 0.5 cm
Assuming that $\rho^\varepsilon_\Theta$ converges to $\rho_\Theta$ and $\xi^\varepsilon_\Theta$ converges to $w_\Theta$ as $\varepsilon \to 0$, we formally obtain the limiting system
\begin{equation}\label{sys:limEM}
(\text{e-MHD}):=\left\{
\begin{aligned}
& \partial_t \rho_\Theta + \nabla_x \cdot(\rho_\Theta w_\Theta) = 0, \\
& \partial_t w_\Theta + \left(w_\Theta \cdot \nabla_x \right)w_\Theta = E + w_\Theta \wedge B, \\
& \int_M \rho_\Theta\, \mu(d\Theta) = 1, \\
& \nabla_x \cdot B = 0, \\
& \nabla_x \wedge E = -\partial_t B, \\
& \nabla_x \wedge B = j,
\end{aligned}
\right.
\end{equation}
which corresponds to the kinetic (e-MHD) equations \eqref{sys:kinMHD} for $f^\varepsilon$ defined as in \eqref{typesoluGrenier}.

We use the notation $w_\Theta$ for the limiting momentum field to reflect the fact that strong convergence will only hold after subtracting suitable oscillatory correctors from $\xi^\varepsilon_\Theta$, a key aspect that will be discussed in detail later in the paper.

\vskip 0.5 cm

In order to prove the quasineutral limit, we begin by studying the Euler--Maxwell system \eqref{sys:EM}. We construct a class of local-in-time solutions that exist on a time interval independent of $\varepsilon$, and remain uniformly bounded in $\varepsilon$ with respect to the analytic norm defined in \eqref{sec1:analyticnormunif}. This is the content of the following theorem.

\begin{theorem}(Local-in-time (uniform in $\varepsilon$) solutions to (EM$^\varepsilon$) system \eqref{sys:EM}.)
    \label{sec3:mainthm}
    Given $\varepsilon>0$ and a probability space $(M,\mu)$, let $\{\rho^\varepsilon_\Theta(0), \xi^\varepsilon_\Theta(0)\}_{\Theta \in M}$ a bounded family belonging to $\widetilde{B}_{\delta_0}\times \widetilde{B}_{\delta_0}$, for a given $\delta_0>1$. Let us also consider initial data for the electromagnetic fields $(E^\varepsilon_0, B^\varepsilon_0, \partial_t E^\varepsilon_0, \partial_t B^\varepsilon_0)$ such that
    \begin{equation}
        \label{sec0:Maxwellaws}
    \begin{aligned}
        &\varepsilon^2\nabla_x \cdot E^\varepsilon_0(x)=\int_M\rho^\varepsilon_\Theta(0) d\mu(\Theta)-1,\quad \nabla_x \cdot B^\varepsilon_0(x)=0,\\ \partial_tB^\varepsilon_0(x)&=-\nabla_x\wedge E^\varepsilon_0(x), \quad 
       \varepsilon^2\partial_tE^\varepsilon_0(x)=\int_M\rho^\varepsilon_\Theta(0)v(\xi^\varepsilon_\Theta(0))d\mu(\Theta)- \nabla_x  \wedge B^\varepsilon_0.
    \end{aligned}
    \end{equation}
    Moreover,  assume that there exist $\delta_1>\delta_0$ and $C_0>0$ such that
    \begin{equation}
    \label{assumption_quasineutral}
    \sup_{\varepsilon>0} \left( \norm{\varepsilon E^\varepsilon_0}_{\delta_1}+ \norm{B_0^\varepsilon}_{\delta_1}\right)\le C_0.
    \end{equation}
 Then there exist $\varepsilon_0>0$ and $\eta>0$ such that the following holds: For every $\varepsilon \in (0,\varepsilon_0]$ there exists a unique solution $\left(\rho^\varepsilon_\Theta,\xi^\varepsilon_\Theta, E^\varepsilon,B^\varepsilon\right)$ to the (EM$^\varepsilon$) system \eqref{sys:EM} in the interval of time $[0,\eta]$ with initial data $(\rho^\varepsilon_\Theta(0),\xi^\varepsilon_\Theta(0),E^\varepsilon_0, B^\varepsilon_0)$ such that the functions $\left(\rho^\varepsilon_\Theta,\xi^\varepsilon_\Theta,\varepsilon E^\varepsilon,B^\varepsilon\right)$ are uniformly bounded in $B_{\delta_0}^\eta$, independently of $\varepsilon$. 
    \noindent Moreover
    $\varepsilon^{-1}\left(\int_M\rho^\varepsilon_\Theta(t,x)\mu(d\Theta)-1 \right)$ 
    is uniformly bounded in $B_{\delta_0}^\eta$.
\end{theorem}

The uniform-in-$\varepsilon$ solutions established in Theorem \ref{sec3:mainthm} allow us to take the limit as $\varepsilon$ goes to zero in the solutions to the Euler--Maxwell equations \eqref{sys:EM} without deteriorating the bounds on the sequence of solutions and on the time of existence. 
It is not, however, reasonable to expect strong convergence of the full macroscopic quantities, as strongly oscillating terms due to the electromagnetic field are present in the equations. In particular, convergence is achieved once appropriate correctors, related to the oscillations of the electromagnetic field, are filtered out.
As a result, we obtain a rigorous derivation of solutions to the (e-MHD) system \eqref{sys:kinMHD}, as stated in the following theorem.

\begin{theorem}(Quasineutral limit - Derivation of the (e-MHD) system \eqref{sys:limEM}.)
\label{sec1:mainthm2}

    Let $(\rho_\Theta^\varepsilon,\xi^\varepsilon_\Theta,E^\varepsilon,B^\varepsilon)$ be solutions to (EM$^\varepsilon$) system \eqref{sys:EM} for $0\le t \le T$.

    Assume that $ j^\varepsilon(0), \varepsilon E^\varepsilon(0),B^\varepsilon(0)$ and, for all $\Theta \in M$, $\rho^\varepsilon_\Theta(0), \xi^\varepsilon_\Theta(0)$ have weak limits (in $\varepsilon$) in the sense of distributions and assume that, for $s>3/2+2$,
    \begin{equation}
    \label{sec0:assumptthm2}
    \sup_{t \le T, \varepsilon, \Theta \in M} \left[\norm{\rho^\varepsilon_\Theta(t)}_{H^s_x}+\norm{\xi^\varepsilon_\Theta(t)}_{H^s_x}+\norm{\varepsilon E^\varepsilon(t)}_{H^s_x}+\norm{B^\varepsilon(t)}_{H^s_x}\right]< + \infty.
    \end{equation}
    Then there exists a solution $(\rho_\Theta, w_\Theta,E,B)$ to the limit (e-MHD) system \eqref{sys:limEM} for $t\in[0,T]$ 
    and two spatially independent correctors $d_{0,+},d_{0,-}$, two irrotational correctors $d_{1,+}, d_{1,-}$ and two solenoidal ones $d_{2,+}, d_{2,-}$ such that
   \small\begin{equation}
        \label{sec0:correctmainthm}
    \begin{aligned}
    &\rho_\Theta^\varepsilon(t,\cdot) \to \rho_\Theta(t,\cdot)\\
    &\xi^\varepsilon_\Theta(t,\cdot)-
      \sum_{\sigma \in\{\pm\}}(-\sigma\imm) \exp\left(\sigma \frac{\imm t}{\varepsilon}\right)d_{0,\sigma}(t)\\
    &\quad-\mathcal{F}^{-1}\Bigg(\Bigg\{
    \sum_{\sigma \in\{\pm\}}(-\sigma\imm) \exp\left(\sigma \frac{\imm t}{\varepsilon}\right) 
    \widehat{d_{1,\sigma}}(t,k)
    +\sum_{\sigma \in \{\pm\}}(-\sigma \imm)\exp\left(\sigma \sqrt{1+|k|^2}\frac{\imm t}{\varepsilon}\right)
    \frac{\widehat{d_{2,\sigma}}(t,k)}{\sqrt{1+|k|^2}}\Bigg\}_{k\in\Z^3}\Bigg)
      \to w_\Theta(t,\cdot)\\
    &\varepsilon E^\varepsilon(t,\cdot) -\mathcal{F}^{-1}\Bigg(\Bigg\{\sum_{\sigma \in\{\pm\}} \exp\left(\sigma \frac{\imm t}{\varepsilon}\right) \left[\widehat{d_{1,\sigma}}(t,k)+\mathbf{1}_{k=0}d_{0,\sigma}(t)\right]
        +\sum_{\sigma \in \{\pm\}}\exp\left(\sigma \sqrt{1+|k|^2}\frac{\imm t}{\varepsilon}\right)\widehat{d_{2,\sigma}}(t,k) \Bigg\}_{k \in \Z^3} \Bigg)
      \to 0\\
     & B^\varepsilon(t,\cdot) - \mathcal{F}^{-1}\Bigg(\Bigg\{ 
     \sum_{\sigma \in \{\pm\}}(-\sigma \imm) \exp\left(\sigma \sqrt{1+|k|^2}\frac{\imm t}{\varepsilon}\right)\frac{k \wedge\widehat{d_{2,\sigma}}(t,k)}{\sqrt{1+|k|^2}}\Bigg\}_{k \in \Z^3}  \Bigg)
      \to B(t,\cdot)
    \end{aligned}
    \end{equation}
    strongly in $C^0([0,T]; H^{s'}(\mathbb{T}^3_x))$, for $s'<s-2$.
    Moreover, the initial data are given by
    \begin{align*}
    &\rho_\Theta(0)=\lim_{\varepsilon\to 0}\rho_\Theta^\varepsilon(0), \qquad w_\Theta(0)=\lim_{\varepsilon \to 0}
    \left[\xi^\varepsilon_\Theta(0)
    - W^\varepsilon(0)\right],\\
    & E(0)=\lim_{\varepsilon\to 0}\left[-W^\varepsilon(0) \cdot \nabla_x W^\varepsilon(0)- W^\varepsilon(0) \wedge (\nabla_x \wedge W^\varepsilon(0))\right]\\
    & B(0)=\lim_{\varepsilon \to 0} \left[B^\varepsilon(0)+ \nabla_x \wedge W^\varepsilon(0)\right]
    \end{align*}
    where 
    \begin{align*}
        W^\varepsilon(0,x)&:= \nabla_x \left( \Delta_x^{-1} (\nabla_x \cdot j^\varepsilon(0,x))\right)\\
        &\quad - (1-\Delta_x)^{-1}\left[\nabla_x \wedge B^\varepsilon(0,x) + \nabla_x \wedge \Delta_x^{-1}(\nabla_x \wedge j^\varepsilon(0,x)) \right]-\frac{1}{(2\pi)^3}\int_{\mathbb{T}^3_x} j^\varepsilon(0,x) dx.
    \end{align*}
    The equations satisfied by the correctors are given in \eqref{sec4:eqcorrector0}, \eqref{sec4:eqcorrector1} and \eqref{sec4:eqcorrector2}.
\end{theorem}

    Below, we list a series of observations regarding the content of Theorems \ref{sec3:mainthm} and \ref{sec1:mainthm2}.
    \begin{remark}

    It is important to note that the quasineutral limit for the Euler--Maxwell system \eqref{sys:EM} may be ill-posed in Sobolev spaces, even in one spatial dimension, as shown in \cite{HKN16}. For this reason, analytic functions, such as those in equation \eqref{sec1:analyticnorm}, provide the infinitely regular framework natural for deriving uniform estimates for a small interval of time as stated in Theorem \ref{sec3:mainthm}. 
To address these issues, we work in the analytic setting, as in \cite{Grenier96}, and rely on a simplified version of the Cauchy--Kovalevskaya theorem due to Caflisch \cite{Caf90}, which provides better control over the region of existence.
\end{remark}

 \begin{remark}
    Notice that, to have an independent set of initial conditions, it is sufficient to consider $\{\rho^\varepsilon_\Theta(0), \xi^\varepsilon_\Theta(0)\}_{\Theta \in M} $ and two vector fields $E^\varepsilon_0, B_0^\varepsilon$ such that $E^\varepsilon_0$ verifies the Gauss's law and $B^\varepsilon_0$ is divergence-free. Indeed, the other Maxwell's equations are obtained defining 
    \begin{align*}
        \partial_t B^\varepsilon_0&:=-\nabla_x \wedge E_0^\varepsilon\\
    \partial_t E^\varepsilon_0&:= \frac{1}{\varepsilon^2}\nabla_x  \wedge  B^\varepsilon_0- \frac{1}{\varepsilon^2} \int_M\rho^\varepsilon_\Theta(0) v(\xi_\Theta^\varepsilon(0)) \mu(d\Theta).
    \end{align*}
\end{remark}

\begin{remark}
    The assumption \eqref{assumption_quasineutral} can be seen as a uniform in $\varepsilon$ bound for the point-wise initial electric and magnetic energies. In particular, by the Gauss's law, \eqref{assumption_quasineutral} implies 
     \begin{equation}
     \label{assumption_quasineutral_2}
     \norm{\int_M\rho^\varepsilon_\Theta(0,x)\mu(d\Theta)-1 }_{\delta_0} \le \varepsilon \norm{\nabla_x \cdot (\varepsilon E^\varepsilon_0)}_{\delta_0}\le C_0 \varepsilon.
    \end{equation}
\end{remark}
\begin{remark}

As will be highlighted in the discussion of key estimates in Section~\ref{sec:key_ingredients}, the electric field is a highly oscillatory term that becomes unbounded in $\varepsilon$ under the quasineutral scaling. On the other hand, Maxwell--Faraday's law, $\partial_t B^\varepsilon = - \nabla_x \wedge E^\varepsilon$, may suggest that the magnetic field $B^\varepsilon$ remains bounded as $\varepsilon \to 0$, since it is the time integral of the curl of $E^\varepsilon$. 
   However, even though the oscillations are bounded in amplitude by 
$\varepsilon$, they are not in frequency. For this reason, even in the term involving $B^\varepsilon$, strong convergence is achieved only after introducing suitable correctors, as stated in Theorem \ref{sec1:mainthm2}.
\end{remark}

\begin{remark}
Although our analysis is carried out on the periodic torus, we expect the same
strategy to extend to the whole space $\R^{3}$. On $\R^{3}$ one gains
additional decay from the dispersive properties of the Klein–Gordon phase, so
the oscillatory part of the magnetic field should radiate away and yield
stronger local convergence of $B^{\varepsilon}$ once $t>0$.  
A rigorous proof would require coupling the present ideas with a dispersive estimate for the oscillatory integral given by the inverse Fourier transform of the corrector.
In this spatially unconfined case also the equations for the correctors would change due to the appearance of a different structure in the nonlinearities. Because the manuscript is already long
and technically involved, we defer this extension to future work.
\end{remark}

\subsubsection{Discussion of proofs}
\label{sec:key_ingredients}
The proof of the quasineutral limit for the (EM$^\varepsilon$) system \eqref{sys:EM} in the high-regularity setting builds on the strategy introduced in \cite{Grenier96}. However, the presence of a self-generated magnetic field introduces significant new difficulties that require novel ideas. We now provide a more detailed discussion of the arguments underlying our main results.

\vskip 0.5 cm
\noindent \emph{Discussion of the proof of Theorem \ref{sec3:mainthm}:}
Theorem \ref{sec3:mainthm} provides an example of a class of analytic solutions for which the assumptions of Theorem \ref{sec1:mainthm2} hold. It constructs local-in-time analytic solutions, with the primary challenge being that both the time interval and the boundedness in norm of these solutions must be independent of $\varepsilon$ in order to study the limit as $\varepsilon$ tends to zero.

It is important to highlight that, contrary to the electrostatic case, the electric field $E^\varepsilon$ is not only irrotational and on the torus $\mathbb{T}_x^3$ it can be decomposed in the following way:
\begin{equation}
\label{sec0:decomposition}
E^\varepsilon(t,x)=E^\varepsilon_{\text{irr}}(t,x) + E^\varepsilon_\text{sol}(t,x)+E^\varepsilon_{\text{mean}}(t),
\end{equation}
where, given a scalar field $\varphi^\varepsilon$ and a vector field $\psi^\varepsilon$, $\ E^\varepsilon_{\text{irr}}(t,x)=\nabla_x \varphi^\varepsilon(t,x)$ is the irrotational component, $E^\varepsilon_{\text{sol}}(t,x)=\nabla_x \wedge \psi^\varepsilon(t,x)$ the solenoidal one and $E^\varepsilon_{\text{mean}}(t)$ is in general an harmonic function that in this case coincides with the spatial mean of $E^\varepsilon(t,x)$.

This decomposition is known as the Helmholtz--Hodge decomposition of $E^\varepsilon$ (see, e.g. \cite{WA83, Bhatia13}),
and it is not unique since, given $C>0$ and $\chi$ a scalar function, the change of potentials $\varphi^\varepsilon \mapsto\varphi^\varepsilon+C$ and $\psi^\varepsilon\mapsto \psi^\varepsilon+\nabla\chi$ generate the same fields.

For the irrotational component, Gauss's law holds, and therefore we expect, as in the electrostatic case, that it exhibits oscillations of order $O(\varepsilon^{-1})$. On the other hand, the solenoidal part is governed by the Maxwell--Faraday's law, which itself also depends on the magnetic field $B^\varepsilon$. 
We show that the solenoidal part of the electric field satisfies a Klein--Gordon-type equation, which has a dispersive character, and it will also exhibit oscillations with amplitude of order $O(\varepsilon^{-1})$.

 One of the main points of study is the scaling of the amplitudes and the frequencies of the oscillations for the magnetic field. The key difference here is that, unlike the electric field, the amplitudes of the oscillations of the magnetic field are of order $O(1)$.
 This boundedness of the magnetic field is suggested by the Maxwell--Faraday equation 
\[
B^\varepsilon(t) = B^\varepsilon(0) - \int_0^t \nabla_x \wedge E^\varepsilon(s) \, ds,
\]
as it is an integral in time of a spatial derivative of the electric field. Proving this behavior is crucial for the proof of the theorem and requires the introduction of
three equations for the irrotational and solenoidal components 
 and the spatial mean of the electric field
(see their expressions in \eqref{sec3:harmonic_irr}
,\eqref{sec3:harmonic_sol}  and \eqref{sec1:spatialmean}), along with several refined estimates on the analytic norm dependence on the parameter \(\varepsilon\).

The proof proceeds by performing appropriate a priori estimates on $(\rho^\varepsilon_\Theta, \xi^\varepsilon_\Theta, G^\varepsilon, B^\varepsilon)$, where 
\[
G^\varepsilon(t,x) = \int_0^t E^\varepsilon(s,x) \, ds= \int_0^t E^\varepsilon_{\text{irr}}(s,x) ds+ \int_0^t E^\varepsilon_{\text{sol}}(s,x) ds+ \int_0^t E^\varepsilon_{\text{mean}}(s) ds,
\]
will also be decomposed into  three components
related to $E^\varepsilon$.
The introduction of the term $G^\varepsilon$ is due, as explained above, to the fact that $E^\varepsilon$ is a quantity that oscillates with amplitude $O(\varepsilon^{-1})$ (in contrast to $B^\varepsilon$, which has oscillations of bounded amplitude). Therefore, it is reasonable to expect that its time integral is a quantity uniformly bounded in $\varepsilon$.

During the a priori estimates, we will also address the study of the relativistic corrections in \eqref{def:v&rho&j} and \eqref{sec0:relcorr2}. These are of a perturbative nature and they will be treated through applications of Lemma \ref{app:lemmarem}.

The uniform in $\varepsilon$ boundedness of the quantities $(\rho^\varepsilon_\Theta, \xi^\varepsilon_\Theta, G^\varepsilon, B^\varepsilon)$ on a $\varepsilon$-independent time interval, will allow the construction of an iterative scheme where the solution of the nonlinear problem to (EM$^\varepsilon$) system \eqref{sys:EM} is obtained by considering a suitable sequence of linear initial value problems.
Thanks to the assumptions on the initial data and the uniform in $\varepsilon$ a priori estimates, the sequence of solutions will converge to a solution to \eqref{sys:EM} verifying the required properties.

\vskip 0.5 cm

\noindent\emph{Discussion of the proof of Theorem \ref{sec1:mainthm2}:}
Once an analytic class of solutions, such as those in Theorem \ref{sec3:mainthm}, has been constructed, we have a non-empty set of data for which the assumptions of Theorem \ref{sec1:mainthm2} are satisfied. 

The derivation of the (e-MHD) system \eqref{sys:limEM} in the limit is obtained by subtracting from $ (\xi_\Theta^\varepsilon, \varepsilon E^\varepsilon, B^\varepsilon) $ correctors with oscillation amplitude bounded in $\varepsilon$ but with frequency $O(\varepsilon^{-1})$.
For this purpose, similarly to \cite{Grenier96} we perform an additional decomposition of the electric field
\begin{equation*}
E^\varepsilon=\mathcal{H}^\varepsilon E^\varepsilon +(1-\mathcal{H}^\varepsilon)E^\varepsilon,
\end{equation*}
where $\mathcal{H}^\varepsilon E^\varepsilon$ is a time-averaged quantity along the oscillation periods of the spatial mean and the irrotational and solenoidal components (see \eqref{sec4: decomp} for the precise definition).

Since $ \mathcal{H}^\varepsilon E^\varepsilon $ is a time-integral average of a function with oscillation amplitude $O( \varepsilon^{-1})$ and frequency $O(\varepsilon^{-1})$, it will be bounded in Sobolev norm. However, this is clearly not true for $ (1 - \mathcal{H}^\varepsilon) E^\varepsilon $, which is the oscillatory term that will contribute to the correctors of the convergent momentum field. In particular, the convergence toward the limit equation is achieved for the corrective term $ w^\varepsilon_\Theta - W^\varepsilon $, where the corrector is defined by 
$$ W^\varepsilon = \int_0^t (1 - \mathcal{H}^\varepsilon) E^\varepsilon(s) \, ds .$$ 
Similarly, we also get the convergence for the magnetic field modulo the corrector $B^\varepsilon + \nabla \wedge W^\varepsilon $. 

Once the convergence to the limit equation is proven, we study the expression of the correctors in the limit $\varepsilon$ going to zero. These are the terms introduced in \eqref{sec0:correctmainthm} and are
given by $d_{1,\pm}$ and $d_{2,\pm}$, limits as $\varepsilon$ goes to zero of $ W^\varepsilon$ (up to a phase).
Specifically, we obtain the equations for the limit correctors in \eqref{sec4:eqcorrector1} and \eqref{sec4:eqcorrector2}. Unlike the electrostatic case, there are interactions between the modes $d_{1,\pm}$ and $d_{2,\pm}$ but always of finite cardinality.

\subsubsection{Analysis of fast electromagnetic oscillations}

This section is aimed at achieving a more physical understanding of the type of oscillations for the correctors introduced in the statement of Theorem \ref{sec1:mainthm2} to prove the convergence to the (e-MHD) system \eqref{sys:limEM}. As will be clear in Section \ref{sec3:eqcorrect}, the expression of the correctors in the limit $\varepsilon$ going to zero depends on the equations satisfied by the components
of the electric field, {which represents the highly oscillatory physical quantity}.

For this reason, we here consider the Vlasov--Maxwell system \eqref{VM_SI} in physical units and we derive these 
equations, characterizing the type of oscillations obtained in the limit as $\varepsilon$ tends to zero by the correctors, showing their dispersion relations. With some abuse of notation, we will use $\rho$ and $j$ to refer to the unscaled quantities and not to \eqref{def:v&rho&j}, that is here we have 
    \begin{equation*}
         j(t,x) = \int_{\br^3} \frac{c\xi}{\sqrt{(cm)^2+\xi^2}} f(t,x,\xi) d\xi.
    \end{equation*}

Before launching into the derivations, we summarise the picture that
will emerge from the next few pages.  When the Maxwell equations are combined
with the quasineutral Vlasov dynamics, the electric field naturally splits
via the Helmholtz-- Hodge decomposition
\(
E = E_{\mathrm{irr}} + E_{\mathrm{sol}} + E_{\mathrm{mean}},
\)
and each component turns out to satisfy its own wave equation:

\begin{itemize}
\item[\textbf{(i)}] The divergence of the irrotational part,
      \(\nabla_x\!\cdot E_{\mathrm{irr}}\), evolves according to a
      simple harmonic oscillator with the electron–plasma frequency
      \(\omega_{pe}\).  These are the familiar \textit{Langmuir (electron-plasma) oscillations}.
\item[\textbf{(ii)}] The curl of the solenoidal part,
      \(\nabla_x {\wedge} E_{\mathrm{sol}}\), satisfies a vector
      Klein–Gordon equation whose symbol
      \(\omega^2(k)=\omega_{pe}^2+|k|^2\)
      reproduces the cold-plasma light wave and, in a strong guide
      field, the short-wavelength whistler/kinetic-Alfvén mode.
\item[\textbf{(iii)}] The spatial mean of the electric field $E_{min}$ solves a forced harmonic oscillator with the same fast natural frequency.
\end{itemize}
We note that all three oscillatory modes live on the ultra-fast time-scale
\(t_{\text{fast}}\sim\varepsilon\), whereas the transport and magnetic
dynamics of interest evolve on order-one times. These dispersion relations are well known in the physics literature, see for instance \cite{Stix1992,book_chen}.

We now enter in the details of the argument.
For the irrotational part, we look at the Gauss's law
\[
 \nabla_x \cdot E_{\text{irr}}(t,x)=-\frac{e}{\epsilon_0}\left(\rho(t,x) -\rho_{\text{ion}}\right), \quad \rho(t,x):=\int_{\mathbb{R}^3} f(t,x,\xi) d\xi.
\]
By taking two-time derivatives and from the continuity equation, $\partial_{t}\rho(t,x) + \nabla_x \cdot j(t,x)=0$, we have
    \begin{equation}
        \label{sec3:auxiliar}
      \partial_{tt}^2 \nabla_x \cdot E_{\text{irr}}(t,x)= \frac{e}{\epsilon_0} \nabla_x \cdot \partial_t j(t,x).
    \end{equation}
Then, we derive an equation for the relativistic current density $j(t,x)$: by the Vlasov--Maxwell equation in \eqref{VM_SI}, we have
    \begin{align*}
        \partial_t j(t,x) & = \partial_t \int_{\br^3} \frac{c\xi}{\sqrt{(cm)^2+\xi^2}} f(t,x,\xi) d\xi \\
        &=  - \int_{\br^3} \frac{c\xi}{\sqrt{(cm)^2+\xi^2}} \left[\frac{c\xi}{\sqrt{(cm)^2+\xi^2}} \cdot \nabla_x f(t,x,\xi) \right. \\
        & \qquad \quad \left. - e\left(E(t,x)+\frac{c\xi}{\sqrt{(cm)^2 +\xi^2}} \wedge B(t,x)\right)
        \cdot\nabla_\xi f(t,x,\xi)\right] d\xi \\
        & = - \nabla_x \cdot \left( \int_{\br^3} \frac{c^2\xi \otimes\xi}{(cm)^2+\xi^2}  f(t,x,\xi) d\xi \right)\\
            & \quad + e  \int_{\br^3} \frac{\xi}{m} \nabla_\xi \cdot \left[\left(E(t,x)+\frac{c\xi}{\sqrt{(cm)^2 +\xi^2}} \wedge B(t,x)\right) 
            f(t,x,\xi) \right] d\xi \\
    & \quad + e  \int_{\br^3} \left(\frac{c\xi}{\sqrt{(cm)^2+\xi^2}} - \frac{\xi}{m} \right)\nabla_\xi \cdot\left[\left(E(t,x)+\frac{c\xi}{\sqrt{(cm)^2 +\xi^2}} \wedge B(t,x)\right)
         f(t,x,\xi)\right] d\xi,
    \end{align*}
where in the last equality we used that  $\nabla_\xi \cdot \left(\frac{c\xi}{\sqrt{(cm)^2 +\xi^2}} \wedge B\right)=0$ since $\nabla_\xi \wedge \left(\frac{c\xi}{\sqrt{(cm)^2 +\xi^2}}\right)=0$. 
By integration by parts,
we get the following equation for $j(t,x)$:
{\begin{equation} \label{dt_j}
    \begin{aligned}
        \partial_t j(t,x)
        & = -\nabla_x \cdot \left( \int_{\br^3} \frac{c^2\xi \otimes\xi}{(cm)^2+\xi^2}  f(t,x,\xi) d\xi \right)
    - \frac{e}{m} E(t,x) \rho(t,x) 
    - \frac{e}{m} j(t,x) \wedge B(t,x)- e R(t,x),
    \end{aligned}
    \end{equation}}
 where 
    \begin{align}
    \label{sec0:relcorr2}
        R(t,x) : =   \int_{\br^3} \nabla_\xi \left(\frac{c\xi}{\sqrt{(cm)^2+\xi^2}} - \frac{\xi}{m} \right) \left(E(t,x)+\frac{c\xi}{\sqrt{(cm)^2 +\xi^2}} \wedge B(t,x)\right)
         f(t,x,\xi) d\xi.
    \end{align}
By \eqref{sec3:auxiliar}, \eqref{dt_j} and substituting the expression $\rho(t,x)=\rho_{\text{ion}} - \frac{\epsilon_0}{e} \nabla_x \cdot E_{\text{irr}}(t,x)$, we get (using the Einstein notation for repeated indices))
\begin{align}
\label{sec3:harmonic_irr}
        & \left( \partial^2_{tt} + \frac{e^2 \rho_{\text{ion}}}{m\epsilon_0} \right)
        \nabla_x \cdot E_{\text{irr}}(t,x)
        = - \frac{e}{\epsilon_0} \partial_{x_i} \partial_{x_j}  \left( \int_{\br^3} \frac{c^2\xi_i \otimes\xi_j}{(cm)^2+\xi^2}  f(t,x,\xi) d\xi \right) \nonumber \\
       & \qquad \qquad + \frac{e}{m} \nabla_x \cdot (E(t,x) \nabla_x \cdot E_{\text{irr}}(t,x)) 
    - \frac{e^2}{m\epsilon_0} \nabla_x \cdot (j(t,x) \wedge B(t,x))- \frac{e^2}{m\epsilon_0} \nabla_x \cdot R(t,x). 
    \end{align}
Hence, the oscillations of the irrotational part of $E$ are given by the following dispersion relation
\[
\omega_{\text{pe}}:=\sqrt{\frac{e^2 \rho_{\text{ion}}}{m \epsilon_0}},
\]
which is the plasma electron frequency. This is the Langmuir dispersion relation in the regime where the electric thermal velocity of the electrons is negligible, coherently with the assumption $\beta=\varepsilon$ in \eqref{sec0:scalassump}. We notice that this kind of dispersion relation is the one obtained in the electrostatic setting considered in \cite{Grenier96}.\\

We now consider the solenoidal part of the electric field: using that $\partial_t B=-\nabla_x \wedge E_{\text{sol}}$, we have
\[
\partial_{tt}^2\left(\nabla_x  \wedge  E_{\text{sol}}\right)(t,x)=-\partial_{t}(\partial_{tt}^2 B)(t,x)=-c^2 \Delta_x (\partial_t B)(t,x) - \frac{ e}{\epsilon_0} \nabla_x  \wedge  \partial_t j(t,x),
\]
where we used that $B(t,x)$ verifies the wave equation
\[
\partial_{tt}^2 B(t,x) -c^2 \Delta_x B(t,x)= \frac{e}{\epsilon_0}\nabla_x  \wedge  j(t,x).
\]
Therefore,
\begin{align*}
    \partial_{tt}^2\left(\nabla_x  \wedge  E_{\text{sol}}\right)(t,x) -c^2 \Delta_x \left(\nabla_x  \wedge  E_{\text{sol}}\right)(t,x) =  - \frac{ e}{\epsilon_0} \nabla_x  \wedge  \partial_t j(t,x).
\end{align*}
Using the equation for the current density in \eqref{dt_j}, and $\rho(t,x)=\rho_{\text{ion}} - \frac{\epsilon_0}{e} \nabla_x \cdot E_{\text{irr}}(t,x)$, we get
\begin{align}
\label{sec3:harmonic_sol}
&\left(\partial_{tt}^2+ \left(\frac{e^2 \rho_{\text{ion}}}{m \epsilon_0}-c^2\Delta_x\right) \right)\nabla_x  \wedge  E_{\text{sol}}(t,x)=
- \frac{e}{\epsilon_0} \nabla_x  \wedge \left( \nabla_x \cdot  \int_{\br^3} \frac{c^2\xi \otimes\xi}{(cm)^2+\xi^2}  f(t,x,\xi) d\xi  \right) \nonumber \\
       & \qquad \qquad + \frac{e}{m} \nabla_x  \wedge  (E(t,x) \nabla_x \cdot E_{\text{irr}}(t,x)) 
    - \frac{e^2}{m\epsilon_0} \nabla_x  \wedge  (j(t,x) \wedge B(t,x))- \frac{e^2}{m\epsilon_0} \nabla_x \cdot R(t,x). 
\end{align}
This is a Klein--Gordon equation with dispersion relation given by
\[
\omega^2(k)=\omega^2_{\text{pe}}+ c^2k^2,
\]
which is the frequency of the light waves in an electromagnetic plasma (see, e.g.\cite{book_chen}).\\

Finally, we study the equation for the average of $E(t,x)$: Let us consider the Maxwell--Ampère equation 
\begin{equation*}
\nabla_x \wedge B(t,x)=-e\mu_0j(t,x)+\frac{1}{c^2}\partial_t E(t,x).\end{equation*}
Taking a second time derivative on $E$ and computing the spatial mean, we get
\[
\partial^2_{tt} E_{\text{mean}}(t)=\frac{c^2e\mu_0}{(2\pi)^3}\partial_t\widehat{j}(t,0).
\]
By \eqref{dt_j}, \eqref{sec0:relcorr2} and substituting the expression $\rho(t,x)=\rho_{\text{ion}} - \frac{\epsilon_0}{e} \nabla_x \cdot E_{\text{irr}}(t,x)$, we get
\begin{equation}
\label{sec1:spatialmean}
\Big(\partial^2_{tt}+\frac{e^2\rho_{\text{ion}}}{m\epsilon_0}\Bigr) E_{\text{mean}}(t)= \frac{c^2 e \mu_0}{(2\pi)^3}\Biggl( \frac{\epsilon_0}{m} \int_{\mathbb{T}^3_x}E(t,x)\nabla_x \cdot E_{\text{irr}}(t,x)dx 
    - \frac{e}{m} \int_{\mathbb{T}^3_x}j(t,x) \wedge B(t,x)dx- e \widehat{R}(t,0)\Biggr).
\end{equation}
In this case, the equation doesn't depend on the spatial variable and so we get the dispersion relation $\omega(k)=0$ for $k \neq 0$.

\begin{remark}
    As expected, the frequencies of the dispersion relations obtained in our setting correspond to the eigenvalues of the singular operator studied by Puel and Saint-Raymond \cite{PSR04}, which capture the oscillatory components of the Vlasov-Maxwell system. However, the origin of the oscillations differs: in their case, the oscillations arise from almost monokinetic initial data that are not well-prepared (see \cite[Theorem 2.3]{PSR04}), whereas in our case, they result from the use of non-monokinetic initial data. This distinction is analogous to the Vlasov-Poisson setting, where similar oscillations appear both in Grenier’s work (which parallels ours) and in Masmoudi’s work (which parallels \cite[Theorem 2.3]{PSR04}).
\end{remark}

\subsection{Plan of the paper}
This article is structured as follows. In Section~\ref{Technical_inequalities}, we state two technical lemmas about the treatment of the relativistic terms appearing in the equations. In Section~\ref{sec2}, we construct local-in-time solutions to the Euler--Maxwell system \eqref{sys:EM}, with times of existence and analytic norms that are uniform in~$\varepsilon$, thus proving Theorem~\ref{sec3:mainthm}. Subsections~\ref{sec_esti_rho_w} and~\ref{sec_esti_G_B} provide a priori estimates for the hydrodynamic and electromagnetic quantities, and the iterative scheme used to construct the nonlinear solutions is introduced in Subsection~\ref{sec2:iterativescheme}. 
Section~\ref{sec3} is devoted to the proof of Theorem~\ref{sec1:mainthm2}, which establishes the quasineutral limit and derives the (e-MHD) system~\eqref{sys:limEM}. After introducing a suitable decomposition of the electric field, we prove convergence to the (e-MHD) system in Subsection~\ref{sec3:limitsyst}, and then describe the correctors that appear in the limit in Subsections~\ref{sec3:limitcorrect} and~\ref{sec3:eqcorrect}.
The Appendix is dedicated to the proofs of Lemma~\ref{sec1:lemmarem} and Lemma~\ref{sec2:lemma_sobolev}.

\section{Technical inequalities}
\label{Technical_inequalities}
In this section, we recall and state some useful inequalities for the analytic norm and the relativistic velocity field respectively introduced in \eqref{sec1:analyticnorm}, \eqref{sec1:analyticnormunif} and \eqref{def:v&rho&j}.

The properties of the analytic norms are recalled in the following lemma (see, e.g. \cite[Lemma 2.2.2--2.2.4]{Grenier96} for a proof).
\begin{lemma}
\label{app:lemmarem}
For $\eta>0$, let $f,g: [0,\eta]\times \mathbb{T}_x^3 \to \R$ two given analytic functions, and let $1<\delta<\delta_0-t/\eta$. It holds that:
\begin{equation}
    \label{app:algprop}
    |f(t)g(t)|_\delta \le |f(t)|_\delta |g(t)|_\delta, \quad t \in [0,\eta] \quad \text{and} \quad \|f g\|_{\delta_0}\le\|f\|_{\delta_0}\|g\|_{\delta_0}.
\end{equation}
Moreover, for $i,j \in \{1,2,3\}$,
\begin{equation}
    \label{app:killderiv1}
    \left| \partial_{x_i} f (t)\right|_\delta \le \left(\delta_0 - \delta- \frac{t}{\eta}\right)^{-\beta}\norm{f}_{\delta_0},
\end{equation}
\begin{equation}
    \label{app:killderiv2}
    \left| \partial^2_{x_i x_j} f (t)\right|_\delta \le \delta_0 \left(\delta_0 - \delta- \frac{t}{\eta}\right)^{-\beta-1}\norm{f}_{\delta_0}.
\end{equation}
\end{lemma}

As in \cite[Lemma 5.6]{BHK22}, we need a technical lemma to control the analytic norm of the relativistic velocity field and to treat relativistic corrections as remainders of higher order terms in $\varepsilon$. Recalling that $v(y):=y/\sqrt{1+\varepsilon^2 |y|^2}$,
we state these technical inequalities in the following lemma and refer to the Appendix \ref{appA} for a proof.
\begin{lemma}
\label{sec1:lemmarem}
Let $0<\varepsilon\leq 1$ and $\xi^\varepsilon_\Theta \in B^\eta_{\delta_0}$ for $\Theta \in M$ with
\begin{equation}
\label{lemmarem_assumption}
    \sup_{0 \leq t \leq \eta(\delta_0 - \delta)} \sup_{\Theta\in M } |\xi_\Theta^\varepsilon(t) |_{\delta} \leq \frac{1}{\sqrt{2}\varepsilon}.
\end{equation} 
Then there exists ${C}>0$ such that for all $\Theta \in M$,
\begin{equation} \label{vxiteta}
\norme{v(\xi_\Theta^\varepsilon)}_{\delta_0} \leq {C} \norme{ \xi_\Theta^\varepsilon }_{\delta_0},
\end{equation}
    and by defining $\lambda(y):=\nabla_y\left(v(y)-y\right)$ we have
\begin{align}
    \label{Lemma_remainder}
         \norm{\lambda(\xi^\varepsilon_\Theta)}_{\delta_0} 
        \leq
        C \varepsilon^2   \norm{\xi^\varepsilon_\Theta}_{\delta_0}^2.
    \end{align} 
    Moreover, for two functions $\xi^{\varepsilon,(1)}_\Theta$ and $\xi^{\varepsilon,(2)}_\Theta$ that satisfy assumption \eqref{lemmarem_assumption} and such that
    \begin{align}
    \label{lemmarem_assumption_2}
        \norm{\xi^{\varepsilon,(1)}_\Theta}_{\delta_0} \leq \bar{C}, \quad \text{and} \quad \norm{\xi^{\varepsilon,(2)}_\Theta}_{\delta_0} \leq \bar{C}
    \end{align}
    for some constant $\bar{C}$,
     we have
\begin{align}
    \label{Lemma_remainder_2}
         \norm{ v(\xi^{\varepsilon,(1)}_\Theta)-v(\xi^{\varepsilon,(2)}_\Theta)}_{\delta_0} 
        \leq
         C \norm{\xi^{\varepsilon,(1)}_\Theta-\xi^{\varepsilon,(2)}_\Theta}_{\delta_0},
    \end{align} 
    and 
    \begin{align}
    \label{Lemma_remainder_3}
        \norm{\lambda(\xi^{\varepsilon,(1)}_\Theta) - \lambda(\xi^{\varepsilon,(2)}_\Theta)}_{\delta_0} \leq C \varepsilon^2 \norm{\xi^{\varepsilon,(1)}_\Theta-\xi^{\varepsilon,(2)}_\Theta}_{\delta_0}.
    \end{align}
\end{lemma}
Finally, since we will work in the Sobolev setting in Theorem \ref{sec1:mainthm2}, we state a result for the Sobolev bound of a relativistic correction that will be needed in Section \ref{sec3}. A proof of this lemma can be found in the Appendix \ref{appA}.
\begin{lemma}
\label{sec2:lemma_sobolev}
    Given $0<\varepsilon<1$, $s> \frac32$ and $\Theta \in M$, let $\xi^\varepsilon_\Theta \in L_t^\infty H^s_x$ such that 
    \begin{align}
    \label{sec2:assumption_sobolev}
        \sup_{t, \varepsilon, \Theta} \norm{\xi^\varepsilon_\Theta(t)}_{H^s_x} < \bar C,
    \end{align}
    for some constant $\bar{C}$.
    Then, there exists a constant $C>0$ such that
    \begin{align*}
        \norm{v(\xi^\varepsilon_\Theta) - \xi^\varepsilon_\Theta}_{L_t^\infty H^s_x} \le C \varepsilon^2 \norm{  \xi^\varepsilon_\Theta}^{3}_{L_t^\infty H^s_x}.
    \end{align*}
\end{lemma}

\section{Local-in-time (uniform in \texorpdfstring{$\varepsilon$}{Lg}) solutions to the Euler--Maxwell system}
\label{sec2}
In this section,
we prove Theorem \ref{sec3:mainthm},
constructing local-in-time solutions
to the Euler--Maxwell system \eqref{sys:EM}, having interval of time independent of $\varepsilon$
and which are
uniformly bounded with respect to the parameter $\varepsilon$ 
in the analytic norms defined in \eqref{sec1:analyticnormunif}.

According to Gauss's law,
\[
\varepsilon^2 \nabla_x \cdot E^\varepsilon(t,x) = \rho^\varepsilon(t,x) - 1,
\]
the electric field generated by charge densities exhibits highly oscillatory behavior in $\varepsilon$. To filter out these oscillations, we introduce the new quantities
\begin{equation}
\label{sec3:wdef}
    w^\varepsilon_\Theta(t,x) := \xi^\varepsilon_\Theta(t,x) - G^\varepsilon(t,x), \quad \text{where} \quad G^\varepsilon(t,x) := \int_0^t E^\varepsilon(s,x) \, ds.
\end{equation}
Specifically, we expect convergence toward the limiting system \eqref{sys:limEM} to hold at the level of $w^\varepsilon_\Theta$, rather than $\xi^\varepsilon_\Theta$.

Recalling the notation for the relativistic velocity given by $v(y) := y / \sqrt{1 + \varepsilon^2 |y|^2}$, the Euler equation satisfied by the new unknown $w^\varepsilon_\Theta$ becomes
\begin{equation}
    \label{sec3:weq}
    \partial_t w^\varepsilon_\Theta 
    + v\left(w^\varepsilon_\Theta + G^\varepsilon\right)\cdot \nabla_x \left(w^\varepsilon_\Theta + G^\varepsilon \right) =
     v\left(w^\varepsilon_\Theta + G^\varepsilon \right)
     \wedge B^\varepsilon.
\end{equation}

To construct local-in-time solutions, we divide the a priori estimates into two parts:
\begin{itemize}
    \item Estimates on the hydrodynamic quantities $(\rho^\varepsilon_\Theta, w^\varepsilon_\Theta)$;
    \item Estimates on the electromagnetic quantities $(G^\varepsilon, \varepsilon E^\varepsilon, B^\varepsilon)$.
\end{itemize}

\subsection{A priori estimates on \texorpdfstring{$(\rho^\varepsilon_\Theta, w^\varepsilon_\Theta)$}{Lg}}
\label{sec_esti_rho_w}

\underline{\textbf{Estimates on $w^\varepsilon_\Theta$}}:
We start by obtaining uniform estimates in $\varepsilon$ for the Euler--Maxwell equation in \eqref{sec3:weq}. 
We have 
\begin{align}
\label{sec2:in1}
        \left|w^\varepsilon_\Theta(t)\right|_\delta\le \left| w^\varepsilon_\Theta(0)\right|_\delta + \int_0^t \left| \partial_s w^\varepsilon_\Theta(s)\right|_\delta \, ds.
\end{align}
By the algebra property \eqref{app:algprop}
\begin{equation*}
\left| \partial_s w^\varepsilon_\Theta(s)\right|_\delta \le
\left|v(w^\varepsilon_\Theta+G^\varepsilon) \right|_\delta 
\left|\nabla_x \left(w^\varepsilon_\Theta+G^\varepsilon\right) \right|_\delta
+
\left|v(w^\varepsilon_\Theta+G^\varepsilon) \right|_\delta 
\left| B^\varepsilon \right|_\delta
,
\end{equation*}
we then use \eqref{app:killderiv1} to bound $\left|\nabla_x \left(w^\varepsilon_\Theta+G^\varepsilon\right) \right|_\delta$ and \eqref{vxiteta} in Lemma \ref{sec1:lemmarem} for the relativistic term $\left|v(w^\varepsilon_\Theta+G^\varepsilon) \right|_\delta $, obtaining
\begin{equation*}
\left| \partial_s w^\varepsilon_\Theta(s)\right|_\delta \le
C\left(\delta_0-\delta - \frac{s}{\eta}\right)^{-\beta}
\norm{w^\varepsilon_\Theta+G^\varepsilon}^2_{\delta_0}
+
C\norm{w^\varepsilon_\Theta+G^\varepsilon}_{\delta_0} 
\left| B^\varepsilon(s) \right|_\delta.
\end{equation*}
By inequality \eqref{sec2:in1}, we arrive at
\begin{equation*}
\begin{aligned}
    |w^\varepsilon_\Theta&(t)|_\delta
    \le \| w^\varepsilon_\Theta(0)\|_{\delta_0} \\
    &+ C \left(\norm{w^\varepsilon_\Theta}_{\delta_0}
    +\norm{G^\varepsilon}_{\delta_0}\right)^2\int_0^t 
    \left(\delta_0-\delta - \frac{s}{\eta}\right)^{-\beta}
    ds+
    C \eta
    \left(
    \norm{w^\varepsilon_\Theta}_{\delta_0}
    +\norm{G^\varepsilon}_{\delta_0}
    \right)\norm{B^\varepsilon}_{\delta_0}
\end{aligned}
\end{equation*}
and using that, for $\beta \in (0,1)$, 
\begin{equation}\label{beta}
\int_0^t \left(\delta_0 - \delta - \dfrac{s}{\eta}\right)^{- \beta} ds
= \eta \left[-\frac{1}{1-\beta}\left(\delta_0-\delta-\frac{s}{\eta}\right)^{1-\beta}\right]^{t}_0 
\leq \dfrac{\eta}{1-\beta}\delta_0^{1-\beta},
\end{equation}
we conclude
\begin{equation}
\label{sec2:final1}
    \begin{aligned}
      |w^\varepsilon_\Theta&(t)|_\delta \le
    \| w^\varepsilon_\Theta(0)\|_{\delta_0}
    +C \eta 
    \left(\norm{w^\varepsilon_\Theta}_{\delta_0}
    +\norm{G^\varepsilon}_{\delta_0}\right)
    \left(
\norm{w^\varepsilon_\Theta}_{\delta_0}
    +\norm{G^\varepsilon}_{\delta_0} + \norm{B^\varepsilon}_{\delta_0}
    \right),
\end{aligned}
\end{equation}
where $C$ is a generic constant depending on $\delta_0$ and $\beta$.

We now bound the quantities $\left|\partial_{x_i} w^\varepsilon_\Theta\right|_\delta$, $i \in \{1,2,3\}$. In this case, we have
\begin{align*}
        \left|\partial_{x_i}w^\varepsilon_\Theta(t)\right|_\delta
        \le \left|\partial_{x_i}w^\varepsilon_\Theta(0)\right|_\delta 
        + \int_0^t \left| \partial_s \partial_{x_i} w^\varepsilon_\Theta(s)\right|_\delta \, ds.
\end{align*}
By \eqref{sec3:weq}, we have
\begin{align*}
 \partial_s \partial_{x_i} w^\varepsilon_\Theta(s)&=-\partial_{x_i}[v(w^\varepsilon_\Theta+G^\varepsilon)]\cdot \nabla_x(w^\varepsilon_\Theta+G^\varepsilon)-
 v(w^\varepsilon_\Theta+G^\varepsilon)\cdot \partial_{x_i}\nabla_x(w^\varepsilon_\Theta+G^\varepsilon)\\
 &\quad + \partial_{x_i}[v(w^\varepsilon_\Theta+G^\varepsilon)]\wedge B^\varepsilon+ 
 v(w^\varepsilon_\Theta+G^\varepsilon)\wedge \partial_{x_i}B^\varepsilon.
\end{align*}
By the algebra property \eqref{app:algprop} and the inequalities \eqref{app:killderiv1} and \eqref{app:killderiv2} in Lemma \ref{sec1:lemmarem} for the terms with derivatives, we obtain
\begin{align*}
\left| \partial_s \partial_{x_i}w^\varepsilon_\Theta(s)\right|_\delta 
&\le \left[
\left(\delta_0-\delta - \frac{s}{\eta}\right)^{-2\beta}
+
\delta_0
\left(\delta_0-\delta - \frac{s}{\eta}\right)^{-\beta-1}\right]
\left\|v(w^\varepsilon_\Theta+G^\varepsilon) \right\|_{\delta_0} 
\left\|w^\varepsilon_\Theta+G^\varepsilon \right\|_{\delta_0}\\
&\quad +
2\left(\delta_0-\delta-\frac{s}{\eta} \right)^{-\beta}\norm{v(w^\varepsilon_\Theta+G^\varepsilon)}_{\delta_0} 
\norm{B^\varepsilon}_{\delta_0}.
\end{align*}
By treating the relativistic terms using \eqref{vxiteta} in Lemma \ref{sec1:lemmarem}, we conclude
\begin{equation}
\begin{aligned}
\label{sec3:est2}
    |\partial_{x_i}&w^\varepsilon_\Theta(t)|_\delta
    \le | \partial_{x_i}w^\varepsilon_\Theta(0)|_{\delta} \\
    &\quad + C   \left(\norm{w^\varepsilon_\Theta}_{\delta_0}
     +\norm{G^\varepsilon}_{\delta_0}\right)^2 \int_0^t\left[ 
     \left(\delta_0-\delta - \frac{s}{\eta}\right)^{-2\beta}
     +
    \left(\delta_0-\delta - \frac{s}{\eta}\right)^{-\beta-1}
    \right]
 ds\\
    &\quad
    +C\left(\norm{w^\varepsilon_\Theta}_{\delta_0}
    +\norm{G^\varepsilon}_{\delta_0}\right) \norm{B^\varepsilon}_{\delta_0}\int_0^t 
     \left(\delta_0-\delta - \frac{s}{\eta}\right)^{-\beta}
     ds.
     \end{aligned}
     \end{equation}
Multiplying \eqref{sec3:est2} by $\left(\delta_0-\delta -t/\eta\right)^\beta$, using that $\left(\delta_0-\delta -t/\eta\right)^\beta\le \left(\delta_0-\delta -s/\eta\right)^\beta \le \delta_0^\beta$ for $s\le t$, we get
 \begin{equation*}
         \begin{aligned}
             &\left(\delta_0-\delta -\frac{t}{\eta}\right)^\beta |\partial_{x_i}w^\varepsilon_\Theta(t)|_\delta
             \le \delta_0^\beta| \partial_{x_i}w^\varepsilon_\Theta(0)|_{\delta}\\
   &\quad + C   \left(\norm{w^\varepsilon_\Theta}_{\delta_0}
     +\norm{G^\varepsilon}_{\delta_0}\right)^2 \left[ \int_0^t
     \left(\delta_0-\delta - \frac{s}{\eta}\right)^{-\beta}ds
     +
    \int_0^t\frac{\left(\delta_0-\delta - \frac{t}{\eta}\right)^{\beta}}{\left(\delta_0-\delta - \frac{s}{\eta}\right)^{\beta+1}}
    ds\right]
 \\
    &\quad
    +C\left(\norm{w^\varepsilon_\Theta}_{\delta_0}
    +\norm{G^\varepsilon}_{\delta_0}\right) \norm{B^\varepsilon}_{\delta_0}\int_0^t 
     ds.
\end{aligned}
\end{equation*}
Using definition \eqref{sec1:analyticnorm_IC} for the norm of the initial condition, the value of the time integral in \eqref{beta} and that, for $\beta \in (0,1)$,
\begin{equation}\label{beta+1}
\int_0^t \left(\delta_0 - \delta - \dfrac{s}{\eta}\right)^{-\beta-1}ds 
= \eta \left[\frac{1}{\beta}\left(\delta_0-\delta- \frac{s}{\eta}\right)^{-\beta}\right]^{t}_0 
\leq \dfrac{\eta}{\beta}\left(\delta_0-\delta-\frac{t}{\eta}\right)^{-\beta},
\end{equation}
we conclude that
     \begin{equation}
     \label{sec2:final2}
         \begin{aligned}
             \left(\delta_0 - \delta - \frac{t}{\eta}\right)^{\beta}|\partial_{x_i}w^\varepsilon_\Theta(t)|_\delta &\le \|w^\varepsilon_\Theta(0)\|_{\delta_0}\\
   &+C \eta 
    \left(\norm{w^\varepsilon_\Theta}_{\delta_0}
    +\norm{G^\varepsilon}_{\delta_0}\right)
    \left(\norm{w^\varepsilon_\Theta}_{\delta_0}+\norm{G^\varepsilon}_{\delta_0}+ \norm{B^\varepsilon}_{\delta_0}
    \right).
\end{aligned}
\end{equation}
By collecting \eqref{sec2:final1} and \eqref{sec2:final2}, there exists a constant $C$ depending on $\delta_0$ and $\beta$ such that
\begin{equation}
\label{sec3:est_w}
    \norm{w^\varepsilon_\Theta}_{\delta_0} \le 
    \norm{w^\varepsilon_\Theta(0)}_{\delta_0} +C \eta \left(\norm{w^\varepsilon_\Theta}_{\delta_0}+\norm{G^\varepsilon}_{\delta_0} \right)
    \left(
\norm{w^\varepsilon_\Theta}_{\delta_0}+\norm{G^\varepsilon}_{\delta_0}
+ \norm{B^\varepsilon}_{\delta_0}
    \right).
\end{equation}
\underline{\textbf{Estimates on $\rho^\varepsilon_\Theta$}}: We now focus on
the continuity equation 
$$ \partial _t \rho^\varepsilon_\Theta + \nabla_x \cdot (\rho^\varepsilon_\Theta v\left(w^\varepsilon_\Theta+ G^\varepsilon\right))=0.
$$
We proceed as in the previous estimates and we get:
\begin{equation}
\label{sec3:est3}
    |\rho^\varepsilon_\Theta(t)|_\delta \le 
    \norm{\rho^\varepsilon_\Theta(0)}_{\delta_0}+ C \eta \norm{\rho^\varepsilon_\Theta}_{\delta_0}\left(\norm{w^\varepsilon_\Theta}_{\delta_0}+\norm{G^\varepsilon}_{\delta_0} \right)
\end{equation}
and, for $i \in \{1, 2, 3\}$,
\begin{equation}
\label{sec3:est4}
     \left(\delta_0-\delta-\frac{t}{\eta}\right)^{\beta} |\partial_{x_i}\rho^\varepsilon_\Theta(t)|_\delta \le 
    \norm{\rho^\varepsilon_\Theta(0)}_{\delta_0}+ C  \eta 
\norm{\rho^\varepsilon_\Theta}_{\delta_0}\left(\norm{w^\varepsilon_\Theta}_{\delta_0}+\norm{G^\varepsilon}_{\delta_0} \right).
\end{equation}
Collecting \eqref{sec3:est3} and \eqref{sec3:est4}, we conclude that there exists a constant $C$ depending on $\delta_0$ and $\beta$ such that
\begin{equation}
\label{sec3_est_rho}
    \norm{\rho^\varepsilon_\Theta}_{\delta_0}
    \le
    \norm{\rho^\varepsilon_\Theta(0)}_{\delta_0}
    +
    C \eta \norm{\rho^\varepsilon_\Theta}_{\delta_0}\left(\norm{w^\varepsilon_\Theta}_{\delta_0}+\norm{G^\varepsilon}_{\delta_0} \right).
\end{equation}

\subsection{A priori estimates on \texorpdfstring{$(G^\varepsilon, \varepsilon E^\varepsilon, B^\varepsilon)$}{Lg}}
\label{sec_esti_G_B}
We now give a priori estimates uniform in $\varepsilon$ on the quantities $(G^\varepsilon, B^\varepsilon)$, where we recall that
$$G^\varepsilon(t,x):=\int_0^t E^\varepsilon(s,x) ds.$$
To do this, recall the decomposition \eqref{sec0:decomposition} of the electric field $E^\varepsilon$ into 
three parts, one that is irrotational, another that is solenoidal and a third one given by the spatial mean of the electric field:
\[
E^\varepsilon(t,x)=E^\varepsilon_{\text{irr}}(t,x) + E^\varepsilon_{\text{sol}}(t,x) + E^\varepsilon_{\text{mean}}(t).
\]
We will give a priori estimates on the three quantities $E^\varepsilon_{\text{irr}}$, $E^\varepsilon_{\text{sol}}$ and $E^\varepsilon_{\text{mean}}$ by using the three equations 
for the irrotational and solenoidal parts  and for the spatial mean of $E^\varepsilon$ obtained in \eqref{sec3:harmonic_irr}, \eqref{sec3:harmonic_sol} and \eqref{sec1:spatialmean}.
\vskip 0.5 cm
\noindent\underline{\textbf{Wave equation for $\nabla_x \cdot E^\varepsilon_{\text{irr}}$}}: In the quasineutral scaling introduced in \eqref{sys:VMquasi}, equation \eqref{sec3:harmonic_irr} reads as 
\begin{equation}
\label{sec3:harmonic}
        \left(\varepsilon^2\partial^2_{tt} + \text{\rm Id}\right) 
        \nabla_x \cdot E_{\text{irr}}^\varepsilon(t,x)
       = g^\varepsilon(t,x),
    \end{equation}
where (using the Einstein notation for repeated indices)
    \begin{equation}
    \begin{aligned}
    \label{sec3:gepsilon}
    g^\varepsilon (t,x) 
    &:= \partial_{x_j} \partial_{x_i}
        \int_{\br^3} f^\varepsilon(t,x,\xi) v(\xi)_i v(\xi)_j d \xi
        - \varepsilon^2 \nabla_x \cdot \big(E^\varepsilon(t,x) \nabla_x \cdot E_{\text{irr}}^\varepsilon(t,x)\big)  \\
    &\quad -\nabla_x \cdot(j^\varepsilon(t,x) \wedge B^\varepsilon(t,x))
    -\nabla_x \cdot R^\varepsilon(t,x),
    \end{aligned}
    \end{equation}
and 
\begin{align}
    \label{sec3:defR}
        R^\varepsilon(t,x) : = \int_{\br^3} \nabla_\xi (v(\xi) - \xi)
        (E^\varepsilon(t,x) + v(\xi)\wedge B^\varepsilon(t,x))
         f^\varepsilon(t,x,\xi) d\xi.
    \end{align}
    Taking the Fourier transform in \eqref{sec3:harmonic}, 
    we get the following forced harmonic oscillator
    \begin{align*}
    \left(\varepsilon^2\partial^2_{tt} + \text{\rm Id}\right) 
        \imm k \cdot \widehat{E_{\text{irr}}^\varepsilon}(t,k) = \widehat{g^\varepsilon} (t,k).
    \end{align*}
    We solve this second order ODE obtaining, for $k \neq 0$,
    \begin{align}
    \label{Fourier_divE_irr}
	\widehat{E_{\text{irr}}^\varepsilon}(t,k) 
	= -\int_{0}^{t} \frac{\imm k}{\varepsilon \abs{k}^2} \sin\left(\frac{t-s}{\varepsilon} \right)
	\widehat{g^\varepsilon}(s,k)ds+\widehat{E_{\text{irr},0}^\varepsilon} (t,k),    
    \end{align}
    with
	\begin{equation} 
    \label{sec2:initialirr}
	\widehat{E_{\text{irr},0}^\varepsilon}(t,k)
	:=\widehat{E_{\text{irr}}^\varepsilon}(0,k)\cos\left(\frac{t}{\varepsilon}\right)
	+	\varepsilon\widehat{\partial_t E_{\text{irr}}^\varepsilon}(0,k)
	\sin\left(\frac{t}{\varepsilon}\right).
	\end{equation}
    Here, $E^\varepsilon_{\text{irr}}(0)$ and $\partial_tE^\varepsilon_{\text{irr}}(0)$ are given by
    \begin{equation}
        \label{sec2:irrtime0}
        \varepsilon^2 E^\varepsilon_{\text{irr}}(0,x):=\nabla_x \left(\Delta_x^{-1}\left(\rho^\varepsilon(0,x)-1 \right)\right), \quad 
         \varepsilon^2 \partial_tE^\varepsilon_{\text{irr}}(0,x):=-\nabla_x\left( \Delta_x^{-1}\left(\nabla_x \cdot j^\varepsilon(0,x)\right)\right).
    \end{equation}

\vskip 0.5 cm
\noindent \underline{\textbf{Wave equation for $\nabla_x  \wedge  E^\varepsilon_{\text{sol}}$}}: 
    In the quasineutral scaling introduced in \eqref{sys:VMquasi}, equation \eqref{sec3:harmonic_sol} reads as
    \begin{equation}
    \label{sec3:wave}
        \left(\varepsilon^2\partial^2_{tt} + (\text{\rm Id}-\Delta_x)\right) 
        \nabla_x  \wedge  E_{\text{sol}}^\varepsilon(t,x)
       = h^\varepsilon(t,x),
    \end{equation}
    where (using again the Einstein notation for repeated indices)
    \begin{equation}
    \begin{aligned}
        \label{sec3:hepsilon}
    h^\varepsilon (t,x) 
    &:= \nabla_x  \wedge \left( \partial_{x_i}
        \int_{\br^3} f^\varepsilon(t,x,\xi) v(\xi)_i v(\xi) d \xi\right)
        - \varepsilon^2 \nabla_x  \wedge  \big(E^\varepsilon(t,x) \nabla_x \cdot E_{\text{irr}}^\varepsilon(t,x)\big)  \\
    &\quad -\nabla_x  \wedge (j^\varepsilon(t,x) \wedge B^\varepsilon(t,x))
    -\nabla_x  \wedge  R^\varepsilon(t,x).
    \end{aligned}
    \end{equation}
    Taking the Fourier transform in \eqref{sec3:wave}, 
    we get 
    \begin{align*}
    \left(\varepsilon^2\partial^2_{tt} + \left(1+|k|^2\right)\right) 
        \imm k \wedge \widehat{E_{\text{sol}}^\varepsilon}(t,k) = \widehat{h^\varepsilon} (t,k).
    \end{align*}
    We solve this, for $k \neq 0$,
    \begin{align}
    \label{Fourier_E_sol}
	\widehat{E_{\text{sol}}^\varepsilon}(t,k) 
	= \int_{0}^{t} \frac{1}{\varepsilon |k|^2\sqrt{1+|k|^2}} \sin\left(\frac{\sqrt{1+|k|^2}(t-s)}{\varepsilon} \right)
	\left(\imm k\wedge\widehat{h^\varepsilon}(s,k)\right)ds+\widehat{E_{\text{sol},0}^\varepsilon} (t,k),    
    \end{align}
    with
	\begin{equation}
    \label{sec2:initialsol}
	\widehat{E_{\text{sol},0}^\varepsilon}(t,k)
	:=\widehat{E_{\text{sol}}^\varepsilon}(0,k)\cos\left(\frac{t\sqrt{1+|k|^2}}{\varepsilon}\right)
	+\frac{\varepsilon}{\sqrt{1+|k|^2}}\widehat{\partial_t E_{\text{sol}}^\varepsilon}(0,k)
	\sin\left(\frac{t \sqrt{1+|k|^2}}{\varepsilon}\right).
	\end{equation}
    Here, $E^\varepsilon_{\text{sol}}(0)$ and $\partial_tE^\varepsilon_{\text{sol}}(0)$ are given by 
    \begin{equation}
        \label{sec2:soltime0}
        \begin{aligned}
        E^\varepsilon_{\text{sol}}(0,x) 
        &:=\nabla_x  \wedge  \Delta_x^{-1}\left(\partial_tB^\varepsilon(0,x)\right), \\
         \varepsilon^2 \partial_tE^\varepsilon_{\text{sol}}(0,x)
         &:=\nabla_x  \wedge B^\varepsilon(0,x)
         - j^\varepsilon(0,x)
         + \nabla_x\left( \Delta_x^{-1}\left(\nabla_x \cdot j^\varepsilon(0,x)\right)\right).
         \end{aligned}
    \end{equation}

\vskip 0.5 cm
\noindent\underline{\textbf{Wave equation for $E^\varepsilon_{\text{mean}}$}}: In the quasineutral scaling in \eqref{sys:VMquasi}, equation \eqref{sec1:spatialmean} reads as 
\begin{equation}
\label{sec3:spatialmean}
        \left(\varepsilon^2\partial^2_{tt} + \text{\rm Id}\right) 
        E_{\text{mean}}^\varepsilon(t)
       = q^\varepsilon(t),
    \end{equation}
where
    \begin{equation}
    \begin{aligned}
    \label{sec3:qepsilon}
    q^\varepsilon (t) 
    &:= \frac{1}{(2\pi)^3}
   \int_{\mathbb{T}^3_x}\varepsilon E^\varepsilon(t,x)\nabla_x \cdot (\varepsilon E^\varepsilon_{\text{irr}})(t,x)dx 
    - \frac{1}{(2\pi)^3}\int_{\mathbb{T}^3_x}j^\varepsilon(t,x) \wedge B^\varepsilon(t,x)dx- \frac{1}{(2\pi)^3} \int_{\mathbb{T}^3_x} R^\varepsilon(t,x) dx,
    \end{aligned}
    \end{equation}
where $R^\varepsilon$ is defined by \eqref{sec3:defR}.
    We solve this forced harmonic oscillator obtaining, 
    \begin{align}
    \label{E_mean}
	E_{\text{mean}}^\varepsilon(t) 
	= \frac{1}{\varepsilon}\int_{0}^{t}  \sin\left(\frac{t-s}{\varepsilon} \right)
	q^\varepsilon(s)ds+E_{\text{mean},0}^\varepsilon (t),    
    \end{align}
    with
	\begin{equation} 
    \label{sec2:initialmean}
	E_{\text{mean},0}^\varepsilon(t)
	:=E_{\text{mean}}^\varepsilon(0)\cos\left(\frac{t}{\varepsilon}\right)
	+	\varepsilon \partial_t E_{\text{mean}}^\varepsilon(0)
	\sin\left(\frac{t}{\varepsilon}\right),
	\end{equation}
    and where 
    \begin{equation}
        \label{sec2:meantime0}
          E^\varepsilon_{\text{mean}}(0):=\frac{1}{(2\pi)^{3}}\int_{\mathbb{T}^3} E^\varepsilon_0(x) dx \quad \text{and} \quad \partial_t E^\varepsilon_{\text{mean}}(0):=\frac{1}{(2\pi)^{3}}\int_{\mathbb{T}^3}\partial_t E^\varepsilon_0(x) dx.
    \end{equation}

\subsubsection{Estimates on \texorpdfstring{$G_{\text{irr}}^\varepsilon$}{Lg} and \texorpdfstring{$\varepsilon E^\varepsilon_{\text{irr}}$}{Lg}}  
\underline{\textbf{Estimates on \texorpdfstring{$G_{\text{irr}}^\varepsilon$}{Lg}}}:
  We want to bound 
	\begin{align*}
	\widehat{G_{\text{irr}}^\varepsilon}(t,k)
	& = \int_0^t \widehat{E_{\text{irr}}^\varepsilon}(s,k) ds.
    \end{align*}
Using formula \eqref{Fourier_divE_irr}, we have 
    \begin{align*}
	\widehat{G_{\text{irr}}^\varepsilon}(t,k)  &= -\int_0^t \int_{0}^{s} \frac{\imm k}{\varepsilon \abs{k}^2} \sin\left(\frac{s-\tau}{\varepsilon} \right)
	\widehat{g^\varepsilon}(\tau,k)d\tau ds
	+ \widehat{G^\varepsilon_{\text{irr},0}}(t,k),
	\end{align*}
	where 
    \begin{equation}
    \label{Girr0}
        \begin{aligned}
            \widehat{G^\varepsilon_{\text{irr},0}}(t,k) = \int_0^t \widehat{E^\varepsilon_{\text{irr},0}}(s,k) ds=\varepsilon\widehat{E_{\text{irr}}^\varepsilon}(0,k)\sin\left(\frac{t}{\varepsilon}\right)
	+	\varepsilon^2\widehat{\partial_t E_{\text{irr}}^\varepsilon}(0,k)
	\left(1-\cos\left(\frac{t}{\varepsilon}\right)\right).
        \end{aligned}
    \end{equation}
    Then, by Fubini on the double integral and recalling that $\int_0^t \int_{0}^{s} d\tau ds = \int_0^t \int_{\tau}^{t}  ds d\tau$, we get
     \begin{align}
     \label{Fourier_G_irr}
	\widehat{G_{\text{irr}}^\varepsilon}(t,k)
	&= -\int_0^t \frac{\imm k}{\varepsilon \abs{k}^2} \widehat{g^\varepsilon}(\tau,k) 
	\int_{\tau}^{t}  \sin\left(\frac{s-\tau}{\varepsilon} \right) ds d\tau
	+ \widehat{G^\varepsilon_{\text{irr},0}}(t,k) \nonumber \\
	& =-\int_0^t \frac{\imm k}{ \abs{k}^2} \widehat{g^\varepsilon}(\tau,k) 
	\left( 1 - \cos \left(\frac{t-\tau}{\varepsilon} \right) \right) d\tau
	+ \widehat{G^\varepsilon_{\text{irr},0}}(t,k)  \nonumber \\
    & =: \widehat{I_1}(t,k) + \widehat{I_2}(t,k) + \widehat{I_3}(t,k) + \widehat{I_4}(t,k)
    + \widehat{G^\varepsilon_{\text{irr},0}}(t,k).
	\end{align}
By recalling the formula \eqref{sec3:gepsilon} for $g^\varepsilon$, $\{\widehat{I_\ell}(t,k)\}_{\ell=1}^4$ are given by
\begin{align*}
    \widehat{I_1}(t,k) := -\int_0^t  \frac{\imm k}{ \abs{k}^2}
    \mathcal{F} \left( 
    \partial_{x_j} \partial_{x_i}
        \int_{\br^3} f^\varepsilon(\tau,x,\xi) v(\xi)_i v(\xi)_j d \xi
    \right) (\tau ,k)
	\left( 1 - \cos \left(\frac{t-\tau}{\varepsilon} \right) \right) d\tau,
\end{align*}
\begin{align*}
    \widehat{I_2}(t,k) := \int_0^t  \frac{\imm k}{ \abs{k}^2}
    \mathcal{F} \left( 
    \varepsilon^2 \nabla_x \cdot \big(E^\varepsilon(\tau,x) \nabla_x \cdot E_{\text{irr}}^\varepsilon(\tau,x)\big)
    \right) (\tau ,k)
	\left( 1 - \cos \left(\frac{t-\tau}{\varepsilon} \right) \right)  d\tau,
\end{align*}
\begin{align*}
    \widehat{I_3}(t,k) :=\int_0^t  \frac{\imm k}{ \abs{k}^2}
    \mathcal{F} \left( 
    \nabla_x \cdot(j^\varepsilon(\tau,x) \wedge B^\varepsilon(\tau,x))
    \right) (\tau ,k)
	\left( 1 - \cos \left(\frac{t-\tau}{\varepsilon} \right) \right)  d\tau,
\end{align*}
and 
\begin{align*}
    \widehat{I_4}(t,k) :=\int_0^t  \frac{\imm k}{ \abs{k}^2}
    \mathcal{F} \left( 
    \nabla_x \cdot R^\varepsilon(\tau,x)
    \right) (\tau ,k)
	\left( 1 - \cos \left(\frac{t-\tau}{\varepsilon} \right) \right)  d\tau.
\end{align*}
In the following, by the definition of analytic norms in \eqref{sec1:analyticnorm} and \eqref{sec1:analyticnormunif}, in order to estimate $\norm{G_{\text{irr}}^\varepsilon}_{\delta_0}$, we compute $\abs{G_{\text{irr}}^\varepsilon}_{\delta}$ and $\abs{\partial_{x_j} G_{\text{irr}}^\varepsilon}_{\delta}$.
\vskip 0.5 cm
\noindent
\underline{\emph{Estimate on $I_1$}}:
We start by estimating $ I_1=\mathcal{F}^{-1}\big(\big\{\widehat{I_1}(k)\big\}_{k \in \mathbb Z^3}\big)$.
	First, by using the bound
	\begin{align}
    \label{bound_cos}
	\abs{\frac{k_j k}{ \abs{k}^2} \left( 1 - \cos \left(\frac{t-\tau}{\varepsilon} \right) \right)} \leq 2,
	\end{align}
	we get, 
	\begin{align}
	\abs{I_1(t)}_\delta 
    & \leq  \int_0^t \abs{\mathcal{F}^{-1} \Bigg(\Bigg\{
	\frac{k_j k}{ \abs{k}^2} 
	\mathcal{F} \left(  
	 \partial_{x_i}  \int_{\br^3} f^\varepsilon(\tau,x,\xi) v(\xi)_i v(\xi)_j d \xi  
	 \right)
     \left( 1 - \cos \left(\frac{t-\tau}{\varepsilon} \right) \right)
     \Bigg\}_{k \in \Z^3}\Bigg) }_\delta d \tau \nonumber \\ 
	&  \leq 2 \int_0^t \abs{  \partial_{x_i}
       \int_M \rho_\Theta^\varepsilon(\tau)  v(\xi_\Theta^\varepsilon(\tau))_i 
       v(\xi_\Theta^\varepsilon(\tau))_j \mu(d\Theta) }_\delta d \tau, \nonumber
    \end{align}
    where we used expression \eqref{typesoluGrenier} for $f^\varepsilon(t,x,\xi)$. Then by using \eqref{app:killderiv1} to handle the derivative, the algebra property \eqref{app:algprop} and \eqref{vxiteta} to bound the relativistic velocity, we have
    \begin{align}
    \label{esti_G_irr_1}
    \abs{I_1(t)}_\delta
       & \leq 2 \int_0^t \left(\delta_0 - \delta -\frac{\tau}{\eta} \right)^{-\beta} \norm{ 
       \int_M \rho_\Theta^\varepsilon  v(\xi_\Theta^\varepsilon)_i 
       v(\xi_\Theta^\varepsilon)_j \mu(d\Theta) }_{\delta_0} d \tau \nonumber \\
       & \leq 2  \sup_\Theta \left( \norm{ 
       \rho_\Theta^\varepsilon}_{\delta_0}  \norm{v(\xi_\Theta^\varepsilon)}_{\delta_0}^2 \right) \int_0^t \left(\delta_0 - \delta -\frac{\tau}{\eta} \right)^{-\beta} d \tau \nonumber  \\
       & \leq  C \eta  \sup_\Theta \left( \norm{ 
       \rho_\Theta^\varepsilon}_{\delta_0}  \norm{\xi_\Theta^\varepsilon}_{\delta_0}^2 \right),
	\end{align}
    where we used \eqref{beta} to estimate the time-integral.
    Next, we estimate $\abs{\partial_{x_\ell} I_1(t)}_{\delta}$ for $\ell \in \{1,2,3\}$. We use the bound \eqref{bound_cos} and the expression \eqref{typesoluGrenier} for $f^\varepsilon(t,x,\xi)$ to get 
    \begin{align}
    \label{esti_d_G_irr_1}
	\abs{\partial_{x_\ell} I_1(t)}_\delta 
    & \leq  \int_0^t \abs{\mathcal{F}^{-1} \Bigg(\Bigg\{ 
	\frac{ k_j k}{ \abs{k}^2} 
	\mathcal{F} \left(  
	\partial_{x_\ell}  \partial_{x_i}  \int_{\br^3} f^\varepsilon(t,x,\xi) v(\xi)_i v(\xi)_j d \xi  
	 \right)
     \left( 1 - \cos \left(\frac{t-\tau}{\varepsilon} \right) \right)\Bigg\}_{k \in \Z^3}
     \Bigg) }_\delta d \tau \nonumber \\
	&  \leq 2 \int_0^t \abs{ \partial_{x_\ell}  \partial_{x_i}
       \int_M \rho_\Theta^\varepsilon(\tau)  v(\xi_\Theta^\varepsilon(\tau))_i 
       v(\xi_\Theta^\varepsilon(\tau))_j \mu(d\Theta) }_\delta d \tau \nonumber \\
    & \leq C \norm{ 
       \int_M \rho_\Theta^\varepsilon  v(\xi_\Theta^\varepsilon)_i 
       v(\xi_\Theta^\varepsilon)_j \mu(d\Theta) }_{\delta_0} \int_{0}^{t} \left(\delta_0 - \delta -\frac{\tau}{\eta} \right)^{-(\beta+1)} d\tau \nonumber \\
       & \leq C \eta \left(\delta_0-\delta-\frac{t}{\eta}\right)^{-\beta}  \sup_\Theta \left( \norm{ 
       \rho_\Theta^\varepsilon}_{\delta_0}  \norm{\xi_\Theta^\varepsilon}_{\delta_0}^2 \right),
	\end{align}
where we used \eqref{app:killderiv2} to handle the two derivatives, the algebra property \eqref{app:algprop}, the relativistic bound in \eqref{vxiteta}, and \eqref{beta+1} for the integral in time.
Thus, multiplying \eqref{esti_d_G_irr_1} by $(\delta_0-\delta-t/\eta)^\beta$ and summing it to \eqref{esti_G_irr_1}, we have 
\begin{align}
\label{esti_norm_G_irr_1}
    \norm{I_1}_{\delta_0} \leq  C \eta \sup_\Theta \left( \norm{ 
       \rho_\Theta^\varepsilon}_{\delta_0}  \norm{\xi_\Theta^\varepsilon}_{\delta_0}^2 \right).
\end{align}
    \vskip 0.5 cm

\noindent\underline{\emph{Estimate on $I_2$}}:
We now estimate $I_2=\mathcal{F}^{-1}\big(\big\{\widehat{I_2}(k)\big\}_{k \in \mathbb Z^3}\big)$. As before we use \eqref{bound_cos}, therefore having
	\begin{align}
    \abs{I_2(t)}_\delta &\leq \int_0^t \abs{\mathcal{F}^{-1} \Bigg(\Bigg\{ \frac{k}{ \abs{k}^2}
    k \cdot \mathcal{F} \left( 
    \varepsilon^2 \big(E^\varepsilon(\tau) \nabla_x \cdot E_{\text{irr}}^\varepsilon(\tau)\big)
    \right) 
	\left( 1 - \cos \left(\frac{t-\tau}{\varepsilon} \right) \right) \Bigg\}_{k \in \Z^3}\Bigg)}_\delta d\tau \nonumber \\
	   &\leq 2 \int_0^t \abs{ \varepsilon^2
	      E^\varepsilon(\tau) \nabla_x \cdot E_{\text{irr}}^\varepsilon(\tau)  
	   }_\delta  d \tau. \nonumber
    \end{align}
    Then using the algebra property \eqref{app:algprop}, inequality \eqref{app:killderiv1} to handle the divergence operator, we obtain
    \begin{align}
    \label{esti_G_irr_2}
	\abs{I_2(t)}_\delta &\leq 2 \int_0^t \left(\abs{ \varepsilon  E_{\text{irr}}^\varepsilon(\tau)}_\delta \abs{\varepsilon \nabla_x \cdot E_{\text{irr}}^\varepsilon(\tau)  
	   }_\delta +  \abs{ \varepsilon  E_{\text{sol}}^\varepsilon(\tau)}_\delta \abs{\varepsilon \nabla_x \cdot E_{\text{irr}}^\varepsilon(\tau)  
	   }_\delta \right)d \tau  \nonumber \\
	   &\leq 2 
	       \norm{\varepsilon E^\varepsilon}_{\delta_0} \norm{\varepsilon E_{\text{irr}}^\varepsilon}_{\delta_0} \int_0^t \left(\delta_0 - \delta -\frac{\tau}{\eta} \right)^{-\beta} d \tau \nonumber \\
    & \leq C \eta 
	      \norm{\varepsilon E^\varepsilon}_{\delta_0} \norm{\varepsilon E_{\text{irr}}^\varepsilon}_{\delta_0},
	\end{align}
	where we used \eqref{beta} to estimate the integral in time.

    Next, we compute $\abs{\partial_{x_\ell} I_2(t)}_{\delta}$ for $\ell \in \{ 1,2,3 \}$. As before we use \eqref{bound_cos}, therefore, we have
	\begin{align}
    \abs{\partial_{x_\ell} I_2(t)}_\delta 
    &\leq \int_0^t \abs{\mathcal{F}^{-1} \Bigg(\Bigg\{ \frac{k}{ \abs{k}^2} k \cdot
    \mathcal{F} \left( 
    \varepsilon^2 \partial_{x_\ell}  \big(E^\varepsilon(\tau) \nabla_x \cdot E_{\text{irr}}^\varepsilon(\tau)\big)
    \right) 
	\left( 1 - \cos \left(\frac{t-\tau}{\varepsilon} \right) \right) \Bigg\}_{k \in \Z^3}\Bigg)}_\delta d\tau \nonumber \\
	&\leq 2 \int_0^t \abs{ \varepsilon^2 \partial_{x_\ell} \left(
	     E^\varepsilon(\tau)  \nabla_x \cdot E_{\text{irr}}^\varepsilon(\tau)  \right)
	   }_\delta  d \tau. \nonumber
       \end{align}
    By the algebra property \eqref{app:algprop} and inequalities \eqref{app:killderiv1} and \eqref{app:killderiv2} to handle the derivatives we get
    {\begin{equation}
    \begin{aligned}
    \label{esti_norm_G_irr_2_aux}
    &\abs{\partial_{x_\ell} I_2(t)}_\delta
	\leq 2 \int_0^t \abs{ \varepsilon \partial_{x_\ell} E^\varepsilon(\tau)}_\delta \abs{\varepsilon \nabla_x \cdot E_{\text{irr}}^\varepsilon(\tau)  
	   }_\delta d\tau
       \\
       &\quad+2\int_0^t  \abs{ \varepsilon  E^\varepsilon(\tau)}_\delta \abs{\varepsilon \partial_{x_\ell} \nabla_x \cdot E_{\text{irr}}^\varepsilon(\tau)  
	   }_\delta
       d \tau  \\
       & \leq C    \norm{\varepsilon E^\varepsilon}_{\delta_0} \norm{\varepsilon E_{\text{irr}}^\varepsilon}_{\delta_0}
       \int_0^t \left[\left(\delta_0 - \delta -\frac{\tau}{\eta} \right)^{-2\beta} + \left(\delta_0 - \delta -\frac{\tau}{\eta} \right)^{-(\beta+1)}\right] d \tau.
	\end{aligned}
    \end{equation}}
    Multiplying \eqref{esti_norm_G_irr_2_aux} by $\left(\delta_0-\delta -\frac{t}{\eta}\right)^\beta$ and using that $\left(\delta_0-\delta -\frac{t}{\eta}\right)^\beta\le \left(\delta_0-\delta -\frac{\tau}{\eta}\right)^\beta$ for $\tau\le t$, we get
    \begin{align*}
    & \left(\delta_0-\delta -\frac{t}{\eta}\right)^\beta \abs{\partial_{x_\ell} I_2(t)}_\delta  \\
    & \quad \leq 
    C    \norm{\varepsilon E^\varepsilon}_{\delta_0} \norm{\varepsilon E_{\text{irr}}^\varepsilon}_{\delta_0}
       \int_0^t \Bigg[\left(\delta_0 - \delta -\frac{\tau}{\eta} \right)^{-\beta} + \frac{\big(\delta_0-\delta -\frac{t}{\eta}\big)^\beta}{\big(\delta_0 - \delta -\frac{\tau}{\eta} \big)^{\beta+1}}\Bigg]  d \tau. \nonumber
    \end{align*}
    Using the values of the time integral \eqref{beta} and \eqref{beta+1} we conclude
    \begin{align}
    \label{esti_d_G_irr_2}
        \left(\delta_0-\delta -\frac{t}{\eta}\right)^\beta  \abs{\partial_{x_\ell} I_2(t)}_\delta
        & \leq C \eta     \norm{\varepsilon E^\varepsilon}_{\delta_0} \norm{\varepsilon E_{\text{irr}}^\varepsilon}_{\delta_0}.
    \end{align}
	Thus, summing the two bounds in \eqref{esti_G_irr_2} and \eqref{esti_d_G_irr_2}, we get
    \begin{align}
    \label{esti_norm_G_irr_2}
        \norm{I_2}_{\delta_0} \leq  C \eta   \norm{\varepsilon E^\varepsilon}_{\delta_0} \norm{\varepsilon E_{\text{irr}}^\varepsilon}_{\delta_0}.
    \end{align}

\vskip 0.5cm

\noindent
\underline{\emph{Estimate on $I_3$}}:
We now estimate $I_3=\mathcal{F}^{-1}\big(\big\{\widehat{I_3}(k)\big\}_{k \in \mathbb Z^3}\big)$. As before, using \eqref{bound_cos}, we have
    \begin{align}
	\abs{I_3(t)}_{\delta} 
    & \leq  \int_0^t\abs{\mathcal{F}^{-1} \Bigg(\Bigg\{ \frac{k}{ \abs{k}^2} k \cdot
	 \mathcal{F} \left( j^\varepsilon(\tau) \wedge B^\varepsilon(\tau) \right)  
     \left( 1 - \cos \left(\frac{t-\tau}{\varepsilon} \right) \right)
	    \Bigg\}_{k\in \Z^3}\Bigg)}_\delta d \tau \leq 2 \int_0^t \abs{(j^\varepsilon \wedge B^\varepsilon)   
	    }_\delta d \tau. \nonumber
    \end{align}
    Then, by the algebra property \eqref{app:algprop},
    \begin{align}
    \label{esti_G_irr_3}
	\abs{I_3(t)}_{\delta}  
    \leq 2 \int_0^t \abs{j^\varepsilon(\tau)}_\delta \abs{ B^\varepsilon(\tau) }_\delta d \tau  
     \leq C \eta
       \norm{j^\varepsilon  }_{\delta_0} \norm{ B^\varepsilon  }_{\delta_0},
	\end{align}
where we used that $t\leq \eta$ for the last inequality.

Next, we compute $\abs{\partial_{x_\ell} I_3(t)}_{\delta}$ for $\ell \in \{ 1,2,3 \}$. As before, using \eqref{bound_cos}, we have
\begin{align*}
	\abs{\partial_{x_\ell} I_3(t)}_{\delta} 
    & \leq  \int_0^t\abs{\mathcal{F}^{-1} \Bigg(\Bigg\{ \frac{k}{ \abs{k}^2} k \cdot
	 \mathcal{F} \left(\partial_{x_\ell} ( j^\varepsilon(\tau) \wedge B^\varepsilon(\tau)) \right)  
     \left( 1 - \cos \left(\frac{t-\tau}{\varepsilon} \right) \right)
	    \Bigg\}_{k \in \Z^3}\Bigg)}_\delta d \tau \nonumber \\
    & \leq 2 \int_0^t \abs{\partial_{x_\ell} (j^\varepsilon \wedge B^\varepsilon)  (\tau) 
	    }_\delta d \tau  \leq C \int_{0}^{t} \left( \delta_0 - \delta - \frac{\tau}{\eta}\right)^{-\beta} \abs{ j^\varepsilon(\tau) }_{\delta} \abs{  B^\varepsilon(\tau) }_{\delta} d\tau,
\end{align*}
where we used the algebra property \eqref{app:algprop} and inequality \eqref{app:killderiv1} to handle the derivative.
Finally, using that $\left(\delta_0-\delta -\frac{\tau}{\eta}\right)^{-\beta} \le \left(\delta_0-\delta -\frac{t}{\eta}\right)^{-\beta}$ for $\tau\le t$, we conclude
\begin{align}
\label{esti_d_G_irr_3}
    \abs{ \partial_{x_\ell} I_3(t)}_{\delta} 
    & \leq \left( \delta_0 - \delta - \frac{t}{\eta}\right)^{-\beta} \norm{ j^\varepsilon }_{\delta_0} \norm{  B^\varepsilon }_{\delta_0}  \int_{0}^{t}   ds  \leq C \eta \left(\delta_0-\delta-\frac{t}{\eta}\right)^{-\beta} \norm{ j^\varepsilon }_{\delta_0} \norm{  B^\varepsilon }_{\delta_0},
\end{align}
where we used $t\leq \eta$ for the last inequality.
Thus, multiplying \eqref{esti_d_G_irr_3} by $(\delta_0-\delta-t/\eta)^\beta$ and summing it to \eqref{esti_G_irr_3}, we get
\begin{align}\label{esti_norm_G_irr_3}
    \norm{I_3}_{\delta_0} \leq C \eta \norm{ j^\varepsilon }_{\delta_0} \norm{  B^\varepsilon }_{\delta_0}.
\end{align}

\vskip 0.5cm
\noindent\underline{\emph{Estimate on $I_4$}}:
Finally, we bound $I_4=\mathcal{F}^{-1}\big(\big\{\widehat{I_4}(k)\big\}_{k \in \mathbb Z^3}\big)$. As before we use \eqref{bound_cos}, then we have
	\begin{align}
    \label{esti_G_irr_4_aux}
    \abs{I_4(t)}_{\delta} 
     \leq \int_0^t \abs{\mathcal{F}^{-1} \Bigg( \Bigg\{\frac{k}{ \abs{k}^2} k \cdot
    \mathcal{F} \left( 
     R^\varepsilon(\tau)
    \right) \left(1-\cos\left(\frac{t-\tau}{\varepsilon}\right) \right)
    \Bigg\}_{k \in \Z^3}\Bigg) }_{\delta}
	  d\tau
     \leq 2 \int_0^t \abs{  R^\varepsilon(\tau)
     }_{\delta}
	  d\tau.
    \end{align}
    Recalling expression \eqref{sec3:defR} for the remainder and formula \eqref{typesoluGrenier} for $f^\varepsilon$, we have
      \begin{align}
    \abs{I_4(t)}_{\delta}
	&\le 2 \int_0^t \abs{ \int_{\br^3} f^\varepsilon(\tau,x,\xi)
    \Bigl(
    \lambda(\xi)
    (
    E^\varepsilon(\tau,x)+ v(\xi)\wedge B^\varepsilon(\tau,x))
    \Bigr) 
     d\xi     
	   }_\delta d \tau \nonumber \\
	& = 2 \int_0^t   \abs{\int_M \rho_\Theta^\varepsilon
    \Bigl(\lambda(\xi_\Theta^\varepsilon)(
    E^\varepsilon(\tau,x)+ v(\xi_\Theta^\varepsilon)\wedge B^\varepsilon(\tau,x))
    \Bigr) 
     \mu (d \Theta) }_\delta d \tau \nonumber,
\end{align}
where $\lambda(\xi):=\nabla_\xi\left(v(\xi)-\xi\right)$.
By the algebra property \eqref{app:algprop} and using $t\leq \eta$ to bound the time integral, we obtain
\begin{align*}
        \abs{I_4(t)}_{\delta} 
\leq 2 \eta \sup_\Theta \norm{\rho_\Theta^\varepsilon}_{\delta_0}   
    &
   \left[\left( \norm{\varepsilon E^\varepsilon}_{\delta_0} + \varepsilon\sup_\Theta \norm{v(\xi_\Theta^\varepsilon)}_{\delta_0}
     \norm{B^\varepsilon}_{\delta_0}\right)
     \sup_\Theta \norm{\varepsilon^{-1}\lambda(\xi_\Theta^\varepsilon)}_{\delta_0}
     \right]. \nonumber
\end{align*}
Using inequalities \eqref{vxiteta} and \eqref{Lemma_remainder} for the relativistic corrections $v(\xi^\varepsilon_\Theta)$ and $\lambda(\xi^\varepsilon_\Theta)$, we get
\begin{align}
\label{esti_G_irr_4}
     |I_4(t)|_\delta \leq \varepsilon C \eta \sup_\Theta \norm{\rho_\Theta^\varepsilon}_{\delta_0}   
     \left(
     \norm{\varepsilon E^\varepsilon}_{\delta_0} \sup_{\Theta} \norm{\xi_\Theta^\varepsilon}_{\delta_0}^2 
    + 
    \varepsilon
      \norm{ B^\varepsilon}_{\delta_0}
      \sup_{\Theta} \norm{\xi_\Theta^\varepsilon}_{\delta_0}^3
      \right).
	\end{align}
Next, we compute $\abs{\partial_{x_\ell} I_4(t)}_{\delta}$ for $\ell \in \{ 1,2,3 \}$. Proceeding as for \eqref{esti_G_irr_4_aux} and using inequality \eqref{app:killderiv1} to handle the  $\partial_{x_\ell}$ derivative, we obtain
\begin{align}
    \abs{ \partial_{x_\ell} I_4(t)}_{\delta} 
    & \leq 2 \int_0^t \abs{  \partial_{x_\ell} R^\varepsilon(\tau)
     }_{\delta}
	  d\tau  \leq 2 \int_0^t \left(\delta_0 - \delta -\frac{\tau}{\eta} \right)^{-\beta} \abs{  R^\varepsilon(\tau)
     }_{\delta}
	  d\tau \nonumber.
      \end{align}
      Next, using that $\left(\delta_0-\delta -\frac{\tau}{\eta}\right)^{-\beta} \le \left(\delta_0-\delta -\frac{t}{\eta}\right)^{-\beta}$ for $\tau\le t$, we conclude
      \begin{align}
      \label{esti_d_G_irr_4}
      \abs{ \partial_{x_\ell} I_4(t)}_{\delta} 
      & \leq 2 \left(\delta_0 - \delta -\frac{t}{\eta} \right)^{-\beta} \int_0^t  \abs{  R^\varepsilon(\tau)
     }_{\delta}
	  d\tau \nonumber \\
     & \leq \varepsilon C \eta  \left(\delta_0 - \delta -\frac{t}{\eta} \right)^{-\beta} \sup_\Theta \norm{\rho_\Theta^\varepsilon}_{\delta_0}   
     \left(
     \norm{\varepsilon E^\varepsilon}_{\delta_0} \sup_{\Theta} \norm{\xi_\Theta^\varepsilon}_{\delta_0}^2 
    + 
    \varepsilon
      \norm{ B^\varepsilon}_{\delta_0}
      \sup_{\Theta} \norm{\xi_\Theta^\varepsilon}_{\delta_0}^3
      \right),
	\end{align}
where we used the same estimates as in \eqref{esti_G_irr_4} to bound $\abs{  R^\varepsilon(\tau)
     }_{\delta}$. 
     Multiplying \eqref{esti_d_G_irr_4} by $(\delta_0-\delta-t/\eta)^\beta$
     and summing it to \eqref{esti_G_irr_4}, we get 
\begin{align}
    \label{esti_norm_G_irr_4}
    \norm{I_4}_{\delta_0 } \leq \varepsilon C  \eta  \sup_\Theta \norm{\rho_\Theta^\varepsilon}_{\delta_0}   
     \left(
     \norm{\varepsilon E^\varepsilon}_{\delta_0} \sup_{\Theta} \norm{\xi_\Theta^\varepsilon}_{\delta_0}^2 
    + 
    \varepsilon
      \norm{ B^\varepsilon}_{\delta_0}
      \sup_{\Theta} \norm{\xi_\Theta^\varepsilon}_{\delta_0}^3
      \right).
\end{align}
\vskip 0.5 cm
\noindent
\underline{\emph{Final steps of the estimate on $G^\varepsilon_{\text{irr}}$}}: Finally, combining estimates \eqref{esti_norm_G_irr_1}, \eqref{esti_norm_G_irr_2}, \eqref{esti_norm_G_irr_3}, and \eqref{esti_norm_G_irr_4}, there exists a constant $C$ depending on $\delta_0$ and $\beta$ such that
\begin{align}
\label{esti_norm_G_irr}
    \norm{G^\varepsilon_{\text{irr}}}_{\delta_0}  
    & \leq C \eta  \left( \sup_\Theta \left( \norm{ 
       \rho_\Theta^\varepsilon}_{\delta_0}  \norm{\xi_\Theta^\varepsilon}_{\delta_0}^2 \right) 
    +    \norm{\varepsilon E^\varepsilon}_{\delta_0} \norm{\varepsilon E_{\text{irr}}^\varepsilon}_{\delta_0}
    + \norm{j^\varepsilon  }_{\delta_0} \norm{ B^\varepsilon  }_{\delta_0} \right) \nonumber  \\
    & \quad 
    +  \varepsilon C \eta   \sup_\Theta \norm{\rho_\Theta^\varepsilon}_{\delta_0}   
    \left(
     \norm{\varepsilon E^\varepsilon}_{\delta_0} \sup_{\Theta} \norm{\xi_\Theta^\varepsilon}_{\delta_0}^2 
    + 
    \varepsilon
      \norm{ B^\varepsilon}_{\delta_0}
      \sup_{\Theta} \norm{\xi_\Theta^\varepsilon}_{\delta_0}^3
      \right)
     + \norm{G^\varepsilon_{\text{irr},0}}_{\delta_0}.
\end{align}
\vskip 0.5cm
\noindent
\underline{\textbf{Estimates on $\varepsilon E_{\text{irr}}^\varepsilon$}:}
Recalling \eqref{Fourier_divE_irr}, we have
\begin{align*}
    \varepsilon \widehat{E_{\text{irr}}^\varepsilon}(t,k) 
	= -\int_{0}^{t} \frac{\imm k}{ \abs{k}^2} \sin\left(\frac{t-s}{\varepsilon} \right)
	\widehat{g^\varepsilon}(s,k)ds+ \varepsilon \widehat{E_{\text{irr},0}^\varepsilon} (t,k).
\end{align*}
We observe that the formula for $\varepsilon \widehat{E_{\text{irr}}^\varepsilon}(t,k)$  has  the same structure as the one for $\widehat{G_{\text{irr}}^\varepsilon}(t,k)$ \eqref{Fourier_G_irr}, the only difference is that
$(1 - \cos\left(\frac{t-s}{\varepsilon} \right))$ in \eqref{Fourier_G_irr} 
is now replaced by $\sin\left(\frac{t-s}{\varepsilon} \right)$. Therefore, using
\begin{align*}
    \abs{\frac{k_j k}{ \abs{k}^2}  \sin \left(\frac{t-\tau}{\varepsilon} \right)} \leq 1,
\end{align*}
we can estimate $\varepsilon \widehat{E_{\text{irr}}^\varepsilon}$ in the same way as $ \widehat{G_{\text{irr}}^\varepsilon}$ and we get the same inequality as in \eqref{esti_norm_G_irr}. Namely, there exists a constant $C$ depending on $\delta_0$ and $\beta$ such that
\begin{align}
\label{esti_norm_epsilon_E_irr}
    \norm{\varepsilon E^\varepsilon_{\text{irr}}}_{\delta_0}  
    & \leq  C \eta  \left( \sup_\Theta \left( \norm{ 
       \rho_\Theta^\varepsilon}_{\delta_0}  \norm{\xi_\Theta^\varepsilon}_{\delta_0}^2 \right) 
    +    \norm{\varepsilon E^\varepsilon}_{\delta_0} \norm{\varepsilon E_{\text{irr}}^\varepsilon}_{\delta_0}
    + \norm{j^\varepsilon  }_{\delta_0} \norm{ B^\varepsilon  }_{\delta_0} \right) \nonumber  \\
    & \quad 
    +  \varepsilon C  \eta   \sup_\Theta \norm{\rho_\Theta^\varepsilon}_{\delta_0}   
    \left(
     \norm{\varepsilon E^\varepsilon}_{\delta_0} \sup_{\Theta} \norm{\xi_\Theta^\varepsilon}_{\delta_0}^2 
    + 
    \varepsilon
      \norm{ B^\varepsilon}_{\delta_0}
      \sup_{\Theta} \norm{\xi_\Theta^\varepsilon}_{\delta_0}^3
      \right)
     + \norm{\varepsilon E^\varepsilon_{\text{irr},0}}_{\delta_0}.
\end{align}

\subsubsection{Estimates on \texorpdfstring{$G_{\text{sol}}^\varepsilon, \varepsilon E^\varepsilon_{\text{sol}}$}{Lg} and \texorpdfstring{$B^\varepsilon$}{Lg}}  
\underline{\textbf{Estimates on \texorpdfstring{$G_{\text{sol}}^\varepsilon$}{Lg}}}:
Using formula \eqref{Fourier_E_sol}, we have 
    \begin{align*}
	\widehat{G_{\text{sol}}^\varepsilon}(t,k)  &= \int_0^t \int_{0}^{s} \frac{\imm k\wedge\widehat{h^\varepsilon}(\tau,k)}{\varepsilon \abs{k}^2\sqrt{1+|k|^2}} \sin\left(\frac{\sqrt{1+|k|^2}(s-\tau)}{\varepsilon} \right)
	d\tau ds
	+ \widehat{G^\varepsilon_{\text{sol},0}}(t,k),
	\end{align*}
	where 
    \begin{equation}
    \label{Gsol0}
        \begin{aligned}
            &\widehat{G^\varepsilon_{\text{sol},0}}(t,k) = \int_0^t \widehat{E^\varepsilon_{\text{sol},0}}(s,k) ds\\
            &=\frac{\varepsilon}{\sqrt{1+|k|^2}}\widehat{E_{\text{sol}}^\varepsilon}(0,k)\sin\left(\frac{t\sqrt{1+|k|^2}}{\varepsilon}\right)
	+\frac{\varepsilon^2}{1+|k|^2}\widehat{\partial_t E_{\text{sol}}^\varepsilon}(0,k)
	\left(1-\cos\left(\frac{t \sqrt{1+|k|^2}}{\varepsilon}\right)\right).
        \end{aligned}
    \end{equation}
    Using Fubini on the double integral and recalling that $\int_0^t \int_{0}^{s} d\tau ds = \int_0^t \int_{\tau}^{t}  ds d\tau$, we get
    \begin{align}
    \label{Fourier_G_sol}
	\widehat{G_{\text{sol}}^\varepsilon}(t,k)
	&= \int_0^t \frac{\imm k\wedge \widehat{h^\varepsilon}(\tau,k) }{\varepsilon \abs{k}^2\sqrt{1+|k|^2}} 
	\int_{\tau}^{t}  \sin\left(\frac{(s-\tau)\sqrt{1+|k|^2}}{\varepsilon} \right) ds d\tau
	+ \widehat{G^\varepsilon_{\text{sol},0}}(t,k) \nonumber \\
	& = \int_0^t \frac{\imm k\wedge \widehat{h^\varepsilon}(\tau,k)}{ \abs{k}^2 (1+|k|^2)} 
	\left( 1 - \cos \left(\frac{(t-\tau)\sqrt{1+|k|^2}}{\varepsilon} \right) \right) d\tau
	+ \widehat{G^\varepsilon_{\text{sol},0}}(t,k) \nonumber \\
    & =: \widehat{J_1}(t,k) + \widehat{J_2}(t,k) + \widehat{J_3}(t,k) + \widehat{J_4}(t,k)
    + \widehat{G^\varepsilon_{\text{sol},0}}(t,k),
	\end{align}
where, by recalling the formula \eqref{sec3:hepsilon} for $h^\varepsilon$, $\{\widehat{J_\ell}(t,k)\}_{\ell=1}^4$ are given by
{\small\begin{align*}
    \widehat{J_1}(t,k):= \int_0^t  \frac{\imm k}{ \abs{k}^2(1+|k|^2)}\wedge
    \mathcal{F} \left( 
    \nabla_x  \wedge  \partial_{x_i}
        \int_{\br^3} f^\varepsilon(\tau) v(\xi)_i v(\xi)d \xi
    \right) (\tau ,k)
	\left( 1 - \cos \left(\frac{(t-\tau)\sqrt{1+|k|^2}}{\varepsilon} \right) \right) d\tau,
\end{align*}
\begin{align*}
    \widehat{J_2}(t,k) := \int_0^t  \frac{\imm k}{ \abs{k}^2(1+|k|^2)}\wedge
    \mathcal{F} \left( 
    \varepsilon^2 \nabla_x  \wedge  \big(E^\varepsilon(\tau,x) \nabla_x \cdot E_{\text{irr}}^\varepsilon(\tau,x)\big)
    \right) (\tau ,k)
	\left( 1 - \cos \left(\frac{(t-\tau)\sqrt{1+|k|^2}}{\varepsilon} \right) \right)  d\tau,
\end{align*}}
\begin{align*}
    \widehat{J_3}(t,k) := \int_0^t  \frac{\imm k}{ \abs{k}^2(1+|k|^2)}\wedge
    \mathcal{F} \left( 
    \nabla_x  \wedge (j^\varepsilon(\tau,x) \wedge B^\varepsilon(\tau,x))
    \right) (\tau ,k)
	\left( 1 - \cos \left(\frac{(t-\tau)\sqrt{1+|k|^2}}{\varepsilon} \right) \right)  d\tau,
\end{align*}
and 
\begin{align*}
    \widehat{J_4}(t,k) := \int_0^t  \frac{\imm k}{ \abs{k}^2(1+|k|^2)}\wedge
    \mathcal{F} \left( 
    \nabla_x  \wedge  R^\varepsilon(\tau,x)
    \right) (\tau ,k)
	\left( 1 - \cos \left(\frac{(t-\tau)\sqrt{1+|k|^2}}{\varepsilon} \right) \right) d\tau.
\end{align*}
We now derive the a priori estimates for the quantities $J_\ell=\mathcal{F}^{-1}\big(\big\{\widehat{J_\ell}(k)\big\}_{k \in \mathbb Z^3}\big)$ with $\ell\in \{1,2,3,4\}$.
\vskip 0.5 cm 
\noindent
\underline{\emph{Estimate on $J_1$}}:
We start by estimating $J_1=\mathcal{F}^{-1}\big(\big\{\widehat{J_1}(k)\big\}_{k \in \mathbb Z^3}\big)$. By using that
\begin{align}
    \label{bound_cos2}
	\abs{\frac{\imm k}{ \abs{k}^2(1+|k|^2)}\wedge \left( k \wedge
    \mathcal{F} (\cdot) \right) \left( 1 - \cos \left(\frac{(t-\tau)\sqrt{1+|k|^2}}{\varepsilon} \right) \right)} \leq 2 \abs{\mathcal{F} (\cdot)},
	\end{align}
	and expression \eqref{typesoluGrenier} for $f^\varepsilon(t,x,\xi)$, we obtain
    \begin{align}
    \label{esti_G_sol_1}
	\abs{J_1(t)}_\delta 
    & \leq \int_0^t \left| \mathcal{F}^{-1} \left(\Bigg\{\frac{k}{ \abs{k}^2(1+|k|^2)}\wedge 
    \left( k \wedge
    \mathcal{F} \left( 
    \partial_{x_i}
        \int_{\br^3} f^\varepsilon(t,x,\xi) v(\xi)_i v(\xi)d \xi
    \right) (\tau ,k) \right) \right. \right. \nonumber \\
	& \qquad \qquad \times \left. \left. \left( 1 - \cos \left(\frac{(t-\tau)\sqrt{1+|k|^2}}{\varepsilon} \right) \right)\Bigg\}_{k \in \Z^3} \right) \right|_\delta d\tau \nonumber \\
    &  \leq 2 \int_0^t \abs{
       \partial_{x_i}
       \int_M \rho_\Theta^\varepsilon(\tau)  v(\xi_\Theta^\varepsilon(\tau))_i 
       v(\xi_\Theta^\varepsilon(\tau)) \mu(d\Theta) }_\delta d \tau \nonumber \\
       & \leq  2 \eta  \sup_\Theta \left( \norm{ 
       \rho_\Theta^\varepsilon}_{\delta_0}  \norm{\xi_\Theta^\varepsilon}_{\delta_0}^2 \right),
	\end{align}
	where the last inequality is deduced using the same arguments as for \eqref{esti_G_irr_1}.
    Next, we estimate $\abs{\partial_{x_\ell} J_1(t)}_{\delta}$ for $\ell \in \{1,2,3\}$. We use the bound \eqref{bound_cos2} and expression \eqref{typesoluGrenier} for $f^\varepsilon(t,x,\xi)$ to get 
    \begin{align}
    \label{esti_d_G_sol_1}
	\abs{\partial_{x_\ell} J_1(t)}_\delta 
    & \leq 
    \int_0^t \left| \mathcal{F}^{-1} \left(\Bigg\{ \frac{k}{ \abs{k}^2(1+|k|^2)}\wedge k \wedge
    \mathcal{F} \left( 
    \partial_{x_\ell}
    \partial_{x_i}
        \int_{\br^3} f^\varepsilon(t,x,\xi) v(\xi)_i v(\xi)d \xi
    \right) (\tau ,k) \right. \right. \nonumber \\
	& \qquad \qquad \times \left. \left. \left( 1 - \cos \left(\frac{(t-\tau)\sqrt{1+|k|^2}}{\varepsilon} \right) \right) \Bigg\}_{k \in \Z^3}\right) \right|_\delta d\tau \nonumber \\
	&  \leq 2 \int_0^t \abs{ \partial_{x_\ell}  \partial_{x_i}
       \int_M \rho_\Theta^\varepsilon(\tau)  v(\xi_\Theta^\varepsilon(\tau))_i 
       v(\xi_\Theta^\varepsilon(\tau)) \mu(d\Theta) }_\delta d \tau \nonumber \\
       & \leq C \eta \left(\delta_0-\delta-\frac{t}{\eta}\right)^{-\beta}  \sup_\Theta \left( \norm{ 
       \rho_\Theta^\varepsilon}_{\delta_0}  \norm{\xi_\Theta^\varepsilon}_{\delta_0}^2 \right),
	\end{align}
where the last inequality is deduced using the same arguments as for  \eqref{esti_d_G_irr_1}.
Thus, multiplying \eqref{esti_d_G_sol_1} by $\left(\delta_0-\delta-\frac{t}{\eta}\right)^\beta$ and summing it to \eqref{esti_G_sol_1}, we get 
\begin{align}
\label{esti_norm_G_sol_1}
    \norm{J_1}_{\delta_0} \leq  C \eta \sup_\Theta \left( \norm{ 
       \rho_\Theta^\varepsilon}_{\delta_0}  \norm{\xi_\Theta^\varepsilon}_{\delta_0}^2 \right).
\end{align}
    \vskip 0.5cm

\noindent\underline{\emph{Estimate on $J_2$}}: 
We now estimate $J_2=\mathcal{F}^{-1}\big(\big\{\widehat{J_2}(k)\big\}_{k \in \mathbb Z^3}\big)$. As before we use \eqref{bound_cos2}, therefore, we obtain
	\begin{align}\label{esti_G_sol_2}
    \abs{J_2(t)}_\delta 
    &\leq \int_0^t \left| \mathcal{F}^{-1} \left( \Bigg\{\frac{k}{ \abs{k}^2(1+|k|^2)}\wedge 
    \left( k \wedge
    \mathcal{F} \left( 
    \varepsilon^2 \big(E^\varepsilon(\tau,x) \nabla_x \cdot E_{\text{irr}}^\varepsilon(\tau,x)\big)
    \right) (\tau ,k) \right) \right. \right. \nonumber \\
    & \qquad \qquad \times \left. \left. \left( 1 - \cos \left(\frac{(t-\tau)\sqrt{1+|k|^2}}{\varepsilon} \right) \right) \Bigg\}_{k \in \Z^3}\right) \right|_\delta d\tau \nonumber \\
	   &\leq 2 \int_0^t \abs{ \varepsilon^2
	      E^\varepsilon(\tau)  \nabla_x \cdot E_{\text{irr}}^\varepsilon(\tau)  
	   }_\delta  d \tau \nonumber \\
    & \leq C \eta 
	     \norm{\varepsilon  E^\varepsilon}_{\delta_0} \norm{\varepsilon  E_{\text{irr}}^\varepsilon}_{\delta_0},
	\end{align}
	where we used the same estimates as for \eqref{esti_G_irr_2} for the last inequality.

    Next, we compute $\abs{\partial_{x_\ell} J_2(t)}_{\delta}$ for $\ell \in \{ 1,2,3 \}$. As before we use \eqref{bound_cos2}, therefore, we have
	\begin{align}
    \abs{\partial_{x_\ell} J_2(t)}_\delta 
    &\leq \int_0^t \left| \mathcal{F}^{-1} \left(\Bigg\{ \frac{k}{ \abs{k}^2(1+|k|^2)}\wedge 
    \left( k \wedge
    \mathcal{F} \left( 
    \varepsilon^2 \partial_{x_\ell} \big(E^\varepsilon(\tau,x) \nabla_x \cdot E_{\text{irr}}^\varepsilon(\tau,x)\big)
    \right) (\tau ,k) \right) \right. \right. \nonumber \\
    & \qquad \qquad \times \left. \left. \Bigg( 1 - \cos \Bigg(\frac{(t-\tau)\sqrt{1+|k|^2}}{\varepsilon} \Bigg) \Bigg)\Bigg\}_{k \in \Z^3} \right) \right|_\delta d\tau \nonumber \\
	   &\leq 2 \int_0^t \abs{ \varepsilon^2 \partial_{x_\ell} \left(
	      E^\varepsilon(\tau) \nabla_x \cdot E_{\text{irr}}^\varepsilon(\tau)  \right)
	   }_\delta  d \tau \nonumber.
    \end{align}
    Next, using the same argument as obtaining \eqref{esti_d_G_irr_2}, we conclude
    \begin{align}
    \label{esti_d_G_sol_2}
    \left(\delta_0 - \delta -\frac{t}{\eta} \right)^{\beta} \abs{\partial_{x_\ell} J_2(t)}_\delta 
    & \leq C \eta    \norm{\varepsilon E^\varepsilon}_{\delta_0} \norm{\varepsilon E_{\text{irr}}^\varepsilon}_{\delta_0}.
	\end{align}
	Thus, summing inequalities \eqref{esti_G_sol_2} and \eqref{esti_d_G_sol_2} we get
    \begin{align}
    \label{esti_norm_G_sol_2}
        \norm{J_2}_{\delta_0} \leq  C \eta   \norm{\varepsilon E^\varepsilon}_{\delta_0} \norm{\varepsilon E_{\text{irr}}^\varepsilon}_{\delta_0}.
    \end{align}
\vskip 0.5 cm
\noindent
\underline{\emph{Estimate on $J_3$}}: 
We estimate $J_3=\mathcal{F}^{-1}\big(\big\{\widehat{J_3}(k)\big\}_{k \in \mathbb Z^3}\big)$. As before we use \eqref{bound_cos2}, therefore, we have
    \begin{align}\label{esti_G_sol_3}
	\abs{J_3(t)}_{\delta} 
    &\leq \int_0^t \left| \mathcal{F}^{-1} \left(\Bigg\{ \frac{k}{ \abs{k}^2(1+|k|^2)}\wedge 
    \left( k \wedge
    \mathcal{F} \left( 
    (j^\varepsilon(\tau,x) \wedge B^\varepsilon(\tau,x))
    \right) (\tau ,k) \right) \right. \right. \nonumber \\
    & \qquad \qquad \times \left. \left. \left( 1 - \cos \left(\frac{(t-\tau)\sqrt{1+|k|^2}}{\varepsilon} \right) \right) \Bigg\}_{k \in \Z^3}\right) \right|_\delta d\tau \nonumber \\
    & \leq 2 \int_0^t \abs{(j^\varepsilon \wedge B^\varepsilon)   
	    }_\delta d \tau 
    \leq C \eta
       \norm{j^\varepsilon  }_{\delta_0} \norm{ B^\varepsilon  }_{\delta_0},
	\end{align}
where we used the same inequalities as for \eqref{esti_G_irr_3}.

Next, we compute $\abs{\partial_{x_\ell} J_3(t)}_{\delta}$ for $\ell \in \{ 1,2,3 \}$. We use again \eqref{bound_cos2}, therefore, we have
\begin{align}
\label{esti_d_G_sol_3}
	\abs{\partial_{x_\ell} J_3(t)}_{\delta} 
    &\leq \int_0^t \left| \mathcal{F}^{-1} \left(\Bigg\{ \frac{k}{ \abs{k}^2(1+|k|^2)}\wedge 
    \left( k \wedge
    \mathcal{F} \left( \partial_{x_\ell}
    (j^\varepsilon(\tau,x) \wedge B^\varepsilon(\tau,x))
    \right) (\tau ,k) \right) \right. \right. \nonumber \\
    & \qquad \qquad \times \left. \left. \left( 1 - \cos \left(\frac{(t-\tau)\sqrt{1+|k|^2}}{\varepsilon} \right) \right) \Bigg\}_{k \in \Z^3}\right) \right|_\delta d\tau \nonumber \\
    & \leq 2 \int_0^t \abs{\partial_{x_\ell} (j^\varepsilon \wedge B^\varepsilon)  (\tau) 
	    }_\delta d \tau 
    \leq C \eta \left(\delta_0-\delta-\frac{t}{\eta}\right)^{-\beta} \norm{ j^\varepsilon }_{\delta_0} \norm{  B^\varepsilon }_{\delta_0},
\end{align}
where we used the same reasoning as for \eqref{esti_d_G_irr_3} to deduce the last inequality.
Thus, multiplying \eqref{esti_d_G_sol_3} by $\left(\delta_0-\delta-\frac{t}{\eta}\right)^\beta$ and putting it together with \eqref{esti_G_sol_3}, we have 
\begin{align}\label{esti_norm_G_sol_3}
    \norm{J_3}_{\delta_0} \leq C \eta \norm{ j^\varepsilon }_{\delta_0} \norm{  B^\varepsilon }_{\delta_0}.
\end{align}
\vskip 0.5 cm
\noindent
\underline{\emph{Estimate on $J_4$}}: 
We conclude by estimating $J_4=\mathcal{F}^{-1}\big(\big\{\widehat{J_4}(k)\big\}_{k \in \mathbb Z^3}\big)$. As before, using \eqref{bound_cos2},
    \begin{equation}
	\begin{aligned}\label{esti_G_sol_4}
    \abs{J_4(t)}_{\delta} 
    &\leq \int_0^t \left| \mathcal{F}^{-1} \left( \Bigg\{\frac{k}{ \abs{k}^2(1+|k|^2)}\wedge 
    \left( k \wedge
    \mathcal{F} \left( \partial_{x_\ell}
    R^\varepsilon(\tau,x)
    \right) (\tau ,k) \right) \right. \right.  \\
    & \qquad \qquad \times \left. \left. \left( 1 - \cos \left(\frac{(t-\tau)\sqrt{1+|k|^2}}{\varepsilon} \right) \right) \Bigg\}_{k \in \Z^3}\right) \right|_\delta d\tau  \\
    & \leq 2 \int_0^t \abs{  R^\varepsilon(\tau)
     }_{\delta}
	  d\tau  \leq \varepsilon C  \eta  \sup_\Theta \norm{\rho_\Theta^\varepsilon}_{\delta_0}   
     \left(
     \norm{\varepsilon E^\varepsilon}_{\delta_0} \sup_{\Theta} \norm{\xi_\Theta^\varepsilon}_{\delta_0}^2 
    + 
    \varepsilon
      \norm{ B^\varepsilon}_{\delta_0}
      \sup_{\Theta} \norm{\xi_\Theta^\varepsilon}_{\delta_0}^3
      \right),
	\end{aligned}
    \end{equation}
where we used the same estimates as in \eqref{esti_G_irr_4} for the last inequality.

Next, we compute $\abs{\partial_{x_\ell} J_4}_{\delta}$ for $\ell \in \{ 1,2,3 \}$.  We use \eqref{bound_cos2} again  and proceeding as for \eqref{esti_d_G_irr_4} we obtain
\begin{align}\label{esti_d_G_sol_4}
    \abs{ \partial_{x_\ell} J_4(t)}_{\delta} 
    &\leq \int_0^t \left| \mathcal{F}^{-1} \left( \Bigg\{\frac{k}{ \abs{k}^2(1+|k|^2)}\wedge 
    \left( k \wedge
    \mathcal{F} \left( \partial_{x_\ell}
    \partial_{x_\ell} R^\varepsilon(\tau,x)
    \right) (\tau ,k) \right) \right. \right. \nonumber \\
    & \qquad \qquad \times \left. \left. \left( 1 - \cos \left(\frac{(t-\tau)\sqrt{1+|k|^2}}{\varepsilon} \right) \right) \Bigg\}_{k \in \Z^3}\right) \right|_\delta d\tau \nonumber \\
    & \leq 2 \int_0^t \abs{  
    \partial_{x_\ell} R^\varepsilon(\tau)
     }_{\delta}
	  d\tau \nonumber \\
     & \leq \varepsilon C \eta  \left(\delta_0 - \delta -\frac{t}{\eta} \right)^{-\beta} \sup_\Theta \norm{\rho_\Theta^\varepsilon}_{\delta_0}   
     \left(
     \norm{\varepsilon E^\varepsilon}_{\delta_0} \sup_{\Theta} \norm{\xi_\Theta^\varepsilon}_{\delta_0}^2 
    + 
    \varepsilon
      \norm{ B^\varepsilon}_{\delta_0}
      \sup_{\Theta} \norm{\xi_\Theta^\varepsilon}_{\delta_0}^3
      \right).
	\end{align}
Thus, multiplying \eqref{esti_d_G_sol_4} by $\left(\delta_0 - \delta -\frac{t}{\eta} \right)^\beta$ and summing it with \eqref{esti_G_sol_4}, we get
\begin{align}
    \label{esti_norm_G_sol_4}
    \norm{J_4}_{\delta_0 } \leq  \varepsilon C \eta \sup_\Theta \norm{\rho_\Theta^\varepsilon}_{\delta_0}   
     \left(
     \norm{\varepsilon E^\varepsilon}_{\delta_0} \sup_{\Theta} \norm{\xi_\Theta^\varepsilon}_{\delta_0}^2 
    + 
    \varepsilon
      \norm{ B^\varepsilon}_{\delta_0}
      \sup_{\Theta} \norm{\xi_\Theta^\varepsilon}_{\delta_0}^3
      \right).
\end{align}
\vskip 0.5 cm
\noindent
\underline{\emph{Final steps of the estimate on $G^\varepsilon_{\text{sol}}$}}:
Finally, combining estimates \eqref{esti_norm_G_sol_1}, \eqref{esti_norm_G_sol_2}, \eqref{esti_norm_G_sol_3}, and \eqref{esti_norm_G_sol_4}, there exists a constant $C$ depending on $\delta_0$ and $\beta$ such that
\begin{align}
\label{esti_norm_G_sol}
    \norm{G^\varepsilon_{\text{sol}}}_{\delta_0}  
    & \leq C \eta  \left( \sup_\Theta \left( \norm{ 
       \rho_\Theta^\varepsilon}_{\delta_0}  \norm{\xi_\Theta^\varepsilon}_{\delta_0}^2 \right) 
    +   \norm{\varepsilon E^\varepsilon}_{\delta_0} \norm{\varepsilon E_{\text{irr}}^\varepsilon}_{\delta_0}
    + \norm{j^\varepsilon  }_{\delta_0} \norm{ B^\varepsilon  }_{\delta_0} \right) \nonumber  \\
    & \quad 
    +  \varepsilon C \eta   \sup_\Theta \norm{\rho_\Theta^\varepsilon}_{\delta_0}   
    \left(
     \norm{\varepsilon E^\varepsilon}_{\delta_0} \sup_{\Theta} \norm{\xi_\Theta^\varepsilon}_{\delta_0}^2 
    + 
    \varepsilon
      \norm{ B^\varepsilon}_{\delta_0}
      \sup_{\Theta} \norm{\xi_\Theta^\varepsilon}_{\delta_0}^3
      \right)
     + \norm{G^\varepsilon_{\text{sol},0}}_{\delta_0}.
\end{align}
\vskip 0.5cm
\noindent\underline{\textbf{Estimates on $\varepsilon E_{\text{sol}}^\varepsilon$}}:
Recalling \eqref{Fourier_E_sol}, we have
\begin{align*}
    \varepsilon \widehat{E_{\text{sol}}^\varepsilon}(t,k) 
	= \int_{0}^{t} \frac{1}{|k|^2\sqrt{1+|k|^2}} \sin\left(\frac{\sqrt{1+|k|^2}(t-s)}{\varepsilon} \right)
	\left(\imm k\wedge\widehat{h^\varepsilon}(s,k)\right)ds 
    + \varepsilon \widehat{E_{\text{sol},0}^\varepsilon} (t,k).
\end{align*}
We observe that $\varepsilon \widehat{E_{\text{sol}}^\varepsilon}(t,k)$ is given by a similar formula that the one for $\widehat{G_{\text{sol}}^\varepsilon}(t,k)$, i.e., \eqref{Fourier_G_sol}. The only difference is that now we have $\sin\left(\frac{\sqrt{1+|k|^2}(t-s)}{\varepsilon} \right)$ instead of $\left( 1 - \cos \left(\frac{\sqrt{1+|k|^2}(t-s)}{\varepsilon} \right) \right)$ in \eqref{Fourier_G_sol}. Therefore, using
\begin{align*}
    \abs{\frac{\imm k}{ \abs{k}^2(1+|k|^2)}\wedge \left( k \wedge
    \mathcal{F} (\cdot) \right) \sin\left(\frac{\sqrt{1+|k|^2}(t-s)}{\varepsilon} \right)} \leq  \abs{\mathcal{F} (\cdot)},
\end{align*}
we can estimate $\varepsilon E_{\text{sol}}^\varepsilon$ in the same way as $ G_{\text{sol}}^\varepsilon$ and we get the same inequality as in \eqref{esti_norm_G_sol}. Namely, there exists a constant $C$ depending on $\delta_0$ and $\beta$ such that
\begin{align}
\label{esti_norm_epsilon_E_sol}
    \norm{\varepsilon E^\varepsilon_{\text{sol}}}_{\delta_0}  
    & \leq C \eta  \left( \sup_\Theta \left( \norm{ 
       \rho_\Theta^\varepsilon}_{\delta_0}  \norm{\xi_\Theta^\varepsilon}_{\delta_0}^2 \right) 
    +    \norm{\varepsilon E^\varepsilon}_{\delta_0} \norm{\varepsilon E_{\text{irr}}^\varepsilon}_{\delta_0}
    + \norm{j^\varepsilon  }_{\delta_0} \norm{ B^\varepsilon  }_{\delta_0} \right) \nonumber  \\
    & \quad 
    +  C\varepsilon  \eta   \sup_\Theta \norm{\rho_\Theta^\varepsilon}_{\delta_0}   
    \left(
     \norm{\varepsilon E^\varepsilon}_{\delta_0} \sup_{\Theta} \norm{\xi_\Theta^\varepsilon}_{\delta_0}^2 
    + 
    \varepsilon
      \norm{ B^\varepsilon}_{\delta_0}
      \sup_{\Theta} \norm{\xi_\Theta^\varepsilon}_{\delta_0}^3
      \right)
     + \norm{\varepsilon E^\varepsilon_{\text{sol},0}}_{\delta_0}.
\end{align}

\vskip 0.5 cm
\noindent\underline{\textbf{Estimates on $B^\varepsilon$}}:
Recall that by the Maxwell--Faraday equation, we have
    \[
    B^\varepsilon(t,x)=B^\varepsilon_0(x)-\int_0^t \nabla_x  \wedge  E^\varepsilon_{\text{sol}}(s,x) ds.
    \]
From the formula for $\nabla_x  \wedge  E^\varepsilon_{\text{sol}}$ in \eqref{sec3:wave}, we get
\begin{align*}
	\reallywidehat{\nabla_x  \wedge  E_{\text{sol}}^\varepsilon}(t,k) 
	= \int_{0}^{t} \frac{1}{\varepsilon \sqrt{1+|k|^2}} \sin\left(\frac{\sqrt{1+|k|^2}(t-s)}{\varepsilon} \right)
	\widehat{h^\varepsilon}(s,k) ds
    +\reallywidehat{\nabla_x  \wedge  E_{\text{sol},0}^\varepsilon} (t,k). 
\end{align*}
Therefore, by Fubini
{\small\begin{align*}
    \int_0^t \reallywidehat{\nabla_x  \wedge  E^\varepsilon_{\text{sol}}}(s,k) ds 
    & = \int_0^t \int_{0}^{s} \frac{1}{\varepsilon \sqrt{1+|k|^2}} \sin\left(\frac{\sqrt{1+|k|^2}(s-\tau)}{\varepsilon} \right)
	\widehat{h^\varepsilon}(\tau,k) d \tau ds
    +\int_0^t \reallywidehat{\nabla_x  \wedge  E_{\text{sol},0}^\varepsilon} (s,k) ds \\
    & = \int_0^t \frac{1}{ 1+|k|^2} 
    \left( 1 - \cos \left(\frac{\sqrt{1+|k|^2}(t-\tau)}{\varepsilon} \right) \right)
	\widehat{h^\varepsilon}(\tau,k) d \tau
    +\int_0^t \reallywidehat{\nabla_x  \wedge  E_{\text{sol},0}^\varepsilon} (s,k) ds.
\end{align*}}
We note again that the last expression has a similar structure as $\widehat{G_{\text{sol}}^\varepsilon}(t,k)$, i.e., \eqref{Fourier_G_sol}. Therefore using
\begin{align*}
    \abs{\frac{1}{1+|k|^2} \left( k\wedge 
    \mathcal{F} (\cdot) \right) \left( 1 - \cos \left(\frac{(t-\tau)\sqrt{1+|k|^2}}{\varepsilon} \right) \right)} \leq 2 \abs{\mathcal{F} (\cdot)},
\end{align*}
we can bound $B^\varepsilon$ with the same estimates as for $G_{\text{sol}}^\varepsilon$. Thus, by \eqref{esti_norm_G_sol}, we obtain that there exists a constant $C$ depending on $\delta_0$ and $\beta$ such that
\begin{align}
\label{esti_norm_B}
    \norm{B^\varepsilon}_{\delta_0}  
    &\leq \norm{B^\varepsilon_0}_{\delta_0} 
    + \norm{\int_0^t \nabla_x  \wedge  E^\varepsilon_{\text{sol}}(s) ds}_{\delta_0} \nonumber \\
    &\leq \norm{B^\varepsilon_0}_{\delta_0} \nonumber \\
    & \quad + C \eta  \left( \sup_\Theta \left( \norm{ 
       \rho_\Theta^\varepsilon}_{\delta_0}  \norm{\xi_\Theta^\varepsilon}_{\delta_0}^2 \right) 
    +   \norm{\varepsilon E^\varepsilon}_{\delta_0} \norm{\varepsilon E_{\text{irr}}^\varepsilon}_{\delta_0}
    + \norm{j^\varepsilon  }_{\delta_0} \norm{ B^\varepsilon  }_{\delta_0} \right) \nonumber  \\
    & \quad 
    +  C \varepsilon \eta   \sup_\Theta \norm{\rho_\Theta^\varepsilon}_{\delta_0}   
    \left(
     \norm{\varepsilon E^\varepsilon}_{\delta_0} \sup_{\Theta} \norm{\xi_\Theta^\varepsilon}_{\delta_0}^2 
    + 
    \varepsilon
      \norm{ B^\varepsilon}_{\delta_0}
      \sup_{\Theta} \norm{\xi_\Theta^\varepsilon}_{\delta_0}^3
      \right)
     + \norm{\varepsilon E^\varepsilon_{\text{sol},0}}_{\delta_0}.
\end{align}

\subsubsection{Estimates on \texorpdfstring{$G_{\text{mean}}^\varepsilon, \varepsilon E^\varepsilon_{\text{mean}}$}{Lg}}  
\underline{\textbf{Estimates on \texorpdfstring{$G_{\text{mean}}^\varepsilon$}{Lg}}}:
Using formula \eqref{E_mean}, we have 
    \begin{align*}
	G_{\text{mean}}^\varepsilon(t)  &=\frac{1}{\varepsilon}\int_0^t \int_{0}^{s}  \sin\left(\frac{s-\tau}{\varepsilon} \right)
	q^\varepsilon(\tau)d\tau ds
	+ G^\varepsilon_{\text{mean},0}(t),
	\end{align*}
	where 
    \begin{equation}
    \label{Gmean0}
        \begin{aligned}
            G^\varepsilon_{\text{mean},0}(t) = \int_0^t E^\varepsilon_{\text{mean},0}(s) ds=\varepsilon E_{\text{mean}}^\varepsilon(0)\sin\left(\frac{t}{\varepsilon}\right)
	+	\varepsilon^2\partial_t E_{\text{mean}}^\varepsilon(0)
	\left(1-\cos\Bigl(\frac{t}{\varepsilon}\Bigr)\right).
        \end{aligned}
    \end{equation}
    Then, by Fubini on the double integral and recalling that $\int_0^t \int_{0}^{s} d\tau ds = \int_0^t \int_{\tau}^{t}  ds d\tau$, we get
     \begin{align}
     \label{Fourier_G_mean}
	G_{\text{mean}}^\varepsilon(t)
	&= \frac{1}{\varepsilon} \int_0^t q^\varepsilon(\tau) 
	\int_{\tau}^{t}  \sin\left(\frac{s-\tau}{\varepsilon} \right) ds d\tau
	+ G^\varepsilon_{\text{mean},0}(t) \nonumber \\
	& =\int_0^t q^\varepsilon(\tau) 
	\left( 1 - \cos \left(\frac{t-\tau}{\varepsilon} \right) \right) d\tau
	+ G^\varepsilon_{\text{mean},0}(t)  \nonumber \\
    & =: K_1(t) + K_2(t) + K_3(t) + 
    G^\varepsilon_{\text{mean},0}(t).
	\end{align}
    By the expression \eqref{sec3:qepsilon} for $q^\varepsilon$, $\{K_\ell(t)\}_{\ell=1}^3$ are given by
\begin{align*}
    K_1(t) := \frac{1}{(2\pi)^3}
    \int_0^t\left( 1 - \cos \left(\frac{t-\tau}{\varepsilon} \right) \right)
    \int_{\mathbb{T}^3_x}
    \varepsilon^2 E^\varepsilon(\tau,x) \nabla_x \cdot E_{\text{irr}}^\varepsilon(\tau,x)
     dx
	  d\tau,
\end{align*}
\begin{align*}
    K_2(t) :=\frac{1}{(2\pi)^3}\int_0^t  
    \left( 1 - \cos \left(\frac{t-\tau}{\varepsilon} \right) \right)
    \int_{\mathbb{T}^3_x} j^\varepsilon(\tau,x) \wedge B^\varepsilon(\tau,x)dx
	  d\tau,
\end{align*}
and 
\begin{align*}
    K_3(t) :=\frac{1}{(2\pi)^3}\int_0^t\left( 1 - \cos \left(\frac{t-\tau}{\varepsilon} \right) \right) 
     \widehat{R^\varepsilon}(\tau,0)
	 d\tau.
\end{align*}
In the following, being $G^\varepsilon_{\text{mean}}(t)$ spatially homogeneous, we only need to compute $\norm{G_{\text{mean}}^\varepsilon}_{L^\infty_t}$.

By using the bound 
	\begin{align*}
	\abs{\left( 1 - \cos \left(\frac{t-\tau}{\varepsilon} \right) \right)} \leq 2,
	\end{align*}
and the fact that, for a general analytic function $\varphi:[0,\eta] \times \mathbb{T}^3_x\to\mathbb{R}$,
\[
\abs{\int_{\mathbb{T}^3_x} \varphi(t,x) dx}= \abs{\widehat{\varphi}(t,0)}\le \abs{\varphi(t)}_{\delta_0}\le \norm{\varphi}_{\delta_0},
\]
we get 
\begin{equation}
\label{esti_norm_G_mean_1}
\norm{K_1}_{L^\infty_t}\le C \eta \norm{\varepsilon E^\varepsilon}_{\delta_0}\norm{\varepsilon E^\varepsilon_{\text{irr}}}_{\delta_0},\qquad
\norm{K_2}_{L^\infty_t}\le C \eta \norm{j^\varepsilon}_{\delta_0}\norm{B^\varepsilon}_{\delta_0}
\end{equation}
and
\begin{equation}
\label{esti_norm_G_mean_2}
\norm{K_3}_{L^\infty_t}\le C\varepsilon \eta 
\sup_\Theta \norm{\rho_\Theta^\varepsilon}_{\delta_0}   
     \left(
     \norm{\varepsilon E^\varepsilon}_{\delta_0} \sup_{\Theta} \norm{\xi_\Theta^\varepsilon}_{\delta_0}^2 
    + 
    \varepsilon
      \norm{ B^\varepsilon}_{\delta_0}
      \sup_{\Theta} \norm{\xi_\Theta^\varepsilon}_{\delta_0}^3
      \right),
\end{equation}
where we used the same estimates as in \eqref{esti_G_irr_4} for the last inequality.

Combining estimates \eqref{esti_norm_G_mean_1}, \eqref{esti_norm_G_mean_2}, there exists a constant $C$ depending on $\delta_0$ and $\beta$ such that
\begin{align}
\label{esti_norm_G_mean}
    \norm{G^\varepsilon_{\text{mean}}}_{L^\infty_t}  
    & \leq C \eta  \left( 
      \norm{\varepsilon E^\varepsilon}_{\delta_0} \norm{\varepsilon E_{\text{irr}}^\varepsilon}_{\delta_0}
    + \norm{j^\varepsilon  }_{\delta_0} \norm{ B^\varepsilon  }_{\delta_0} \right) \nonumber  \\
    & \quad 
    + C \varepsilon  \eta   \sup_\Theta \norm{\rho_\Theta^\varepsilon}_{\delta_0}   
    \left(
     \norm{\varepsilon E^\varepsilon}_{\delta_0} \sup_{\Theta} \norm{\xi_\Theta^\varepsilon}_{\delta_0}^2 
    + 
    \varepsilon
      \norm{ B^\varepsilon}_{\delta_0}
      \sup_{\Theta} \norm{\xi_\Theta^\varepsilon}_{\delta_0}^3
      \right)
     + \norm{G^\varepsilon_{\text{mean},0}}_{L^\infty_t}.
\end{align}
\vskip 0.5 cm
\noindent\underline{\textbf{Estimates on \texorpdfstring{$\varepsilon E_{\text{mean}}^\varepsilon$}{Lg}}}:
Recalling \eqref{sec3:spatialmean}, we have
\begin{align*}
    \varepsilon E_{\text{mean}}^\varepsilon(t) 
	= \int_{0}^{t} \sin\left(\frac{t-s}{\varepsilon} \right)
	q^\varepsilon(s)ds+ \varepsilon E_{\text{mean},0}^\varepsilon (t).
\end{align*}
We observe that the formula for $\varepsilon E_{\text{mean}}^\varepsilon(t)$  has  the same structure as the one for $G_{\text{mean}}^\varepsilon(t)$ \eqref{Fourier_G_mean}, the only difference is that
$(1 - \cos\left(\frac{t-s}{\varepsilon} \right))$ in \eqref{Fourier_G_mean} 
is now replaced by $\sin\left(\frac{t-s}{\varepsilon} \right)$. Therefore, using
\begin{align*}
    \abs{\sin \left(\frac{t-\tau}{\varepsilon} \right)} \leq 1,
\end{align*}
we can estimate $\varepsilon E_{\text{mean}}^\varepsilon$ in the same way as $ G_{\text{mean}}^\varepsilon$ and we get the same inequality as in \eqref{esti_norm_G_mean}. Namely, there exists a constant $C$ depending on $\delta_0$ and $\beta$ such that
\begin{align}
\label{esti_norm_epsilon_E_mean}
    \norm{\varepsilon E^\varepsilon_{\text{mean}}}_{L^\infty_t}  
    & \leq  C \eta  \left(
      \norm{\varepsilon E^\varepsilon}_{\delta_0} \norm{\varepsilon E_{\text{irr}}^\varepsilon}_{\delta_0}
    + \norm{j^\varepsilon  }_{\delta_0} \norm{ B^\varepsilon  }_{\delta_0} \right) \nonumber  \\
    & \quad 
    +  \varepsilon C  \eta   \sup_\Theta \norm{\rho_\Theta^\varepsilon}_{\delta_0}   
    \left(
     \norm{\varepsilon E^\varepsilon}_{\delta_0} \sup_{\Theta} \norm{\xi_\Theta^\varepsilon}_{\delta_0}^2 
    + 
    \varepsilon
      \norm{ B^\varepsilon}_{\delta_0}
      \sup_{\Theta} \norm{\xi_\Theta^\varepsilon}_{\delta_0}^3
      \right)
     + \norm{\varepsilon E^\varepsilon_{\text{mean},0}}_{L^\infty_t}.
\end{align}

We now have all the a priori estimates, independent of $\varepsilon$, for the relevant quantities: $\rho^\varepsilon_\Theta, w^\varepsilon_\Theta, G^\varepsilon,  \varepsilon E^\varepsilon$, and $B^\varepsilon$. In the next section, these estimates will allow us to construct an iterative scheme for these quantities. We will then use the a priori estimates within this scheme to demonstrate that its solutions converge to a solution of the Euler--Maxwell system \eqref{sys:EM}.

\vskip 0.5 cm
\noindent
\subsection{Iterative scheme}
\label{sec2:iterativescheme}
We now build an iterative scheme to construct the analytic solutions declared in the statement of Theorem \ref{sec3:mainthm}.
For $t \in[0,\eta]$, let
\[
\rho^{\varepsilon,(0)}_\Theta(t,x):=\rho^\varepsilon_\Theta(0,x), \quad w^{\varepsilon,(0)}_\Theta(t,x):=\xi^\varepsilon_\Theta(0,x)-G^{\varepsilon,(0)}(t,x),
\]
where $\rho^\varepsilon_\Theta(0)$ and $\xi^\varepsilon_\Theta(0)$ are given by the hypothesis of Theorem \ref{sec3:mainthm}, while 
\[G^{\varepsilon,(0)}(t,x):=G^\varepsilon_0(t,x)=G^\varepsilon_{\text{irr},0}(t,x)+G^\varepsilon_{\text{sol},0}(t,x)+ G^\varepsilon_{\text{mean},0}(t),
\]
with $G^\varepsilon_{\text{irr},0}(t,x)$, $G^\varepsilon_{\text{sol},0}(t,x)$ and $G^\varepsilon_{\text{mean},0}(t)$ defined in \eqref{Girr0}, \eqref{Gsol0} and \eqref{Gmean0}.
Then, letting 
\[E^{\varepsilon,(0)}(t,x):=\partial_tG^{\varepsilon,(0)}(t,x), \quad B^{\varepsilon,(0)}(t,x):=B^\varepsilon(0,x)-\nabla_x  \wedge  G_{\text{sol}}^{\varepsilon,(0)}(t,x),
\]
for $n \ge 0$, we let 
\begin{align*}
        \partial_t \rho^{\varepsilon,(n+1)}_\Theta &+\nabla_x \cdot \left(\rho^{\varepsilon,(n)}_\Theta v\left(w^{\varepsilon,(n)}_\Theta+G^{\varepsilon,(n)}\right) \right)=0\\
        \partial_t w^{\varepsilon,(n+1)}_\Theta&+\left[v\left(w^{\varepsilon,(n)}_\Theta+G^{\varepsilon,(n)}\right)\cdot\nabla_x\right]\left(w^{\varepsilon,(n)}_\Theta+G^{\varepsilon,(n)}\right)=v\left(w^{\varepsilon,(n)}_\Theta+G^{\varepsilon,(n)}\right)\wedge B^{\varepsilon,(n)},
\end{align*}
with $\rho^{\varepsilon, (n+1)}_\Theta(0)\equiv\rho^{\varepsilon, (n)}_\Theta(0)$ and $w^{\varepsilon, (n+1)}_\Theta(0)\equiv w^{\varepsilon, (n)}_\Theta(0)$, and where
\begin{align*}
    G^{\varepsilon, (n+1)}&=G_{\text{irr}}^{\varepsilon,(n+1)}+G_{\text{sol}}^{\varepsilon,(n+1)}+ 
    G^{\varepsilon, (n+1)}_{\text{mean}}
\end{align*}
with the terms in the decomposition given by
\begin{align*}
	\widehat{G_{\text{irr}}^{\varepsilon,(n+1)}}(t,k)  &= -\int_0^t \int_{0}^{s} \frac{\imm k}{\varepsilon \abs{k}^2} \sin\left(\frac{s-\tau}{\varepsilon} \right)
	\widehat{g^{\varepsilon,(n)}}(\tau,k)d\tau ds
	+ \widehat{G^{\varepsilon}_{\text{irr},0}}(t,k),
\end{align*}
\begin{align*}
	\widehat{G_{\text{sol}}^{\varepsilon,(n+1)}}(t,k)  &= \int_0^t \int_{0}^{s} \frac{\imm k\wedge\widehat{h^{\varepsilon,(n)}}(\tau,k)}{\varepsilon \abs{k}^2\sqrt{1+|k|^2}} \sin\left(\frac{\sqrt{1+|k|^2}(s-\tau)}{\varepsilon} \right)
	d\tau ds
	+ \widehat{G^{\varepsilon}_{\text{sol},0}}(t,k),
\end{align*}
and
\begin{align*}
	G_{\text{mean}}^{\varepsilon,(n+1)}(t)  &= \int_0^t \int_{0}^{s} \frac{1}{\varepsilon} \sin\left(\frac{s-\tau}{\varepsilon} \right)
	q^{\varepsilon,(n)}(\tau)d\tau ds
	+ G^{\varepsilon}_{\text{mean},0}(t),
\end{align*}
while
\begin{align*}
    &g^{\varepsilon,(n)} (t,x) 
    := \partial_{x_j} \partial_{x_i}
        \int_{M} \rho_\Theta^{\varepsilon,(n)}(t,x) v\left(\xi^{\varepsilon, (n)}
        _\Theta\right)_i v\left(\xi^{\varepsilon,(n)}_\Theta\right)_j d\mu(\Theta)\\
        & \qquad- \varepsilon^2 \nabla_x \cdot \big(E^{\varepsilon,(n)}(t,x) \nabla_x \cdot E_{\text{irr}}^{\varepsilon,(n)}(t,x)\big) -\nabla_x \cdot \left(j^{\varepsilon,(n)}(t,x)  \wedge B^{\varepsilon,(n)}(t,x)\right)
    -\nabla_x \cdot R^{\varepsilon,(n)}(t,x),
    \end{align*}
     \begin{align*}
    &h^{\varepsilon, (n)} (t,x) 
    := \nabla_x  \wedge \left( \partial_{x_i}
        \int_{\br^3} \rho^{\varepsilon,(n)}_\Theta v\left(\xi^{\varepsilon,(n)}_\Theta\right)_i v\left(\xi^{\varepsilon,(n)}_\Theta\right) d \xi\right)\\
    & \quad - \varepsilon^2 \nabla_x  \wedge  \big(E^{\varepsilon, (n)}(t,x) \nabla_x \cdot E_{\text{irr}}^{\varepsilon,(n)}(t,x)\big) \nonumber -\nabla_x  \wedge (j^{\varepsilon,(n)}(t,x) \wedge B^{\varepsilon,(n)}(t,x))
    -\nabla_x  \wedge  R^{\varepsilon,(n)}(t,x),
    \end{align*}
    and
    \begin{align*}
    q^{\varepsilon,(n)} (t) 
    &:=\frac{1}{(2\pi)^3} \int_{\mathbb{T}^3_x}\varepsilon E^{\varepsilon,(n)}(t,x) \nabla_x \cdot \big(\varepsilon E_{\text{irr}}^{\varepsilon,(n)}(t,x)\big)dx\\
    &\qquad+\frac{1}{(2\pi)^3}\int_{\mathbb{T}^3_x}j^{\varepsilon,(n)}(t,x)  \wedge B^{\varepsilon,(n)}(t,x)dx
    + \frac{1}{(2\pi)^3} \widehat{R^{\varepsilon,(n)}}(t,0),
    \end{align*}
and where $R^{\varepsilon, (n)}(t,x)$ is defined as in \eqref{sec3:defR}, while
\begin{align*}
    j^{\varepsilon,(n)}(t,x) = \int_M v\left(\xi^{\varepsilon,(n)}
    _\Theta (t,x) \right)  \rho^{\varepsilon,(n)}_\Theta (t,x) d\mu(\Theta),
\end{align*}
 and
 \begin{align*}
    E^{\varepsilon, (n+1)}&=E^{\varepsilon, (n+1)}_{\text{irr}}+E^{\varepsilon, (n+1)}_{\text{sol}}+E^{\varepsilon, (n+1)}_{\text{mean}}=\partial_tG_{\text{irr}}^{\varepsilon,(n+1)}+\partial_tG_{\text{sol}}^{\varepsilon,(n+1)}+ \partial_tG_{\text{mean}}^{\varepsilon,(n+1)},\\
    B^{\varepsilon,(n+1)}&=B^{\varepsilon}(0) -\nabla_x  \wedge  G_{\text{sol}}^{\varepsilon,(n+1)}.
\end{align*}
\subsubsection{Estimates on the iterations}

\begin{lemma}
\label{lemma_iterative_scheme}
There exist  constants $C_1$ and $C_2$ independent of $\varepsilon$ and $\Theta$ such that, for $\eta$ sufficiently small and independent of $\varepsilon$ and $\Theta$, it holds that:
\begin{enumerate}
    \item \label{lemma3:pt1}
For $n \geq 0$
\begin{align}
    \label{bound_step_n}
        \max \left\{\norm{\rho^{\varepsilon,(n)}_\Theta}_{\delta_0}, \norm{w^{\varepsilon,(n)}_\Theta }_{\delta_0}, \norm{G^{\varepsilon,(n)}}_{\delta_0}, \norm{\varepsilon E^{\varepsilon,(n)}}_{\delta_0}, \norm{B^{\varepsilon,(n)}}_{\delta_0}   \right\} \leq C_1.
    \end{align}
    \item \label{lemma3:pt2}
Moreover, for $n \geq 1$,
    \begin{align}
    \label{bound_step_n_diff}
    &\max \left\{ 
    \norm{\rho^{\varepsilon,(n)}_\Theta - \rho^{\varepsilon,(n-1)}_\Theta}_{\delta_0},
    \norm{w^{\varepsilon,(n)}_\Theta - w^{\varepsilon,(n-1)}_\Theta}_{\delta_0}, 
    \norm{G^{\varepsilon, (n)} - G^{\varepsilon, (n-1)}}_{\delta_0}, \right. \nonumber \\
    & \qquad \qquad \left.
    \norm{\varepsilon E^{\varepsilon, (n)} - \varepsilon E^{\varepsilon, (n-1)}}_{\delta_0},
    \norm{B^{\varepsilon, (n)} -  B^{\varepsilon, (n-1)}}_{\delta_0}
    \right\}
    \leq \frac{C_2}{2^{n}}.
    \end{align}
    \end{enumerate}
\end{lemma}

\vskip 0.5 cm
\noindent\emph{Proof of part \eqref{lemma3:pt1} of Lemma \ref{lemma_iterative_scheme}}:
We prove the lemma using an inductive procedure, beginning with the case $n = 0$.
By the assumptions of Theorem \ref{sec3:mainthm}, we know there exists a constant $C_0$ independent of $\Theta$ and $\varepsilon$ such that $\norm{\rho^{\varepsilon}_\Theta (0)}_{\delta_0}\le C_0$ and $\norm{\xi^{\varepsilon}_\Theta (0) }_{\delta_0}\le C_0$.
Since by definition $\rho^{\varepsilon,(0)}(t,x)\equiv\rho^\varepsilon(0,x)$, this immediately implies that $\norm{\rho^{\varepsilon}_\Theta (0)}_{\delta_0}\le C_0$.

We now show how to obtain a bound on $\norm{G^{\varepsilon,(0)}}_{\delta_0}$, where we remind that
\begin{align}
\label{sec3:remindG0}
\widehat{G^{\varepsilon,(0)}}&(t,k)=\varepsilon\widehat{E_{\text{irr}}^\varepsilon}(0,k)\sin\left(\frac{t}{\varepsilon}\right)
	+	\varepsilon^2\widehat{\partial_t E_{\text{irr}}^\varepsilon}(0,k)
	\left(1-\cos\left(\frac{t}{\varepsilon}\right)\right)  \nonumber \\
    &+\frac{\varepsilon}{\sqrt{1+|k|^2}}\widehat{E_{\text{sol}}^\varepsilon}(0,k)\sin\left(\frac{t\sqrt{1+|k|^2}}{\varepsilon}\right)
	+\frac{\varepsilon^2}{1+|k|^2}\widehat{\partial_t E_{\text{sol}}^\varepsilon}(0,k)
	\left(1-\cos\left(\frac{t \sqrt{1+|k|^2}}{\varepsilon}\right)\right) \nonumber\\
    & +\varepsilon (2\pi)^3 E_{\text{mean}}^\varepsilon(0)\sin\left(\frac{t}{\varepsilon}\right)\mathbf{1}_{k =0}
	+	\varepsilon^2 (2\pi)^3\partial_t E_{\text{mean}}^\varepsilon(0)
	\left(1-\cos\left(\frac{t}{\varepsilon}\right)\right)\mathbf{1}_{k=0}.
\end{align}
First, note that by recalling the expression of $\varepsilon E^\varepsilon_{\text{irr}}(0)$ in \eqref{sec2:irrtime0} and using the quasineutrality bound \eqref{assumption_quasineutral_2}, we get
\begin{align}
\label{esti_E_irr_0}
    \norm{\varepsilon E_{\text{irr}}^\varepsilon(0)}_{\delta_0} \le \frac{\norm{\rho^{\varepsilon}(0)-1}_{\delta_0}}{\varepsilon} \leq C_0.
\end{align}
Next, by the expression of $\varepsilon^2 \partial_t E_{\text{irr}}(0)$ in \eqref{sec2:irrtime0}, recalling that $j^\varepsilon(0) = \int_M \rho^\varepsilon_\Theta (0) v(\xi^\varepsilon_\Theta (0)) \mu(d \Theta)$ and by using inequality \eqref{vxiteta} for the relativistic velocity, we obtain
\begin{align}
\label{esti_dt_E_irr_0}
    \norm{\varepsilon^2 \partial_t E_{\text{irr}}^\varepsilon(0)}_{\delta_0} \le 
    \norm{j^\varepsilon(0)}_{\delta_0} 
    \le
    \sup_\Theta\left(
    \norm{\rho^{\varepsilon}_\Theta(0)}_{\delta_0} \norm{v(\xi^{\varepsilon}_\Theta(0))}_{\delta_0} \right) \leq C_0,
\end{align}
for some constant $C_0$ independent of $\Theta$ and $\varepsilon$.
For the solenoidal part, we use formula \eqref{sec2:soltime0} to get
\begin{align}
\label{esti_E_sol_0}
    \norm{\varepsilon E_{\text{sol}}^\varepsilon (0)}_{\delta_0} \leq  \norm{\varepsilon \partial_t B^\varepsilon_0}_{\delta_0} = \norm{\varepsilon \nabla_x \wedge E^\varepsilon_0}_{\delta_0} \leq C_0,
\end{align}
for some constant $C_0$ independent of $\Theta$ and $\varepsilon$,
where we used that
$
 \partial_t B^\varepsilon_0=-\nabla_x \wedge E^\varepsilon_0
$
and the initial assumption \eqref{assumption_quasineutral}.
Moreover, again by formula \eqref{sec2:soltime0}, we have 
\begin{align}
\label{esti_dt_E_sol_0}
    \norm{\varepsilon^2 \partial_t E^\varepsilon_{\text{sol}}(0)}_{\delta_0} \leq \norm{\nabla_x  \wedge  B^\varepsilon(0)}_{\delta_0}+ 
    \sup_\Theta 
    \left( 
    \norm{\rho^{\varepsilon}_\Theta(0)}_{\delta_0} \norm{v (\xi^{\varepsilon}_\Theta(0))}_{\delta_0}
    \right)
    \le C_0,
\end{align}
where we used the initial assumption \eqref{assumption_quasineutral}.
Finally, we study the spatial mean of $E^\varepsilon_0$. By the definition in \eqref{sec2:meantime0} and the initial assumption \eqref{assumption_quasineutral}, we get
\begin{align}
\label{esti_E_mean_0}
    |\varepsilon E_{\text{mean}}^\varepsilon(0)| \le C_0.
\end{align}
Next, by the expression of $\varepsilon^2 \partial_t E^\varepsilon_{\text{mean}}$ in \eqref{sec2:meantime0}, by using inequality \eqref{vxiteta} for the relativistic velocity, we obtain
\begin{align}
\label{esti_dt_E_mean_0}
    \left|\varepsilon^2 \partial_t E_{\text{mean}}^\varepsilon(0)\right| \le 
    \norm{j^\varepsilon(0)}_{\delta_0} 
    \le
    \sup_\Theta\left(
    \norm{\rho^{\varepsilon}_\Theta(0)}_{\delta_0} \norm{v(\xi^{\varepsilon}_\Theta(0))}_{\delta_0} \right) \leq C_0.
\end{align}

Therefore, by recalling \eqref{sec3:remindG0} and collecting the estimates in \eqref{esti_E_irr_0}, \eqref{esti_dt_E_irr_0}, \eqref{esti_E_sol_0}, \eqref{esti_dt_E_sol_0}, \eqref{esti_E_mean_0} and \eqref{esti_dt_E_mean_0} we get
\begin{align*}
    \norm{G^{\varepsilon,(0)}}_{\delta_0} 
    & \leq \norm{\varepsilon E_{\text{irr}}^\varepsilon(0)}_{\delta_0}  
    +  \norm{\varepsilon^2 \partial_t E_{\text{irr}}^\varepsilon(0)}_{\delta_0}
    + \norm{\varepsilon E_{\text{sol}}^\varepsilon(0)}_{\delta_0}  
    +  \norm{\varepsilon^2 \partial_t E_{\text{sol}}^\varepsilon(0)}_{\delta_0}\\
    &\quad+ C|\varepsilon E_{\text{mean}}^\varepsilon(0)| 
    +  C|\varepsilon^2 \partial_t E_{\text{mean}}^\varepsilon(0)|\le C_0,
\end{align*}
for some constant $C_0$ independent of $\Theta$ and $\varepsilon$.
Similarly, since $E^{\varepsilon,(0)}=\partial_t G^{\varepsilon,(0)}$, we get that $\norm{\varepsilon E^{\varepsilon,(0)}}_{\delta_0}$ is also uniformly bounded by a constant $C_0$ independent of $\Theta$ and $\varepsilon$. 

Moreover, by the bound on $\norm{G^{\varepsilon,(0)}}_{\delta_0}$ and since $w^{\varepsilon,(0)}_\Theta=\xi^\varepsilon_\Theta(0)-G^{\varepsilon,(0)}$, we deduce that $\norm{w^{\varepsilon,(0)}_\Theta }_{\delta_0}$ is also bounded by $C_0$.

Finally, we deduce the same estimate on $\norm{B^{\varepsilon,(0)}}_{\delta_0}$ by noticing that
\begin{align*}
        \nabla_x  \wedge  G_{\text{sol}}^{\varepsilon,(0)}(t,x)
            &=\frac{\varepsilon \imm k}{\sqrt{1+|k|^2}} \wedge \widehat{E_{\text{sol}}^\varepsilon}(0,k)\sin\left(\frac{t\sqrt{1+|k|^2}}{\varepsilon}\right) \\
	& \quad+\frac{\varepsilon^2 \imm k}{1+|k|^2} \wedge \widehat{\partial_t E_{\text{sol}}^\varepsilon}(0,k)
	\left(1-\cos\left(\frac{t \sqrt{1+|k|^2}}{\varepsilon}\right)\right),
\end{align*}
and since  $\frac{\varepsilon \imm k}{\sqrt{1+|k|^2}}$ and $\frac{\varepsilon^2 \imm k}{1+|k|^2}$ are bounded in $k$, we can apply the same analysis as before.

Now, suppose for $n \geq 1$, the bounds \eqref{bound_step_n} are true for a constant $C_1$ defined by $C_1 := 4C_0$. Then, we show by induction that the same bounds hold for $(n+1)$. Note that by the induction hypothesis at step $(n)$, the assumption in Lemma \ref{app:lemmarem} is satisfied, and we can therefore use the a priori estimates from Section \ref{sec_esti_rho_w} and Section \ref{sec_esti_G_B}.

By \eqref{sec3_est_rho}, we have
\begin{align*}
    \norm{\rho^{\varepsilon,(n+1)}_\Theta}_{\delta_0}
    & \le
    \norm{\rho^{\varepsilon,(n)}_\Theta(0)}_{\delta_0}
    +
     C \eta \norm{\rho^{\varepsilon,(n)}_\Theta}_{\delta_0}\left(\norm{w^{\varepsilon,(n)}_\Theta}_{\delta_0}+\norm{G^{\varepsilon,(n)}}_{\delta_0} \right) \\
    & \leq C_0 + 2  C C_1^2 \eta \leq C_1,
\end{align*}
where the last inequality follows by choosing $\eta$ sufficiently small.
Similarly, by \eqref{sec3:est_w} and by taking $\eta$ sufficiently small, we get
{\begin{align*}
    \norm{w^{\varepsilon,(n+1)}_\Theta}_{\delta_0} 
    &\le 
    \norm{w^{\varepsilon,(n)}_\Theta(0)}_{\delta_0} +C \eta \left(\norm{w^{\varepsilon,(n)}_\Theta}_{\delta_0}+\norm{G^{\varepsilon,(n)}}_{\delta_0} \right)
    \left(
\norm{w^{\varepsilon,(n)}_\Theta}_{\delta_0}+\norm{G^{\varepsilon,(n)}}_{\delta_0}
+ \norm{B^{\varepsilon,(n)}}_{\delta_0}
    \right) \\
    & \leq C_0 + 6  C C_1^2 \eta \leq C_1.
\end{align*}}
Next, recalling that $G^{\varepsilon,(n+1)}=G_{\text{irr}}^{\varepsilon,(n+1)}+G_{\text{sol}}^{\varepsilon,(n+1)}+G^{\varepsilon,(n+1)}_{\text{mean}}$ and using the a priori estimates \eqref{esti_norm_G_irr}, \eqref{esti_norm_G_sol} and \eqref{esti_norm_G_mean}, we obtain
\begin{align*}
    \norm{G^{\varepsilon, (n+1)}}_{\delta_0} & \leq  C \eta  \left[ \sup_\Theta \left( \norm{ 
       \rho_\Theta^{\varepsilon,(n)}}_{\delta_0}  \norm{\xi_\Theta^{\varepsilon,(n)}}_{\delta_0}^2 \right)
    +  \ \norm{\varepsilon E^{\varepsilon,(n)}}_{\delta_0} \norm{\varepsilon E_{\text{irr}}^{\varepsilon,(n)}}_{\delta_0} + \norm{j^{\varepsilon,(n)}  }_{\delta_0} \norm{ B^{\varepsilon,(n)}  }_{\delta_0} \right] \nonumber  \\
    & \quad 
    +    \varepsilon C  \eta   \sup_\Theta \norm{\rho_\Theta^{\varepsilon,(n)}}_{\delta_0}   
    \left(
     \norm{\varepsilon E^{\varepsilon,(n)}}_{\delta_0} \sup_{\Theta} \norm{\xi_\Theta^{\varepsilon,(n)}}_{\delta_0}^2 
    + 
    \varepsilon
      \norm{ B^{\varepsilon,(n)}}_{\delta_0}
      \sup_{\Theta} \norm{\xi_\Theta^{\varepsilon,(n)}}_{\delta_0}^3
      \right) \\
     & \quad + \norm{G^\varepsilon_{\text{irr},0}}_{\delta_0}
     + \norm{G^\varepsilon_{\text{sol},0}}_{\delta_0}+ \norm{G^\varepsilon_{\text{mean},0}}_{L^\infty_t}.
\end{align*}
Hence, by the fact that $\norm{j^{\varepsilon,(n)}  }_{\delta_0} \leq C C_1^2$ and by choosing $\eta$ sufficiently small, 
 \[
 \norm{G^{\varepsilon, (n+1)}}_{\delta_0} \le  C (C_1^3 + C_1^2 + C_1^3)\eta +  \varepsilon C  (C_1^4 + \varepsilon C_1^4) \eta + 3 C_0 \leq C_1.
 \]
Similarly, using \eqref{esti_norm_epsilon_E_irr}, \eqref{esti_norm_epsilon_E_sol} and \eqref{esti_norm_epsilon_E_mean}, we get
\begin{align*}
    \norm{\varepsilon E^{\varepsilon, (n+1)}}_{\delta_0} & \leq \norm{\varepsilon E_{\text{irr}}^{\varepsilon,(n+1)}}_{\delta_0} + \norm{\varepsilon E_{\text{sol}}^{\varepsilon,(n+1)}}_{\delta_0}+ \norm{\varepsilon E_{\text{mean}}^{\varepsilon,(n+1)}}_{L^\infty_t}  \leq C_1.
\end{align*}
And finally, using \eqref{esti_norm_B}, we have
\begin{align*}
    \norm{B^{\varepsilon, (n+1)}}_{\delta_0} & \leq C_1.
\end{align*}
This concludes the proof of the first part of the lemma, that is \eqref{bound_step_n}.\\
\vskip 0.5 cm
\noindent\emph{Proof of part \eqref{lemma3:pt2} of Lemma \ref{lemma_iterative_scheme}}: For $n=1$, by choosing $C_2 := 8 C_1$ and by triangle inequality, we have
\begin{align*}
    \max & \left \{
    \norm{\rho^{\varepsilon,(1)}_\Theta - \rho^{\varepsilon,(0)}_\Theta}_{\delta_0}, \norm{w^{\varepsilon,(1)}_\Theta - w^{\varepsilon,(0)}_\Theta}_{\delta_0},
    \norm{G^{\varepsilon,(1)} - G^{\varepsilon,(0)}}_{\delta_0}, \right. \\
    & \left. \qquad \norm{\varepsilon E^{\varepsilon,(1)} - \varepsilon E^{\varepsilon,(0)}}_{\delta_0},
    \norm{B^{\varepsilon,(1)} - B^{\varepsilon,(0)}}_{\delta_0}
    \right \}
    \leq 2 C_1 \leq\frac{C_2}{2}.
\end{align*}
Therefore, we conclude that \eqref{bound_step_n_diff} is true for $n=1$. Now, we assume that \eqref{bound_step_n_diff}  is true for step $(n)$, and we show by induction that the same estimates hold for $(n+1)$.
For $n\geq 2$, the continuity equation for the difference 
$\left(\rho^{\varepsilon,(n+1)}_\Theta - \rho^{\varepsilon,(n)} \right)$ is given by
\begin{align*}
    \partial_t \left( \rho^{\varepsilon,(n+1)}_\Theta - \rho^{\varepsilon,(n)} \right) + \nabla_x \cdot \left( \rho^{\varepsilon,(n)}_\Theta v\left(w^{\varepsilon,(n)}_\Theta+G^{\varepsilon,(n)}\right) - \rho^{\varepsilon,(n-1)}_\Theta v\left(w^{\varepsilon,(n-1)}_\Theta+G^{\varepsilon,(n-1)}\right) \right) = 0.
\end{align*}
We estimate the norm of the difference by adding and subtracting the mixed term, that is
\begin{align*}
    \partial_t \left( \rho^{\varepsilon,(n+1)}_\Theta - \rho^{\varepsilon,(n)} \right) + \nabla_x \cdot \left( \rho^{\varepsilon,(n)}_\Theta v\left(w^{\varepsilon,(n)}_\Theta+G^{\varepsilon,(n)}\right) - \rho^{\varepsilon,(n)}_\Theta v\left(w^{\varepsilon,(n-1)}_\Theta+G^{\varepsilon,(n-1)}\right) \right)  \\
    + \nabla_x \cdot \left(  \rho^{\varepsilon,(n)}_\Theta v\left(w^{\varepsilon,(n-1)}_\Theta+G^{\varepsilon,(n-1)}\right) 
    - \rho^{\varepsilon,(n-1)}_\Theta v\left(w^{\varepsilon,(n-1)}_\Theta+G^{\varepsilon,(n-1)}\right) \right) = 0.
\end{align*}
By applying similar inequality than the a priori estimate \eqref{sec3_est_rho} and by triangle inequality, we get 
\begin{align*}
    \norm{\rho^{\varepsilon,(n+1)}_\Theta - \rho^{\varepsilon,(n)}_\Theta}_{\delta_0} 
    & \le C \eta  \norm{\rho^{\varepsilon,(n)}_\Theta}_{\delta_0} \norm{v(w^{\varepsilon,(n)}_\Theta + G^{\varepsilon,(n)}) - v(w^{\varepsilon,(n-1)}_\Theta + G^{\varepsilon,(n-1)}) }_{\delta_0}  \\
    &\quad + C \eta  \norm{\rho^{\varepsilon,(n)}_\Theta - \rho^{\varepsilon,(n-1)}_\Theta}_{\delta_0} \norm{v(w^{\varepsilon,(n-1)}_\Theta + G^{\varepsilon,(n-1)})}_{\delta_0}.
\end{align*}
Then, using the two relativistic inequalities \eqref{vxiteta}, \eqref{Lemma_remainder_2} and the induction hypothesis \eqref{bound_step_n_diff}, we have
\begin{align*}
    \norm{\rho^{\varepsilon,(n+1)}_\Theta - \rho^{\varepsilon,(n)}_\Theta}_{\delta_0} 
    & \le C \eta   \norm{\rho^{\varepsilon,(n)}_\Theta}_{\delta_0} \left( \norm{w^{\varepsilon,(n)}_\Theta - w^{\varepsilon,(n-1)}_\Theta}_{\delta_0} + \norm{  G^{\varepsilon,(n)}  - G^{\varepsilon,(n-1)} }_{\delta_0} \right)  \\
    & \quad + C \eta  \norm{\rho^{\varepsilon,(n)}_\Theta - \rho^{\varepsilon,(n-1)}_\Theta}_{\delta_0} \left( 
    \norm{w^{\varepsilon,(n-1)}_\Theta}_{\delta_0} +  \norm{ G^{\varepsilon,(n-1)}}_{\delta_0} \right) \\
    & \leq 2 C    C_1 \frac{C_2}{2^n}  \eta
    + 2   C C_1 \frac{C_2}{2^n} \eta \leq \frac{C_2}{2^{n+1}},
\end{align*}
where we used $\eta$ sufficiently small.

Next, we estimate the difference for $w^{\varepsilon}_\Theta$. The equation for the difference is given by
\begin{align*}
        \partial_t \left( w^{\varepsilon,(n+1)}_\Theta - w^{\varepsilon,(n)}_\Theta \right) &+\left[v\left(w^{\varepsilon,(n)}_\Theta+G^{\varepsilon,(n)}\right)\cdot\nabla_x\right]\left(w^{\varepsilon,(n)}_\Theta+G^{\varepsilon,(n)}\right) \\
        & - \left[v\left(w^{\varepsilon,(n-1)}_\Theta+G^{\varepsilon,(n-1)}\right)\cdot\nabla_x\right]\left(w^{\varepsilon,(n-1)}_\Theta+G^{\varepsilon,(n-1)}\right) \\
        & =v\left(w^{\varepsilon,(n)}_\Theta+G^{\varepsilon,(n)}\right)\wedge B^{\varepsilon,(n)}- v\left(w^{\varepsilon,(n-1)}_\Theta+G^{\varepsilon,(n-1)}\right)\wedge B^{\varepsilon,(n-1)}.
\end{align*}
We proceed as we did for $\rho_\Theta^\varepsilon$: we add and subtract the mixed terms in the second term of the r.h.s and also in the l.h.s, then we use the a priori estimate \eqref{sec3:est_w} and the triangle inequality. Hence, we deduce
\begin{align*}
    \norm{w^{\varepsilon,(n+1)}_\Theta - w^{\varepsilon,(n)}_\Theta }_{\delta_0} 
    & \leq \frac{C_2}{2^{n+1}}.
\end{align*}
Finally, we estimate the difference for $G^{\varepsilon}$. First, we have
\begin{align*}
    \norm{G^{\varepsilon, (n+1)} - G^{\varepsilon, (n)}}_{\delta_0} 
    & \leq \norm{G_{\text{irr}}^{\varepsilon,(n+1)} - G_{\text{irr}}^{\varepsilon,(n)}}_{\delta_0} + \norm{G_{\text{sol}}^{\varepsilon,(n+1)} - G_{\text{sol}}^{\varepsilon,(n)}}_{\delta_0}+  \norm{G_{\text{mean}}^{\varepsilon,(n+1)} - G_{\text{mean}}^{\varepsilon,(n)}}_{L^\infty_t}.
\end{align*}
We have to estimate the norms on the r.h.s, so we start with the irrational term. By definition, the difference is given by
\begin{align*}
    \widehat{G_{\text{irr}}^{\varepsilon,(n+1)}}(t,k) - \widehat{G_{\text{irr}}^{\varepsilon,(n)}}(t,k) = \int_0^t \frac{-\imm k}{ \abs{k}^2} 
    \left( \widehat{g^{\varepsilon,(n)}}(\tau,k) - \widehat{g^{\varepsilon,(n-1)}}(\tau,k) \right)
	\left( 1 - \cos \left(\frac{t-\tau}{\varepsilon} \right) \right) d\tau.
\end{align*}
We split this term as in \eqref{Fourier_G_irr} and using the same notation, we get
\begin{align*}
    G_{\text{irr}}^{\varepsilon,(n+1)} - G_{\text{irr}}^{\varepsilon,(n)}
    & =: I_1^{(n+1)} - I_1^{(n)}  + I_2^{(n+1)} - I_2^{(n)} + I_3^{(n+1)} - I_3^{(n)} + I_4^{(n+1)} - I_4^{(n)}.
\end{align*}
Next, we estimate the norm of each $\left( I_j^{(n+1)}- I_j^{(n)}\right)$. We use the same analysis as in Section \ref{sec_esti_G_B}. We proceed as for $\rho_\Theta^\varepsilon$ above: for $I_1$ we add and subtract the mixed term and apply the a priori estimate \eqref{esti_norm_G_irr_1}. Therefore, we obtain
\begin{align}
\label{esti_norm_diff_G_irr_1}
    \norm{I_1^{(n+1)} - I_1^{(n)}}_{\delta_0} 
    & \leq C\eta \sup_\Theta \left( \norm{\rho^{\varepsilon,(n)}_\Theta - \rho^{\varepsilon,(n-1)}_\Theta}_{\delta_0} \norm{\xi^{\varepsilon,(n)}_\Theta}_{\delta_0}^2 \right) \nonumber\\
    & \quad + C \eta  \sup_\Theta \left( \norm{\rho^{\varepsilon,(n)}_\Theta}_{\delta_0} \norm{\xi^{\varepsilon,(n)}_\Theta - \xi^{\varepsilon,(n-1)}_\Theta}_{\delta_0}\norm{\xi^{\varepsilon,(n)}_\Theta + \xi^{\varepsilon,(n-1)}_\Theta}_{\delta_0} \right) \nonumber\\
    & \leq  4C C_1^2 \frac{C_2}{2^{n}} \eta + 8 C C_1^2 \frac{C_2}{2^n} \eta
    \leq \frac{C_2}{12 \cdot2^{n+1}},
\end{align}
where, for the penultimate inequality, we used that  $\xi^\varepsilon_\Theta  = w^\varepsilon_\Theta + G^\varepsilon$, the induction hypothesis and we took $\eta$ sufficiently small. Note that here, we explicitly required the constant to be $\frac{C_2}{12 \cdot 2^{n+1}}$ in order to be able to sum all the terms and finally get the right $\frac{C_2}{3 \cdot 2^{n+1}}$ for $\norm{G^{\varepsilon, (n+1)}_{\text{irr}} - G^{\varepsilon, (n)}_{\text{irr}}}_{\delta_0}$.

Then, for $I_2$, we add and subtract the mixed term and we employ the a priori estimate \eqref{esti_norm_G_irr_2}. By the induction hypothesis and taking $\eta$ sufficiently small, we obtain
\begin{align}
\label{esti_norm_diff_G_irr_2}
    \norm{I_2^{(n+1)} - I_2^{(n)}}_{\delta_0} 
    & \leq C\eta   \norm{\varepsilon E^{\varepsilon,(n)}_{\text{irr}} - \varepsilon E^{\varepsilon,(n-1)}_{\text{irr}}}_{\delta_0} \norm{\varepsilon E^{\varepsilon,(n)}}_{\delta_0} \nonumber \quad + C\eta  \norm{\varepsilon E^{\varepsilon,(n-1)}_{\text{irr}}}_{\delta_0}  \norm{\varepsilon E^{\varepsilon,(n)} - \varepsilon E^{\varepsilon,(n-1)}}_{\delta_0} \nonumber \\
    & \leq 2C  C_1 \frac{C_2}{2^n}\eta \leq \frac{C_2}{12 \cdot 2^{n+1}}.
\end{align}
For $I_3$, first note that by definition of $j^\varepsilon$ we have $\norm{j^{\varepsilon,(n)}  }_{\delta_0} \leq C_1^2$ and moreover by the relativistic inequality \eqref{vxiteta} and \eqref{Lemma_remainder_2} we also get
\begin{align*}
    \norm{j^{\varepsilon,(n)} - j^{\varepsilon,(n-1)}}_{\delta_0} 
    & \leq \norm{\rho^{\varepsilon,(n)}_\Theta}_{\delta_0} \norm{v(\xi^{\varepsilon,(n)}_\Theta) - v(\xi^{\varepsilon,(n-1)}_\Theta)}_{\delta_0} 
    + \norm{v(\xi^{\varepsilon,(n-1)}_\Theta)}_{\delta_0} \norm{\rho^{\varepsilon,(n)}_\Theta - \rho^{\varepsilon,(n-1)}_\Theta}_{\delta_0} \\
    & \leq  C C_1 \frac{C_2}{2^n} + C C_1 \frac{C_2}{2^n}.
\end{align*}
Therefore, by adding  and subtracting  the mixed term, using the a priori estimate \eqref{esti_norm_G_irr_3}, the induction hypothesis and taking $\eta$ sufficiently small, we deduce 
\begin{align}
\label{esti_norm_diff_G_irr_3}
    \norm{I_3^{(n+1)} - I_3^{(n)}}_{\delta_0} 
    & \leq C\eta  \norm{j^{\varepsilon,(n)}}_{\delta_0} \norm{B^{\varepsilon,(n)} - B^{\varepsilon,(n-1)}}_{\delta_0} + C\eta  \norm{j^{\varepsilon,(n)} - j^{\varepsilon,(n-1)}}_{\delta_0}
    \norm{ B^{\varepsilon,(n-1)}}_{\delta_0} \nonumber \\
    & \leq  C C_1^2 \frac{C_2}{2^n}\eta + C  \left( C C_1 \frac{C_2}{2^n} + C C_1 \frac{C_2}{2^n}  \right) \eta\leq \frac{C_2}{12 \cdot 2^{n+1}}.
\end{align}
Similarly, by using the a priori estimate \eqref{esti_norm_G_irr_4} and recalling that $\lambda(\xi)=\nabla_\xi\left(v(\xi)-\xi\right)$, we have
\begin{align*}
    \norm{I_4^{(n+1)} - I_4^{(n)}}_{\delta_0} 
    &\leq C\eta  \norm{\rho^{\varepsilon,(n)}_\Theta - \rho^{\varepsilon,(n-1)}_\Theta}_{\delta_0} \left(
    \norm{E^{\varepsilon,(n)}}_{\delta_0}  \norm{\lambda(\xi_\Theta^{\varepsilon,(n)})}_{\delta_0}\right. \nonumber\\
    &\left.\qquad +  \norm{v(\xi_\Theta^{\varepsilon,(n)})}_{\delta_0}
     \norm{B^{\varepsilon,(n)}}_{\delta_0} 
      \norm{\lambda(\xi_\Theta^{\varepsilon,(n)})}_{\delta_0}
     \right) \nonumber \\
     & \quad + C\eta  \norm{\rho^{\varepsilon,(n-1)}_\Theta}_{\delta_0} 
     \norm{\lambda(\xi_\Theta^{\varepsilon,(n)})}_{\delta_0} \nonumber \\
     & \qquad \times \left( \norm{ E^{\varepsilon,(n)} -  E^{\varepsilon,(n-1)}} 
     + \norm{ v(\xi_\Theta^{\varepsilon,(n)}) \wedge B^{\varepsilon,(n)} - v(\xi_\Theta^{\varepsilon,(n-1)}) \wedge B^{\varepsilon,(n-1)} }_{\delta_0}
     \right) \nonumber \\
     & \quad + C\eta  \norm{\rho^{\varepsilon,(n-1)}_\Theta}_{\delta_0} \left( \norm{  E^{\varepsilon,(n-1)}} 
     + \norm{ v(\xi_\Theta^{\varepsilon,(n-1)}) \wedge B^{\varepsilon,(n-1)} }_{\delta_0}
     \right) \nonumber \\
     & \qquad \times \norm{\lambda(\xi_\Theta^{\varepsilon,(n)}) - \lambda(\xi_\Theta^{\varepsilon,(n-1)})}_{\delta_0}. 
\end{align*}
By the relativistic inequalities  \eqref{vxiteta}, \eqref{Lemma_remainder}, \eqref{Lemma_remainder_2},  \eqref{Lemma_remainder_3}, the induction hypothesis, and by choosing $\eta$ sufficiently small, we get 
\begin{align}
\label{esti_norm_diff_G_irr_4}
    \norm{I_4^{(n+1)} - I_4^{(n)}}_{\delta_0} 
     & \leq C\frac{C_2}{2^n} (C_1 + CC_1^2) \varepsilon^2 C C_1^3 \eta  + C  C_1 \varepsilon^2 C C_1^3 \left(\frac{C_2}{2^n} +  CC_1 \frac{C_2}{2^n} + C_1 C  \frac{C_2}{2^n}  \right) \eta  \nonumber \\
     & \qquad +  C   C_1 \left( C_1 + C C_1^2 \right) c \varepsilon^2  \frac{C_2}{2^n} \eta   \leq \frac{C_2}{12 \cdot 2^{n+1}}.
\end{align}
Putting together \eqref{esti_norm_diff_G_irr_1}, \eqref{esti_norm_diff_G_irr_2}, \eqref{esti_norm_diff_G_irr_3}, and \eqref{esti_norm_diff_G_irr_4}, we obtain
\begin{align*}
    \norm{G_{\text{irr}}^{\varepsilon,(n+1)} - G_{\text{irr}}^{\varepsilon,(n)}}_{\delta_0} \leq \frac{C_2}{3 \cdot 2^{n+1}}.
\end{align*}
 Similar estimates hold for $G_{\text{sol}}^{\varepsilon}$ and $G_{\text{mean}}^\varepsilon$, thus
\begin{align*}
    \norm{G^{\varepsilon, (n+1)} - G^{\varepsilon, (n)}}_{\delta_0} 
    & \leq \norm{G_{\text{irr}}^{\varepsilon,(n+1)} - G_{\text{irr}}^{\varepsilon,(n)}}_{\delta_0} + \norm{G_{\text{sol}}^{\varepsilon,(n+1)} - G_{\text{sol}}^{\varepsilon,(n)}}_{\delta_0} +  \norm{G_{\text{mean}}^{\varepsilon,(n+1)} - G_{\text{mean}}^{\varepsilon,(n)}}_{L^\infty_t}\\&\leq 3 \cdot \frac{C_2}{3\cdot2^{n+1}} = \frac{C_2}{2^{n+1}}.
\end{align*}   
The same analysis applies to $\left(\varepsilon E^{\varepsilon,(n+1)}- \varepsilon E^{\varepsilon,(n)}\right)$ and $\left(B^{\varepsilon,(n+1)} - B^{\varepsilon,(n)}\right)$ and this concludes the proof of \eqref{bound_step_n_diff}.

\qed 

\subsection{Proof of Theorem \ref{sec3:mainthm}}

In this section, we conclude the proof of the local-in-time solutions to the Euler--Maxwell system \eqref{sys:EM} on a time interval $\eta>0$ and independent of $\varepsilon$ and with solutions uniformly bounded with respect to $\varepsilon$ in the analytic norm.

\begin{proof}
    By Lemma \ref{lemma_iterative_scheme}, we have showed that the sequences $\left(\rho^{\varepsilon, (n)}_\Theta,w^{\varepsilon, (n)}_\Theta, G^{\varepsilon, (n)},B^{\varepsilon, (n)}\right)$ are bounded Cauchy sequences in $C\left([0,\eta];B_{\delta_0} \right)$ for a given $\delta_0 >1$. Therefore, there exist 
    $\left(\rho^\varepsilon_\Theta,w^\varepsilon_\Theta, G^\varepsilon,B^\varepsilon\right) \in C\left([0,\eta];B_{\delta_0} \right)$ such that
    \begin{align*}
        \left(\rho^{\varepsilon, (n)}_\Theta,w^{\varepsilon, (n)}_\Theta, G^{\varepsilon, (n)},B^{\varepsilon, (n)}\right) 
        \xrightarrow[n \to \infty]{}
        \left(\rho^\varepsilon_\Theta,w^\varepsilon_\Theta, G^\varepsilon,B^\varepsilon\right).
    \end{align*}
    By defining $\xi_\Theta^\varepsilon (t,x) := w_\Theta^\varepsilon (t,x) + G^\varepsilon (t,x)$, $E^\varepsilon(t,x):=\partial_t G^\varepsilon$ and letting $n$ go to infinity in the definitions of sequences for our iterative scheme defined in Section \ref{sec2:iterativescheme}, we get that $\left(\rho^\varepsilon_\Theta,\xi^\varepsilon_\Theta, E^\varepsilon,B^\varepsilon\right)$ are solutions to the Euler--Maxwell system \eqref{sys:EM} with initial data $(\rho^\varepsilon_\Theta(0),\xi^\varepsilon_\Theta(0),E^\varepsilon(0), B^\varepsilon(0))$ given by Theorem \ref{sec3:mainthm}.

\end{proof}

\section{Derivation of the (e-MHD) system in the quasineutral limit}
\label{sec3}

The goal of this section is to prove Theorem \ref{sec1:mainthm2}. That is, we now want to consider the quasineutral limit as $\varepsilon$ goes to zero and to derive the (e-MHD) system \eqref{sys:limEM} starting from the Euler--Maxwell system \eqref{sys:EM}.
We first introduce a suitable decomposition of the electric field that differs from the one that we already did 
among spatial mean, irrotational and solenoidal parts.
The new splitting divides the electric field into two terms.
The first is oscillatory, of magnitude $\frac{1}{\varepsilon}$ and leads to the momentum corrector, while the second is bounded.

To do this, given $\varphi,\psi \in C^0\left([0,T];H^s(\mathbb{T}_x^d)\right)$ for $d \in \mathbb{N}$, $s \ge 0$, and $t\in[0,T]$, we define
\begin{align*}
\mathcal{H}^\varepsilon_1\varphi(t,x)&:=\frac{1}{2\pi\varepsilon} \int_t^{t+2\pi \varepsilon} \varphi(s,x)ds,\\
\mathcal{H}^\varepsilon_2\psi(t,x)&:=\mathcal{F}^{-1}\Bigg(\Bigg\{\frac{\sqrt{1+|k|^2}}{2\pi\varepsilon} \int_t^{t+\frac{2\pi \varepsilon}{\sqrt{1+|k|^2}}} \widehat{\psi}(s,k)ds\Bigg\}_{k \in \Z^3}\Bigg).
\end{align*}
We hence introduce the following splitting for the electric field $E^\varepsilon$:
\begin{equation}
\label{sec4: decomp}
E^\varepsilon(t,x)=E^\varepsilon_1(t,x)+E^\varepsilon_2(t,x), \quad E^\varepsilon_1(t,x):=\left(Id-\mathcal{H}^\varepsilon\right) E^\varepsilon(t,x), \quad E^\varepsilon_2(t,x):=\mathcal{H}^\varepsilon E^\varepsilon(t,x),
\end{equation}
where
\[
\mathcal{H}^\varepsilon E^\varepsilon(t,x):=\mathcal{H}^\varepsilon_1 E^\varepsilon_{\text{irr}}(t,x)+\mathcal{H}^\varepsilon_2 E^\varepsilon_{\text{sol}}(t,x)+ \mathcal{H}^\varepsilon_1 E^\varepsilon_{\text{mean}}(t) .
\]
To summarize, the electric field $E^\varepsilon$ is decomposed into two terms, $E^\varepsilon_1$ and $E^\varepsilon_2$. The latter is obtained by taking the time averages over the oscillation periods of the irrotational and solenoidal components of $E^\varepsilon$  and of its spatial mean, while the former is its complementary part. 

Unlike $\mathcal{H}_1^\varepsilon E^\varepsilon_{\text{irr}}$ and $\mathcal{H}_1^\varepsilon E^\varepsilon_{\text{mean}}$, 
which are the time averages of the electric field components with a $k$-independent oscillation frequency,
the solenoidal term $\mathcal{H}_2^\varepsilon E^\varepsilon_{\text{sol}}$ is treated differently. It is first averaged in Fourier space mode by mode and then inverted, since the dispersion relation of $E^\varepsilon_{\text{sol}}$ depends on the Fourier mode $k$.
 Moreover notice that, since $E^\varepsilon_{\text{mean}}$ is spatially homogeneous, $\mathcal{H}_1^\varepsilon E_{\text{mean}}^\varepsilon$ depends only on time.

\begin{proposition}
\label{sec4:prop1}
    Under the assumptions of Theorem \ref{sec1:mainthm2}, the following facts hold:
    \begin{enumerate}
    \item \label{sec4:prop1pt1} There exists $C>0$ such that
    \[
    \sup_{t \in [0,T]}\left\|E^\varepsilon_2(t) \right\|_{H^{s-1}(\mathbb{T}_x^3)}\le C,
    \quad
    \sup_{t \in [0,T]}\left\|\varepsilon E^\varepsilon_1(t)\right\|_{H^{s-1}(\mathbb{T}_x^3)}\le C.
    \]
    \item \label{sec4:prop1pt2} Defining the corrector
    \begin{align}
    \label{sec4:corrector}
    W^\varepsilon(t,x)&:=\int_0^t E^\varepsilon_1(s,x)ds=\int_0^t\left(\text{Id}-\mathcal{H}^\varepsilon \right)E^\varepsilon(s,x)ds,
    \end{align}
    there exists $C>0$ such that 
    \[
    \sup_{t \in [0,T]}\|W^\varepsilon(t)\|_{H^{s-1}(\mathbb{T}_x^3)}\le C; \quad W^\varepsilon\rightharpoonup0 \, \text{weakly in }L^2_{t,x}.
    \]
        \end{enumerate}
    \end{proposition}
    \begin{proof}
    \underline{\emph{Part \eqref{sec4:prop1pt1} of Proposition \ref{sec4:prop1}}}: Let us first recall the formulas for $\widehat{E_{\text{irr}}^\varepsilon}$, $\widehat{E_{\text{sol}}^\varepsilon}$ and  $E_{\text{mean}}^\varepsilon$ given by \eqref{Fourier_divE_irr}, \eqref{sec2:initialirr}, \eqref{Fourier_E_sol}, \eqref{sec2:initialsol}, and \eqref{E_mean}, \eqref{sec2:initialmean}
    that is
    \begin{align*}
	\widehat{E_{\text{irr}}^\varepsilon}(t,k) 
	= -\int_{0}^{t} \frac{\imm k}{\varepsilon \abs{k}^2} \sin\left(\frac{t-s}{\varepsilon} \right)
	\widehat{g^\varepsilon}(s,k)ds+\widehat{E_{\text{irr},0}^\varepsilon} (t,k), 
    \end{align*}
    \begin{align*}
	\widehat{E_{\text{sol}}^\varepsilon}(t,k) 
	= \int_{0}^{t} \frac{1}{\varepsilon |k|^2\sqrt{1+|k|^2}} \sin\left(\frac{\sqrt{1+|k|^2}(t-s)}{\varepsilon} \right)
	\left(\imm k\wedge\widehat{h^\varepsilon}(s,k)\right)ds+\widehat{E_{\text{sol},0}^\varepsilon} (t,k),    
    \end{align*}
    and 
    \begin{align*}
	E_{\text{mean}}^\varepsilon(t) 
	=\frac{1}{\varepsilon} \int_{0}^{t} \sin\left(\frac{t-s}{\varepsilon} \right)
	q^\varepsilon(s)ds+E_{\text{mean},0}^\varepsilon (t), 
    \end{align*}
    where
	\begin{equation}
    \label{sec4:initialirr}
	\widehat{E_{\text{irr},0}^\varepsilon}(t,k)
	=\widehat{E_{\text{irr}}^\varepsilon}(0,k)\cos\left(\frac{t}{\varepsilon}\right)
	+	\varepsilon\widehat{\partial_t E_{\text{irr}}^\varepsilon}(0,k)
	\sin\left(\frac{t}{\varepsilon}\right),
    \end{equation}
    \begin{equation}
    \label{sec4:initialsol}
	\widehat{E_{\text{sol},0}^\varepsilon}(t,k)
	=\widehat{E_{\text{sol}}^\varepsilon}(0,k)\cos\left(\frac{t\sqrt{1+|k|^2}}{\varepsilon}\right)
	+\frac{\varepsilon}{\sqrt{1+|k|^2}}\widehat{\partial_t E_{\text{sol}}^\varepsilon}(0,k)
	\sin\left(\frac{t \sqrt{1+|k|^2}}{\varepsilon}\right),
	\end{equation}
    	\begin{equation}
    \label{sec4:initialmean}
	E_{\text{mean},0}^\varepsilon(t)
	=E_{\text{mean}}^\varepsilon(0)\cos\left(\frac{t}{\varepsilon}\right)
	+	\varepsilon \partial_t E_{\text{mean}}^\varepsilon(0)
	\sin\left(\frac{t}{\varepsilon}\right),
    \end{equation}
    and where $g^\varepsilon, h^\varepsilon$ and $q^\varepsilon$ are defined in \eqref{sec3:gepsilon}, \eqref{sec3:hepsilon} and \eqref{sec3:qepsilon}.
    
    We start studying $\norm{E_2^\varepsilon(t)}_{H^{s-1}_x}$.
    Since the time averages of the initial contributions in \eqref{sec4:initialirr}, \eqref{sec4:initialsol}  and \eqref{sec4:initialmean} are vanishing, we get
    \begin{align*}
        \widehat{\mathcal{H}^\varepsilon E^\varepsilon}(t,k)= \widehat{\mathcal{I}_{\text{irr}}}(t,k)+\widehat{\mathcal{I}_{\text{sol}}}(t,k) +\mathcal{I}_{\text{mean}}(t)\mathbf{1}_{k=0},
    \end{align*}
    where
    \begin{align*}
    \widehat{\mathcal{I}_{\text{irr}}}(t,k)&:=-\frac{1}{2\pi\varepsilon^2}\int_t^{t+2\pi \varepsilon}\int_0^s\frac{\imm k}{|k|^2} \widehat{g^\varepsilon}(\sigma,k) \sin\left( \frac{s-\sigma}{\varepsilon}\right)d\sigma ds,\\
        \widehat{\mathcal{I}_{\text{sol}}}(t,k)&:= \frac{1}{2\pi\varepsilon^2 |k|^2}\int_t^{t+\frac{2\pi \varepsilon}{\sqrt{1+|k|^2}}}\int_0^s \left( \imm k\wedge\widehat{h^\varepsilon}(\sigma,k) \right) \sin\left(\frac{\sqrt{1+|k|^2}(s-\sigma)}{\varepsilon} \right)d\sigma ds,\\
        \mathcal{I}_{\text{mean}}(t)&:=\frac{(2\pi)^2}{\varepsilon^2}
        \int_t^{t+2\pi \varepsilon}\int_0^s q^\varepsilon(\sigma) \sin\left( \frac{s-\sigma}{\varepsilon}\right)d\sigma ds.
    \end{align*}
    For the first integral, since
    \[
    \int_t^{t+2\pi \varepsilon}\int_0^s\frac{\imm k}{|k|^2} \widehat{g^\varepsilon}(\sigma,k) \sin\left( \frac{s-\sigma}{\varepsilon}\right)d\sigma ds=
    \int_t^{t+2\pi \varepsilon}\int_t^s\frac{\imm k}{|k|^2} \widehat{g^\varepsilon}(\sigma,k) \sin\left( \frac{s-\sigma}{\varepsilon}\right)d\sigma ds,
    \]
    we get,
    inverting the order of the integration,
    \begin{align}
    \label{sec4:est1}
        \widehat{\mathcal{I}_{\text{irr}}}(t,k)&=-\frac{1}{2\pi \varepsilon^2}\int_t^{t+2\pi\varepsilon} \frac{\imm k}{|k|^2}\widehat{g^\varepsilon}(\sigma,k) 
        \int_\sigma^{t+2\pi\varepsilon} \sin\left(\frac{s-\sigma}{\varepsilon} \right) ds d\sigma \nonumber\\
        &=-\frac{1}{2\pi \varepsilon}\int_t^{t+2\pi\varepsilon} \frac{\imm k}{|k|^2}\widehat{g^\varepsilon}(\sigma,k) 
        \left[1-\cos\left(\frac{t-\sigma}{\varepsilon} \right) \right] d\sigma.
    \end{align}
    Similarly, for the second integral, we get
    \begin{align}
    \label{sec4:est2}
        \widehat{\mathcal{I}_{\text{sol}}}(t,k)&=\frac{1}{2\pi \varepsilon^2|k|^2}\int_t^{t+\frac{2\pi\varepsilon}{\sqrt{1+|k|^2}}} \imm k \wedge\widehat{h^\varepsilon}(\sigma,k) 
        \int_\sigma^{t+\frac{2\pi\varepsilon}{\sqrt{1+|k|^2}}} \sin\left(\frac{\sqrt{1+|k|^2}(s-\sigma)}{\varepsilon} \right) ds d\sigma \nonumber\\
        &=\frac{1}{2\pi \varepsilon|k|^2\sqrt{1+|k|^2}}\int_t^{t+\frac{2\pi\varepsilon}{\sqrt{1+|k|^2}}} \imm k \wedge\widehat{h^\varepsilon}(\sigma,k) 
        \left[1-\cos\left(\frac{\sqrt{1+|k|^2}(t-\sigma)}{\varepsilon} \right) \right] d\sigma,
    \end{align}
    while concerning $\mathcal{I}_{\text{mean}}$, we get
      \begin{align}
    \label{sec4:est3}
        \mathcal{I}_{\text{mean}}(t)&=\frac{(2\pi)^2}{ \varepsilon^2}\int_t^{t+2\pi\varepsilon} q^\varepsilon(\sigma) 
        \int_\sigma^{t+2\pi\varepsilon} \sin\left(\frac{s-\sigma}{\varepsilon} \right) ds d\sigma=\frac{(2\pi)^2}{ \varepsilon}\int_t^{t+2\pi\varepsilon}q^\varepsilon(\sigma) 
        \left[1-\cos\left(\frac{t-\sigma}{\varepsilon} \right) \right] d\sigma.
    \end{align}
    By \eqref{sec4:est1}, \eqref{sec4:est2} and \eqref{sec4:est3}, we arrive at
    \[
    \left|\widehat{\mathcal{H}^\varepsilon E^\varepsilon}(t,k)\right|\le \frac{C}{\pi \varepsilon}\int_t^{t+2\pi\varepsilon}\left[\frac{|\widehat{g^\varepsilon}(\sigma,k)|}{|k|}+ \mathbf{1}_{k=0}|q^\varepsilon(\sigma)|\right] d\sigma + \frac{1}{\pi \varepsilon|k|\sqrt{1+|k|^2}}\int_t^{t+\frac{2\pi\varepsilon}{\sqrt{1+|k|^2}}}|\widehat{h^\varepsilon}(\sigma,k)| d\sigma.
    \]
    Recalling the assumption in \eqref{sec0:assumptthm2}, along with the definitions of \( g^\varepsilon \), \( h^\varepsilon \) and $q^\varepsilon$ in \eqref{sec3:gepsilon}, \eqref{sec3:hepsilon} and \eqref{sec3:qepsilon}, and using the algebra property of Sobolev spaces for \( s-2 > \frac{3}{2} \), we deduce that \( g^\varepsilon, h^\varepsilon \in L^\infty_{t} H^{s-2}_x \) and $q^\varepsilon\in L^\infty_t$. This implies that  
\[
\sup_{t \in [0,T]}\left\|E^\varepsilon_2(t) \right\|_{H^{s-1}(\mathbb{T}_x^3)} \le C.
\]
The estimate on \( \norm{\varepsilon E_1^\varepsilon}_{L^\infty_tH^{s-1}_x} \) follows from the identity \( E^\varepsilon_1 = E^\varepsilon - E^\varepsilon_2 \) and the assumption that \( \norm{\varepsilon E^\varepsilon}_{L^\infty_tH^s_x} \) is bounded by \eqref{sec0:assumptthm2}.

  \vskip 0.5 cm
        \noindent 
    \underline{\emph{Part \eqref{sec4:prop1pt2} of Proposition \ref{sec4:prop1}}}: We start by writing explicitly $W^\varepsilon$:
    \begin{equation}
            \begin{aligned}
    \label{sec4:corrector2}
    \widehat{W^\varepsilon}(t,k)=\mathcal{J}_0(t,k) + \mathcal{J}_1(t,k)+\mathcal{J}_2(t,k)+
    \mathcal{J}_3(t)\mathbf{1}_{k=0} -\int_0^t\widehat{\mathcal{H}^\varepsilon E^\varepsilon}(s,k) ds,
\end{aligned}
\end{equation}
where
\begin{align*}
    \mathcal{J}_0(t,k)&:=\mathbf{1}_{k=0}(2\pi)^3\left[\varepsilon E_{\text{mean}}^\varepsilon(0) \sin\left(\frac{t}{\varepsilon}\right)-\varepsilon^2\partial_t E^\varepsilon_{\text{mean}}(0)\cos\left(\frac{t}{\varepsilon}\right)\right]
	\\
    & \quad+ \varepsilon \widehat{E_{\text{irr}}^\varepsilon}(0,k) \sin\left(\frac{t}{\varepsilon}\right)-\varepsilon^2\widehat{\partial_t E^\varepsilon_{\text{irr}}}(0,k)\cos\left(\frac{t}{\varepsilon}\right)\\
    &\quad+\frac{\varepsilon}{\sqrt{1+|k|^2}}\widehat{E_{\text{sol}}^\varepsilon}(0,k)\sin\left(\frac{t\sqrt{1+|k|^2}}{\varepsilon}\right)
	-\frac{\varepsilon^2}{1+|k|^2}\widehat{\partial_t E_{\text{sol}}^\varepsilon}(0,k)
	\cos\left(\frac{t \sqrt{1+|k|^2}}{\varepsilon}\right),\\
    \mathcal{J}_1(t,k)&:= -\frac{\imm k}{\varepsilon|k|^2}\int_0^t \int_0^s \widehat{g^\varepsilon}(\sigma,k)
    \sin\left(\frac{s-\sigma}{\varepsilon} \right) d\sigma ds,\\
	\mathcal{J}_2(t,k)&:=\ \frac{1}{\varepsilon |k|^2\sqrt{1+|k|^2}}\int_0^t  \int_{0}^{s}\left(\imm k\wedge\widehat{h^\varepsilon}(\sigma,k)\right)  \sin\left(\frac{\sqrt{1+|k|^2}(s-\sigma)}{\varepsilon} \right)
	d\sigma ds,\\
    \mathcal{J}_3(t)&:= \frac{(2\pi)^3}{\varepsilon}\int_0^t \int_0^s q^\varepsilon(\sigma)
    \sin\left(\frac{s-\sigma}{\varepsilon} \right) d\sigma ds.
    \end{align*}
    By the assumption \eqref{sec0:assumptthm2} applied on the initial data, the six terms in $\mathcal{J}_0$ are bounded in $L^\infty_tH^{s-1}_x$: Indeed, by the expressions \eqref{sec2:irrtime0}, \eqref{sec2:soltime0} and \eqref{sec2:meantime0} and reasoning as in  \eqref{esti_E_irr_0}, \eqref{esti_dt_E_irr_0}, \eqref{esti_E_sol_0}, \eqref{esti_dt_E_sol_0}, \eqref{esti_E_mean_0} and \eqref{esti_dt_E_mean_0}, we get
    \[
    \Big|\varepsilon E^\varepsilon_{\text{mean}}(0)\Big|\le \norm{\varepsilon E^\varepsilon(0)}_{H^s_x}, \quad \Big|\varepsilon^2 \partial_t E^\varepsilon_{\text{mean}}(0) \Big|\le \norm{j^\varepsilon(0)}_{H^s_x},
    \]
    \[
     \norm{\varepsilon E^\varepsilon_{\text{irr}}(0)}_{H^{s-1}_x}\le\varepsilon^{-1}\norm{\rho^\varepsilon(0)-1}_{H^{s-2}_x}\le\norm{\varepsilon E^\varepsilon(0)}_{H^{s-1}_x}, \quad 
         \norm{\varepsilon^2 \partial_tE^\varepsilon_{\text{irr}}(0)}_{H^{s-1}_x}\le \norm{j^\varepsilon(0)}_{H^{s-1}_x},
    \]
    and
    {\[
      \norm{\varepsilon E^\varepsilon_{\text{sol}}(0)}_{H^{s-1}_x}\le  \norm{\varepsilon E^\varepsilon(0)}_{H^{s-1}_x}+  \norm{\varepsilon E^\varepsilon_{\text{irr}}(0)}_{H^{s-1}_x}+\Big|\varepsilon E_{\text{mean}}^\varepsilon(0)\Big|,\]
      \[
         \norm{\varepsilon^2 \partial_tE^\varepsilon_{\text{sol}}(0)}_{H^{s-1}_x}\le\norm{B^\varepsilon(0)}_{H^s_x}+\norm{j^\varepsilon(0)}_{H^{s-1}_x}.
    \]}
    Moreover, all the terms in $\mathcal{J}_0$ converge weakly to $0$ in $L^2_{t,x}$.

    We now study $\{\mathcal{J}_{\ell}\}_{\ell=1}^3$: By inverting the order of integration, we get
    \begin{equation*}
    \begin{aligned}
    \mathcal{J}_1(t,k)&= -\frac{\imm k}{\varepsilon|k|^2}\int_0^t \widehat{g^\varepsilon}(\sigma,k) \int_\sigma^t 
    \sin\left(\frac{s-\sigma}{\varepsilon} \right) ds d\sigma\\
    &=-\frac{\imm k}{|k|^2}\int_0^t \widehat{g^\varepsilon}(\sigma,k)\left[1-
    \cos\left(\frac{t-\sigma}{\varepsilon} \right) \right]d\sigma,
    \end{aligned}
    \end{equation*}
    which is bounded in $L^\infty_tH^{s-1}_x$, since $g^\varepsilon$ is bounded in $L^\infty_t H^{s-2}_x$. Similarly
    \begin{equation*}
    \begin{aligned}
	\mathcal{J}_2(t,k)&=\frac{1}{\varepsilon |k|^2\sqrt{1+|k|^2}}\int_0^t \imm k\wedge\widehat{h^\varepsilon}(\sigma,k) \int_\sigma^t\sin\left(\frac{\sqrt{1+|k|^2}(s-\sigma)}{\varepsilon} \right)
	ds d\sigma\\
    &=\frac{1}{(1+|k|^2) |k|^2}\int_0^t \imm k\wedge\widehat{h^\varepsilon}(\sigma,k) 
    \left[1-\cos\left(\frac{\sqrt{1+|k|^2}(t-\sigma)}{\varepsilon} \right)\right]
	 d\sigma,
    \end{aligned}
    \end{equation*}
    is bounded in $L^\infty_t H^{s-1}_x$, since $h^\varepsilon$ is bounded in $L^\infty_t H^{s-2}_x$.
    Finally
     \begin{equation*}
    \begin{aligned}
    \mathcal{J}_3(t)&= \frac{(2\pi)^3}{\varepsilon}\int_0^t q^\varepsilon(\sigma) \int_\sigma^t 
    \sin\left(\frac{s-\sigma}{\varepsilon} \right) ds d\sigma=(2\pi)^3\int_0^t q^\varepsilon(\sigma)\left[1-
    \cos\left(\frac{t-\sigma}{\varepsilon} \right) \right]d\sigma,
    \end{aligned}
    \end{equation*}
    is also bounded uniformly in time, since $q^\varepsilon$ is bounded in $L^\infty_t$.
    By definition, we have
    \[
    \int_0^t\mathcal{H^\varepsilon}{E^\varepsilon}(s,k) ds=\int_0^t E^\varepsilon_2(s,k)ds,
    \]
    and being $E_2^\varepsilon \in L^\infty_tH^{s-1}_x$ by part \eqref{sec4:prop1pt1} of this proposition, it follows that  $\int_0^tE_2^\varepsilon\in L^\infty_t H^{s-1}_x$. 
    By \eqref{sec4:corrector2}, this proves there exists a constant $C>0$ such that $\norm{W^\varepsilon}_{L^\infty_t H^{s-1}_x}\le C$. 
    
    We now prove that $W^\varepsilon(t,x)$ weakly converge to $0$ in $L^2_{t,x}$. We already observed that the initial contributions in $\mathcal{J}_0$ weakly converge to $0$. Moreover, notice that
    \begin{equation}
        \label{sec4:jei1}
        \mathcal{J}_1(t,k)=-\frac{\imm k}{|k|^2}\int_0^t \widehat{g^\varepsilon}(\sigma,k)d\sigma + \mathcal{O}_1^\varepsilon(t,k),\quad \mathcal{O}^\varepsilon_1(t,k):=\frac{\imm k}{|k|^2}\int_0^t \widehat{g^\varepsilon}(\sigma,k)
    \cos\left(\frac{t-\sigma}{\varepsilon} \right)d\sigma,
    \end{equation}
    and
    \begin{align}
        \mathcal{J}_2(t,k)&=\frac{1}{(1+|k|^2) |k|^2}\int_0^t \imm k\wedge\widehat{h^\varepsilon}(\sigma,k)   
    d\sigma+\mathcal{O}_2^\varepsilon(t,k), \label{sec4:jei2}
\\
    \mathcal{O}_2^\varepsilon(t,k)&:=-\frac{1}{(1+|k|^2) |k|^2}\int_0^t \left(\imm k\wedge\widehat{h^\varepsilon}(\sigma,k)\right) \cos\left(\frac{\sqrt{1+|k|^2}(t-\sigma)}{\varepsilon} \right) 
	 d\sigma, \nonumber
    \end{align}
    {while
     \begin{equation}
        \label{sec4:jei3}
        \mathcal{J}_3(t)=(2\pi)^3\int_0^t q^\varepsilon(\sigma)d\sigma + \mathcal{O}_3^\varepsilon(t),\quad \mathcal{O}^\varepsilon_3(t):=-(2\pi)^3\int_0^t q^\varepsilon(\sigma)
    \cos\left(\frac{t-\sigma}{\varepsilon} \right)d\sigma.
    \end{equation}
    
    We now compute $\int_0^t \widehat{\mathcal{H}^\varepsilon E^\varepsilon}(s,k) ds$ integrating in time
    \eqref{sec4:est1}, \eqref{sec4:est2} and \eqref{sec4:est3}. Therefore, we get   
    \begin{align}
    \label{sec4:inth}
    \int_0^t &\widehat{\mathcal{H}^\varepsilon E^\varepsilon}(s,k) ds=-\int_0^t \frac{1}{2\pi \varepsilon}\int_{s}^{s+2\pi \varepsilon} \frac{\imm k}{|k|^2}\widehat{g^\varepsilon}(\sigma,k)
    \left[1-\cos\left(\frac{s-\sigma}{\varepsilon} \right)\right]d\sigma ds \nonumber\\
    &\quad +\frac{1}{2\pi\sqrt{1+|k|^2} |k|^2 \varepsilon}\int_0^t \int_s^{s+\frac{2 \pi \varepsilon}{\sqrt{1+|k|^2}}} \imm k\wedge\widehat{h^\varepsilon}(\sigma,k)\left[1-\cos\left(\frac{\sqrt{1+|k|^2}(s-\sigma)}{\varepsilon} \right)\right] d\sigma ds  \nonumber \\
    &\quad
    +\mathbf{1}_{k=0}\int_0^t \frac{ (2\pi)^2}{\varepsilon}\int_{s}^{s+2\pi \varepsilon} q^\varepsilon(\sigma)
    \left[1-\cos\left(\frac{s-\sigma}{\varepsilon} \right)\right]d\sigma ds \nonumber\\
    &=-\frac{\imm k}{|k|^2}\int_0^t \widehat{g^\varepsilon}(s,k)ds 
    +\frac{1}{(1+|k|^2) |k|^2}\int_0^t \imm k\wedge\widehat{h^\varepsilon}(s,k) ds+
    \mathbf{1}_{k=0}(2\pi)^3\int_0^t q^\varepsilon(s)ds+ \mathcal{R}^\varepsilon(t,k),
    \end{align}
    where 
    \begin{equation}
    \label{sec4:Reps}
        \mathcal{R}^\varepsilon:=\mathcal{R}^\varepsilon_1 + \mathcal{R}^\varepsilon_2+ \mathcal{R}^\varepsilon_3,
    \end{equation}
    with
    \[
    \mathcal{R}^\varepsilon_1:=\frac{\imm k}{|k|^2}\int_0^t \widehat{g^\varepsilon}(s,k)ds-\int_0^t \frac{1}{2\pi \varepsilon}\int_{s}^{s+2\pi \varepsilon} \frac{\imm k}{|k|^2}\widehat{g^\varepsilon}(\sigma,k)
    \left[1-\cos\left(\frac{s-\sigma}{\varepsilon} \right)\right]d\sigma ds,
    \]
    \begin{align*}
    \mathcal{R}_2^\varepsilon:&= -\frac{1}{(1+|k|^2) |k|^2}\int_0^t \imm k\wedge\widehat{h^\varepsilon}(s,k) ds \\
    &\quad+ \int_0^t \int_s^{s+\frac{2 \pi \varepsilon}{\sqrt{1+|k|^2}}}\frac{\imm k\wedge\widehat{h^\varepsilon}(\sigma,k)}{2\pi\sqrt{1+|k|^2} |k|^2 \varepsilon}\left[1-\cos\left(\frac{\sqrt{1+|k|^2}(s-\sigma)}{\varepsilon} \right)\right] d\sigma ds,
    \end{align*}
    and
      \[
    \mathcal{R}^\varepsilon_3:=\mathbf{1}_{k=0}\left[-(2\pi)^3\int_0^t q^\varepsilon(s)ds+\int_0^t \frac{(2\pi)^2}{ \varepsilon}\int_{s}^{s+2\pi \varepsilon}q^\varepsilon(\sigma)
    \left(1-\cos\left(\frac{s-\sigma}{\varepsilon} \right)\right)d\sigma ds\right].
    \]
    By summing \eqref{sec4:jei1}, \eqref{sec4:jei2} and \eqref{sec4:jei3} and taking the difference with \eqref{sec4:inth}, we get
    \begin{align}
    \label{sec4:formula_O}
    \mathcal{J}_1(t,k) + \mathcal{J}_2(t,k)
    +\mathcal{J}_3(t)\mathbf{1}_{k=0}-\int_0^t \mathcal{H^\varepsilon}E^\varepsilon(s,k)ds=\mathcal{O}_1^\varepsilon+\mathcal{O}_2^\varepsilon+
    \mathcal{O}_3^\varepsilon-\mathcal{R}^\varepsilon
    \end{align}
    and, since they are oscillatory integrals, $\mathcal{O}_1^\varepsilon$, $\mathcal{O}_2^\varepsilon$ and $\mathcal{O}_3^\varepsilon$ weakly converge to $0$ in $L^2_{t,x}$. Concerning $\mathcal{R}^\varepsilon$, we actually have strong convergence, indeed $\|\mathcal{R}^\varepsilon\|_{L^\infty_t H^{s-1}_x}\le C \varepsilon$.
    To see this, note that we can switch the order of integration in the second terms of $\mathcal{R}^\varepsilon_1$, $\mathcal{R}^\varepsilon_2$ and $\mathcal{R}^\varepsilon_3$ using that
    $$\int_0^t \int_s^{s+2\pi \varepsilon} d\sigma ds = \int_0^{2\pi \varepsilon} \int_0^{\sigma} ds d\sigma + \int_{2\pi \varepsilon}^t \int_{\sigma -2 \pi \varepsilon}^\sigma ds d\sigma + \int_t^{t +2\pi \varepsilon} \int_{\sigma -2 \pi \varepsilon}^t ds d\sigma.$$ 
    Therefore, we get 
    \begin{align*}
        \mathcal{R}^\varepsilon_1 &=\frac{\imm k}{|k|^2}\int_0^t \widehat{g^\varepsilon}(s,k)ds
        -\int_0^{2\pi \varepsilon}   \frac{1}{2\pi \varepsilon} \frac{\imm k}{|k|^2}\widehat{g^\varepsilon}(\sigma,k)
    \int_0^{\sigma} \left[1-\cos\left(\frac{s-\sigma}{\varepsilon} \right)\right] ds d\sigma \\
    & \qquad - \int_{2\pi \varepsilon}^t    \frac{1}{2\pi \varepsilon} \frac{\imm k}{|k|^2} \widehat{g^\varepsilon}(\sigma,k)
    \int_{\sigma -2 \pi \varepsilon}^\sigma \left[1-\cos\left(\frac{s-\sigma}{\varepsilon} \right)\right] ds d\sigma \\
    & \qquad - \int_t^{t +2\pi \varepsilon}   \frac{1}{2\pi \varepsilon} \frac{\imm k}{|k|^2}\widehat{g^\varepsilon}(\sigma,k)
    \int_{\sigma -2 \pi \varepsilon}^t \left[1-\cos\left(\frac{s-\sigma}{\varepsilon} \right)\right] ds d\sigma.
    \end{align*}
    By computing the three integrals in $ds$, we get 
    \begin{align*}
        \mathcal{R}^\varepsilon_1 &=\frac{\imm k}{|k|^2}\int_0^t \widehat{g^\varepsilon}(s,k)ds
        -\int_0^{2\pi \varepsilon}   \frac{1}{2\pi \varepsilon} \frac{\imm k}{|k|^2}\widehat{g^\varepsilon}(\sigma,k)
    \left( \sigma + \varepsilon \sin\left( \frac{-\sigma}{\varepsilon} \right) \right) d\sigma -  \int_{2\pi \varepsilon}^t   \frac{\imm k}{|k|^2}\widehat{g^\varepsilon}(\sigma,k)
     d\sigma \\
    & \qquad - \int_t^{t +2\pi \varepsilon}   \frac{1}{2\pi \varepsilon} \frac{\imm k}{|k|^2}\widehat{g^\varepsilon}(\sigma,k)
    \left( t-\sigma + 2 \pi \varepsilon - \varepsilon \sin\left( \frac{t-\sigma}{\varepsilon} \right) \right) 
     d\sigma.
    \end{align*}
    Hence, we can bound $\mathcal{R}^\varepsilon_1$ as follow,
    \begin{align*}
       \abs{ \mathcal{R}^\varepsilon_1} & \leq \frac{1}{|k|} \int_0^{2 \pi \varepsilon} \abs{\widehat{g^\varepsilon}(s,k)} ds
        + \frac{1}{|k|} \int_0^{2\pi \varepsilon}   \frac{1}{2\pi \varepsilon}  \abs{\widehat{g^\varepsilon}(\sigma,k)}
    \left( 2 \pi \varepsilon + \varepsilon \right) d\sigma \\
     & \qquad + \frac{1}{|k|} \int_t^{t +2\pi \varepsilon}   \frac{1}{2\pi \varepsilon} \abs{\widehat{g^\varepsilon}(\sigma,k)}
    \left(  2 \pi \varepsilon + \varepsilon \right) 
     d\sigma.
    \end{align*}
    A similar estimate holds for both $\mathcal{R}^\varepsilon_2$ and $\mathcal{R}^\varepsilon_3$. Therefore, we have 
    \begin{align}
    \label{sec4:estrem}
    \left| \mathcal{R}^\varepsilon\right| & \le C \frac{1}{|k|}\int_0^{2\pi \varepsilon}|\widehat{g^{\varepsilon}}(s,k)|ds + C \frac{1}{|k|}\int_t^{t + 2\pi \varepsilon}|\widehat{g^{\varepsilon}}(s,k)|ds \nonumber \\
    & \quad + C \frac{1}{(1+|k|^2)|k|}\int_0^{\frac{2\pi \varepsilon}{\sqrt{1 + |k|^2}}}|\widehat{h^{\varepsilon}}(s,k)|ds 
    + C \frac{1}{(1+|k|^2)|k|}\int_t^{t + \frac{2\pi \varepsilon}{\sqrt{1 + |k|^2}}}|\widehat{h^{\varepsilon}}(s,k)|ds\\
    & \quad + C\mathbf{1}_{k=0}\left[\int_0^{2\pi \varepsilon}|q^{\varepsilon}(s)|ds + \int_t^{t + 2\pi \varepsilon}|q^{\varepsilon}(s)|ds \right].\nonumber
    \end{align}
    Since $q^\varepsilon$ is uniformly bounded in time and $g^\varepsilon$ and $h^\varepsilon$ are bounded in $L^\infty_t H^{s-2}_x$}, this proves that $\|\mathcal{R}^\varepsilon\|_{L^\infty_t H^{s-1}_x}\le C \varepsilon$ and concludes the proof.
    \end{proof}

   \subsection{The limit (e-MHD) system}
   \label{sec3:limitsyst}
   We now rigorously derive the limit (e-MHD) system \eqref{sys:limEM}.
   \begin{proposition}
   \label{sec4:prop2}
       Given $0<\varepsilon<1$ and $\Theta \in M$, let 
       \begin{equation}
           \label{sec4:smallb}
           w^\varepsilon_\Theta(t,x):=\xi^\varepsilon_\Theta(t,x)-W^\varepsilon(t,x), \quad b^\varepsilon(t,x):=B^\varepsilon(t,x)+\nabla_x  \wedge W^\varepsilon(t,x)
       \end{equation} 
       where $W^\varepsilon$ is the corrector defined in \eqref{sec4:corrector}. 
       There exist a subsequence in $\varepsilon$ and two vector fields $E,B:[0,T]\times \mathbb{T}^3_x \to \R^3$ where $B$ is solenoidal such that $w^\varepsilon_\Theta, \rho^\varepsilon_\Theta$ (for every $\Theta \in M$) and $b^\varepsilon$ converge in $C^0\left( [0,T];H^{s'-2}_x\right)$ with $s'<s$ respectively to $w_\Theta, \rho_\Theta$ and $B$ with
    \[
    \partial_t w_\Theta+ (w_\Theta \cdot \nabla_x)w_\Theta=E+w_\Theta \wedge B,\quad   \partial_t \rho_\Theta + \nabla_x \cdot \left(\rho_\Theta w_\Theta \right)=0;
    \]     
    \[
    \int_M \rho_\Theta(t,x) d\mu(\Theta)=1, \quad \nabla_x  \wedge E(t,x)=-\partial_tB(t,x), \quad\nabla_x  \wedge B(t,x)=\int_M\rho_\Theta(t,x) w_\Theta(t,x) \mu(d\Theta).
    \]    
   \end{proposition}

   \begin{proof}
    \underline{\emph{Limit of the hydrodynamic quantities $w^\varepsilon_\Theta, \rho^\varepsilon_\Theta$}}:  We begin by taking the limit of $w^\varepsilon_\Theta$.  
    
    By assumption \eqref{sec0:assumptthm2}, we have $\xi^\varepsilon_\Theta \in L^\infty_t H^s_x$, and by part \eqref{sec4:prop1pt2} of Proposition \ref{sec4:prop1}, we have $W^\varepsilon \in L^\infty_t H^{s-1}_x$. Consequently, we obtain $w^\varepsilon_\Theta \in L^\infty_t H^{s-1}_x$. Moreover, since $\partial_t W^\varepsilon = E^\varepsilon_1$, the function $w^\varepsilon_\Theta$ satisfies the equation  
    \[
    \partial_t w^\varepsilon_\Theta + (v(w^\varepsilon_\Theta+W^\varepsilon)\cdot\nabla_x)(w^\varepsilon_\Theta+W^\varepsilon) = E^\varepsilon_2+v(w^\varepsilon_\Theta+W^\varepsilon)\wedge B^\varepsilon.
    \]
    By adding and subtracting the non relativistic velocity and rearranging terms, we rewrite this equation as follows
    \begin{equation}
    \label{sec4:prelimit}
    \partial_t w^\varepsilon_\Theta + (w^\varepsilon_\Theta\cdot\nabla_x)w^\varepsilon_\Theta =
    \left[E^\varepsilon_2-(W^\varepsilon\cdot\nabla_x) W^\varepsilon-W^\varepsilon\wedge(\nabla_x  \wedge W^\varepsilon)\right]
    +w^\varepsilon_\Theta\wedge b^\varepsilon+\bar{\mathcal{R}}^\varepsilon_1 + \bar{\mathcal{R}}^\varepsilon_2,
    \end{equation}
    where $b^\varepsilon$ is defined in \eqref{sec4:smallb}, and the remainder terms are given by  
    \[
    \bar{\mathcal{R}}^\varepsilon_1 := -(w^\varepsilon_\Theta\cdot\nabla_x) W^\varepsilon-(W^\varepsilon\cdot\nabla_x) w^\varepsilon_\Theta+W^\varepsilon\wedge b^\varepsilon-w^\varepsilon_\Theta \wedge (\nabla_x  \wedge W^\varepsilon),
    \]
    \[
    \bar{\mathcal{R}}^\varepsilon_2 := \left[(w^\varepsilon_\Theta+W^\varepsilon)-v(w^\varepsilon_\Theta+W^\varepsilon)\right]\cdot\nabla_x (w^\varepsilon_\Theta+W^\varepsilon) + \left[v(w^\varepsilon_\Theta+W^\varepsilon)-(w^\varepsilon_\Theta+W^\varepsilon)\right]\wedge B^\varepsilon.
    \]

    We now study the terms on the r.h.s. in \eqref{sec4:prelimit}.
    Since the term  
    \[
    E_2^\varepsilon-(W^\varepsilon\cdot \nabla_x) W^\varepsilon - W^\varepsilon\wedge (\nabla_x  \wedge W^\varepsilon)
    \]  
    is bounded in $L^\infty_t H^{s-2}_x$, there exists a subsequence $\varepsilon'$ and a vector field $E \in L^\infty_t H^{s-2}_x$ such that  
    \begin{equation}
    \label{sec4:limitE}
    E_2^{\varepsilon'}-(W^{\varepsilon'}\cdot \nabla_x) W^{\varepsilon'}-W^{\varepsilon'}\wedge(\nabla_x  \wedge W^{\varepsilon'}) \rightharpoonup  E
    \end{equation}
    in the sense of distributions.
    
    For the term $w^\varepsilon_\Theta \wedge b^\varepsilon$, we recall that $w^\varepsilon_\Theta \in L^\infty_t H^{s-1}_x$, and from the definition \eqref{sec4:smallb}, we deduce that $b^\varepsilon \in L^\infty_t H^{s-2}_x$ because $B^\varepsilon \in L^\infty_t H^s_x$ by assumption \eqref{sec0:assumptthm2} and $W^\varepsilon \in L^\infty_t H^{s-1}_x$ by part \eqref{sec4:prop1pt2} of Proposition \ref{sec4:prop1}.  
    
    Additionally, we compute  
    \[
    \partial_t b^\varepsilon = \partial_t B^\varepsilon + \nabla_x  \wedge \partial_t W^\varepsilon = -\nabla_x  \wedge E^\varepsilon + \nabla_x  \wedge E^\varepsilon_1 = -\nabla_x  \wedge E^\varepsilon_2.
    \]
    Since $E^\varepsilon_2 \in L^\infty_t H^{s-1}_x$ by part \eqref{sec4:prop1pt1} of Proposition \ref{sec4:prop1}, we conclude that $\partial_t b^\varepsilon \in L^\infty_t H^{s-2}_x$. Thus, there exists a subsequence $\varepsilon'$ and a solenoidal field $B \in L^\infty_t H^{s-2}_x$ such that  
    \[
    b^{\varepsilon'} \to B \quad \text{in} \quad C^0([0,T]; H^{s'-2}_x) \quad \text{for} \quad s' < s.
    \]
        Now, we analyze the remainder terms $\bar{\mathcal{R}}_1^\varepsilon$ and $\bar{\mathcal{R}}_2^\varepsilon$.  
        By Proposition \ref{sec4:prop1}, since $W^\varepsilon$ and $\partial_{x_i} W^\varepsilon$ converge weakly to zero in $L^2_{t,x}$ and $w^\varepsilon_\Theta$ and $b^\varepsilon$ are bounded in $L^\infty_t H^{s-2}_x$, it follows that  $\bar{\mathcal{R}}^\varepsilon_1$ converges weakly to zero in $L^2_{t,x}$,
        since 
        \[
        (w^\varepsilon_\Theta \cdot \nabla_x) W^\varepsilon \rightharpoonup 0, \quad  
        (W^\varepsilon\cdot \nabla_x) w^\varepsilon_\Theta \rightharpoonup 0, \quad  
        W^\varepsilon\wedge b^\varepsilon \rightharpoonup 0, \quad  
        \text{and} \quad w^\varepsilon_\Theta\wedge (\nabla_x  \wedge W^\varepsilon) \rightharpoonup 0.
        \]
        For the term $\bar{\mathcal{R}}^\varepsilon_2$, using the algebra property for Sobolev norm, we obtain  
        \begin{align*}
        \left\| \bar{\mathcal{R}}^\varepsilon_2 \right\|_{L^\infty_t H^{s-2}_x} 
        &\leq \norm{(w^\varepsilon_\Theta+W^\varepsilon)-v(w^\varepsilon_\Theta+W^\varepsilon)}_{L^\infty_t H^{s-2}_x} 
        \norm{w^\varepsilon_\Theta+W^\varepsilon}_{L^\infty_t H^{s-1}_x} \\
        & \quad +  \norm{v(w^\varepsilon_\Theta+W^\varepsilon)-(w^\varepsilon_\Theta+W^\varepsilon)}_{L^\infty_t H^{s-2}_x} 
        \norm{B^\varepsilon}_{L^\infty_t H^{s-2}_x}  \\
        & \leq C \varepsilon^2 \norm{w^\varepsilon_\Theta+W^\varepsilon}^3_{L^\infty_t H^{s-2}_x} \left(\norm{w^\varepsilon_\Theta+W^\varepsilon}_{L^\infty_t H^{s-1}_x} +  \norm{B^\varepsilon}_{L^\infty_t H^{s-2}_x}\right),
        \end{align*}
        where in the last inequality we used Lemma \ref{sec2:lemma_sobolev} to treat the difference between the relativistic and non relativistic velocity.
        Finally, since $w^\varepsilon_\Theta$ and $W^\varepsilon \in L^\infty_t H^{s-1}$ and $B^\varepsilon \in L_t^\infty H_x^{s-2}$, we deduce that $\left\| \bar{\mathcal{R}}^\varepsilon_2 \right\|_{L^\infty_t H^{s-2}_x} $ goes to zero as $\varepsilon$ goes to zero.
        
       From \eqref{sec4:prelimit}, since all terms on the right-hand side are bounded in $L^\infty_t H^{s-2}_x$, it follows that $\partial_t w^\varepsilon_\Theta \in L^\infty_t H^{s-2}_x$. Thus, there exists a subsequence $\varepsilon''$ dependent on $\Theta$ such that  
        \[
        w^{\varepsilon''}_\Theta \to w_\Theta \quad \text{strongly in} \quad C^0([0,T]; H^{s'-2}_x) \quad \text{for} \quad s' < s.
        \] 
        We now consider the initial datum $w^\varepsilon_\Theta(0)$. Since $W^\varepsilon(0,x)$, defined by \eqref{sec4:corrector2}, satisfies  
        \[
        \widehat{W^\varepsilon}(0,k)=-\varepsilon^2\widehat{\partial_t E^\varepsilon_{\text{irr}}}(0,k)-\frac{\varepsilon^2}{1+|k|^2}\widehat{\partial_tE^\varepsilon_{\text{sol}}}(0,k)-\mathbf{1}_{k=0} (2\pi)^3\varepsilon^2\partial_tE^\varepsilon_{\text{mean}}(0),
        \]
        it follows that $W^\varepsilon(0)$ is bounded in $H_x^{s-1}$. By the Kondrachov embedding theorem, there exists a subsequence $\varepsilon'$ such that $W^{\varepsilon'}(0)$ converges strongly in $H^{s'-1}_x$ for any $s' < s$.  
        
        By assumption \eqref{sec0:assumptthm2}, we conclude that there exists $w_\Theta(0) \in H^{s'-1}_x$ such that  
        \[
        w^{\varepsilon'}_\Theta(0)=\xi^{\varepsilon'}_\Theta(0)-W^{\varepsilon'}(0) \to w_\Theta(0), \quad \text{in} \quad H^{s'-1}_x.
        \]
        Moreover, the limit $w_\Theta(0)$ does not depend on the subsequence $\varepsilon'$ since, in Theorem \ref{sec1:mainthm2}, we assumed that the entire sequences $\xi^\varepsilon_\Theta(0)$, $j^\varepsilon(0)$, and $B^\varepsilon(0)$ weakly converge in $L^2_x$.  
        
        Taking the weak $L^2_{t,x}$ limit of \eqref{sec4:prelimit}, we obtain  
        \[
        \partial_t w_\Theta+w_\Theta \cdot \nabla_x w_\Theta= E+ w_\Theta\wedge B.
        \]
        Since the solution of this equation is unique for fixed $E,B \in L^\infty_t H^{s-2}_x$, it follows that the limit solution $w_\Theta$ is independent of the chosen subsequence. Thus, we can apply the same argument for all $\Theta \in M$.

        For the continuity equation, by adding and subtracting $\nabla_x \cdot (\rho_\Theta^\varepsilon \xi_\Theta^\varepsilon)$, we get  
        \begin{align*}
            \partial_t \rho_\Theta^\varepsilon +\nabla_x \cdot(\rho_\Theta^\varepsilon \xi_\Theta^\varepsilon) +  \nabla_x \cdot \left(\rho_\Theta^\varepsilon \left( v(\xi_\Theta^\varepsilon) - \xi_\Theta^\varepsilon \right) \right) = 0.
        \end{align*}
        We treat the third term as a remainder of order $\varepsilon^2$, using Lemma \ref{sec2:lemma_sobolev} for the relativistic velocity in Sobolev space.  
        Therefore, the derivation of the continuity equation in the limit follows similarly to the case of $w_\Theta^\varepsilon$. As a result, we obtain that there exists $\rho_\Theta \in C([0,T];H^{s'-2}_x)$ for $s' < s$ such that  
        \[
        \rho_\Theta^\varepsilon \to \rho_\Theta \quad \text{in} \quad C^0([0,T];H^{s'-2}_x),
        \]
        where $\rho_\Theta$ satisfies the continuity equation  
        \[
        \partial_t \rho_\Theta+ \nabla_x \cdot (\rho_\Theta w_\Theta) = 0. 
        \]

         \vskip 0.5 cm
        \noindent \underline{\emph{Limit of the Maxwell system}}:
        By assumption \eqref{sec0:assumptthm2}, we have $\varepsilon E^\varepsilon \in L^\infty_t H^{s-1}_x$. Recalling Gauss's law,  
        \[
        \varepsilon^2 \nabla_x \cdot E^\varepsilon(t,x) = \int_M \rho^\varepsilon_\Theta(t,x) \mu(d \Theta) - 1,
        \]
        we deduce that $\int_M \rho^\varepsilon_\Theta \mu(d\Theta) -1$ converges to zero in $L^2_{t,x}$. Since $\rho^\varepsilon_\Theta$ strongly converges to $\rho_\Theta$ in $L^2_{t,x}$, we obtain the neutrality condition  
        \[
        \int_M \rho_\Theta(t,x) \mu(d \Theta) = 1.
        \]
Next, we consider the Maxwell--Faraday equation. Rewriting it, we obtain  
        \[
        B^\varepsilon(t,x) = B^\varepsilon(0,x) - \int_0^t \nabla_x \wedge E^\varepsilon(s,x) ds = B^\varepsilon(0,x) - \nabla_x \wedge W^\varepsilon(t,x) - \int_0^t \nabla_x \wedge E^\varepsilon_2(s,x) ds.
        \]
        Recalling that $b^\varepsilon(t,x) = B^\varepsilon(t,x) + \nabla_x \wedge W^\varepsilon(t,x)$, it follows that  
        \[
        b^\varepsilon(t,x) = B^\varepsilon(0,x) - \int_0^t \nabla_x \wedge E^\varepsilon_2(s,x) ds.
        \]
        By assumption, $B^\varepsilon(0)$ is bounded in $H^s_x$ and converges weakly. Consequently, it converges in $H^{s'}_x$ to a limit $B(0)$ for any $s' < s$. 
        For the time integral of $\nabla_x \wedge E_2^\varepsilon$, we can rewrite it as  
        \[
        \int_0^t\nabla_x  \wedge E^\varepsilon_2(s) ds = \int_0^t\nabla_x  \wedge \left[ E^\varepsilon_2(s) - (W^\varepsilon\cdot\nabla_x)W^\varepsilon(s) - W^\varepsilon\wedge\left(\nabla_x  \wedge W^\varepsilon(s)\right) \right] ds,
        \]
        using the identity  
        \[
        \nabla_x \wedge ((W^\varepsilon\cdot \nabla_x) W^\varepsilon) + \nabla_x \wedge \left(W^\varepsilon\wedge (\nabla_x \wedge W^\varepsilon) \right) = 0.
        \]
        This follows from the vector calculus identity  
        \[
        (W^\varepsilon\cdot \nabla_x) W^\varepsilon = \frac{1}{2} | W^\varepsilon |^2 + (\nabla_x \wedge W^\varepsilon) \wedge W^\varepsilon.
        \]
        Recalling \eqref{sec4:limitE}, we pass to the limit in $L^2_{t,x}$ and obtain  
        \[
        B(t,x) = B_0(x) - \int_0^t \nabla_x \wedge E(s,x) ds.
        \]
        Finally, we analyze the Maxwell--Ampère law:  
        \[
        \nabla_x \wedge B^\varepsilon(t,x) = \varepsilon^2\partial_t E^\varepsilon (t,x) + \int_M \rho^\varepsilon_\Theta(t,x) v(\xi^\varepsilon_\Theta(t,x))\mu(d\Theta).
        \]
        By adding and subtracting the $\rho^\varepsilon_\Theta(t,x) \xi^\varepsilon_\Theta(t,x)$, we obtain
        \begin{align*}
        \nabla_x \wedge B^\varepsilon(t,x) 
        &= \varepsilon^2\partial_t E^\varepsilon (t,x) 
        + \int_M \rho^\varepsilon_\Theta(t,x) \xi^\varepsilon_\Theta(t,x)\mu(d\Theta) \\
        & \qquad + \int_M \rho^\varepsilon_\Theta(t,x) \left( v(\xi^\varepsilon_\Theta(t,x)) - \xi^\varepsilon_\Theta(t,x) \right) \mu(d\Theta).
        \end{align*}
        The last term can be treated as a remainder of order $\varepsilon^2$ using Lemma \ref{sec2:lemma_sobolev} on the relativistic velocity for Sobolev spaces. Then, using the relation $b^\varepsilon = B^\varepsilon + \nabla_x \wedge W^\varepsilon$, we obtain  
        \[
        \nabla_x \wedge b^\varepsilon(t,x) = \int_M \rho^\varepsilon_\Theta(t,x) w^\varepsilon_\Theta(t,x) \mu(d \Theta) + \nabla_x \wedge (\nabla_x \wedge W^\varepsilon) + \rho^\varepsilon(t,x) W^\varepsilon(t,x) + \varepsilon^2 \partial_t E^\varepsilon.
        \]
        Therefore, taking the $L^2_{t,x}$ limit on both sides and using the fact that $W^\varepsilon, \partial^2_{x_i} W^\varepsilon$, and $\varepsilon^2\partial_t E^\varepsilon$ weakly converge to zero in $L^2_{t,x}$, we obtain  
        \[
        \nabla_x \wedge B(t,x) = \int_M \rho_\Theta(t,x) w_\Theta(t,x) \mu(d \Theta).
        \]
   \end{proof}

\subsection{Correctors in the limit}
\label{sec3:limitcorrect}
   In this section, we prove that as $\varepsilon$ goes to zero, the corrector introduced in \eqref{sec4:corrector} has a limit that is the sum of six terms, which arise from the expressions for the spatial average, irrotational and solenoidal parts in which the electric field is decomposed. As will be clear in the proof, a similar limit decomposition also holds for $\varepsilon E^\varepsilon$ and $B^\varepsilon$. Indeed, as shown in Proposition \ref{sec4:prop2}, the magnetic field $B^\varepsilon$ also converges modulo the corrector $\nabla_x \wedge W^\varepsilon$, from which we obtain the limiting expression for the corrector of the magnetic field.
   
   Given $\phi,\psi \in L^\infty \left( [0,T]; L^2(\mathbb{T}^d_x)\right)$ for $d \in \mathbb{N}$, we define for $t \in [0,T]$ and  $k \in \Z^d$,
   \begin{equation}
   \label{sec4:def_T}
   \begin{aligned}
   \widehat{T^\varepsilon_{1,\pm} \phi}(t,k)&:=\exp\left(\mp \imm \frac{t}{\varepsilon}\right)\widehat{\phi}(t,k),\\
    \widehat{T^\varepsilon_{2,\pm} \psi}(t,k)&:=\exp\left(\mp \imm \sqrt{1+|k|^2}\frac{t}{\varepsilon}\right)\widehat{\psi}(t,k).
    \end{aligned}
   \end{equation}
   We start by stating the following lemma, the proof of which can be found in \cite[Lemma 3.3.3]{Grenier96}.
   \begin{lemma}
   \label{sec4:lemma1}
       Let $j \in \{1,2\}$, then the following results hold:
       \begin{enumerate}
           \item
           \label{sec4:lemmapt1} 
            $T^\varepsilon_{j,+}$
           and $T^\varepsilon_{j,-}$ are adjoints, they are isometries on $L^\infty_t H_x^s$ for any $s \ge 0$ and 
           $T_{j,-}^\varepsilon T_{j,+}^\varepsilon=T_{j,+}^\varepsilon T_{j,-}^\varepsilon=\text{Id}$;
           \item \label{sec4:lemmapt2} If $\phi^\varepsilon \to \phi$ strongly in $L^2_{t,x}$, then $T_{j,+}^\varepsilon\phi^\varepsilon \rightharpoonup 0$ weakly in $L^2_{t,x}$;
           \item \label{sec4:lemmapt3} If $\phi$ and $\psi$ are in $L^\infty_t H^s_x$ with $s>d/2$, then for $i,j\in\{1,2\}$ $T_{i,+}^\varepsilon\left(\phi T_{j,+}^\varepsilon\psi\right) \rightharpoonup 0$ weakly in $L^2_{t,x}$.
       \end{enumerate}
   \end{lemma}
   We can now prove the main result of this section.
   \begin{proposition}
   \label{sec4.2:prop4.4}
   Let $E^\varepsilon_1$ and $W^\varepsilon$ be defined as in \eqref{sec4: decomp} and \eqref{sec4:corrector} and $B$ the limit magnetic field introduced in Proposition \ref{sec4:prop2}.
   
       Given $s'<s$ with $s>\frac32 + 2$, there exist  two spatially independent functions $d_{0,+}, d_{0,-} \in C^0_t$, two irrotational components $d_{1,+}, d_{1,-} \in C^0_tH^{s'-1}_x$ and two solenoidal components $d_{2,+}, d_{2,-} \in C^0_tH^{s'-1}_x$ such that:
       \begin{enumerate}
           \item \label{prop4.4:pt1} $\norm{\varepsilon E^\varepsilon_1 - T_{1,-}^\varepsilon d_{0,+}-T_{1,+}^\varepsilon d_{0,-}-T_{1,-}^\varepsilon d_{1,+}-T_{1,+}^\varepsilon d_{1,-}
        -T_{2,-}^\varepsilon d_{2,+} - T_{2,+}^\varepsilon d_{2,-}
        }_{C^0_t H_x^{s'-1}}\to 0;$
        
          \item \label{prop4.4:pt2} $\norm{W^\varepsilon - T_{1,-}^\varepsilon \widetilde{d}_{0,+}-T_{1,+}^\varepsilon \widetilde{d}_{0,-}- T_{1,-}^\varepsilon \widetilde{d}_{1,+}-T_{1,+}^\varepsilon \widetilde{d}_{1,-}
        -T_{2,-}^\varepsilon \widetilde{d}_{2,+} - T_{2,+}^\varepsilon \widetilde{d}_{2,-}
        }_{C^0_tH_x^{s'-1}}\to 0$;
        \item \label{prop4.4:pt3}
        $\norm{\left(B^\varepsilon +
        T_{2,-}^\varepsilon \nabla_x  \wedge \widetilde{d}_{2,+} + T_{2,+}^\varepsilon \nabla_x  \wedge \widetilde{d}_{2,-}\right) -B
        }_{C^0_tH_x^{s'-2}}\to 0$;
       \end{enumerate}
       where
        \begin{equation}
        \begin{aligned}
            \label{sec4:dd}
        \widetilde{d}_{0,\pm}(t)=\mp\imm\,\, d_{0,\pm}(t) , \qquad \widetilde{d}_{1,\pm}(t,x)=\mathcal{F}^{-1}\Big(\Big\{\mp\imm\,\, \widehat{d_{1,\pm}}(t,k)\Big\}_{k \in \Z^3}\Big) \\
        \text{and} \quad
        \widetilde{d}_{2,\pm}(t,x)=\mathcal{F}^{-1}\Big(\Big\{\mp \imm (1+|k|^2)^{-\frac12}\,\, \widehat{d_{2,\pm}}(t,k)\Big\}_{k \in \Z^3}\Big).
        \end{aligned}
        \end{equation}
   \end{proposition}
   \begin{proof}
    \underline{\emph{Part \eqref{prop4.4:pt1} of Proposition \ref{sec4.2:prop4.4}}}. Given $\varepsilon E^\varepsilon$ we split it into 
    six
    parts so that:
    \[
    \varepsilon E^\varepsilon=E^\varepsilon_{\text{mean,+}}+ E^\varepsilon_{\text{mean},-}+E^\varepsilon_{\text{irr,+}}+E^\varepsilon_{\text{irr},-}+E^\varepsilon_{\text{sol},+}+E^\varepsilon_{\text{sol},-}
    ,
    \]
     where, using Euler's formula, we have
     
    \begin{equation}
    \label{sec4:E3}
    \begin{aligned}
      E^\varepsilon_{\text{mean},\pm}(t)&:= \varepsilon\exp\left(\pm\frac{\imm t}{\varepsilon} \right)\left[
    \frac12 E_{\text{mean}}^\varepsilon(0) \pm \frac{\varepsilon}{2 \imm} \partial_t E^\varepsilon_{\text{mean}}(0) \pm \int_0^t \frac{q^\varepsilon(s)}{2 \imm\varepsilon} \exp\left(\mp\frac{ \imm s}{\varepsilon} \right)ds
    \right],
   \end{aligned}
   \end{equation}
     \begin{align}
       \label{sec4:E1}
    \widehat{E^\varepsilon_{\text{irr},\pm}}(t,k)&:= \varepsilon\exp\left(\pm\frac{\imm t}{\varepsilon} \right)\left[
    \frac12 \widehat{E_{\text{irr}}^\varepsilon}(0,k) \pm \frac{\varepsilon}{2 \imm} \widehat{\partial_t E^\varepsilon_{\text{irr}}}(0,k) \mp \int_0^t k\frac{\widehat{g^\varepsilon}(s,k)}{2|k|^2 \varepsilon} \exp\left(\mp\frac{\imm s}{\varepsilon} \right)ds
    \right],
   \end{align}
   and
    \begin{equation}
    \label{sec4:E2}
    \begin{aligned}
    \widehat{E^\varepsilon_{\text{sol},\pm}}(t,k)&:= \varepsilon\exp\left(\pm\frac{\imm\sqrt{1+|k|^2} t}{\varepsilon} \right)\Biggl[
    \frac12 \widehat{E_{\text{sol}}^\varepsilon}(0,k) \pm  \frac{\varepsilon \widehat{\partial_t E^\varepsilon_{\text{sol}}}(0,k)}{2 \imm\sqrt{1+|k|^2}} \\
    &\quad \pm \int_0^t \frac{k\wedge\widehat{h^\varepsilon}(s,k)}{2|k|^2 \sqrt{1+|k|^2}\varepsilon} \exp\left(\mp\frac{\imm\sqrt{1+|k|^2}s}{\varepsilon} \right)ds
    \Biggr].
   \end{aligned}
   \end{equation}
   
   By part \eqref{sec4:lemmapt1} of Lemma \ref{sec4:lemma1}, we have
\[
\norm{T_{1,\pm}^\varepsilon E^\varepsilon_{\text{mean},\pm}}_{L^\infty_t} = \norm{E^\varepsilon_{\text{mean},\pm}}_{L^\infty_t} \le C,\quad \norm{T_{1,\pm}^\varepsilon E^\varepsilon_{\text{irr},\pm}}_{L^\infty_t H^{s-1}_x} = \norm{E^\varepsilon_{\text{irr},\pm}}_{L^\infty_t H^{s-1}_x} \le C,
\]
and
\[
\norm{T_{2,\pm}^\varepsilon E^\varepsilon_{\text{sol},\pm}}_{L^\infty_t H^{s-1}_x} = \norm{E^\varepsilon_{\text{sol},\pm}}_{L^\infty_t H^{s-1}_x} \le C,
\]
since $E^\varepsilon_{\text{mean},\pm}$, $E^\varepsilon_{\text{irr},\pm}$ and $E^\varepsilon_{\text{sol},\pm}$ contain an extra factor of $\varepsilon$ in the numerator, and thus the same estimates for $\varepsilon E^\varepsilon_1$ in Proposition \ref{sec4:prop1} hold.

    Moreover,
    
   \[
    \partial_t T_{1,\pm}^\varepsilon E^\varepsilon_{\text{mean},\pm}(t)
    =\pm \frac{ q^\varepsilon(t)}{2\imm}\exp\left( \mp \frac{\imm t}{\varepsilon}\right),
    \]
   \[
   \partial_t \reallywidehat{T_{1,\pm}^\varepsilon E^\varepsilon_{\text{irr},\pm}}(t,k)
    =\mp \frac{k\widehat{g^\varepsilon}(t,k)}{2|k|^2}\exp\left( \mp \frac{\imm t}{\varepsilon}\right),
   \]
   and
    \[
   \partial_t \reallywidehat{T_{2,\pm}^\varepsilon E^\varepsilon_{\text{sol},\pm}}(t,k)
    =\pm \frac{k\wedge\widehat{h^\varepsilon}(t,k)}{2|k|^2\sqrt{1+|k|^2}}\exp\left( \mp \frac{\imm t\sqrt{1+|k|^2}}{\varepsilon}\right).
   \]

   By the already mentioned boundedness of $g^\varepsilon$,  $h^\varepsilon$ in $L^\infty_t H^{s-2}_x$ and of$q^\varepsilon$ in $L^\infty_t$, it follows that there exists $C>0$ such that
   \[
    \norm{ \partial_t T_{1,\pm}^\varepsilon E^\varepsilon_{\text{mean},\pm}}_{L^\infty_t}\le C, \quad
   \norm{ \partial_t T_{1,\pm}^\varepsilon E^\varepsilon_{\text{irr},\pm}}_{L^\infty_t H_x^{s-1}}\le C,\quad  
     \norm{ \partial_t T_{2,\pm}^\varepsilon E^\varepsilon_{\text{sol},\pm}}_{L^\infty_t H_x^{s-1}}\le C.
   \]
   Hence, by compactness, there exist  two spatially independent functions $d_{0,\pm}$, two irrotational components $d_{1,\pm}$
   and two solenoidal components $d_{2,\pm}$ such that, up to a subsequence,
   \[
     T_{1,\pm}^\varepsilon E^\varepsilon_{\text{mean},\pm} \to d_{0,\pm} \quad \text{in} \quad C^0_t, \quad \text{i.e.,}\quad E^\varepsilon_{\text{mean},\pm} - T_{1,\mp}^\varepsilon d_{0,\pm} \to 0,
   \]
   \[
    T_{1,\pm}^\varepsilon E^\varepsilon_{\text{irr},\pm} \to d_{1,\pm} \quad \text{in} \quad C^0_t H^{s'-1}_x, \quad \text{i.e.,}\quad E^\varepsilon_{\text{irr},\pm} - T_{1,\mp}^\varepsilon d_{1,\pm} \to 0,
   \]
   and
    \[
    T_{2,\pm}^\varepsilon E^\varepsilon_{\text{sol},\pm} \to d_{2,\pm} \quad \text{in} \quad C^0_tH^{s'-1}_x, 
     \quad \text{i.e.,}\quad E^\varepsilon_{\text{sol},\pm} - T_{2,\mp}^\varepsilon d_{2,\pm} \to 0,
   \]
  
   for $s'<s$.
   The statement follows from the decomposition $\varepsilon E^\varepsilon(t,x)=\varepsilon E^\varepsilon_1(t,x)+\varepsilon E^\varepsilon_2(t,x)$ in \eqref{sec4: decomp} and the fact that $\norm{\varepsilon E_2^\varepsilon}_{L^\infty_t H^{s-1}_x} \to 0$
   from part \eqref{sec4:prop1pt1} of Proposition \ref{sec4:prop1}.
     \vskip 0.5 cm
    
   \noindent\underline{\emph{Part \eqref{prop4.4:pt2} of Proposition \ref{sec4.2:prop4.4}}}.
From the decomposition in \eqref{sec4:corrector2} with \eqref{sec4:formula_O}, we recall that we can write 
    \[
   W^\varepsilon= \mathcal{J}_0^\varepsilon + \mathcal{O}_1^\varepsilon + \mathcal{O}_2^\varepsilon + \mathcal{O}_3^\varepsilon -  \mathcal{R}^\varepsilon,
   \]
   where $\mathcal{O}_1^\varepsilon, \mathcal{O}_2^\varepsilon, \mathcal{O}^\varepsilon_3$ and $\mathcal{R}^\varepsilon$ are defined in \eqref{sec4:jei1}, \eqref{sec4:jei2}, \eqref{sec4:jei3} and \eqref{sec4:Reps}.
Using Euler's formula again, we introduce the following splitting for $W^\varepsilon$:
   \[
   W^\varepsilon(t,x)= W^\varepsilon_{\text{mean},+}(t)+W^\varepsilon_{\text{mean},-}(t)+W^\varepsilon_{\text{irr},+}(t,x)
   + W^\varepsilon_{\text{irr},-}(t,x)+ W^\varepsilon_{\text{sol},+}(t,x)
   +  W^\varepsilon_{\text{sol},-}(t,x) - \mathcal{R}^\varepsilon(t,x),
   \]
  where 
    \begin{align}
    \label{sec4:w1}
    \widehat{W^\varepsilon_{\text{irr},\pm}}(t,k):= \varepsilon\exp\left(\pm \frac{\imm t}{\varepsilon} \right)\left[\pm
    \frac{1}{2\imm} \widehat{E_{\text{irr}}^\varepsilon}(0,k)-\frac{\varepsilon}{2} \widehat{\partial_t E^\varepsilon_{\text{irr}}}(0,k) +
    \int_0^t \imm k\frac{\widehat{g^\varepsilon}(s,k)}{2|k|^2 \varepsilon} \exp\left(\mp\frac{\imm s}{\varepsilon} \right)ds
    \right],
   \end{align}
   \begin{equation}
   \label{sec4:w2}
   \begin{aligned}
    \widehat{W^\varepsilon_{\text{sol},\pm}}(t,k)&:= \varepsilon\exp\left(\pm\frac{\imm \sqrt{1+|k|^2} t}{\varepsilon} \right)\Biggl[\pm
    \frac{ \widehat{E_{\text{sol}}^\varepsilon}(0,k)}{2\imm \sqrt{1+|k|^2}}-\frac{\varepsilon \widehat{\partial_t E^\varepsilon_{\text{sol}}}(0,k)}{2 (1+|k|^2)}  \\
    & \quad -
    \int_0^t \frac{\imm k \wedge \widehat{h^\varepsilon}(s,k)}{2 (1+|k|^2)|k|^2 \varepsilon} \exp\left(\mp\frac{\imm \sqrt{1+|k|^2} s}{\varepsilon} \right)ds
    \Biggr],
   \end{aligned}
   \end{equation}
   and
     \begin{align}
    \label{sec4:w3}
    W^\varepsilon_{\text{mean},\pm}(t):= \varepsilon\exp\left(\pm \frac{\imm t}{\varepsilon} \right)\left[\pm
    \frac{1}{2\imm} E_{\text{mean}}^\varepsilon(0)-\frac{\varepsilon}{2} \partial_t E^\varepsilon_{\text{mean}}(0) +
    \int_0^t \frac{q^\varepsilon(s)}{2 \varepsilon} \exp\left(\mp\frac{\imm s}{\varepsilon} \right)ds
    \right].
   \end{align}
    As for $E^\varepsilon_{\text{irr},\pm}$, 
     $E^\varepsilon_{\text{sol},\pm}$ and $E^\varepsilon_{\text{mean},\pm}$, we deduce that $T_{1,\pm}^\varepsilon W^\varepsilon_{\text{irr},\pm}$, $\partial_t T_{1,\pm}^\varepsilon W^\varepsilon_{\text{irr},\pm}$ and $T_{2,\pm}^\varepsilon W^\varepsilon_{\text{sol},\pm}$, $\partial_t T_{2,\pm}^\varepsilon W^\varepsilon_{\text{sol},\pm}$ are bounded in $L^\infty_t H^{s-1}_x$ and $T_{1,\pm}^\varepsilon W^\varepsilon_{\text{mean},\pm}$, $\partial_t T_{1,\pm}^\varepsilon W^\varepsilon_{\text{mean},\pm}$ are bounded in $L^\infty_t$. It follows there exist two spatially independent functions $\widetilde{d}_{0,\pm}$, two irrotational components $\widetilde{d}_{1,\pm}$ and two solenoidal components $\widetilde{d}_{2,\pm}$ such that
      \[
     W^\varepsilon_{\text{mean},\pm} - T^\varepsilon_{1,\mp}\widetilde{d}_{0,\pm}\to 0 \quad \text{in} \quad C^0_t,
   \]
   and, for $s'<s$,
     \[
     W^\varepsilon_{\text{irr},\pm} - T^\varepsilon_{1,\mp}\widetilde{d}_{1,\pm}\to 0 \quad \text{in} \quad C^0_t H^{s'-1}_x, \quad 
   \]
    \[
 W^\varepsilon_{\text{sol},\pm} -T^\varepsilon_{2,\mp}\widetilde{d}_{2,\pm}  \to 0\quad \text{in} \quad C^0_tH^{s'-1}_x.
   \]
 Since, by \eqref{sec4:estrem}, $\mathcal{R}^\varepsilon$ is converging to $0$ strongly in $L^\infty_t H^{s-1}_x$, the statement of part \eqref{sec4:lemmapt2} of Proposition \ref{sec4:prop1} follows. 
 
 Comparing expressions \eqref{sec4:E3}, \eqref{sec4:E1}, \eqref{sec4:E2} and \eqref{sec4:w1}, \eqref{sec4:w2}, \eqref{sec4:w3} we notice that
 \[
 (\pm \imm)W^\varepsilon_{\text{mean},\pm}=E^\varepsilon_{\text{mean},\pm}, \quad(\pm \imm) W^\varepsilon_{\text{irr},\pm}=E^\varepsilon_{\text{irr},\pm} \quad \text{and} \quad (\pm \imm)\sqrt{1+|k|^2}W^\varepsilon_{\text{sol},\pm}=E^\varepsilon_{\text{sol},\pm}.
 \]
 Hence we get the formula for $\widetilde{d}_{0,\pm}$, $\widetilde{d}_{1,\pm}$ and $\widetilde{d}_{2,\pm}$ given in \eqref{sec4:dd}.
 \vskip 0.5 cm

 \noindent \underline{\emph{Part \eqref{prop4.4:pt3} of Proposition \ref{sec4.2:prop4.4}}}.
 Recalling that $b^\varepsilon=B^\varepsilon+\nabla_x \wedge W^\varepsilon$, by the triangle inequality we have
 \begin{align}
 \label{sec4:correctorB}
 \|B^\varepsilon &+
        T_{2,-}^\varepsilon \nabla_x  \wedge \widetilde{d}_{2,+} + T_{2,+}^\varepsilon \nabla_x  \wedge \widetilde{d}_{2,-} -B
    \|_{C^0_tH_x^{s'-2}}\le \norm{b^\varepsilon - B
        }_{C^0_tH_x^{s'-2}} \nonumber\\
        &\quad + \norm{\nabla_x  \wedge \left(W^\varepsilon- T_{1,-}^\varepsilon \widetilde{d}_{0,+} - T_{1,+}^\varepsilon\widetilde{d}_{0,-}-T_{1,-}^\varepsilon \widetilde{d}_{1,+} - T_{1,+}^\varepsilon\widetilde{d}_{1,-}
        -T_{2,-}^\varepsilon\widetilde{d}_{2,+} - T_{2,+}^\varepsilon\widetilde{d}_{2,-}\right)
        }_{C^0_tH_x^{s'-2}},
 \end{align}
 where we used that 
 $\nabla_x \wedge \widetilde{d}_{0,\pm}=\nabla_x \wedge \widetilde{d}_{1,\pm} = 0$ since $\widetilde{d}_{0,\pm}$ are spatially independent and $\widetilde{d}_{1,\pm}$ are irrotational.
 Hence, since the first term on the r.h.s of \eqref{sec4:correctorB} is going to zero by Proposition \ref{sec4:prop2} and the second term is going to zero by part \eqref{prop4.4:pt2} of this proposition, we obtain the conclusion.
 \end{proof}
   \subsection{Equation for the correctors}
   \label{sec3:eqcorrect}
   The goal here is to derive the equations satisfied by the correctors $d_{0,\pm}$, $d_{1,\pm}$ and $d_{2,\pm}$. To find these equations, we take the weak limit in \eqref{sec3:harmonic}, \eqref{sec3:wave} and \eqref{sec3:spatialmean}. Specifically, for test functions $\phi_1 \in C^\infty_c((0,T))$ and $ \phi_2 \in C^\infty_c((0,T)\times \mathbb{T}^3_x)$, we compute
   \begin{align}
\label{sec4:limitT_0}
    \lim_{\varepsilon\to 0} \langle T^\varepsilon_{1,\pm}\left( (\varepsilon^2\partial^2_{tt} + \text{\rm Id}) 
         E_{\text{mean}}^\varepsilon \right), \phi_1  \rangle
               &= \lim_{\varepsilon\to 0} \langle T^\varepsilon_{1,\pm}q^\varepsilon, \phi_1  \rangle,
\end{align}
\begin{align}
\label{sec4:limitT_1}
    \lim_{\varepsilon\to 0} \langle T^\varepsilon_{1,\pm}\left( (\varepsilon^2\partial^2_{tt} + \text{\rm Id}) 
        \nabla_x \cdot E_{\text{irr}}^\varepsilon \right), \phi_2  \rangle
               &= \lim_{\varepsilon\to 0} \langle T^\varepsilon_{1,\pm}g^\varepsilon, \phi_2  \rangle,
\end{align}
and
\begin{align}
\label{sec4:limitT_2}
    \lim_{\varepsilon\to 0} \langle T^\varepsilon_{2,\pm}\left( \left(\varepsilon^2\partial^2_{tt} + (\text{\rm Id}-\Delta_x)\right) 
        \nabla_x  \wedge  E_{\text{sol}}^\varepsilon \right) , \phi_2  \rangle
               &= \lim_{\varepsilon\to 0} \langle T^\varepsilon_{2,\pm}h^\varepsilon, \phi_2  \rangle,
\end{align}
where $g^\varepsilon$, $h^\varepsilon$ and $q^\varepsilon$ are defined in \eqref{sec3:gepsilon}, \eqref{sec3:hepsilon} and \eqref{sec3:qepsilon}. 

The following proposition states the equations obtained by the limit correctors.
   \begin{proposition}
       Under the assumption of Theorem \ref{sec1:mainthm2}, the equations satisfied by the correctors {$d_{0,+}$, $d_{1,+}$ and $d_{2,+}$ are given by
      \begin{equation}
           \label{sec4:eqcorrector0}
           \begin{aligned}
                2\imm \partial_t  d_{0,+}(t)&=  \frac{\imm}{(2\pi)^6} \sum_{\ell \in \boldsymbol{1}}  \widehat{d_{2,+}}(t,-\ell)\ell\cdot \widehat{d_{1,-}}(t,\ell)\\
               &\quad +\imm   
             \Big[d_{0,+}(t) \wedge \Big(\frac{1}{(2\pi)^3}\int_{\mathbb{T}^3_x} B(0,x)dx \Big)+ \frac{1}{(2\pi)^6}\sum_{\ell \in \Z^3}   \left( \widehat{d_{1,+}}(t,\ell)\wedge \widehat{B}(t,-\ell) \right)\Big] \\
             & \quad + \frac{\imm}{(2\pi)^6}    \sum_{\ell \in \boldsymbol{1}} \widehat{d_{1,-}}(t,-\ell)\wedge\left[\ell\wedge \widehat{d_{2,+}}(t,\ell)\left(1+|\ell|^2\right)^{-\frac{1}{2}}\right],
           \end{aligned}
       \end{equation}
       and
       \begin{equation}
       \label{sec4:eqcorrector1}
       \begin{aligned}
           &-2 k\cdot \partial_t\widehat{d_{1,+}}(t,k)
           =-2  k_i k_j 
           \int_M \reallywidehat{\rho_\Theta (w_\Theta)_i}(t,k) \mu(d\Theta) (d_{0,+})_j(t) \\
           &\quad +\frac{2\imm}{(2\pi)^3} k_i k_j \sum_{\ell \in \Z^3} \left(\int_M(\widehat{\rho_\Theta {w_\Theta}})_i(t,k-\ell)\mu(d\Theta)\right) (\widehat{d_{1,+}})_j(t,\ell) \\
           &\quad - 2\delta_{k \in \boldsymbol{1}} k_i k_j (d_{0,-})_i(t)(\widehat{d_{2,+}})_j(t,k)\left(1+|k|^2\right)^{-\frac{1}{2}} \\
           & \quad -\frac{2}{(2\pi)^3} k_i k_j\sum_{\ell \in \boldsymbol{1}} (\widehat{d_{1,-}})_i(t,k-\ell)
           \left(1+|\ell|^2\right)^{-\frac{1}{2}}(\widehat{d_{2,+}})_j(t,\ell)\\
           & \quad -\frac{2}{(2\pi)^3}k_i k_j \sum_{\ell \in \Omega^{(1)}_{+,-}(k)}  \left(1+|k-\ell|^2\right)^{-\frac{1}{2}}(\widehat{d_{2,-}})_i(t,k-\ell) \left(1+|\ell|^2\right)^{-\frac{1}{2}}(\widehat{d_{2,+}})_j(t,\ell) \\
        &\quad + \frac{k}{(2\pi)^3} \cdot  \sum_{\ell \in \boldsymbol{1}} \widehat{d_{2,+}} (t,\ell)(k-\ell)\cdot \widehat{d_{1,-}}(t,k-\ell)
        \\
        &\quad -k \cdot \left(
             d_{0,+}(t) \wedge \widehat{B}(t,k)  \right)- \frac{k}{(2\pi)^3} \cdot \sum_{\ell \in \Z^3}  \widehat{d}_{1,+}(t,\ell)\wedge \widehat{B}(t,k-\ell)  \\
             & \quad \-\delta_{k \in \boldsymbol{1}} k \cdot \left(d_{0,-}(t)\wedge\left[k\wedge \widehat{d_{2,+}}(t,k)\left(1+|k|^2\right)^{-\frac{1}{2}}\right] \right)\\ 
            & \quad - \frac{k}{(2\pi)^3} \cdot \sum_{\ell \in \boldsymbol{1}} \widehat{d_{1,-}}(t,k-\ell)\wedge\left[\ell\wedge \widehat{d_{2,+}}(t,\ell) \left(1+|\ell|^2\right)^{-\frac{1}{2}} \right]  \\
            & \quad - \frac{k}{(2\pi)^3}\cdot \sum_{\sigma \in\{\pm\}}\sum_{\ell \in \Omega^{(1)}_{\sigma,-\sigma}(k)} \left(1+|k-\ell|^2\right)^{-\frac{1}{2}} \widehat{d_{2,-\sigma}}(t,k-\ell)\wedge\left[\ell \wedge \left(1+|\ell|^2\right)^{-\frac{1}{2}} \widehat{d_{2,\sigma}}(t,\ell)\right],
       \end{aligned}  
       \end{equation}}
       and
       {\small 
        \begin{equation}
        \label{sec4:eqcorrector2}
       \begin{aligned}
           -2& \left(1+|k|^2\right)^{\frac{1}{2}} k\wedge \partial_t\widehat{d_{2,+}}(t,k)
           = \frac{\imm}{(2\pi)^3}
           \sum_{\substack{\ell \in \Z^3 \\ |\ell| =  |k|}}  k \wedge \left( k_i \int_M \reallywidehat{\rho_\Theta (w_\Theta)_i}(t,k-\ell) \mu(d\Theta) \left(1+|\ell|^2\right)^{-\frac{1}{2}}\widehat{d}_{2,+}(t,\ell) \right) \\
           &+ \frac{\imm}{(2\pi)^3}   \sum_{\substack{\ell \in \Z^3 \\ |\ell| =  |k|}}  k \wedge \left( k_i \int_M \reallywidehat{\rho_\Theta w_\Theta}(t,k-\ell) \mu(d\Theta) \left(1+|\ell|^2\right)^{-\frac{1}{2}} (\widehat{d}_{2,+})_i (t,\ell) \right) \\
           &+  k \wedge \left(k_i ((d_{0,+})_i (t) \widehat{d_{1,+}}(t,k) + d_{0,+} (t) (\widehat{d_{1,+}})_i (t,k)) \right) \delta_{k \in \boldsymbol{1}} \nonumber \\
           & +  k \wedge \left(\frac{k_i}{(2\pi)^3} \sum_{\ell \in \mathbb{Z}^3} (\widehat{d_{1,+}})_i(t,k-\ell)\widehat{d_{1,+}}(t,\ell) \right) \delta_{k \in \boldsymbol{1}} \\
           &- k \wedge \left( \frac{k_i}{(2\pi)^3} \sum_{\sigma \in \{\pm\}}\sigma \sum_{\ell \in \Omega^{(2)}_{\sigma,-\sigma}(k)} \left(1+|\ell|^2\right)^{-\frac{1}{2}}[(\widehat{d_{2,+}})_i(t,\ell)\widehat{d_{1,-\sigma}} (t,k-\ell)+(\widehat{d_{2,+}})(t,\ell)(\widehat{d_{1,-\sigma}})_i (t,k-\ell) ]\right)
         \\
         & +  k \wedge \left(  d_{0,+} (t) k \cdot \widehat{d_{1,+}} (t,k) \right) \delta_{k \in \boldsymbol{1}}     + k \wedge \left(   \frac{1}{(2\pi)^3} \sum_{\ell \in \Z^3} \widehat{d_{1,+}} (t,k-\ell)\ell \cdot \widehat{d_{1,+}} (t,\ell) \right) \delta_{k \in \boldsymbol{1}} \\
             & +  k \wedge \left(  \frac{1}{(2\pi)^3} \sum_{\sigma \in \{\pm\}} \sum_{\ell \in \Omega^{(2)}_{\sigma, -\sigma}(k)} \widehat{d_{2,+}} (t,\ell)(k-\ell) \cdot \widehat{d_{1,-\sigma}} (t,k-\ell) \right) \\
        & -  k \wedge \left(\frac{1}{(2\pi)^3} \sum_{\substack{\ell \in \Z^3 \\ |\ell|=|k|}} \widehat{d_{2,+}}(t,\ell) \left(1+\abs{\ell}^2\right)^{-\frac12} \wedge \widehat{B}(t,k-\ell) \right) \\
        & + \frac{\imm}{(2\pi)^3}
        \sum_{\substack{\ell \in \Z^3 \\ |\ell| = |k|}}  
           k \wedge \Biggl(\Bigl(\int_M \widehat{\rho_\Theta w_\Theta}(t,k-\ell) d\mu(\Theta) \Bigr)\wedge \left[\ell\wedge\widehat{d_{2,+}}(t,\ell)\left(1+|\ell|^2\right)^{-\frac{1}{2}}\right] \Biggr)\\
         &-  k \wedge \left(\frac{1}{(2\pi)^3} \sum_{\sigma \in \{\pm\}} \sigma\sum_{\ell \in \Omega^{(2)}_{\sigma,-\sigma}(k)} \widehat{d_{1,-\sigma}}(t,k-\ell)\wedge\left[\ell\wedge \widehat{d_{2,+}}(t,\ell)\left(1+|\ell|^2\right)^{-\frac{1}{2}}\right]  \right),
       \end{aligned}  
       \end{equation}}
       where $\boldsymbol{1}:=\{\ell\in \Z^3 : |\ell|=\sqrt{3}\}$, and
       \begin{equation}
           \label{sec4:eqsigma}
       \begin{aligned}
       \Omega^{(1)}_{\eta_1,\eta_2} (k)&:=\{\ell \in \Z^3 \, : \, 1+\eta_1 \sqrt{1+|k-\ell|^2}+\eta_2 \sqrt{1+|\ell|^2}=0\}, \quad \eta_1, \eta_2 \in \{\pm 1\},\\
        \Omega^{(2)}_{\eta_1,\eta_2} (k)&:=\{\ell \in \Z^3 \, : \, 1+\eta_1 \sqrt{1+|k|^2}+\eta_2 \sqrt{1+|\ell|^2}=0\}, \quad \eta_1, \eta_2 \in \{\pm 1\}.
       \end{aligned}
       \end{equation}
       Similar equations hold for $d_{0,-}$, $d_{1,-}$ and $d_{2,-}$.
       \end{proposition}

       \begin{proof}
       \noindent \underline{\textbf{General idea of the proof.}} To compute the weak limits of the integral quantities in \eqref{sec4:limitT_1} and \eqref{sec4:limitT_2}, we use the Plancherel's identity to express them in the Fourier variables.
       
       It turns out that the quantities involved are oscillatory integrals with different phase functions due to the type of interaction of the electromagnetic waves.  
        The idea is to determine when the phase of the oscillatory integral is nonzero in order to identify which terms survive in the limit as $\varepsilon$ goes to zero. Indeed, by the Riemann-Lebesgue lemma, when this phase is nonzero, the limit tends to zero. 

        When not specified, in this section the symbol $\underset{\varepsilon \to 0}{\lim}$
        refers to the weak limit.

 \vskip 0.5 cm

\noindent\underline{\textbf{Equation for the spatially homogeneous correctors}}: We start studying the equation \eqref{sec3:spatialmean} for the spatially independent term $E^\varepsilon_{\text{mean}}$. Given $\phi \in C^\infty_c((0,T))$, we compute the l.h.s. of \eqref{sec4:limitT_0}. By taking the adjoints of $T^\varepsilon_{1,+}$ and $\partial_{tt}^2$, we have
        \begin{align*}
                \langle T^\varepsilon_{1,+}(\varepsilon^2 \partial_{tt}^2 E^\varepsilon_{\text{mean}} &+ E^\varepsilon_{\text{mean}}), \phi  \rangle
               =
                 \left \langle  E^\varepsilon_{\text{mean}}, \varepsilon^2 \partial_{tt}^2T^\varepsilon_{1,-}\phi + T_{1,-}^\varepsilon\phi \right \rangle\\
                 &=\left \langle T^\varepsilon_{1,-}T^\varepsilon_{1,+}E^\varepsilon_{\text{mean}}\,,\, \varepsilon^2 \partial_{tt}^2T^\varepsilon_{1,-}\phi + T_{1,-}^\varepsilon\phi\right \rangle=\left \langle \varepsilon T_{1,+}^\varepsilon E^\varepsilon_{\text{mean}}\,,\, \frac{1}{\varepsilon}\Big(\varepsilon^2 T_{1,+}^\varepsilon\partial_{tt}^2T^\varepsilon_{1,-}\phi + \phi\Big)\right \rangle.
           \end{align*}
            
        By part \eqref{prop4.4:pt1} of Proposition \ref{sec4.2:prop4.4}, we have that $\varepsilon T_{1,+}^\varepsilon E^\varepsilon_{\text{mean}} \rightharpoonup d_{0,+}$ in $L^2_{t}$, moreover
           \begin{align*}
               &\frac{\varepsilon^2 T^\varepsilon_{1,+}\partial_{tt}^2 T^\varepsilon_{1,-}\phi(t)+\phi(t)}{\varepsilon}
               =
               \frac{1}{\varepsilon}\exp\left(-\frac{\mathrm{i} t}{\varepsilon} \right)\varepsilon^2 \partial_{tt}^2 \left( \exp\left(\frac{\mathrm{i} t}{\varepsilon}\right) \phi(t)\right)
                + \frac{1}{\varepsilon} \phi(t)\\
             &=\varepsilon \exp\left(-\frac{\imm t}{\varepsilon}\right) \left[-\frac{1}{\varepsilon^2}\phi(t) + \frac{2 \mathrm{i}}{\varepsilon}\partial_t \phi(t)+\partial_{tt}^2\phi(t) \right]\exp\left(\frac{\imm t}{\varepsilon} \right)+ \frac{1}{\varepsilon} \phi(t)=
               2 \imm\partial_t \phi(t) + \varepsilon \partial_{tt}^2 \phi(t).
           \end{align*}
      It follows that 
           \[
           \frac{\varepsilon^2 T^\varepsilon_{1,+} \partial_{tt}^2 T_{1,-}^\varepsilon \phi +\phi}{\varepsilon} \to 2\imm \partial_t \phi \quad \text{in} \quad L^2_t \quad \text{strongly}.
           \]
           We conclude that
           \begin{align}
           \label{sec4:limit_LHS_mean}
              \lim_{\varepsilon\to 0} 
              \langle T^\varepsilon_{1,+}(\varepsilon^2 \partial_{tt}^2 E^\varepsilon_{\text{mean}} + E^\varepsilon_{\text{mean}}), \phi  \rangle
               &= \left \langle d_{0,+},  2 \imm \partial_t \phi \right \rangle =\left \langle 2\imm \partial_t  d_{0,+}, \phi  \right \rangle.
           \end{align}
        We now want to analyse the r.h.s. of \eqref{sec4:limitT_0}.
        Let us define 
        \begin{equation}
        \label{sec4:def_E_B_mean}
        \begin{aligned}
           &\mathcal{E}_{\text{mean}}^\varepsilon(t):=\fint_{\mathbb{T}^3_x} \varepsilon^2 E^\varepsilon(t,x)\nabla_x \cdot E_{\text{irr}}^\varepsilon(t,x) dx, \quad\quad  \mathcal{B}^\varepsilon_{\text{mean}}(t):=\fint_{\mathbb{T}^3_x} j^\varepsilon(t,x) \wedge B^\varepsilon(t,x)dx,\\
        & \qquad \qquad \qquad \qquad \qquad\qquad \mathcal{R}^\varepsilon_{\text{mean}}(t):=\fint_{\mathbb{T}^3_x}R^\varepsilon(t,x) dx,
        \end{aligned}
        \end{equation}
        where $\fint_{\mathbb{T}^3_x}= (2\pi)^{-3}\int_{\mathbb{T}^3_x}$.
        Therefore, we study
               \begin{equation}
               \label{sec4:differenterms_mean}
    \begin{aligned}
    \lim_{\varepsilon\to 0} \langle T^\varepsilon_{1,+}q^\varepsilon, \phi  \rangle= \lim_{\varepsilon\to 0} \langle
    T_{1,+}^\varepsilon\left(\mathcal{E}_{\text{mean}}^\varepsilon-\mathcal{B}_{\text{mean}}^\varepsilon
    - \mathcal{R}_{\text{mean}}^\varepsilon\right), \phi \rangle,
    \end{aligned}
    \end{equation}
    where $R^\varepsilon$ is defined in \eqref{sec3:defR}.
    \vskip 0.5 cm

        \noindent\underline{\emph{Limit of the $ \mathcal{E}_{\text{mean}}^\varepsilon$ term}}: We start studying the first term in \eqref{sec4:differenterms_mean}.
        By part \eqref{prop4.4:pt1} of Proposition \ref{sec4.2:prop4.4}, we have for  $\phi\in C^\infty_c((0,T) )$
         \begin{align}
         \label{sec4:eqmean1}
            &\lim_{\varepsilon\to 0} \left \langle 
            T^\varepsilon_{1,+}\left(\fint_{\mathbb{T}^3_x}\varepsilon 
           E^\varepsilon \nabla_x \cdot(\varepsilon E_{\text{irr}}^\varepsilon)dx
            \right), \phi \right \rangle =\nonumber\\
            &\qquad \sum_{\substack{\sigma_3,\sigma_4 \in \{\pm\}}}\lim_{\varepsilon\to 0}
            \left \langle 
            T^\varepsilon_{1,+}\left( \fint_{\mathbb{T}^3_x}
            T^\varepsilon_{1,\sigma_3}({d_{0,-\sigma_3}})
            \nabla_x \cdot T^\varepsilon_{1,\sigma_4}({d_{1,-\sigma_4}})
            dx\right), \phi \right \rangle
            \nonumber \\&\qquad+ \sum_{\substack{\sigma_1 \in\{1,2\}, \\\sigma_3,\sigma_4 \in \{\pm\}}}\lim_{\varepsilon\to 0}
            \left \langle 
            T^\varepsilon_{1,+}\left( \fint_{\mathbb{T}^3_x}
            T^\varepsilon_{\sigma_1,\sigma_3}({d_{\sigma_1,-\sigma_3}})
            \nabla_x \cdot T^\varepsilon_{1,\sigma_4}({d_{1,-\sigma_4}})
            dx\right), \phi \right \rangle,
        \end{align}
      where we note that only $T^\varepsilon_{1,\pm}({d_{1,\mp}})$ appears with the divergence since $\nabla_x \cdot T^\varepsilon_{2,\pm}({d_{2,\mp}}) =\nabla_x \cdot T^\varepsilon_{0,\pm}({d_{0,\mp}})= 0$, as $d_{0,\mp}$ are spatially independent and $d_{2,\mp}$ are solenoidal.
      
        The first term on the r.h.s. of \eqref{sec4:eqmean1} is vanishing by the divergence theorem, since $d_{0,\pm}$ is spatially independent.
        To compute the second term on the r.h.s. of \eqref{sec4:eqmean1}, for $\sigma_1 \in\{1,2\}, \sigma_3,\sigma_4 \in\{\pm\}$ and $\psi_1:=d_{\sigma_1,-\sigma_3}, \psi_2:=d_{1,-\sigma_4} \in L^\infty_t H^{s-2}_x$, we look at which terms in
        \[
        \mathcal{M}_{\sigma_1,1, \sigma_3, \sigma_4}(t) :=(2\pi)^{3}\mathcal{F}\left(T^\varepsilon_{1,+}\left( 
            T^\varepsilon_{\sigma_1,\sigma_3}(\psi_1)
            T^\varepsilon_{1,\sigma_4}(\nabla_x \cdot\psi_2)
            \right)
            \right)(t,0)
        \]
        give a non vanishing contribution as $\varepsilon$ goes to zero. 

        If $\sigma_1=1$, using the formula for $T^\varepsilon_{1,\pm}$ given in \eqref{sec4:def_T}, we have
        \[
        \mathcal{M}_{1,1,\sigma_3,\sigma_4}(t)=\imm\exp\left(-\frac{\imm t}{\varepsilon}[1+\sigma_31+\sigma_41] \right)\sum_{\ell \in \Z^3} \widehat{\psi_1}(t,-\ell)\ell\cdot \widehat{\psi_2}(t,\ell),
        \]
        with $\sigma_3,\sigma_4 \in \{\pm\}$. It follows that the phase of the oscillatory integrand $\mathcal{M}_{1,1,\sigma_3,\sigma_4}$ is always non-zero 
and therefore gives a vanishing contribution as $\varepsilon$ goes to zero.

         If $\sigma_1=2$ we have
        \[
        \mathcal{M}_{2,1,\sigma_3,\sigma_4}(t)= -\imm \sum_{\ell \in \Z^3}\exp\left(-\frac{ \imm t}{\varepsilon}[1+\sigma_3\sqrt{1+|\ell|^2}+\sigma_41] \right) \widehat{\psi_1}(t,\ell)\ell \cdot\widehat{\psi_2}(t,-\ell),
        \]
        with $\sigma_3,\sigma_4 \in \{\pm\}$. We notice that the phase of the oscillatory integrand $\mathcal{M}_{2,1,\sigma_3,\sigma_4}$ is non-zero, except when $\sigma_3=-, \sigma_4=+$ so that the phase is zero for $|\ell|=\sqrt3$, i.e.,
        \[
        \mathcal{M}_{2,1,-,+}(t)=-\imm \sum_{\ell \in \boldsymbol{1}} \widehat{\psi_1}(t,\ell)\ell\cdot \widehat{\psi_2}(t,-\ell) + \text{oscillatory terms},
        \]
        where $\boldsymbol{1}=\{ \ell \in \Z^3 : \ell_i \in \{\pm1 \}, i \in \{1,2,3\}\}$. 

        Hence,
        \begin{align}
        \label{sec4:limit_RHS_2_mean}
            \lim_{\varepsilon\to 0} \left \langle 
            T^\varepsilon_{1,+} \fint_{\mathbb{T}^3_x} \left( 
           E^\varepsilon \nabla_x \cdot(\varepsilon E_{\text{irr}}^\varepsilon)
            \right), \phi \right \rangle 
            =  \frac{\imm}{(2\pi)^6} \sum_{\ell \in \boldsymbol{1}} \int_{-\infty}^{+\infty}  \widehat{d_{2,+}}(t,-\ell)\ell\cdot \widehat{d_{1,-}}(t,\ell)
             \phi(t) dt.
        \end{align}
        \vskip 0.5 cm
        \noindent\underline{\emph{Limit of the $ \mathcal{B}^\varepsilon_{\text{mean}}$ term}}: We now study the second term in \eqref{sec4:differenterms_mean}, i.e., we consider
        \begin{equation*}
         \mathcal{B}_{\text{mean}}^\varepsilon = \mathcal{B}_{\text{mean},1}^\varepsilon +\mathcal{B}_{\text{mean},2}^\varepsilon,
        \end{equation*}
        where 
        \[
        \mathcal{B}_{\text{mean},1}^\varepsilon:= \fint_{\mathbb{T}^3_x}\int_{M}\rho^\varepsilon_\Theta \xi^\varepsilon_\Theta \mu(d\Theta)\wedge B^\varepsilon dx,
        \quad \text{and} \quad
        \mathcal{B}_{\text{mean},2}^\varepsilon :=\fint_{\mathbb{T}^3_x}\int_{M}\rho^\varepsilon_\Theta \left(v(\xi^\varepsilon_\Theta)-\xi^\varepsilon_\Theta \right) \mu(d\Theta)\wedge B^\varepsilon dx.
        \]
        The treatment of $\mathcal{B}_{\text{mean},2}^\varepsilon$ is done using Lemma \ref{sec2:lemma_sobolev} on the difference between the relativistic and non relativistic velocity. Hence, $\mathcal{B}_{\text{mean},2}^\varepsilon$ is a remainder term of order $\varepsilon^2$ which is strongly converging to zero.
        Therefore, we only focus on the $\mathcal{B}_{\text{mean},1}^\varepsilon$ term. By recalling the two formulas in \eqref{sec4:smallb}, we have
        \[
        \mathcal{B}_{\text{mean},1}^\varepsilon=\fint_{\mathbb{T}^3_x}\int_M \rho^\varepsilon_\Theta(w^\varepsilon_\Theta+W^\varepsilon)\mu(d\Theta)
        \wedge \left( b^\varepsilon(t,x) - \nabla_x  \wedge W^\varepsilon \right) dx
        =:\mathcal{A}_{\text{m},1}+\mathcal{A}_{\text{m},2}+\mathcal{A}_{\text{m},3} + \mathcal{A}_{\text{m},4},
        \]
        where
        \[
        \mathcal{A}_{\text{m},1}:=\int_M \rho^\varepsilon_\Theta w^\varepsilon_\Theta \mu(d\Theta)\wedge b^\varepsilon,\quad \mathcal{A}_{\text{m},2}:=\int_M\rho^\varepsilon_\Theta \mu(d\Theta)W^\varepsilon\wedge b^\varepsilon,
        \]
        \[
        \mathcal{A}_{\text{m},3}:=-\int_M\rho^\varepsilon_\Theta w^\varepsilon_\Theta \mu(d\Theta)\wedge (\nabla_x  \wedge W^\varepsilon),\quad 
        \mathcal{A}_{\text{m},4}:=-\int_M\rho^\varepsilon_\Theta \mu(d\Theta)W^\varepsilon\wedge (\nabla_x  \wedge  W^\varepsilon).
        \]
  For $\mathcal{A}_{\text{m},1}$, since $\fint\int_M \rho^\varepsilon_\Theta w^\varepsilon_\Theta \mu(d\Theta) \wedge b^\varepsilon \to \fint\int_M \rho_\Theta w_\Theta \mu(d\Theta) \wedge B$ in $L^2$ strongly by Proposition \ref{sec4:prop2}, we obtain, using part \eqref{sec4:lemmapt2} of Lemma \ref{sec4:lemma1},
        \[
         T^\varepsilon_{1,+}\mathcal{A}_{\text{m},1}=T^\varepsilon_{1,+} \left(\fint_{\mathbb{T}^3_x}\int_M \rho^\varepsilon_\Theta w^\varepsilon_\Theta  \mu(d\Theta) \wedge b^\varepsilon dx\right) \rightharpoonup0 \quad \text{in} \quad L^2_{t}.
        \]
        Concerning $\mathcal{A}_{\text{m},2}$, we know $\int_M \rho^\varepsilon_\Theta  \mu(d\Theta) \to 1$ and $b^\varepsilon \to B$ in $L^2$ strongly by Proposition \ref{sec4:prop2}. Therefore, by adding and subtracting these two limits to $\mathcal{A}_{\text{m},2}$, we get
        \begin{align*}
            \lim_{\varepsilon\to 0} T^\varepsilon_{1,+} \mathcal{A}_{\text{m},2} 
            = \lim_{\varepsilon\to 0} T^\varepsilon_{1,+} \left(\fint_{\mathbb{T}^3_x}\int_M\rho^\varepsilon_\Theta \mu(d\Theta)  W^\varepsilon \wedge b^\varepsilon dx\right)  =
         \lim_{\varepsilon\to 0} T^\varepsilon_{1,+}\left(\fint_{\mathbb{T}^3_x}W^\varepsilon \wedge B dx \right).
        \end{align*}
        By part \eqref{prop4.4:pt2} of Proposition \ref{sec4.2:prop4.4}, we have
        \begin{align*}
            &\lim_{\varepsilon \to 0} T^\varepsilon_{1,+}\mathcal{A}_{\text{m},2}=\\
        &\quad\lim_{\varepsilon \to 0} T^\varepsilon_{1,+}\left( \fint_{\mathbb{T}^3_x}\left((T^\varepsilon_{1,+} (\widetilde{d}_{0,-})+(T^\varepsilon_{1,-} (\widetilde{d}_{0,+})+T^\varepsilon_{1,+} (\widetilde{d}_{1,-}) + T^\varepsilon_{1,-} (\widetilde{d}_{1,+})+  T^\varepsilon_{2,+} (\widetilde{d}_{2,-}) + T^\varepsilon_{2,-} (\widetilde{d}_{2,+}) \right) \wedge B \right).
        \end{align*}
        However, by part \eqref{sec4:lemmapt3} of Lemma \ref{sec4:lemma1}, 
        \[
        \lim_{\varepsilon \to 0} T^\varepsilon_{1,+} \left[ \fint_{\mathbb{T}^3_x}\left(T^\varepsilon_{1,+} (\widetilde{d}_{0,-}) +T^\varepsilon_{1,+} (\widetilde{d}_{1,-}) + T^\varepsilon_{2,+} (\widetilde{d}_{2,-}) \right)\wedge B dx\right]=0.
        \]
        Hence, we conclude that the following identity holds,
        \[
        \lim_{\varepsilon \to 0} T^\varepsilon_{1,+}\mathcal{A}_{\text{m},2}=
        \lim_{\varepsilon \to 0} T^\varepsilon_{1,+}\left[ \fint_{\mathbb{T}^3_x}\left(T^\varepsilon_{1,-} (\widetilde{d}_{0,+})+ T^\varepsilon_{1,-} (\widetilde{d}_{1,+})+T^\varepsilon_{2,-} (\widetilde{d}_{2,+})\right)\wedge B dx\right].
        \]
        Taking  $\phi\in C^\infty_c((0,T))$, we compute the weak limit of the first term in the last equation.
        By using the expression for $T_{1,\pm}^\varepsilon$ given by \eqref{sec4:def_T} and the one of $\widetilde{d}_{0,+}$ given by \eqref{sec4:dd}, we obtain
        \begin{equation}
        \label{sec4:corrector_B_term_A_2_mean1}
            \lim_{\varepsilon \to 0}\left \langle 
            T^\varepsilon_{1,+}\left( \fint
             T^\varepsilon_{1,-} (\widetilde{d}_{0,+})\wedge B dx
            \right), \phi \right \rangle 
             =
            - \imm \int_{-\infty}^{+\infty}  
             d_{0,+}(t) \wedge \Big(\fint B(t,x)dx \Big)\phi(t) dt,
        \end{equation}
         which contributes to the limit equation. 
        
        Concerning the weak limit of the second term: By using the expression for $T_{1,\pm}^\varepsilon$ given by \eqref{sec4:def_T} and the one of $\widetilde{d}_{1,+}$ given by \eqref{sec4:dd}, we obtain
        \begin{equation}
        \label{sec4:corrector_B_term_A_2_mean2}
            \lim_{\varepsilon \to 0}\left \langle 
            T^\varepsilon_{1,+}\left( \fint 
             T^\varepsilon_{1,-} (\widetilde{d}_{1,+})\wedge B dx
            \right), \phi \right \rangle 
             =
            - \frac{\imm}{(2\pi)^6} \int_{-\infty}^{+\infty}  \sum_{\ell \in \Z^3}
             \widehat{d_{1,+}}(t,\ell) \wedge \widehat{B}(t,-\ell) \phi(t) dt,
        \end{equation}
         which contributes to the limit equation. We now look at the weak limit of the third term, i.e.,
       \begin{align*}
            &\lim_{\varepsilon \to 0} \left \langle 
            T^\varepsilon_{1,+}\left( 
            \fint  T^\varepsilon_{2,-} (\widetilde{d}_{2,+})\wedge B dx
            \right), \phi \right \rangle 
            \\
            &  =  \lim_{\varepsilon \to 0}
            -\frac{\imm}{(2\pi)^6} \sum_{\ell \in \Z^3}\int_{-\infty}^{+\infty}  \exp\left(-\frac{\imm}{\varepsilon}\left(-\sqrt{1+|\ell|^2}+1\right)t \right)
              \widehat{d_{2,+}}(t,\ell) \wedge \widehat{B}(t,-\ell)  \left(1+|\ell|^2\right)^{-\frac{1}{2}} \phi(t)dt.
            \end{align*}
        Hence, the last integral is always of oscillatory type except when $\ell=0$, but in this case, $(\widehat{d_{2,+}})_j(t,0)=0$ since it is solenoidal. It follows that \begin{align*}
            &\lim_{\varepsilon \to 0} \left \langle 
            T^\varepsilon_{1,+}\left( \fint
             T^\varepsilon_{2,-} (\widetilde{d}_{2,+})\wedge B dx
            \right), \phi \right \rangle   = 0.
        \end{align*}
        We conclude that the contributions for $\mathcal{A}_{\text{m},2}$ are given by \eqref{sec4:corrector_B_term_A_2_mean1} and \eqref{sec4:corrector_B_term_A_2_mean2}. Hence
        \begin{align}
        \label{sec4:limit_RHS_3_mean}
            \lim_{\varepsilon\to 0} &\left \langle T_{1,+}^\varepsilon  \mathcal{A}_{\text{m},2},\phi\right \rangle =  - \imm \int_{-\infty}^{+\infty}  
             \left[d_{0,+}(t) \wedge \Big(\fint B(t,x)dx \Big)+ \frac{1}{(2\pi)^6}\sum_{\ell \in \Z^3}   \left( \widehat{d_{1,+}}(t,\ell)\wedge \widehat{B}(t,-\ell) \right)\right] \phi(t)dt.
        \end{align}
        Next, we study $\mathcal{A}_{\text{m},3}$. Since $\int_M\rho^\varepsilon_\Theta w^\varepsilon_\Theta \mu(d\Theta) \to \int_M\rho_\Theta w_\Theta \mu(d\Theta)$, by adding and subtracting the limit, we get
        \begin{align*}
        \lim_{\varepsilon \to 0} T^\varepsilon_{1,+}\mathcal{A}_{\text{m},3} =
            - \lim_{\varepsilon \to 0} T^\varepsilon_{1,+} \left(
            \fint \int_M\rho_\Theta w_\Theta \mu(d\Theta)\wedge (\nabla_x  \wedge W^\varepsilon) dx\right).
        \end{align*}
        Next, by part \eqref{prop4.4:pt2} of Proposition \ref{sec4.2:prop4.4}, we have
        {\small\begin{align*}
        &\lim_{\varepsilon \to 0} T^\varepsilon_{1,+}\mathcal{A}_{\text{m},3} =\\
            &- \lim_{\varepsilon \to 0} T^\varepsilon_{1,+} \left(
            \fint \int_M\rho_\Theta w_\Theta \mu(d\Theta)\wedge \left(\nabla_x  \wedge \left(T^\varepsilon_{1,+} (\widetilde{d}_{0,-}+\widetilde{d}_{1,-})+ 
            T^\varepsilon_{1,-} (\widetilde{d}_{0,+}+\widetilde{d}_{1,+})+
            T^\varepsilon_{2,+} (\widetilde{d}_{2,-}) + T^\varepsilon_{2,-} (\widetilde{d}_{2,+}) \right) \right) dx\right).
        \end{align*}}
        Observe that $\nabla_x  \wedge \left(T^\varepsilon_{1,+} (\widetilde{d}_{0,-}+\widetilde{d}_{1,-}) + T^\varepsilon_{1,-} (\widetilde{d}_{0,+}+\widetilde{d}_{1,+})\right)  = 0$ since $\widetilde{d}_{0,\pm}$ are spatially independent and $\widetilde{d}_{1,\pm}$ are irrotational. Moreover, by part \eqref{sec4:lemmapt3} of Lemma \ref{sec4:lemma1} and by commuting the curl with $T_{2,+}^\varepsilon$, we have
        \begin{align*}
             \lim_{\varepsilon \to 0} T^\varepsilon_{1,+} \left(\fint
            \int_M\rho_\Theta w_\Theta \mu(d\Theta)\wedge \left( T^\varepsilon_{2,+} (\nabla_x  \wedge \widetilde{d}_{2,-})  \right) dx\right) = 0.
        \end{align*}
        Hence, the following identity holds:
        \[
        \lim_{\varepsilon \to 0} T^\varepsilon_{1,+}\mathcal{A}_{\text{m},3}=
        - \lim_{\varepsilon \to 0} T^\varepsilon_{1,+}\left( \fint \int_M \rho_\Theta w_\Theta\mu(d\Theta) \wedge T^\varepsilon_{2,-} (\nabla_x  \wedge \widetilde{d}_{2,+}) dx\right).
        \]
         Taking  $\phi\in C^\infty_c((0,T) )$ and  $\psi(t,x):=  \int_M \rho_\Theta w_\Theta\mu(d\Theta) \in L^\infty_tH^{s-2}_x$, we compute the  weak limit of the last equation. By the expression for $T_{1,+}^\varepsilon$, $T_{2,-}^\varepsilon$ and the one of $\widetilde{d}_{2,+}$ given by \eqref{sec4:dd}, we obtain
       \begin{align*}
        & 
            \lim_{\varepsilon \to 0} \left \langle 
            T^\varepsilon_{1,+}\left( \fint 
            \psi \wedge T^\varepsilon_{2,-} (\nabla_x  \wedge \widetilde{d}_{2,+})dx
            \right), \phi \right \rangle 
            =  \lim_{\varepsilon \to 0}
            \int_{-\infty}^{+\infty}  
            \exp\left(-\frac{\imm t}{\varepsilon}\right) \\
            &\qquad \times \frac{1}{(2\pi)^6} \sum_{\ell \in \Z^3} \widehat{\psi}(t,-\ell)\wedge \left[\ell\wedge\widehat{d_{2,+}}(t,\ell)\left(1+|\ell|^2\right)^{-\frac{1}{2}}\right]\exp\left(\imm \frac{\sqrt{1+|\ell|^2}t}{\varepsilon} \right) \phi(t)  dt\\
            &= \lim_{\varepsilon \to 0}\frac{1}{(2\pi)^6}\sum_{\ell \in \Z^3}\int_{-\infty}^{+\infty}  \exp\left(\frac{\imm}{\varepsilon}\left(\sqrt{1+|\ell|^2}-1\right)t \right)
           \widehat{\psi}(t,-\ell)\wedge \left[\ell\wedge\widehat{d_{2,+}}(t,\ell)\left(1+|\ell|^2\right)^{-\frac{1}{2}}\right] \phi(t) dt.
            \end{align*}
        Hence, the limit is vanishing except when $\ell=0$, which gives a zero contribution.

        We finally study $\mathcal{A}_{\text{m},4}$. We know that $\int_M \rho^\varepsilon_\Theta  \mu(d\Theta) \to 1$ in $L^2$ strongly by Proposition \ref{sec4:prop2}. Therefore, by adding and subtracting the limit, we obtain
        \begin{align*}
            \lim_{\varepsilon \to 0} T^\varepsilon_{1,+}\mathcal{A}_{\text{m},4} 
            = -\lim_{\varepsilon \to 0} T^\varepsilon_{1,+} 
            \left(\fint \int_M\rho^\varepsilon_\Theta \mu(d\Theta)W^\varepsilon\wedge (\nabla_x  \wedge  W^\varepsilon) dx\right) 
            = -\lim_{\varepsilon \to 0} T^\varepsilon_{1,+} 
            \left(\fint W^\varepsilon\wedge (\nabla_x  \wedge  W^\varepsilon) dx\right).
        \end{align*}
        By part \eqref{prop4.4:pt2} of Proposition \ref{sec4.2:prop4.4}, we have
        \begin{align}
        \label{sec4:eqmean11}
            &\lim_{\varepsilon\to 0}\left \langle 
            T^\varepsilon_{1,+}\left( \fint 
            W^\varepsilon\wedge (\nabla_x  \wedge  W^\varepsilon)dx
            \right), \phi \right \rangle 
            =\nonumber \\
            &\sum_{\substack{\sigma_3,\sigma_4 \in \{\pm\}}}\lim_{\varepsilon\to 0}
            \left \langle 
            T^\varepsilon_{1,+}\left( 
            T^\varepsilon_{1,\sigma_3}(\widetilde{d}_{0,-\sigma_3})\wedge
            \left(\fint \nabla_x  \wedge T^\varepsilon_{2,\sigma_4}(\widetilde{d}_{2,-\sigma_4})dx\right)
            \right), \phi \right \rangle \nonumber\\
            &+\sum_{\substack{\sigma_1 \in\{1,2\}, \\\sigma_3,\sigma_4 \in \{\pm\}}}\lim_{\varepsilon\to 0}
            \left \langle 
            T^\varepsilon_{1,+}\left( \fint 
            T^\varepsilon_{\sigma_1,\sigma_3}(\widetilde{d}_{\sigma_1,-\sigma_3})\wedge
            (\nabla_x  \wedge T^\varepsilon_{2,\sigma_4} (\widetilde{d}_{2,-\sigma_4}))dx
            \right), \phi \right \rangle,
        \end{align}
        where we note that only $T^\varepsilon_{2,\pm}({d_{2,\mp}})$ appears with the curl operator since $\nabla_x \wedge T^\varepsilon_{1,\pm}({d_{1,\mp}}) = \nabla_x \wedge T^\varepsilon_{1,\pm}({d_{0,\mp}})=0$, as $d_{0,\mp}$ are spatially independent and $d_{1,\mp}$ are irrotational.
        
        The first term on the r.h.s. of \eqref{sec4:eqmean11} is vanishing by the Stokes' theorem. Concerning the second term on the r.h.s of \eqref{sec4:eqmean11} we study, for $\sigma_1 \in\{1,2\}, \sigma_3,\sigma_4 \in\{\pm\}$ and $\psi_1:=\widetilde{d}_{\sigma_1,-\sigma_3}, \psi_2:=\widetilde{d}_{2,-\sigma_4} \in L^\infty_t H^{s-2}_x$, which terms in 
        \[
        \mathcal{\bar B}_{\sigma_1,2, \sigma_3, \sigma_4}(t) : =(2\pi)^3\mathcal{F}\left(T^\varepsilon_{1,+}\left(\fint 
            T^\varepsilon_{\sigma_1,\sigma_3}(\psi_1)\wedge
            T^\varepsilon_{2,\sigma_4}(\nabla_x  \wedge \psi_2)dx
            \right)
            \right)(t,0)
        \]
        give a non vanishing contribution as $\varepsilon$ goes to zero. 

              If $\sigma_1=1$, using the formula for $T^\varepsilon_{1,\pm}$ and $T^\varepsilon_{2,\pm}$ given in \eqref{sec4:def_T}, we have
        \[
        \mathcal{\bar B}_{1,2,\sigma_3,\sigma_4}(t)=\imm\sum_{\ell \in \Z^3}\exp\left(-\frac{\imm t}{\varepsilon}[1+\sigma_31+\sigma_4\sqrt{1+|\ell|^2}] \right) \widehat{\psi_1}(t,-\ell)\wedge[\ell\wedge \widehat{\psi_2}(t,\ell)],
        \]
        with $\sigma_3,\sigma_4 \in \{\pm\}$. We notice that the phase of the oscillatory integrand $\mathcal{\bar B}_{1,2,\sigma_3,\sigma_4}$ is non-zero, except when $\sigma_3=+, \sigma_4=-$ so that the phase is zero for $|\ell|=\sqrt3$, i.e.,
        \[
        \mathcal{\bar B}_{1,2,+,-}(t)=\imm \sum_{\ell \in \boldsymbol{1}} \widehat{\psi_1}(t,-\ell)\wedge[\ell\wedge \widehat{\psi_2}(t,\ell)]+ \text{oscillatory terms},
        \]
        where $\boldsymbol{1}=\{ \ell \in \Z^3 : \ell_i \in \{\pm1 \}, i \in \{1,2,3\}\}.$

         If $\sigma_1=2$ we have
        \[
        \mathcal{\bar B}_{2,2,\sigma_3,\sigma_4}(t)=\imm\sum_{\ell \in \Z^3}\exp\left(-\frac{\imm t}{\varepsilon}\left[1+\sigma_3\sqrt{1+|\ell|^2}+\sigma_4\sqrt{1+|\ell|^2}\right] \right) \widehat{\psi_1}(t,-\ell)\wedge[\ell \wedge \widehat{\psi_2}(t,\ell)],
        \]
        with $\sigma_3,\sigma_4 \in \{\pm\}$. We notice that the phase of the oscillatory integrand $\mathcal{B}_{2,2,\sigma_3,\sigma_4}$ is always non-zero.

          We conclude that 
        \begin{align}
        \label{sec4:limit_RHS_3_mean_bis}
            - \lim_{\varepsilon\to 0}   \left \langle T_{1,+}^\varepsilon \mathcal{A}_{\text{m},4},\phi\right \rangle = \frac{\imm}{(2\pi)^6}   \int_{-\infty}^{+\infty}   \sum_{\ell \in \boldsymbol{1}} \widehat{d_{1,-}}(t,-\ell)\wedge\left[\ell\wedge \widehat{d_{2,+}}(t,\ell)\left(1+|\ell|^2\right)^{-\frac{1}{2}}\right]\phi(t)dt.
        \end{align}

         The study of the fourth term in \eqref{sec4:differenterms_mean}, which corresponds to the remainder, follows by using the Sobolev version of inequality \eqref{Lemma_remainder} in Lemma \ref{sec1:lemmarem}  and is converging to zero.
        \vskip 0.5 cm
        \noindent\underline{\emph{Conclusion}}: Finally, by collecting \eqref{sec4:limit_LHS_mean}, \eqref{sec4:limit_RHS_2_mean}, \eqref{sec4:limit_RHS_3_mean}, and \eqref{sec4:limit_RHS_3_mean_bis} we get the equation satisfied by the corrector $d_{0,+}$ given by \eqref{sec4:eqcorrector0}. \\

        \vskip 0.5 cm

 \noindent\underline{\textbf{Equation for the irrotational correctors}}: We now study the equation \eqref{sec3:harmonic} for the irrotational part $E^\varepsilon_{\text{irr}}$. Given $\phi \in C^\infty_c((0,T)\times \mathbb{T}^3_x)$, we compute the l.h.s. of \eqref{sec4:limitT_1}. By taking the adjoints of $T^\varepsilon_{1,+}$, $\partial_{tt}^2$ and the divergence operator, we have
        \begin{align*}
                \langle T^\varepsilon_{1,+}(\varepsilon^2 \partial_{tt}^2 \nabla_x \cdot E^\varepsilon_{\text{irr}} &+\nabla_x \cdot E^\varepsilon_{\text{irr}}), \phi  \rangle
               =
                 \left \langle \nabla_x \cdot E^\varepsilon_{\text{irr}}, \varepsilon^2 \partial_{tt}^2T^\varepsilon_{1,-}\phi + T_{1,-}^\varepsilon\phi \right \rangle=-   \left \langle E^\varepsilon_{\text{irr}}, \nabla_x \left(\varepsilon^2 \partial_{tt}^2T^\varepsilon_{1,-}\phi + T_{1,-}^\varepsilon\phi\right) \right \rangle\\
                 &=-   \left \langle T^\varepsilon_{1,-}T^\varepsilon_{1,+}E^\varepsilon_{\text{irr}}, \nabla_x \left(\varepsilon^2 \partial_{tt}^2T^\varepsilon_{1,-}\phi + T_{1,-}^\varepsilon\phi\right) \right \rangle\\
                 &=- \left \langle \varepsilon T_{1,+}^\varepsilon E^\varepsilon_{\text{irr}}, \frac{1}{\varepsilon}\nabla_x \left(\varepsilon^2 T_{1,+}^\varepsilon\partial_{tt}^2T^\varepsilon_{1,-}\phi + \phi\right) \right \rangle,
           \end{align*}
            where in the last equality we used that $\nabla_x$ commute with $T^\varepsilon_{1,\pm}$.
            
        By part \eqref{prop4.4:pt1} of Proposition \ref{sec4.2:prop4.4}, we have that $\varepsilon T_{1,+}^\varepsilon E^\varepsilon_{\text{irr}} \rightharpoonup d_{1,+}$ in $L^2_{t,x}$, moreover
           \begin{align*}
               &\frac{\reallywidehat{\varepsilon^2 T^\varepsilon_{1,+}\partial_{tt}^2 T^\varepsilon_{1,-}\phi}+\widehat{\phi}}{\varepsilon}(t,k)
               =
               \frac{1}{\varepsilon}\exp\left(-\frac{\mathrm{i} t}{\varepsilon} \right)\varepsilon^2 \partial_{tt}^2 \left( \exp\left(\frac{\mathrm{i} t}{\varepsilon}\right) \widehat{\phi}(t,k)\right)
                + \frac{1}{\varepsilon} \widehat{\phi}(t,k)\\
             &=\varepsilon \exp\left(-\frac{\imm t}{\varepsilon}\right) \left[-\frac{1}{\varepsilon^2}\widehat{\phi}(t,k) + \frac{2 \mathrm{i}}{\varepsilon}\partial_t \widehat{\phi}(t,k)+\partial_{tt}^2\widehat{\phi}(t,k) \right]\exp\left(\frac{\imm t}{\varepsilon} \right)+ \frac{1}{\varepsilon} \widehat{\phi}(t,k)\\
            &=
               2 \imm\partial_t \widehat{\phi}(t,k) + \varepsilon \partial_{tt}^2 \widehat{\phi}(t,k).
           \end{align*}
      It follows that 
           \[
           \frac{\reallywidehat{\varepsilon^2 T^\varepsilon_{1,+} \partial_{tt}^2 T_{1,-}^\varepsilon \phi} +\widehat{\phi}}{\varepsilon} \to 2\imm \partial_t \widehat{\phi} \quad \text{in} \quad L^2_t H^1_x \quad \text{strongly}.
           \]
           We conclude that
           \begin{align}
           \label{sec4:limit_LHS}
              \lim_{\varepsilon\to 0} 
              \langle T^\varepsilon_{1,+}(\varepsilon^2 \partial_{tt}^2 \nabla_x \cdot E^\varepsilon_{\text{irr}} +\nabla_x \cdot E^\varepsilon_{\text{irr}}), \phi  \rangle
               &= -\left \langle d_{1,+}, \nabla_x \mathcal{F}^{-1}\left(\left\{2 \imm \partial_t \widehat{\phi} \right\}_{k \in \Z^3}\right) \right \rangle   =\left \langle 2\imm \partial_t \nabla_x \cdot (d_{1,+}), \phi  \right \rangle.
           \end{align}
        We now want to analyse the r.h.s. of \eqref{sec4:limitT_1}.
        Let us define 
        \begin{equation}
        \label{sec4:def_J_E_B}
        \begin{aligned}
           &\mathcal{J}^\varepsilon(t,x):=  \partial_{x_i} \left(
        \int_{M} \rho_\Theta^\varepsilon v(\xi^\varepsilon_\Theta)_i v(\xi^\varepsilon_\Theta) \mu(d\Theta) \right), \quad\quad \mathcal{E}^\varepsilon(t,x):= \varepsilon^2 E^\varepsilon(t,x) \left(\nabla_x \cdot E_{\text{irr}}^\varepsilon(t,x) \right), \\
        & \qquad \qquad \qquad \qquad \qquad\qquad \mathcal{B}^\varepsilon(t,x):=j^\varepsilon(t,x) \wedge B^\varepsilon(t,x).
        \end{aligned}
        \end{equation}
        Therefore, we study
               \begin{equation}
               \label{sec4:differenterms}
    \begin{aligned}
    \lim_{\varepsilon\to 0} \langle T^\varepsilon_{1,+}g^\varepsilon, \phi  \rangle= \lim_{\varepsilon\to 0} \langle
    T_{1,+}^\varepsilon\left(\nabla_x \cdot \mathcal{J}^\varepsilon(t,x)-\nabla_x \cdot \mathcal{E}^\varepsilon(t,x)-\nabla_x \cdot \mathcal{B}^\varepsilon(t,x)
    -\nabla_x \cdot R^\varepsilon(t,x)\right), \phi \rangle,
    \end{aligned}
    \end{equation}
    where $R^\varepsilon$ is defined in \eqref{sec3:defR}.
    \vskip 0.5 cm
    
    \noindent\underline{\emph{Limit of the $\nabla_x \cdot \mathcal{J}^\varepsilon$ term}}: We study the first term in \eqref{sec4:differenterms}, i.e., we consider
        $$\nabla_x \cdot \mathcal{J}^\varepsilon(t,x)=
        \sum_{i,j=1}^3\partial_{x_i}\partial_{x_j} \mathcal{J}_1^\varepsilon(t,x)
        +\sum_{i,j=1}^3\partial_{x_i}\partial_{x_j}\mathcal{J}_2^\varepsilon(t,x),$$
        where
        $$
        \mathcal{J}_1^\varepsilon(t,x):= \int_{M}\rho^\varepsilon_\Theta(t,x) (\xi^\varepsilon_\Theta(t,x))_i (\xi^\varepsilon_\Theta(t,x))_j \mu(d\Theta)
        $$
        and 
        $$
        \mathcal{J}_2^\varepsilon(t,x):=\int_{M}\rho^\varepsilon_\Theta(t,x) [v(\xi^\varepsilon_\Theta(t,x))_iv(\xi^\varepsilon_\Theta(t,x))_j-(\xi^\varepsilon_\Theta(t,x))_i(\xi^\varepsilon_\Theta(t,x))_j] \mu(d\Theta),
        $$
        for $i,j \in \{1, 2, 3\}$. The treatment of $\mathcal{J}_2^\varepsilon$ is done using Lemma \ref{sec2:lemma_sobolev} on the difference between the relativistic and non relativistic velocity. Therefore, $\mathcal{J}_2^\varepsilon$ is a remainder term of order $\varepsilon^2$ which is strongly converging to zero.

        By the definition of $w^\varepsilon_\Theta=\xi^\varepsilon_\Theta - W^\varepsilon$ and \eqref{sec4:corrector}, we have
        \[
        \mathcal{J}_1^\varepsilon=\int_M \rho^\varepsilon_\Theta((w^\varepsilon_\Theta)_i +W^\varepsilon_i)((w^\varepsilon_\Theta)_j + W^\varepsilon_j)\mu(d\Theta)=:\mathcal{A}^{i,j}_1+\mathcal{A}^{i,j}_2+\mathcal{A}^{j,i}_2+\mathcal{A}^{i,j}_3,
        \]
        where
        \[
        \mathcal{A}^{i,j}_1:=\int_M \rho^\varepsilon_\Theta(w^\varepsilon_\Theta)_i(w^\varepsilon_\Theta)_j \mu(d\Theta),\, \,\mathcal{A}_2^{i,j}:=W^\varepsilon_j\int_M\rho^\varepsilon_\Theta (w^\varepsilon_\Theta)_i \mu(d\Theta), \,\, \mathcal{A}^{i,j}_3:=W^\varepsilon_i W^\varepsilon_j\int_M\rho^\varepsilon_\Theta \mu(d\Theta).
        \]
        For $\mathcal{A}^{i,j}_1$, since $\int_M \rho^\varepsilon_\Theta (w^\varepsilon_\Theta)_i (w^\varepsilon_\Theta)_j \mu(d\Theta) \to \int_M \rho_\Theta (w_\Theta)_i (w_\Theta)_j \mu(d\Theta)$ in $L^2$ strongly by Proposition \ref{sec4:prop2}, we obtain, using part \eqref{sec4:lemmapt2} of Lemma \ref{sec4:lemma1},
        \[
         T^\varepsilon_{1,+}\mathcal{A}^{i,j}_1=T^\varepsilon_{1,+} \left(\int_M \rho^\varepsilon_\Theta (w^\varepsilon_\Theta)_i (w^\varepsilon_\Theta)_j \mu(d\Theta) \right) \rightharpoonup0 \quad \text{in} \quad L^2_{t,x}.
        \]
        Concerning $\mathcal{A}^{i,j}_2$, since $\int_M \rho^\varepsilon_\Theta (w^\varepsilon_\Theta)_i \mu(d\Theta) \to \int_M \rho_\Theta (w_\Theta)_i \mu(d\Theta)$ in $L^2$ strongly by Proposition \ref{sec4:prop2}, we get by adding and subtracting $\int_M \rho_\Theta (w_\Theta)_i \mu(d\Theta)$
        \begin{align}
        \label{sec4:J_limit_A_2}
        \lim_{\varepsilon\to 0} T^\varepsilon_{1,+}\mathcal{A}^{i,j}_2
        &=
        \lim_{\varepsilon\to 0} T^\varepsilon_{1,+}\left( \int_M \rho^\varepsilon_\Theta (w^\varepsilon_\Theta)_i -\rho_\Theta (w_\Theta)_i \mu(d\Theta) \right) W^\varepsilon_j  \nonumber \\
        & \quad + 
         \lim_{\varepsilon\to 0} T^\varepsilon_{1,+}\left(\int_M \rho_\Theta (w_\Theta)_i \mu(d\Theta) \right) W^\varepsilon_j \nonumber \\
         & = \lim_{\varepsilon\to 0} T^\varepsilon_{1,+}\left(\int_M \rho_\Theta (w_\Theta)_i \mu(d\Theta) \right) W^\varepsilon_j .
        \end{align}
        We now remind that, by part \eqref{prop4.4:pt2} Proposition \ref{sec4.2:prop4.4}, we have
        \[
        W^\varepsilon_j - T_{1,-}^\varepsilon (\widetilde{d}_{0,+})_j-T_{1,+}^\varepsilon (\widetilde{d}_{0,-})_j -T^\varepsilon_{1,-} (\widetilde{d}_{1,+})_j - T^\varepsilon_{1,+} (\widetilde{d}_{1,-})_j- T^\varepsilon_{2,-} (\widetilde{d}_{2,+})_j - T^\varepsilon_{2,+} (\widetilde{d}_{2,-})_j \to 0 \quad \text{in} \quad C^0_t H^{s-1}_x,
        \]
        where $\widetilde{d}_{0,\pm}$, $\widetilde{d}_{1,\pm}$ and $\widetilde{d}_{2,\pm}$ are defined in \eqref{sec4:dd}. Therefore, we can rewrite the last limit as 
        \begin{align}
        \label{sec4:splitting_W}
        &\lim_{\varepsilon\to 0} T^\varepsilon_{1,+} \mathcal{A}^{i,j}_2 
        = \lim_{\varepsilon\to 0} T^\varepsilon_{1,+}\left(\int_M \rho_\Theta (w_\Theta)_i \mu(d\Theta) \right) \nonumber \\
         & \qquad \qquad \qquad\times \left(  T_{1,-}^\varepsilon (\widetilde{d}_{0,+})_j-T_{1,+}^\varepsilon (\widetilde{d}_{0,-})_j + T^\varepsilon_{1,-} (\widetilde{d}_{1,+})_j +T^\varepsilon_{1,+} (\widetilde{d}_{1,-})_j  + T^\varepsilon_{2,-} (\widetilde{d}_{2,+})_j + T^\varepsilon_{2,+} (\widetilde{d}_{2,-})_j \right) \nonumber \\
         & \quad + \lim_{\varepsilon\to 0} T^\varepsilon_{1,+}\left(\int_M \rho_\Theta (w_\Theta)_i \mu(d\Theta) \right) \nonumber \\
         & \qquad \qquad \times
         \left( W^\varepsilon_j - T_{1,-}^\varepsilon (\widetilde{d}_{0,+})_j-T_{1,+}^\varepsilon (\widetilde{d}_{0,-})_j  -T^\varepsilon_{1,-} (\widetilde{d}_{1,+})_j - T^\varepsilon_{1,+} (\widetilde{d}_{1,-})_j- T^\varepsilon_{2,-} (\widetilde{d}_{2,+})_j - T^\varepsilon_{2,+} (\widetilde{d}_{2,-})_j \right),
        \end{align}
        and the last term is going to zero thanks to part \eqref{prop4.4:pt2} of Proposition \ref{sec4.2:prop4.4}.
        Moreover, by part \eqref{sec4:lemmapt3} of Lemma \ref{sec4:lemma1}, we notice that,
        \[
        \lim_{\varepsilon \to 0} T^\varepsilon_{1,+} \left[\left(\int_M \rho_\Theta (w_\Theta)_i \mu(d\Theta)\right)\left(T_{1,+}^\varepsilon (\widetilde{d}_{0,-})_j+ T^\varepsilon_{1,+} (\widetilde{d}_{1,-})_j + T^\varepsilon_{2,+} (\widetilde{d}_{2,-})_j \right)\right]=0.
        \]
        We conclude that the following identity holds:
        \begin{equation}
        \label{sec4:lim_T_A2_mean}
        \lim_{\varepsilon \to 0} T^\varepsilon_{1,+}\mathcal{A}^{i,j}_2=
        \lim_{\varepsilon \to 0} T^\varepsilon_{1,+}\left[\left( \int_M \rho_\Theta (w_\Theta)_i \mu(d\Theta)\right) \left(T_{1,-}^\varepsilon (\widetilde{d}_{0,+})_j +T^\varepsilon_{1,-} (\widetilde{d}_{1,+})_j+T^\varepsilon_{2,-} (\widetilde{d}_{2,+})_j\right)\right].
        \end{equation}
        Taking  $\phi\in C^\infty_c((0,T) \times \mathbb{T}^3_x)$ and defining $\psi(t,x):=  \int_M \rho_\Theta (w_\Theta)_i \mu(d\Theta)\in L^\infty_tH^{s-2}_x$, we compute the  weak limit of the first term in \eqref{sec4:lim_T_A2_mean}. That is,
        \begin{align}
        \label{sec4:correctorlim1_mean}
             \lim_{\varepsilon \to 0}\left \langle 
            T^\varepsilon_{1,+}\left( 
            \psi T^\varepsilon_{1,-} (\widetilde{d}_{0,+})_j
            \right), \phi \right \rangle 
            = \lim_{\varepsilon \to 0} \left \langle 
            \psi (\widetilde{d}_{0,+})_j
            \, ,\,  \phi \right \rangle 
        \end{align}
       where we used that $d_{0,+}$ is spatially homogeneous and that $T^\varepsilon_{1,+} T^\varepsilon_{1,-}= \text{Id}$. Therefore, this term contributes to the limit equation.
        Then, we compute the  weak limit of the second term in \eqref{sec4:lim_T_A2_mean}. By using Plancherel identity, the expression for $T_{1,\pm}^\varepsilon$ given by \eqref{sec4:def_T} and the one of $\widetilde{d}_{1,+}$ given by \eqref{sec4:dd}, we obtain
        \begin{align}
        \label{sec4:correctorlim1}
             &\lim_{\varepsilon \to 0}\left \langle 
            T^\varepsilon_{1,+}\left( 
            \psi T^\varepsilon_{1,-} (\widetilde{d}_{1,+})_j
            \right), \phi \right \rangle 
            = \lim_{\varepsilon \to 0}\frac{1}{(2\pi)^3} \int_{-\infty}^{+\infty}  \sum_{k \in \Z^3} \exp{\left( \frac{ - \imm t}{\varepsilon} \right)} \reallywidehat{\left( 
            \psi T^\varepsilon_{1,-} (\widetilde{d}_{1,+})_j
            \right)} (t,k) \overline{\widehat{\phi}}(t,k) dt \nonumber \\
            & \qquad= \lim_{\varepsilon \to 0} \frac{1}{(2\pi)^6}\int_{-\infty}^{+\infty}  \sum_{k \in \Z^3} \exp{\left( \frac{ - \imm t}{\varepsilon} \right)} \sum_{\ell \in \Z^3}  
            \widehat{\psi} (t,k-\ell) \exp{\left( \frac{\imm t}{\varepsilon} \right)} \widehat{(\widetilde{d}_{1,+})_j}
             (t,\ell) \overline{\widehat{\phi}}(t,k) dt \nonumber \\
            &\qquad = - \frac{\imm}{(2\pi)^6} \int_{-\infty}^{+\infty}  \sum_{k,\ell \in \Z^3}
            \widehat{\psi}(t,k-\ell) (\widehat{d_{1,+}})_j(t,\ell) \overline{\widehat{\phi}}(t,k) dt.
        \end{align}
        Therefore, this term contributes to the limit equation.
       We now look at the weak limit of the third term in \eqref{sec4:lim_T_A2_mean}, i.e.,
       \begin{align*}
              &\lim_{\varepsilon \to 0} \left \langle 
            T^\varepsilon_{1,+}\left( 
            \psi T^\varepsilon_{2,-} (\widetilde{d}_{2,+})_j
            \right), \phi \right \rangle 
            =  \lim_{\varepsilon \to 0} - \frac{\imm}{(2\pi)^6}
            \int_{-\infty}^{+\infty}  \sum_{k \in \Z^3}
            \exp\left(-\frac{\imm t}{\varepsilon}\right) \\
            &\qquad \qquad\times \sum_{\ell \in \Z^3} \widehat{\psi}(t,k-\ell) \exp\left(\imm \frac{\sqrt{1+|\ell|^2}t}{\varepsilon} \right) (\widehat{d_{2,+}})_j(t,\ell) \left(1+|\ell|^2\right)^{-\frac{1}{2}} \overline{\widehat{\phi}}(t,k)  dt\\
            & \quad  =  \lim_{\varepsilon \to 0} -\frac{\imm}{(2\pi)^6}\sum_{k \in \Z^3}\sum_{\ell \in \Z^3}\int_{-\infty}^{+\infty}  \exp\left(\frac{\imm}{\varepsilon}\left(\sqrt{1+|\ell|^2}-1\right)t \right)
             \widehat{\psi}(t,k-\ell) (\widehat{d_{2,+}})_j(t,\ell) \left(1+|\ell|^2\right)^{-\frac{1}{2}} \overline{\widehat{\phi}}(t,k)dt.
            \end{align*}
        Hence, the last integral is always of oscillatory type except when $\ell=0$, but in this case $(\widehat{d_{2,+}})_j(t,0)=0$ since it is solenoidal. It follows that
        \begin{align*}
            \lim_{\varepsilon \to 0} \left \langle 
            T^\varepsilon_{1,+}\left( 
            \psi T^\varepsilon_{2,-} (\widetilde{d}_{2,+})_j
            \right), \phi \right \rangle  = 0.
        \end{align*}
        We conclude that the contributions for $\mathcal{A}_2^{i,j}$ are given by \eqref{sec4:correctorlim1_mean} and \eqref{sec4:correctorlim1}. Hence
        \begin{align}
        \label{sec4:limit_RHS_1}
            \lim_{\varepsilon\to 0} &\left \langle  T_{1,+}^\varepsilon \partial_{x_i}\partial_{x_j} (\mathcal{A}_2^{i,j}+\mathcal{A}_2^{j,i}),\phi\right \rangle \nonumber\\
            &= -\frac{2}{(2\pi)^3} \sum_{k \in \Z^3} \int_{-\infty}^{+\infty} k_i k_j 
           \left( \int_M \reallywidehat{\rho_\Theta (w_\Theta)_i}(t,k) \mu(d\Theta) (d_{0,+})_j(t) \right) \overline{\widehat{\phi}}(t,k) dt \nonumber \\
            & \qquad + \frac{2\imm}{(2\pi)^6}\sum_{k,\ell \in \Z^3}  \int_{-\infty}^{+\infty}  k_i k_j \left( \int_M \reallywidehat{\rho_\Theta (w_\Theta)_i}(t,k-\ell) \mu(d\Theta) (\widehat{d}_{1,+})_j(t,\ell)
            \right)
            \bar{\widehat{\phi}}(t,k)dt.
        \end{align}
        Finally, we study $\mathcal{A}^{i,j}_3$. We know that $\int_M \rho^\varepsilon_\Theta  \mu(d\Theta) \to 1$ in $L^2$ strongly by Proposition \ref{sec4:prop2}. Therefore, by adding and subtracting the limit as in \eqref{sec4:J_limit_A_2}, we obtain
        \begin{align*}
            \lim_{\varepsilon\to 0}  T^\varepsilon_{1,+}\mathcal{A}^{i,j}_3 = \lim_{\varepsilon\to 0}  T^\varepsilon_{1,+}
            \left( W^\varepsilon_i W^\varepsilon_j\int_M\rho^\varepsilon_\Theta \mu(d\Theta) \right) = \lim_{\varepsilon\to 0}  T^\varepsilon_{1,+} \left(W^\varepsilon_i W^\varepsilon_j \right).
        \end{align*}
        By part \eqref{prop4.4:pt2} of Proposition \ref{sec4.2:prop4.4} and proceeding as in \eqref{sec4:splitting_W}, we have
        \begin{align*}
            &\lim_{\varepsilon\to 0}\left \langle 
            T^\varepsilon_{1,+}\left( 
             W^\varepsilon_i W^\varepsilon_j
            \right), \phi \right \rangle 
            =\sum_{\substack{\sigma_1,\sigma_2 \in\{0,1,2\}, \\\sigma_3,\sigma_4 \in \{\pm\}}}\lim_{\varepsilon\to 0}
            \left \langle 
            T^\varepsilon_{1,+}\left( 
            T^\varepsilon_{\sigma_1,\sigma_3}(\widetilde{d}_{\sigma_1,-\sigma_3})_i
            T^\varepsilon_{\sigma_2,\sigma_4}(\widetilde{d}_{\sigma_2,-\sigma_4})_j
            \right), \phi \right \rangle.
        \end{align*}
        where we defined $T_{0,\pm}^\varepsilon = T_{1,\pm}^\varepsilon$ for notational convenience.
        We now study, for $\sigma_1,\sigma_2 \in\{0,1,2\}$,  $\sigma_3,\sigma_4 \in\{\pm\}$ and $\psi_1:=(\widetilde{d}_{\sigma_1,-\sigma_3})_i, \psi_2:=(\widetilde{d}_{\sigma_2,-\sigma_4})_j \in L^\infty_t H^{s-2}_x$, which terms in 
        \[
        \mathcal{J}_{\sigma_1, \sigma_2, \sigma_3, \sigma_4}(t,k):=\mathcal{F}\left(T^\varepsilon_{1,+}\left( 
            T^\varepsilon_{\sigma_1,\sigma_3}(\psi_1)
            T^\varepsilon_{\sigma_2,\sigma_4}(\psi_2)
            \right)
            \right)(t,k)
        \]
        give a non vanishing contribution as $\varepsilon$ goes to zero.

        If $\sigma_1, \sigma_2 \in \{ 0, 1 \}$, using the formula for $T^\varepsilon_{1,\pm}$ given in \eqref{sec4:def_T}, we have
        \[
        \mathcal{J}_{\sigma_1,\sigma_2,\sigma_3,\sigma_4}(t,k)=\exp\left(-\frac{\imm t}{\varepsilon}[1+\sigma_31+\sigma_41] \right)\frac{1}{(2\pi)^3}\sum_{\ell \in \Z^3} \widehat{\psi_1}(t,k-\ell)\widehat{\psi_2}(t,\ell),
        \]
        with $\sigma_3,\sigma_4 \in \{\pm\}$. It follows that, for $\sigma_1, \sigma_2 \in \{ 0, 1 \}$, the phase of the oscillatory integrand $\mathcal{J}_{\sigma_1,\sigma_2,\sigma_3,\sigma_4}$ is always non-zero and therefore gives a vanishing contribution as $\varepsilon$ goes to zero.

        If $\sigma_1 =0$ and $\sigma_2=2$, we get
        \[
        \mathcal{J}_{0,2,\sigma_3,\sigma_4}(t,k)=\frac{1}{(2\pi)^3}\sum_{\ell \in \Z^3}\exp\left(-\frac{\imm t}{\varepsilon}[1+\sigma_31+\sigma_4 \sqrt{1+|\ell|^2}] \right) \widehat{\psi_1}(t,k-\ell)\widehat{\psi_2}(t,\ell),
        \]
        with $\sigma_3,\sigma_4 \in \{\pm\}$. We notice that the phase of the oscillatory integrand $\mathcal{J}_{0,2,\sigma_3,\sigma_4}$ is non-zero, except when $\sigma_3=+, \sigma_4=-$ so that the phase is zero for $|\ell|=\sqrt3$, i.e.,
        \[
        \mathcal{J}_{0,2,+,-}(t,k)=\frac{1}{(2\pi)^3}\sum_{\ell \in \boldsymbol{1}} \widehat{\psi_1}(t,k-\ell)\widehat{\psi_2}(t,\ell)+ \text{oscillatory terms}
        \]
        where $\boldsymbol{1}=\{ \ell \in \Z^3 : \ell_i \in \{\pm1 \}, i \in \{1,2,3\}\}$. The analogous conclusion holds for $\mathcal{J}_{2,0,-,+}$

         If $\sigma_1\sigma_2=2$, we can assume w.l.o.g. that $\sigma_1=1, \sigma_2=2$ and we get
        \[
        \mathcal{J}_{1,2,\sigma_3,\sigma_4}(t,k)=\frac{1}{(2\pi)^3}\sum_{\ell \in \Z^3}\exp\left(-\frac{\imm t}{\varepsilon}[1+\sigma_31+\sigma_4 \sqrt{1+|\ell|^2}] \right) \widehat{\psi_1}(t,k-\ell)\widehat{\psi_2}(t,\ell),
        \]
        with $\sigma_3,\sigma_4 \in \{\pm\}$. We notice that the phase of the oscillatory integrand $\mathcal{J}_{1,2,\sigma_3,\sigma_4}$ is non-zero, except when $\sigma_3=+, \sigma_4=-$ so that the phase is zero for $|\ell|=\sqrt3$, i.e.,
        \[
        \mathcal{J}_{1,2,+,-}(t,k)=\frac{1}{(2\pi)^3}\sum_{\ell \in \boldsymbol{1}} \widehat{\psi_1}(t,k-\ell)\widehat{\psi_2}(t,\ell)+ \text{oscillatory terms}
        \]
        where $\boldsymbol{1}=\{ \ell \in \Z^3 : \ell_i \in \{\pm1 \}, i \in \{1,2,3\}\}$. The analogous conclusion holds for $\mathcal{J}_{2,1,-,+}$.

        If $\sigma_1=\sigma_2=2$ we have
           \[
        \mathcal{J}_{2,2,\sigma_3,\sigma_4}(t,k)=\frac{1}{(2\pi)^3}\sum_{\ell \in \Z^3}\exp\left(-\frac{\imm t}{\varepsilon}[1+\sigma_3\sqrt{1+|k-\ell|^2}+\sigma_4\sqrt{1+|\ell|^2}] \right) \widehat{\psi_1}(t,k-\ell)\widehat{\psi_2}(t,\ell),
        \]
        with $\sigma_3,\sigma_4 \in \{\pm\}$. We notice that the phase of the oscillatory integrand $\mathcal{J}_{2,2,\sigma_3,\sigma_4}$ is non-zero, except when $\sigma_3\sigma_4=-$. Hence
        \[
        \mathcal{J}_{2,2,\sigma_3,-\sigma_3}(t,k)=\frac{1}{(2\pi)^3}\sum_{\ell \in \Omega^{(1)}_{\sigma_3,-\sigma_3}(k)} \widehat{\psi_1}(t,k-\ell)\widehat{\psi_2}(t,\ell)+ \text{oscillatory terms},
        \]
        where $\Omega^{(1)}_{\sigma_3,-\sigma_3}(k)$ is defined by \eqref{sec4:eqsigma}.
        
        Collecting the limit contributions given by $\mathcal{J}_{0,2,+,-}, \mathcal{J}_{2,0,-,+}, $ $\mathcal{J}_{1,2,+,-}, \mathcal{J}_{2,1,-,+}, \mathcal{J}_{2,2,+,-}$ and $\mathcal{J}_{2,2,-,+}$ we arrive at
        \begin{align}
        \label{sec4:limit_RHS_1_bis}
        &\lim_{\varepsilon \to 0} \left \langle  T_{1,+}^\varepsilon \partial_{x_i} \partial_{x_j} \mathcal{A}_3^{i,j}, \phi\right \rangle=
        - \frac{2}{(2\pi)^3}\sum_{k \in \boldsymbol{1}}\int_{-\infty}^{+\infty} k_i k_j \left((d_{0,-})_i(t)(\widehat{d_{2,+}})_j(t,k)\left(1+|k|^2\right)^{-\frac{1}{2}} \right)
        \,\overline{\widehat{\phi}}(t,k)dt \nonumber\\
        &- \frac{2}{(2\pi)^6}\sum_{k \in \Z^3}\int_{-\infty}^{+\infty} k_i k_j \left( \sum_{\ell \in \boldsymbol{1}} (\widehat{d_{1,-}})_i(t,k-\ell)(\widehat{d_{2,+}})_j(t,\ell)\left(1+|\ell|^2\right)^{-\frac{1}{2}} \right)
        \,\overline{\widehat{\phi}}(t,k)dt \nonumber\\
        & -\frac{2}{(2\pi)^6}\sum_{k \in \Z^3}\int_{-\infty}^{+\infty} k_i k_j \Bigg( \sum_{\ell \in \Omega^{(1)}_{+,-}(k)} (\widehat{d_{2,-}})_i(t,k-\ell)\left(1+|k-\ell|^2\right)^{-\frac{1}{2}}(\widehat{d_{2,+}})_j(t,\ell)\left(1+|\ell|^2\right)^{-\frac{1}{2}} \Bigg)
        \,\overline{\widehat{\phi}}(t,k)dt.
        \end{align}
        
        \vskip 0.5 cm
        
        \noindent\underline{\emph{Limit of the $\nabla_x \cdot \mathcal{E}^\varepsilon$ term}}: We now study the second term in \eqref{sec4:differenterms}.
        By part \eqref{prop4.4:pt1} of Proposition \ref{sec4.2:prop4.4} and proceeding as in \eqref{sec4:splitting_W}, we have for  $\phi\in C^\infty_c((0,T) \times \mathbb{T}^3_x)$
        \begin{equation}
        \label{sec4:eq_E_1}
         \begin{aligned}
            \lim_{\varepsilon\to 0} \left \langle 
            T^\varepsilon_{1,+}\left(\varepsilon 
           E^\varepsilon \nabla_x \cdot(\varepsilon E_{\text{irr}}^\varepsilon)
            \right), \phi \right \rangle 
            & = 
            \sum_{\substack{\sigma_3,\sigma_4 \in \{\pm\}}}\lim_{\varepsilon\to 0}
            \left \langle 
            T^\varepsilon_{1,+}\left( 
            T^\varepsilon_{1,\sigma_3}({d_{0,-\sigma_3}})
            \nabla_x \cdot T^\varepsilon_{1,\sigma_4}({d_{1,-\sigma_4}})
            \right), \phi \right \rangle \\
            & \quad + \sum_{\substack{\sigma_1 \in\{1,2\}, \\\sigma_3,\sigma_4 \in \{\pm\}}}\lim_{\varepsilon\to 0}
            \left \langle 
            T^\varepsilon_{1,+}\left( 
            T^\varepsilon_{\sigma_1,\sigma_3}({d_{\sigma_1,-\sigma_3}})
            \nabla_x \cdot T^\varepsilon_{1,\sigma_4}({d_{1,-\sigma_4}})
            \right), \phi \right \rangle,
        \end{aligned}
        \end{equation}
       where we note that only $T^\varepsilon_{1,\pm}({d_{1,\mp}})$ appears with the divergence since $\nabla_x \cdot T^\varepsilon_{1,\pm}({d_{0,\mp}}) = \nabla_x \cdot T^\varepsilon_{2,\pm}({d_{2,\mp}}) = 0$, as as $d_{0,\mp}$ are spatially independent and $d_{2,\mp}$ are solenoidal.

       The first term on the r.h.s. of \eqref{sec4:eq_E_1} is vanishing since the oscillatory phase is always non-zero.
        To compute the second term on the r.h.s. of \eqref{sec4:eq_E_1}, for $\sigma_1 \in\{1,2\}, \sigma_3,\sigma_4 \in\{\pm\}$ and $\psi_1:=d_{\sigma_1,-\sigma_3}, \psi_2:=d_{1,-\sigma_4} \in L^\infty_t H^{s-2}_x$, we look at which terms in
        \[
        \mathcal{E}_{\sigma_1,1, \sigma_3, \sigma_4}(t,k) :=\mathcal{F}\left(T^\varepsilon_{1,+}\left( 
            T^\varepsilon_{\sigma_1,\sigma_3}(\psi_1)
            T^\varepsilon_{1,\sigma_4}(\nabla_x \cdot\psi_2)
            \right)
            \right)(t,k)
        \]
        give a non vanishing contribution as $\varepsilon$ goes to zero. 

        If $\sigma_1=1$, using the formula for $T^\varepsilon_{1,\pm}$ given in \eqref{sec4:def_T}, we have
        \[
        \mathcal{E}_{1,1,\sigma_3,\sigma_4}(t,k)=\imm\exp\left(-\frac{\imm t}{\varepsilon}[1+\sigma_31+\sigma_41] \right)\frac{1}{(2\pi)^3}\sum_{\ell \in \Z^3} \widehat{\psi_1}(t,k-\ell)\ell\cdot \widehat{\psi_2}(t,\ell),
        \]
        with $\sigma_3,\sigma_4 \in \{\pm\}$. It follows that the phase of the oscillatory integrand $\mathcal{E}_{1,1,\sigma_3,\sigma_4}$ is always non-zero 
and therefore gives a vanishing contribution as $\varepsilon$ goes to zero.

         If $\sigma_1=2$ we have
        \[
        \mathcal{E}_{2,1,\sigma_3,\sigma_4}(t,k)= \frac{\imm}{(2\pi)^3} \sum_{\ell \in \Z^3}\exp\left(-\frac{ \imm t}{\varepsilon}[1+\sigma_3\sqrt{1+|\ell|^2}+\sigma_41] \right) \widehat{\psi_1}(t,\ell)(k-\ell) \cdot\widehat{\psi_2}(t,k-\ell),
        \]
        with $\sigma_3,\sigma_4 \in \{\pm\}$. We notice that the phase of the oscillatory integrand $\mathcal{E}_{2,1,\sigma_3,\sigma_4}$ is non-zero, except when $\sigma_3=-, \sigma_4=+$ so that the phase is zero for $|\ell|=\sqrt3$, i.e.,
        \[
        \mathcal{E}_{2,1,-,+}(t,k)=\frac{\imm}{(2\pi)^3} \sum_{\ell \in \boldsymbol{1}} \widehat{\psi_1}(t,\ell)(k-\ell)\cdot \widehat{\psi_2}(t,k-\ell) + \text{oscillatory terms},
        \]
        where $\boldsymbol{1}=\{ \ell \in \Z^3 : \ell_i \in \{\pm1 \}, i \in \{1,2,3\}\}$. 

        Hence,
        {\small\begin{align}
        \label{sec4:limit_RHS_2}
            \lim_{\varepsilon\to 0} - \left \langle 
            T^\varepsilon_{1,+} \nabla_x \cdot \left( 
           E^\varepsilon \nabla_x \cdot(\varepsilon E_{\text{irr}}^\varepsilon)
            \right), \phi \right \rangle 
            =    \frac{1}{(2\pi)^6} \sum_{k \in \mathbb{Z}^3} \sum_{\ell \in \boldsymbol{1}} \int_{-\infty}^{+\infty}  k \cdot \left( \widehat{d_{2,+}}(t,\ell)(k-\ell)\cdot \widehat{d_{1,-}}(t,k-\ell) \right)
             \overline{\widehat{\phi}}(t,k) dt.
        \end{align}}
        \vskip 0.5 cm
        \noindent\underline{\emph{Limit of the $\nabla_x \cdot \mathcal{B}^\varepsilon$ term}}: We now study the third term in \eqref{sec4:differenterms}, i.e., we consider
        \begin{equation*}
            \nabla_x \cdot \mathcal{B}^\varepsilon = \nabla_x \cdot \mathcal{B}_1^\varepsilon + \nabla_x \cdot \mathcal{B}_2^\varepsilon,
        \end{equation*}
        where 
        \[
        \mathcal{B}_1^\varepsilon:= \int_{M}\rho^\varepsilon_\Theta \xi^\varepsilon_\Theta \mu(d\Theta)\wedge B^\varepsilon,
        \quad \text{and} \quad
        \mathcal{B}_2^\varepsilon :=\int_{M}\rho^\varepsilon_\Theta \left(v(\xi^\varepsilon_\Theta)-\xi^\varepsilon_\Theta \right) \mu(d\Theta)\wedge B^\varepsilon.
        \]
        The treatment of $\mathcal{B}_2^\varepsilon$ is done using Lemma \ref{sec2:lemma_sobolev} on the difference between the relativistic and non relativistic velocity. Therefore, $\mathcal{B}_2^\varepsilon$ is a remainder term of order $\varepsilon^2$ which is strongly converging to zero.
        Therefore, we only focus on the $\mathcal{B}_1^\varepsilon$ term. By recalling the two formulas in \eqref{sec4:smallb}, we have
        \[
        \mathcal{B}_1^\varepsilon=\int_M \rho^\varepsilon_\Theta(w^\varepsilon_\Theta+W^\varepsilon)\mu(d\Theta)
        \wedge \left( b^\varepsilon(t,x) - \nabla_x  \wedge W^\varepsilon \right)
        =:\bar{\mathcal{A}}_1+\bar{\mathcal{A}}_2+\bar{\mathcal{A}}_3 + \bar{\mathcal{A}}_4,
        \]
        where
        \[
        \bar{\mathcal{A}}_1:=\int_M \rho^\varepsilon_\Theta w^\varepsilon_\Theta \mu(d\Theta)\wedge b^\varepsilon,\quad \bar{\mathcal{A}}_2:=\int_M\rho^\varepsilon_\Theta \mu(d\Theta)W^\varepsilon\wedge b^\varepsilon,
        \]
        \[
        \bar{\mathcal{A}}_3:=-\int_M\rho^\varepsilon_\Theta w^\varepsilon_\Theta \mu(d\Theta)\wedge (\nabla_x  \wedge W^\varepsilon),\quad 
        \bar{\mathcal{A}}_4:=-\int_M\rho^\varepsilon_\Theta \mu(d\Theta)W^\varepsilon\wedge (\nabla_x  \wedge  W^\varepsilon).
        \]
  For $\bar{\mathcal{A}}_1$, since $\int_M \rho^\varepsilon_\Theta w^\varepsilon_\Theta \mu(d\Theta) \wedge b^\varepsilon \to \int_M \rho_\Theta w_\Theta \mu(d\Theta) \wedge B$ in $L^2$ strongly by Proposition \ref{sec4:prop2}, we obtain, using part \eqref{sec4:lemmapt2} of Lemma \ref{sec4:lemma1},
        \[
         T^\varepsilon_{1,+}\bar{\mathcal{A}}_1=T^\varepsilon_{1,+} \left(\int_M \rho^\varepsilon_\Theta w^\varepsilon_\Theta  \mu(d\Theta) \wedge b^\varepsilon\right) \rightharpoonup0 \quad \text{in} \quad L^2_{t,x}.
        \]
        Concerning $\bar{\mathcal{A}}_2$, we know $\int_M \rho^\varepsilon_\Theta  \mu(d\Theta) \to 1$ and $b^\varepsilon \to B$ in $L^2$ strongly by Proposition \ref{sec4:prop2}. Therefore, similarly to \eqref{sec4:J_limit_A_2} by adding and subtracting these two limits to $\bar{\mathcal{A}}_2$, we get
        \begin{align*}
            \lim_{\varepsilon\to 0} T^\varepsilon_{1,+} \bar{\mathcal{A}}_2 
            = \lim_{\varepsilon\to 0} T^\varepsilon_{1,+} \left(\int_M\rho^\varepsilon_\Theta \mu(d\Theta)  W^\varepsilon \wedge b^\varepsilon\right)  =
         \lim_{\varepsilon\to 0} T^\varepsilon_{1,+}\left(W^\varepsilon \wedge B \right).
        \end{align*}
        By part \eqref{prop4.4:pt2} of Proposition \ref{sec4.2:prop4.4} and similarly as in \eqref{sec4:splitting_W}, we have
        \begin{small}
        \begin{align*}
            \lim_{\varepsilon \to 0} T^\varepsilon_{1,+}\bar{\mathcal{A}}_2=
        \lim_{\varepsilon \to 0} T^\varepsilon_{1,+}\left( \left(T^\varepsilon_{1,+} (\widetilde{d}_{0,-}) + T^\varepsilon_{1,-} (\widetilde{d}_{0,+}) + T^\varepsilon_{1,+} (\widetilde{d}_{1,-}) + T^\varepsilon_{1,-} (\widetilde{d}_{1,+})+  T^\varepsilon_{2,+} (\widetilde{d}_{2,-}) + T^\varepsilon_{2,-} (\widetilde{d}_{2,+}) \right) \wedge B\right).
        \end{align*}
        \end{small}
        However, by part \eqref{sec4:lemmapt3} of Lemma \ref{sec4:lemma1}, 
        \[
        \lim_{\varepsilon \to 0} T^\varepsilon_{1,+} \left[ \left( T^\varepsilon_{1,+} (\widetilde{d}_{0,-}) + T^\varepsilon_{1,+} (\widetilde{d}_{1,-}) + T^\varepsilon_{2,+} (\widetilde{d}_{2,-}) \right)\wedge B\right]=0.
        \]
        Hence, we conclude that the following identity holds,
        \begin{equation}
            \label{sec4:lim_T_A2bar_mean}
        \lim_{\varepsilon \to 0} T^\varepsilon_{1,+}\bar{\mathcal{A}}_2=
        \lim_{\varepsilon \to 0} T^\varepsilon_{1,+}\left[ \left(T^\varepsilon_{1,-} (\widetilde{d}_{0,+}) + T^\varepsilon_{1,-} (\widetilde{d}_{1,+})+T^\varepsilon_{2,-} (\widetilde{d}_{2,+})\right)\wedge B\right].
        \end{equation}
        
        Taking  $\phi\in C^\infty_c((0,T) \times \mathbb{T}^3_x)$, we compute the weak limit of the first term in \eqref{sec4:lim_T_A2bar_mean}. By the expression for $T_{1,\pm}^\varepsilon$ given by \eqref{sec4:def_T}, we obtain
        \begin{align}
        \label{sec4:corrector_B_term_A_2_mean}
            \lim_{\varepsilon \to 0}\left \langle 
            T^\varepsilon_{1,+}\left( 
             T^\varepsilon_{1,-} (\widetilde{d}_{0,+})\wedge B
            \right), \phi \right \rangle 
              =\left \langle 
             \widetilde{d}_{0,+}\wedge B
            , \phi \right \rangle 
        \end{align}
        where we used that $d_{0,+}$ is spatially homogeneous. Then, we compute the weak limit of the second term in \eqref{sec4:lim_T_A2bar_mean}. By using again Plancherel identity, the expression for $T_{1,\pm}^\varepsilon$ given by \eqref{sec4:def_T} and the one of $\widetilde{d}_{1,+}$ given by \eqref{sec4:dd}, we obtain
        \begin{equation}
        \label{sec4:corrector_B_term_A_2}
            \lim_{\varepsilon \to 0}\left \langle 
            T^\varepsilon_{1,+}\left( 
             T^\varepsilon_{1,-} (\widetilde{d}_{1,+})\wedge B
            \right), \phi \right \rangle 
             =
            - \frac{\imm}{(2\pi)^6} \int_{-\infty}^{+\infty}  \sum_{k,\ell \in \Z^3}
             \widehat{d_{1,+}}(t,\ell) \wedge \widehat{B}(t,k-\ell)  \overline{\widehat{\phi}}(t,k) dt,
        \end{equation}
         which contributes to the limit equation. We now look at the weak limit of the second term, i.e.,
       \begin{align*}
            &\lim_{\varepsilon \to 0} \left \langle 
            T^\varepsilon_{1,+}\left( 
             T^\varepsilon_{2,-} (\widetilde{d}_{2,+})\wedge B
            \right), \phi \right \rangle 
            \\
            &  = - \lim_{\varepsilon \to 0}\frac{\imm}{(2\pi)^6}\sum_{k \in \Z^3}
            \sum_{\ell \in \Z^3}\int_{-\infty}^{+\infty}  \exp\left(-\frac{\imm}{\varepsilon}\left(-\sqrt{1+|\ell|^2}+1\right)t \right)
              \widehat{d_{2,+}}(t,\ell) \wedge \widehat{B}(t,k-\ell)  \left(1+|\ell|^2\right)^{-\frac{1}{2}} \overline{\widehat{\phi}}(t,k)dt.
            \end{align*}
        Hence, the last integral is always of oscillatory type except when $\ell=0$, but in this case, $(\widehat{d_{2,+}})_j(t,0)=0$ since it is solenoidal. It follows that \begin{align*}
            &\lim_{\varepsilon \to 0} \left \langle 
            T^\varepsilon_{1,+}\left( 
             T^\varepsilon_{2,-} (\widetilde{d}_{2,+})\wedge B
            \right), \phi \right \rangle   = 0.
        \end{align*}
        We conclude that the contributions for $\bar{\mathcal{A}}_2$ are given by \eqref{sec4:corrector_B_term_A_2_mean} and
        \eqref{sec4:corrector_B_term_A_2}. Hence
        \begin{align}
        \label{sec4:limit_RHS_3}
            \lim_{\varepsilon\to 0} - \left \langle T_{1,+}^\varepsilon \nabla_x \cdot \bar{\mathcal{A}}_2,\phi\right \rangle 
            &=
           - \frac{1}{(2\pi)^3}\sum_{k \in \Z^3} \int_{-\infty}^{+\infty}  k \cdot \left(
             d_{0,+}(t) \wedge \widehat{B}(t,k)  \right)\overline{\widehat{\phi}}(t,k) dt  \nonumber \\
             & \quad - \frac{1}{(2\pi)^6}\sum_{k,\ell \in \Z^3}  \int_{-\infty}^{+\infty} k \cdot \left( \widehat{d_{1,+}}(t,\ell)\wedge \widehat{B}(t,k-\ell) \right) \overline{\widehat{\phi}}(t,k)dt.
        \end{align}
        Next, we study $\bar{\mathcal{A}}_3$. Since $\int_M\rho^\varepsilon_\Theta w^\varepsilon_\Theta \mu(d\Theta) \to \int_M\rho_\Theta w_\Theta \mu(d\Theta)$, by adding and subtracting the limit, we get
        \begin{align*}
        \lim_{\varepsilon \to 0} T^\varepsilon_{1,+}\bar{\mathcal{A}}_3 =
            - \lim_{\varepsilon \to 0} T^\varepsilon_{1,+} \left(
            \int_M\rho_\Theta w_\Theta \mu(d\Theta)\wedge (\nabla_x  \wedge W^\varepsilon) \right).
        \end{align*}
        Next, by part \eqref{prop4.4:pt2} of Proposition \ref{sec4.2:prop4.4} and similarly as in \eqref{sec4:splitting_W}, we have
        \begin{align*}
         \lim_{\varepsilon \to 0} T^\varepsilon_{1,+}\bar{\mathcal{A}}_3 & =
            - \lim_{\varepsilon \to 0} T^\varepsilon_{1,+} \left(
            \int_M\rho_\Theta w_\Theta \mu(d\Theta) \right. \\
            & \left. \wedge \left(\nabla_x  \wedge \left(T^\varepsilon_{1,+} (\widetilde{d}_{0,-}) + T^\varepsilon_{1,-} (\widetilde{d}_{0,+}) + T^\varepsilon_{1,+} (\widetilde{d}_{1,-}) + T^\varepsilon_{1,-} (\widetilde{d}_{1,+})+  T^\varepsilon_{2,+} (\widetilde{d}_{2,-}) + T^\varepsilon_{2,-} (\widetilde{d}_{2,+}) \right) \right) \right).
        \end{align*}
        Observe that $\nabla_x  \wedge \left(T^\varepsilon_{1,+} (\widetilde{d}_{0,-}) + T^\varepsilon_{1,-} (\widetilde{d}_{0,+}) +T^\varepsilon_{1,+} (\widetilde{d}_{1,-}) + T^\varepsilon_{1,-} (\widetilde{d}_{1,+})\right)  = 0$ since $\widetilde{d}_{0,\pm}$ are spatially homogeneous $\widetilde{d}_{1,\pm}$ are irrotational. Moreover, by  by part \eqref{sec4:lemmapt3} of Lemma \ref{sec4:lemma1} and by commuting the curl with $T_{2,+}^\varepsilon$, we have
        \begin{align*}
             \lim_{\varepsilon \to 0} T^\varepsilon_{1,+} \left(
            \int_M\rho_\Theta w_\Theta \mu(d\Theta)\wedge \left( T^\varepsilon_{2,+} (\nabla_x  \wedge \widetilde{d}_{2,-})  \right) \right) = 0.
        \end{align*}
        Hence, the following identity holds:
        \[
        \lim_{\varepsilon \to 0} T^\varepsilon_{1,+}\bar{\mathcal{A}}_3=
        - \lim_{\varepsilon \to 0} T^\varepsilon_{1,+}\left( \int_M \rho_\Theta w_\Theta\mu(d\Theta) \wedge T^\varepsilon_{2,-} (\nabla_x  \wedge \widetilde{d}_{2,+})\right).
        \]
         Taking  $\phi\in C^\infty_c((0,T) \times \mathbb{T}^3_x)$ and  $\psi(t,x):=  \int_M \rho_\Theta w_\Theta\mu(d\Theta) \in L^\infty_tH^{s-2}_x$, we compute the  weak limit of the last equation. By using Plancherel identity, the expression for $T_{1,+}^\varepsilon$, $T_{2,-}^\varepsilon$ and the one of $\widetilde{d}_{2,+}$ given by \eqref{sec4:dd}, we obtain
       \begin{align*}
        & 
            \lim_{\varepsilon \to 0} \left \langle 
            T^\varepsilon_{1,+}\left( 
            \psi \wedge T^\varepsilon_{2,-} (\nabla_x  \wedge \widetilde{d}_{2,+})
            \right), \phi \right \rangle 
            =  \lim_{\varepsilon \to 0}\frac{1}{(2\pi)^6}
            \int_{-\infty}^{+\infty}  \sum_{k \in \Z^3}
            \exp\left(-\frac{\imm t}{\varepsilon}\right) \\
            &\qquad \times \sum_{\ell \in \Z^3} \widehat{\psi}(t,k-\ell)\wedge \left[\ell\wedge\widehat{d_{2,+}}(t,\ell)\left(1+|\ell|^2\right)^{-\frac{1}{2}}\right]\exp\left(\imm \frac{\sqrt{1+|\ell|^2}t}{\varepsilon} \right) \overline{\widehat{\phi}}(t,k)  dt\\
            &= \lim_{\varepsilon \to 0}\frac{1}{(2\pi)^6}\sum_{k \in \Z^3}\sum_{\ell \in \Z^3}\int_{-\infty}^{+\infty}  \exp\left(\frac{\imm}{\varepsilon}\left(\sqrt{1+|\ell|^2}-1\right)t \right)
           \widehat{\psi}(t,k-\ell)\wedge \left[\ell\wedge\widehat{d_{2,+}}(t,\ell)\left(1+|\ell|^2\right)^{-\frac{1}{2}}\right] \overline{\widehat{\phi}}(t,k) dt.
            \end{align*}
        Hence, the limit is vanishing except when $\ell=0$, which gives a zero contribution.

        We finally study $\bar{\mathcal{A}}_4$. We know that $\int_M \rho^\varepsilon_\Theta  \mu(d\Theta) \to 1$ in $L^2$ strongly by Proposition \ref{sec4:prop2}. Therefore, by adding and subtracting the limit as in \eqref{sec4:J_limit_A_2}, we obtain
        \begin{align*}
            \lim_{\varepsilon \to 0} T^\varepsilon_{1,+}\bar{\mathcal{A}}_4 
            = -\lim_{\varepsilon \to 0} T^\varepsilon_{1,+} 
            \left(\int_M\rho^\varepsilon_\Theta \mu(d\Theta)W^\varepsilon\wedge (\nabla_x  \wedge  W^\varepsilon) \right) 
            = -\lim_{\varepsilon \to 0} T^\varepsilon_{1,+} 
            \left(W^\varepsilon\wedge (\nabla_x  \wedge  W^\varepsilon) \right).
        \end{align*}
        By part \eqref{prop4.4:pt2} of Proposition \ref{sec4.2:prop4.4} and proceeding as in \eqref{sec4:splitting_W}, we have
        \begin{align}
        \label{sec4:eq_B_A4}
            \lim_{\varepsilon\to 0}\left \langle 
            T^\varepsilon_{1,+}\left( 
            W^\varepsilon\wedge (\nabla_x  \wedge  W^\varepsilon)
            \right), \phi \right \rangle 
            &=
            \sum_{\substack{\sigma_1 \in\{0,1,2\}, \\\sigma_3,\sigma_4 \in \{\pm\}}}\lim_{\varepsilon\to 0}
            \left \langle 
            T^\varepsilon_{1,+}\left( 
            T^\varepsilon_{\sigma_1,\sigma_3} (\widetilde{d}_{\sigma_1,-\sigma_3}) \wedge
            (\nabla_x  \wedge T^\varepsilon_{2,\sigma_4} (\widetilde{d}_{2,-\sigma_4}))
            \right), \phi \right \rangle,
        \end{align}
        where we introduced the notation $T^\varepsilon_{0,+}:=T_{1,+}^\varepsilon$. We note that only $T^\varepsilon_{2,\pm}({d_{2,\mp}})$ appears with the curl operator since $\nabla_x \wedge T^\varepsilon_{1,\pm}({d_{0,\mp}}) = \nabla_x \wedge T^\varepsilon_{1,\pm}({d_{1,\mp}}) = 0$, as $d_{0,\mp}$ are spatially homogeneous and $d_{1,\mp}$ are irrotational.

        We study for $\sigma_1 \in\{0,1,2\}, \sigma_3,\sigma_4 \in\{\pm\}$ and $\psi_1:=\widetilde{d}_{\sigma_1,-\sigma_3}, \psi_2:=\widetilde{d}_{2,-\sigma_4} \in L^\infty_t H^{s-2}_x$, which terms in 
        \[
        \mathcal{B}_{\sigma_1,2, \sigma_3, \sigma_4}(t,k) : =\mathcal{F}\left(T^\varepsilon_{1,+}\left( 
            T^\varepsilon_{\sigma_1,\sigma_3}(\psi_1)\wedge
            T^\varepsilon_{2,\sigma_4}(\nabla_x  \wedge \psi_2)
            \right)
            \right)
        \]
        give a non vanishing contribution as $\varepsilon$ goes to zero. 

              If $\sigma_1\in\{0,1\}$, using the formula for $T^\varepsilon_{1,\pm}$ and $T^\varepsilon_{2,\pm}$ given in \eqref{sec4:def_T}, we have
        \[
        \mathcal{B}_{\sigma_1,2,\sigma_3,\sigma_4}(t,k)=\frac{\imm}{(2\pi)^3}\sum_{\ell \in \Z^3}\exp\left(-\frac{\imm t}{\varepsilon}[1+\sigma_31+\sigma_4\sqrt{1+|\ell|^2}] \right) \widehat{\psi_1}(t,k-\ell)\wedge[\ell\wedge \widehat{\psi_2}(t,\ell)],
        \]
        with $\sigma_3,\sigma_4 \in \{\pm\}$. We notice that the phase of the oscillatory integrand $\mathcal{B}_{\sigma_1,2,\sigma_3,\sigma_4}$ is non-zero, except when $\sigma_3=+, \sigma_4=-$ so that the phase is zero for $|\ell|=\sqrt3$, i.e.,
        \[
        \mathcal{B}_{\sigma_1,2,+,-}(t,k)=\frac{\imm}{(2\pi)^3} \sum_{\ell \in \boldsymbol{1}} \widehat{\psi_1}(t,k-\ell)\wedge[\ell\wedge \widehat{\psi_2}(t,\ell)]+ \text{oscillatory terms},
        \]
        where $\boldsymbol{1}=\{ \ell \in \Z^3 : \ell_i \in \{\pm1 \}, i \in \{1,2,3\}\}.$

         If $\sigma_1=2$ we have
        \[
        \mathcal{B}_{2,2,\sigma_3,\sigma_4}(t,k)=\frac{\imm}{(2\pi)^3}\sum_{\ell \in \Z^3}\exp\left(-\frac{\imm t}{\varepsilon}\left[1+\sigma_3\sqrt{1+|k-\ell|^2}+\sigma_4\sqrt{1+|\ell|^2}\right] \right) \widehat{\psi_1}(t,k-\ell)\wedge[\ell \wedge \widehat{\psi_2}(t,\ell)],
        \]
        with $\sigma_3,\sigma_4 \in \{\pm\}$. We notice that the phase of the oscillatory integrand $\mathcal{B}_{2,2,\sigma_3,\sigma_4}$ is non-zero, except when $\sigma_3\sigma_4=-$. Hence
        \[
        \mathcal{B}_{2,2,\sigma_3,-\sigma_3}(t,k)= \frac{\imm}{(2\pi)^3} \sum_{\ell \in \Omega^{(1)}_{\sigma_3,-\sigma_3}(k)} \widehat{\psi_1}(t,k-\ell)\wedge[\ell \wedge \widehat{\psi_2}(t,\ell)] + \text{oscillatory terms},
        \]
        where $\Omega^{(1)}_{\sigma_3,\sigma_4}(k)$ is defined in \eqref{sec4:eqsigma}. 

          We conclude that 
          \begin{small}
        \begin{align}
        \label{sec4:limit_RHS_3_bis}
            &- \lim_{\varepsilon\to 0}   \left \langle T_{1,+}^\varepsilon \nabla_x \cdot \bar{\mathcal{A}}_4,\phi\right \rangle =
            - \frac{1}{(2\pi)^3}\sum_{k \in \boldsymbol{1}}   \int_{-\infty}^{+\infty} k \cdot \left(d_{0,-}(t)\wedge\left[k\wedge \widehat{d_{2,+}}(t,k)\left(1+|k|^2\right)^{-\frac{1}{2}}\right] \right)\overline{\widehat{\phi}}(t,k)dt \nonumber \\
            &- \frac{1}{(2\pi)^6}\sum_{k \in \Z^3}   \int_{-\infty}^{+\infty} k \cdot \left(\sum_{\ell \in \boldsymbol{1}} \widehat{d_{1,-}}(t,k-\ell)\wedge\left[\ell\wedge \widehat{d_{2,+}}(t,\ell)\left(1+|\ell|^2\right)^{-\frac{1}{2}}\right] \right)\overline{\widehat{\phi}}(t,k)dt \nonumber \\
            & - \frac{1}{(2\pi)^6}\sum_{\sigma \in \{ \pm \}}  \sum_{k \in \Z^3}  \int_{-\infty}^{+\infty} k \cdot \left( \sum_{\ell \in \Omega^{(1)}_{\sigma,-\sigma}(k)}\left(1+|k-\ell|^2\right)^{-\frac{1}{2}}\widehat{d_{2,-\sigma}}(t,k-\ell)\wedge[\ell \wedge \widehat{d_{2,\sigma}}(t,\ell)\left(1+|\ell|^2\right)^{-\frac{1}{2}}] \right)
            \overline{\widehat{\phi}}(t,k)dt.
        \end{align}
        \end{small}

         The study of the fourth term in \eqref{sec4:differenterms}, which corresponds to the remainder, follows by using the Sobolev version of inequality \eqref{Lemma_remainder} in Lemma \ref{sec1:lemmarem}  and is converging to zero.
        \vskip 0.5 cm
        \noindent\underline{\emph{Conclusion}}: Finally, by collecting \eqref{sec4:limit_LHS}, \eqref{sec4:limit_RHS_1}, \eqref{sec4:limit_RHS_1_bis},  \eqref{sec4:limit_RHS_2}, \eqref{sec4:limit_RHS_3}, and \eqref{sec4:limit_RHS_3_bis} we get the equation satisfied by the corrector $d_{1,+}$ given by \eqref{sec4:eqcorrector1}. \\

        \vskip 0.5 cm

        \noindent\underline{\textbf{Equation for the solenoidal correctors}}:
           We now derive the equation for the solenoidal correctors by looking at \eqref{sec3:wave}. We proceed similarly to the previous case starting from the l.h.s. of \eqref{sec4:limitT_2}: Given $\phi \in C^\infty_c((0,T)\times \mathbb{T}^3_x)$ and by taking the adjoints of the respective operators, we get
           \begin{align*}
                \langle T^\varepsilon_{2,+}(\varepsilon^2 \partial_{tt}^2 \nabla_x  \wedge  E^\varepsilon_{\text{sol}} &+ (1-\Delta_x)\nabla_x  \wedge E^\varepsilon_{\text{sol}}, \phi  \rangle
               =
                 \left \langle \nabla_x  \wedge  E^\varepsilon_{\text{sol}}, \varepsilon^2 \partial_{tt}^2T^\varepsilon_{2,-}\phi + (1-\Delta_x)T_{2,-}^\varepsilon\phi \right \rangle\\
                 &=   \left \langle E^\varepsilon_{\text{sol}}, \nabla_x  \wedge  \left(\varepsilon^2 \partial_{tt}^2T^\varepsilon_{2,-}\phi + (1-\Delta_x)T_{2,-}^\varepsilon\phi\right) \right \rangle\\
                 &=   \left \langle T^\varepsilon_{2,-}T^\varepsilon_{2,+}E^\varepsilon_{\text{sol}}, \nabla_x  \wedge  \left(\varepsilon^2 \partial_{tt}^2T^\varepsilon_{2,-}\phi + (1-\Delta_x)T_{2,-}^\varepsilon\phi\right) \right \rangle\\
                 &= \left \langle \varepsilon T_{2,+}^\varepsilon E^\varepsilon_{\text{sol}}, \frac{1}{\varepsilon}\nabla_x  \wedge  \left(\varepsilon^2 T_{2,+}^\varepsilon\partial_{tt}^2T^\varepsilon_{2,-}\phi + (1-\Delta_x)\phi\right) \right \rangle,
           \end{align*}
           since the curl operator and $(1-\Delta_x)$ commute with $T^\varepsilon_{2,\pm}$.
           By part \eqref{prop4.4:pt1} of Proposition \ref{sec4.2:prop4.4}, we know that $\varepsilon T_{2,+}^\varepsilon E^\varepsilon_{\text{sol}} \rightharpoonup d_{2,+}$ in $L^2_{t,x}$, moreover
           \begin{align*}
               &\frac{\reallywidehat{\varepsilon^2 T^\varepsilon_{2,+}\partial_{tt}^2 T^\varepsilon_{2,-}\phi+(1-\Delta_x)\phi}}{\varepsilon}(t,k)
               =
               \frac{1}{\varepsilon}\exp\left(-\imm \frac{\sqrt{1+|k|^2}t}{\varepsilon} \right)\varepsilon^2 \partial_{tt}^2 \left( \exp\left(\imm\frac{\sqrt{1+|k|^2}t}{\varepsilon}\right) \widehat{\phi}(t,k)\right)
               \quad \\
               & \qquad + \frac{1}{\varepsilon} (1+|k|^2) \widehat{\phi}(t,k)\\
             &=\varepsilon \exp\left(-\imm \frac{\sqrt{1+|k|^2}t}{\varepsilon}\right) \left[-\frac{1+|k|^2}{\varepsilon^2}\widehat{\phi}(t,k) + \frac{2 \imm \sqrt{1+|k|^2}}{\varepsilon}\partial_t \widehat{\phi}(t,k)+\partial_{tt}^2\widehat{\phi}(t,k) \right]\exp\left(\imm \frac{\sqrt{1+|k|^2 }t}{\varepsilon} \right)\\
            &\qquad + \frac{1}{\varepsilon} (1+|k|^2) \widehat{\phi}(t,k)\\
            &=
               2 \imm \sqrt{1+|k|^2}\partial_t \widehat{\phi}(t,k) + \varepsilon \partial_{tt}^2 \widehat{\phi}(t,k).
           \end{align*}
           It follows that 
           \[
           \frac{\reallywidehat{\varepsilon^2 T^\varepsilon_{2,+} \partial_{tt}^2 T_{2,-}^\varepsilon \phi +(1-\Delta_x)\phi}}{\varepsilon} \to 2\imm \sqrt{1+|k|^2} \partial_t \widehat{\phi} \quad \text{in} \quad H^1 \quad \text{strongly}.
           \]
           We conclude that
           \begin{align} 
           \label{sec4:limit_sol_LHS}
              \lim_{\varepsilon\to 0} & \langle T^\varepsilon_{2,+}(\varepsilon^2 \partial_{tt}^2 \nabla_x  \wedge  E^\varepsilon_{\text{sol}} + (1-\Delta_x)\nabla_x  \wedge E^\varepsilon_{\text{sol}}, \phi  \rangle
              = \left \langle d_{2,+}, \nabla_x  \wedge  \mathcal{F}^{-1}\left(\left\{2\imm \sqrt{1+|k|^2}\partial_t \widehat{\phi}\right\}_{k \in \Z^3} \right) \right \rangle \nonumber\\
              &=\left \langle \nabla_x  \wedge d_{2,+},   \mathcal{F}^{-1}\left(\left\{2\imm \sqrt{1+|k|^2}\partial_t \widehat{\phi} \right\}_{k \in \Z^3}\right) \right \rangle 
              =- \left \langle \mathcal{F}^{-1} \left(\left\{ 2 \sqrt{1+|k|^2} k \wedge \partial_t  \widehat{d_{2,+}}  \right\}_{k \in \Z^3}\right), \phi  \right \rangle.
           \end{align}   
           
        Next, we want to analyse the r.h.s. of \eqref{sec4:limitT_2}.
        By recalling definitions \eqref{sec4:def_J_E_B}, we study
        \begin{equation}
        \label{sec4:differenterms_2}
    \begin{aligned}
    \lim_{\varepsilon\to 0} \langle T^\varepsilon_{2,+}h^\varepsilon, \phi  \rangle = \lim_{t \to 0} \langle T^\varepsilon_{2,+}\left(\nabla_x \wedge \mathcal{J}^\varepsilon(t,x)-\nabla_x \wedge \mathcal{E}^\varepsilon(t,x)-\nabla_x \wedge \mathcal{B}^\varepsilon(t,x)
    -\nabla_x \wedge R^\varepsilon(t,x)\right), \phi \rangle,
    \end{aligned}
    \end{equation}
    where $R^\varepsilon$ is defined in \eqref{sec3:defR}.
    \vskip 0.5 cm
    \noindent
     \underline{\emph{Limit of the $\nabla_x \wedge \mathcal{J}^\varepsilon$ term}}: We study the first term on the r.h.s. of \eqref{sec4:differenterms_2}, i.e., we consider
        \begin{equation}
        \label{sec4:curl_J}
            \nabla_x \wedge \mathcal{J}^\varepsilon(t,x)=
        \nabla_x \wedge (\partial_{x_i} \widetilde{\mathcal{J}}_1^\varepsilon(t,x))
        + \nabla_x \wedge (\partial_{x_i} \mathcal{\widetilde{J}}_2^\varepsilon(t,x)),
        \end{equation}
        where
        $$
        \mathcal{\widetilde{J}}_1^\varepsilon(t,x):= \int_{M}\rho^\varepsilon_\Theta(t,x) (\xi^\varepsilon_\Theta(t,x))_i \xi^\varepsilon_\Theta(t,x) \mu(d\Theta)
        $$
        and 
        $$
        \mathcal{\widetilde{J}}_2^\varepsilon(t,x):=\int_{M}\rho^\varepsilon_\Theta(t,x) [v(\xi^\varepsilon_\Theta(t,x))_iv(\xi^\varepsilon_\Theta(t,x))-(\xi^\varepsilon_\Theta(t,x))_i\xi^\varepsilon_\Theta(t,x)] \mu(d\Theta),
        $$
        for $i \in \{1, 2, 3\}$.  The treatment of $\mathcal{\widetilde J}_2^\varepsilon$ is done using Lemma \ref{sec2:lemma_sobolev} on the difference between the relativistic and non relativistic velocity. Therefore, $\mathcal{\widetilde J}_2^\varepsilon$ is a remainder term of order $\varepsilon^2$ which is strongly converging to zero.
         
       By the definition of $w^\varepsilon_\Theta=\xi^\varepsilon_\Theta - W^\varepsilon$ and \eqref{sec4:corrector}, we have
        \[
        \mathcal{\widetilde J}_1^\varepsilon=\int_M \rho^\varepsilon_\Theta((w^\varepsilon_\Theta)_i +W^\varepsilon_i)(w^\varepsilon_\Theta + W^\varepsilon)\mu(d\Theta)=:\widetilde{\mathcal{A}}_1+\widetilde{\mathcal{A}}_2+\widetilde{\mathcal{A}}_3 + \widetilde{\mathcal{A}}_4,
        \]
        where
        \begin{align*}
        \widetilde{\mathcal{A}}_1 & :=\int_M \rho^\varepsilon_\Theta(w^\varepsilon_\Theta)_i w^\varepsilon_\Theta \mu(d\Theta),
        \qquad \widetilde{\mathcal{A}}_2:=W^\varepsilon \int_M\rho^\varepsilon_\Theta (w^\varepsilon_\Theta)_i \mu(d\Theta), \\ 
        \widetilde{\mathcal{A}}_3 & :=W^\varepsilon_i  \int_M\rho^\varepsilon_\Theta w^\varepsilon_\Theta \mu(d\Theta), 
        \qquad \widetilde{\mathcal{A}}_4  :=W^\varepsilon_i W^\varepsilon \int_M\rho^\varepsilon_\Theta \mu(d\Theta).
        \end{align*}
        For $\widetilde{\mathcal{A}}_1$, since $\int_M \rho^\varepsilon_\Theta (w^\varepsilon_\Theta)_i w^\varepsilon_\Theta \mu(d\Theta) \to \int_M \rho_\Theta (w_\Theta)_i w_\Theta \mu(d\Theta)$ in $L^2$ strongly by Proposition \ref{sec4:prop2}, we obtain, using part \eqref{sec4:lemmapt2} of Lemma \ref{sec4:lemma1},
        \[
         T^\varepsilon_{2,+}\widetilde{\mathcal{A}}_1=T^\varepsilon_{2,+} \left(\int_M \rho^\varepsilon_\Theta (w^\varepsilon_\Theta)_i w^\varepsilon_\Theta \mu(d\Theta) \right) \rightharpoonup0 \quad \text{in} \quad L^2_{t,x}.
        \]
        Concerning $\widetilde{\mathcal{A}}_2$, we have $\int_M \rho^\varepsilon_\Theta (w^\varepsilon_\Theta)_i \mu(d\Theta) \to \int_M \rho_\Theta (w_\Theta)_i \mu(d\Theta)$ in $L^2$ strongly by Proposition \ref{sec4:prop2}. Therefore, as in \eqref{sec4:J_limit_A_2}, we get
        \[
        \lim_{\varepsilon \to 0} T^\varepsilon_{2,+} \widetilde{\mathcal{A}}_2 =
        \lim_{\varepsilon\to 0} T^\varepsilon_{2,+}\left(\int_M \rho^\varepsilon_\Theta (w^\varepsilon_\Theta)_i \mu(d\Theta) \right) W^\varepsilon=
         \lim_{\varepsilon\to 0} T^\varepsilon_{2,+}\left(\int_M \rho_\Theta (w_\Theta)_i \mu(d\Theta) \right) W^\varepsilon.
        \]
        Then, by part \eqref{prop4.4:pt2} Proposition \ref{sec4.2:prop4.4} and proceeding as in \eqref{sec4:splitting_W} we get
        \begin{align*}
            \lim_{\varepsilon \to 0} T^\varepsilon_{2,+}\widetilde{\mathcal{A}}_2 
            & = \lim_{\varepsilon \to 0} T^\varepsilon_{2,+} \left[\left(\int_M \rho_\Theta (w_\Theta)_i \mu(d\Theta)\right) \right. \\
            &  \qquad \times \left. \left(T^\varepsilon_{1,+} (\widetilde{d}_{0,-}) + T^\varepsilon_{1,-} (\widetilde{d}_{0,+}) + T^\varepsilon_{1,+} (\widetilde{d}_{1,-}) + T^\varepsilon_{1,-} (\widetilde{d}_{1,+}) + T^\varepsilon_{2,+} (\widetilde{d}_{2,-}) + T^\varepsilon_{2,-} (\widetilde{d}_{2,+}) \right)\right].
        \end{align*}
        Moreover, by part \eqref{sec4:lemmapt3} of Lemma \ref{sec4:lemma1}, we notice that,
        \[
        \lim_{\varepsilon \to 0} T^\varepsilon_{2,+} \left[\left(\int_M \rho_\Theta (w_\Theta)_i \mu(d\Theta)\right)\left( T^\varepsilon_{1,+} (\widetilde{d}_{0,-}) + T^\varepsilon_{1,+} (\widetilde{d}_{1,-}) + T^\varepsilon_{2,+} (\widetilde{d}_{2,-}) \right)\right]=0.
        \]
        Therefore, we conclude that the following identity holds:
        \begin{equation}
            \label{sec4:lim_T_A2_sol}
        \lim_{\varepsilon \to 0} T^\varepsilon_{2,+}\widetilde{\mathcal{A}}_2=
        \lim_{\varepsilon \to 0} T^\varepsilon_{2,+}\left[\left( \int_M \rho_\Theta (w_\Theta)_i \mu(d\Theta)\right) \left(T^\varepsilon_{1,-} (\widetilde{d}_{0,+}) + T^\varepsilon_{1,-} (\widetilde{d}_{1,+})+T^\varepsilon_{2,-} (\widetilde{d}_{2,+}) \right)\right].
        \end{equation}
        Taking  $\phi\in C^\infty_c((0,T) \times \mathbb{T}^3_x)$ and defining $\psi(t,x):=  \int_M \rho_\Theta (w_\Theta)_i \mu(d\Theta)\in L^\infty_tH^{s-2}_x$, we compute the  weak limit of the first term in \eqref{sec4:lim_T_A2_sol}. By using the expression for $T_{2,+}^\varepsilon$, $T_{1,-}^\varepsilon$ and the one of $\widetilde{d}_{0,+}$ given by \eqref{sec4:dd}, we obtain
         \begin{align*}
           &\lim_{\varepsilon \to 0}\left \langle 
            T^\varepsilon_{2,+}\left( 
            \psi T^\varepsilon_{1,-} (\widetilde{d}_{0,+})
            \right), \phi \right \rangle  \\
            &\qquad = - \frac{\imm}{(2\pi)^3}
            \lim_{\varepsilon \to 0}
            \int_{-\infty}^{+\infty}  \sum_{k \in \Z^3}
            \exp\left(\frac{\imm}{\varepsilon}\left(1 - \sqrt{1+|k|^2}\right)t \right)  
             \widehat{\psi}(t,k)
             d_{0,+}(t)   \overline{\widehat{\phi}}(t,k)  dt,
            \end{align*}
            where we used that $d_{0,+}$ is spatially homogeneous.
            Therefore, the last integral is always of oscillatory type except when $k=0$. However, recalling expression \eqref{sec4:curl_J}, we note that we still have to take a derivative $\partial_{x_i}$ for $\widetilde{\mathcal{A}}_2$. Thus, we obtain
        \begin{align*}
            \lim_{\varepsilon \to 0}\left \langle 
            T^\varepsilon_{2,+} \partial_{x_i}\left( 
            \psi T^\varepsilon_{1,-} \widetilde{d}_{0,+}
            \right), \phi \right \rangle = 0.
        \end{align*}
            Then, we compute the  weak limit of the second term in \eqref{sec4:lim_T_A2_sol}. By using Plancherel identity, the expression for $T_{2,+}^\varepsilon$, $T_{1,-}^\varepsilon$ and the one of $\widetilde{d}_{1,+}$ given by \eqref{sec4:dd}, we obtain
         \begin{align*}
           &\lim_{\varepsilon \to 0}\left \langle 
            T^\varepsilon_{2,+}\left( 
            \psi T^\varepsilon_{1,-} (\widetilde{d}_{1,+})
            \right), \phi \right \rangle  \\
            &\qquad = - \imm 
            \lim_{\varepsilon \to 0} \frac{1}{(2\pi)^6}
            \int_{-\infty}^{+\infty}  \sum_{k \in \Z^3}
            \exp\left(-\frac{\imm\sqrt{1 + |k|^2} t}{\varepsilon}\right) 
            \sum_{\ell \in \Z^3} \widehat{\psi}(t,k-\ell)
            \exp\left(\imm \frac{t}{\varepsilon} \right) \widehat{d_{1,+}}(t,\ell)   \overline{\widehat{\phi}}(t,k)  dt\\
            &\qquad= - \imm \lim_{\varepsilon \to 0}  \frac{1}{(2\pi)^6} \sum_{k \in \Z^3}\sum_{\ell \in \Z^3}\int_{-\infty}^{+\infty}  \exp\left(\frac{\imm}{\varepsilon}\left(1 - \sqrt{1+|k|^2}\right)t \right)
             \widehat{\psi}(t,k-\ell) \widehat{d_{1,+}}(t,\ell)  \overline{\widehat{\phi}}(t,k)dt.
            \end{align*}
        Therefore, the last integral is always of oscillatory type except when $k=0$. However, recalling expression \eqref{sec4:curl_J}, we note that we still have to take a derivative $\partial_{x_i}$ for $\widetilde{\mathcal{A}}_2$. Thus, we obtain
        \begin{align*}
            \lim_{\varepsilon \to 0}\left \langle 
            T^\varepsilon_{2,+} \partial_{x_i}\left( 
            \psi T^\varepsilon_{1,-} \widetilde{d}_{1,+}
            \right), \phi \right \rangle = 0.
        \end{align*}
        We now look at the third term in \eqref{sec4:lim_T_A2_sol}, that is,
        \begin{align}
        \label{sec4:correctorlim2}
           &\left \langle 
            T^\varepsilon_{2,+}\left( 
            \psi T^\varepsilon_{2,-} (\widetilde{d}_{2,+})
            \right), \phi \right \rangle 
            =-\frac{\imm}{(2\pi)^6}
            \int_{-\infty}^{+\infty}  \sum_{k \in \Z^3}
            \exp{\left(-\frac{\imm \sqrt{1+|k|^2} t}{\varepsilon} \right)}   \nonumber \\
            &
            \quad \times \sum_{\ell \in \Z^3}
            \widehat{\psi}(t,k-\ell) \exp{\left(\frac{\imm \sqrt{1+|\ell|^2} t}{\varepsilon} \right)} \widehat{d_{2,+}}(t,\ell)\left(1+\abs{\ell}^2\right)^{-\frac12} \overline{\widehat{\phi}}(t,k) dt \nonumber \\
            & = -   \frac{\imm}{(2\pi)^6} \sum_{k \in \Z^3}\sum_{\ell \in \Z^3}\int_{-\infty}^{+\infty}  \exp\left(\frac{\imm}{\varepsilon}\left(\sqrt{1+|\ell|^2} - \sqrt{1+|k|^2}\right)t \right)
             \widehat{\psi}(t,k-\ell) \widehat{d_{2,+}}(t,\ell) \left(1+\abs{\ell}^2\right)^{-\frac12} \overline{\widehat{\phi}}(t,k)dt
        \end{align}
        Therefore, the last integral is always of oscillatory type except when $\abs{\ell} =  \abs{k} $.

        We conclude that the only contribution for $\widetilde{\mathcal{A}}_2$ is given by \eqref{sec4:correctorlim2}. Hence
        \begin{align}
        \label{sec4:limit_sol_RHS_A2}
            \lim_{\varepsilon\to 0} &\left \langle  T_{2,+}^\varepsilon \nabla_x \wedge (\partial_{x_i} \widetilde{\mathcal{A}}_2),\phi\right \rangle \nonumber\\
            &= \frac{\imm}{(2\pi)^6} \int_{-\infty}^{+\infty}\sum_{k \in \Z^3} \sum_{\substack{\ell \in \Z^3 \\ |\ell| =  |k|}}  k \wedge \left( k_i \int_M \reallywidehat{\rho_\Theta (w_\Theta)_i}(t,k-\ell) \mu(d\Theta) \left(1+|\ell|^2\right)^{-\frac{1}{2}}\widehat{d}_{2,+}(t,\ell) \right) \overline{\widehat{\phi}}(t,k)  dt.
        \end{align}
        By noticing the symmetry between the definitions of $\widetilde{\mathcal{A}}_2$ and $\widetilde{\mathcal{A}}_3$, we get
        \begin{align}
        \label{sec4:limit_sol_RHS_A3}
            \lim_{\varepsilon\to 0} &\left \langle  T_{2,+}^\varepsilon \nabla_x \wedge (\partial_{x_i} \widetilde{\mathcal{A}}_3),\phi\right \rangle \nonumber\\
            &= \frac{\imm}{(2\pi)^6} \int_{-\infty}^{+\infty}\sum_{k \in \Z^3} \sum_{\substack{\ell \in \Z^3 \\ |\ell| =  |k|}}  k \wedge \left( k_i \int_M \reallywidehat{\rho_\Theta w_\Theta}(t,k-\ell) \mu(d\Theta) \left(1+|\ell|^2\right)^{-\frac{1}{2}}(\widehat{d}_{2,+})_i (t,\ell) \right) \overline{\widehat{\phi}}(t,k)  dt.
        \end{align}
        Finally, we study $\widetilde{\mathcal{A}}_4$.
        We know that $\int_M \rho^\varepsilon_\Theta  \mu(d\Theta) \to 1$ in $L^2$ strongly by Proposition \ref{sec4:prop2}. Therefore, by adding and subtracting the limit as in \eqref{sec4:J_limit_A_2}, we obtain
        \begin{align*}
            \lim_{\varepsilon\to 0}  T^\varepsilon_{2,+}\widetilde{\mathcal{A}}_4 
            = \lim_{\varepsilon\to 0}  T^\varepsilon_{2,+}W^\varepsilon_i W^\varepsilon \int_M\rho^\varepsilon_\Theta \mu(d\Theta) 
            = \lim_{\varepsilon\to 0}  T^\varepsilon_{2,+}W^\varepsilon_i W^\varepsilon.
        \end{align*}
        By part \eqref{prop4.4:pt2} of Proposition \ref{sec4.2:prop4.4} and proceeding as in \eqref{sec4:splitting_W}, we have
        \begin{align*}
            &\lim_{\varepsilon\to 0}
            T^\varepsilon_{2,+}\left( 
             W^\varepsilon_i W^\varepsilon
            \right)  
            =\sum_{\substack{\sigma_1,\sigma_2 \in\{0,1,2\}, \\\sigma_3,\sigma_4 \in \{\pm\}}}\lim_{\varepsilon\to 0}
            T^\varepsilon_{2,+}\left( 
            T^\varepsilon_{\sigma_1,\sigma_3}(\widetilde{d}_{\sigma_1,-\sigma_3})_i
            T^\varepsilon_{\sigma_2,\sigma_4}(\widetilde{d}_{\sigma_2,-\sigma_4})
            \right).
        \end{align*}
        where we defined $T_{0,\pm}^\varepsilon = T_{1,\pm}^\varepsilon$ for notational convenience.
        We now study, for $\sigma_1,\sigma_2 \in \{0,1,2\}$, $\sigma_3,\sigma_4 \in\{\pm\}$ and $\psi_1:=(\widetilde{d}_{\sigma_1,-\sigma_3})_i, \psi_2:=\widetilde{d}_{\sigma_2,-\sigma_4} \in L^\infty_t H^{s-2}_x$, which terms in 
        \[
        \widetilde{\mathcal{J}}_{\sigma_1, \sigma_2, \sigma_3, \sigma_4}(t,k):=\mathcal{F}\left(T^\varepsilon_{2,+}\left( 
            T^\varepsilon_{\sigma_1,\sigma_3}(\psi_1)
            T^\varepsilon_{\sigma_2,\sigma_4}(\psi_2)
            \right)
            \right)(t,k)
        \]
        give a non vanishing contribution as $\varepsilon$ goes to zero.

        If $\sigma_1, \sigma_2 \in \{0, 1\}$, using the formula for $T^\varepsilon_{1,\pm}$ and $T^\varepsilon_{2,\pm}$ given in \eqref{sec4:def_T}, we have
           \[
        \widetilde{\mathcal{J}}_{\sigma_1,\sigma_2,\sigma_3,\sigma_4}(t,k)= \exp\left(-\frac{\imm t}{\varepsilon}\left[\sqrt{1+|k|^2}+\sigma_31+\sigma_41 \right] \right) \frac{1}{(2\pi)^3}\sum_{\ell \in \Z^3} \widehat{\psi_1}(t,k-\ell)\widehat{\psi_2}(t,\ell)
        \]
        with $\sigma_3,\sigma_4 \in \{\pm\}$. We notice that, for $\sigma_1, \sigma_2 \in \{0, 1\}$, the phase of the oscillatory integrand $\widetilde{\mathcal{J}}_{\sigma_1,\sigma_2,\sigma_3,\sigma_4}$ is non-zero, except when $\sigma_3=\sigma_4=-$, in this case the phase is zero for $|k|=\sqrt3$, i.e.,
        \[
        \widetilde{\mathcal{J}}_{\sigma_1,\sigma_2,-,-}(t,k)= \frac{1}{(2\pi)^3} \sum_{\ell \in \mathbb{Z}^3} \widehat{\psi_1}(t,k-\ell)\widehat{\psi_2}(t,\ell)+ \text{oscillatory terms}
        \]
        for $k \in \boldsymbol{1}=\{ \ell \in \Z^3 : \ell_i \in \{\pm1 \}, i \in \{1,2,3\}\}$.

       We observe that for $\sigma_1 = \sigma_2 = 0$,
        $\widetilde{\mathcal{J}}_{0,0,-,-}(t,k)$
        gives a zero contribution in the limit, since $\widetilde{d}_{0,\pm}(t)$ are spatially independent.
        Next, for $\sigma_1 = 0$ and  $\sigma_2 = 1$,
        \begin{align*}
        \widetilde{\mathcal{J}}_{0,1,-,-}(t,k)
        &= (\widetilde{d}_{0,+}(t))_i \widehat{\widetilde{d}_{1,+}}(t,k)+ \text{oscillatory terms}
        \end{align*}
        for $k \in \boldsymbol{1}=\{ \ell \in \Z^3 : \ell_i \in \{\pm1 \}, i \in \{1,2,3\}\}.$ By symmetry, the same formula holds for $\widetilde{\mathcal{J}}_{1,0,-,-}(t,k)$.

        If $\sigma_1=\sigma_2=2$ we have
        \[
        \widetilde{\mathcal{J}}_{2,2,\sigma_3,\sigma_4}(t,k)= \frac{1}{(2\pi)^3} \sum_{\ell \in \Z^3}\exp\left(-\frac{\imm t}{\varepsilon}\left[\sqrt{1+|k|^2}+\sigma_3 \sqrt{1+|k-\ell|^2}+\sigma_4\sqrt{1+|\ell|^2} \right] \right) \widehat{\psi_1}(t,k-\ell)\widehat{\psi_2}(t,\ell)
        \]
        with $\sigma_3,\sigma_4 \in \{\pm\}$. We observe that the phase of the oscillatory integrand $\widetilde{\mathcal{J}}_{2,2,\sigma_3,\sigma_4}$ is always non-zero.

        If $\sigma_1 = 2$ and $\sigma_2=0$, we have 
          \[
        \widetilde{\mathcal{J}}_{2,0,\sigma_3,\sigma_4}(t,k)
        = \frac{1}{(2\pi)^3} \sum_{\ell \in \Z^3}\exp\left(-\frac{\imm t}{\varepsilon}[\sqrt{1+|k|^2}+\sigma_3\sqrt{1+|\ell|^2}+\sigma_41] \right) \widehat{\psi_1}(t,\ell)\widehat{\psi_2}(t,k-\ell).
        \]
         We notice that the phase of the oscillatory integrand $\widetilde{\mathcal{J}}_{2,0,\sigma_3,\sigma_4}$ can be zero, if $\sigma_3=-$. Therefore, we have
    \begin{align*}
        \widetilde{\mathcal{J}}_{2,0,-,\sigma_4}(t,k) 
        = \frac{1}{(2\pi)^3} \sum_{\ell \in \Omega^{(2)}_{\sigma_4,-\sigma_4}(k)} (\widehat{\widetilde{d}_{2,+}})_i(t,\ell)\widehat{\widetilde{d}_{0,-\sigma_4}}(t,k-\ell) + \text{oscillatory terms},
    \end{align*}
        where $\Omega^{(2)}_{\sigma_4,-\sigma_4}(k)$ is defined by \eqref{sec4:eqsigma}. But this implies that $\widetilde{\mathcal{J}}_{2,0,-,\sigma_4}(t,k)$ gives a zero contribution in the limit because $\widetilde{d}_{0,-\sigma_4}$ is spatially homogeneous. The same holds for $\widetilde{\mathcal{J}}_{0,2,\sigma_3,-}(t,k)$.

         If $\sigma_1\sigma_2=2$, we can assume w.l.o.g. that $\sigma_1=2, \sigma_2=1$ and we get
        \[
        \widetilde{\mathcal{J}}_{2,1,\sigma_3,\sigma_4}(t,k)
        = \frac{1}{(2\pi)^3} \sum_{\ell \in \Z^3}\exp\left(-\frac{\imm t}{\varepsilon}[\sqrt{1+|k|^2}+\sigma_3\sqrt{1+|\ell|^2}+\sigma_41] \right) \widehat{\psi_1}(t,\ell)\widehat{\psi_2}(t,k-\ell)
        \]
        with $\sigma_3,\sigma_4 \in \{\pm\}$. We notice that the phase of the oscillatory integrand $\widetilde{\mathcal{J}}_{2,1,\sigma_3,\sigma_4}$ can be zero, if $\sigma_3=-$. Therefore, we have
    \begin{align*}
        \widetilde{\mathcal{J}}_{2,1,-,\sigma_4}(t,k) 
        = \frac{1}{(2\pi)^3} \sum_{\ell \in \Omega^{(2)}_{\sigma_4,-\sigma_4}(k)} \widehat{\psi_1}(t,\ell)\widehat{\psi_2}(t,k-\ell) + \text{oscillatory terms},
    \end{align*}
        where $\Omega^{(2)}_{\sigma_4,-\sigma_4}(k)$ is defined by \eqref{sec4:eqsigma}.

        Collecting the limit contributions given by $\widetilde{\mathcal{J}}_{\sigma_1, \sigma_2, \sigma_3, \sigma_4}$, we arrive at
        \begin{align}
        \label{sec4:limit_sol_RHS_A4}
        &\lim_{\varepsilon \to 0} \left \langle  T_{2,+}^\varepsilon \nabla_x \wedge \partial_{x_i} \widetilde{\mathcal{A}}_4, \phi\right \rangle
        = \frac{1}{(2\pi)^3} \sum_{k \in \boldsymbol{1}}  \int_{-\infty}^{+\infty} k \wedge \left(k_i (d_{0,+})_i (t) \widehat{d_{1,+}}(t,k)\right) \overline{\widehat{\phi}}(t,k)dt \nonumber \\
        & + \frac{1}{(2\pi)^3} \sum_{k \in \boldsymbol{1}}  \int_{-\infty}^{+\infty} k \wedge \left(k_i d_{0,+} (t) (\widehat{d_{1,+}})_i (t,k)\right) \overline{\widehat{\phi}}(t,k)dt \nonumber \\
        &+\frac{1}{(2\pi)^6} \sum_{k \in \boldsymbol{1}}  \int_{-\infty}^{+\infty} k \wedge \left(k_i \sum_{\ell \in \mathbb{Z}^3} (\widehat{d_{1,+}})_i(t,k-\ell)\widehat{d_{1,+}}(t,\ell) \right) \overline{\widehat{\phi}}(t,k)dt \nonumber\\
        &\quad - \frac{1}{(2\pi)^6} \sum_{\sigma \in \{\pm\}} \sigma \sum_{k \in \Z^3}\int_{-\infty}^{+\infty}  k \wedge \left( k_i \sum_{\ell \in \Omega^{(2)}_{\sigma, -\sigma}(k)} \left(1+|\ell|^2\right)^{-\frac{1}{2}}(\widehat{d_{2,+}})_i(t,\ell)\widehat{d_{1,-\sigma}} (t,k-\ell) \right)
        \overline{\widehat{\phi}}(t,k)dt \nonumber \\
        &\quad - \frac{1}{(2\pi)^6} \sum_{\sigma \in \{\pm\}} \sigma \sum_{k \in \Z^3}\int_{-\infty}^{+\infty}  k \wedge \left( k_i \sum_{\ell \in \Omega^{(2)}_{\sigma, -\sigma}(k)} \left(1+|\ell|^2\right)^{-\frac{1}{2}}\widehat{d_{2,+}}(t,\ell)(\widehat{d_{1,-\sigma}})_i (t,k-\ell) \right)
        \overline{\widehat{\phi}}(t,k)dt.
        \end{align}

        \vskip 0.5 cm

       \noindent \underline{\emph{Limit of the $\nabla_x \wedge \mathcal{E}^\varepsilon$ term}}: We now study the second term in \eqref{sec4:differenterms_2}.
        By part \eqref{prop4.4:pt1} of Proposition \ref{sec4.2:prop4.4} and proceeding as in \eqref{sec4:splitting_W}, we have for  $\phi\in C^\infty_c((0,T) \times \mathbb{T}^3_x)$
         \begin{align*}
            \lim_{\varepsilon\to 0}  \left \langle 
            T^\varepsilon_{2,+}\left(\varepsilon 
           E^\varepsilon \nabla_x \cdot(\varepsilon E_{\text{irr}}^\varepsilon)
            \right), \phi \right \rangle 
            &= 
            \sum_{\substack{\sigma_1 \in\{0, 1,2\}, \\\sigma_3,\sigma_4 \in \{\pm\}}}\lim_{\varepsilon\to 0}
            \left \langle 
            T^\varepsilon_{2,+}\left( 
            T^\varepsilon_{\sigma_1,\sigma_3}({d_{\sigma_1,-\sigma_3}})
            \nabla_x \cdot T^\varepsilon_{1,\sigma_4}({d_{1,-\sigma_4}})
            \right), \phi \right \rangle,
        \end{align*}
        where we introduced the notation $T^\varepsilon_{0,+}:=T_{1,+}^\varepsilon$. We note that only $T^\varepsilon_{1,\pm}({d_{1,\mp}})$ appears with the divergence since $ \nabla_x \cdot T^\varepsilon_{2,\pm}({d_{2,\mp}}) = \nabla_x \cdot T^\varepsilon_{2,\pm}({d_{2,\mp}}) = 0$, as $d_{0,\mp}$ are spatially independent and $d_{2,\mp}$ are solenoidal.

        Next, we study for $\sigma_1 \in\{0, 1,2\},  \sigma_3,\sigma_4 \in\{\pm\}$ and $\psi_1:=d_{\sigma_1,-\sigma_3}, \psi_2:=d_{1,-\sigma_4}\in L^\infty_t H^{s-2}_x$, which terms in
        \[
        \widetilde{\mathcal{E}}_{\sigma_1,1, \sigma_3, \sigma_4}(t,k)=\mathcal{F}\left(T^\varepsilon_{2,+}\left( 
            T^\varepsilon_{\sigma_1,\sigma_3}(\psi_1)
            T^\varepsilon_{1,\sigma_4}(\nabla_x \cdot \psi_2)
            \right)
            \right)(t,k)
        \]
        give a non vanishing contribution as $\varepsilon$ goes to zero.

        If $\sigma_1 \in \{0, 1 \}$, using the formula for $T^\varepsilon_{1,\pm}$ and $T^\varepsilon_{2,+}$ given in \eqref{sec4:def_T}, we have
        \[
        \widetilde{\mathcal{E}}_{\sigma_1,1,\sigma_3,\sigma_4}(t,k)=\imm\exp\left(-\frac{\imm t}{\varepsilon}[\sqrt{1+|k|^2}+\sigma_31+\sigma_41] \right) \frac{1}{(2 \pi)^3}\sum_{\ell \in \Z^3} \widehat{\psi_1}(t,k-\ell)\ell\cdot \widehat{\psi_2}(t,\ell)
        \]
        with $\sigma_3,\sigma_4 \in \{\pm\}$. We notice that, for $\sigma_1 \in \{0, 1 \}$,  the phase of the oscillatory integrand $\widetilde{\mathcal{E}}_{\sigma_1,1,\sigma_3,\sigma_4}$ is non-zero, except when $\sigma_3=\sigma_4=-$ so that the phase is zero for $|k|=\sqrt3$, i.e.,
        \[
        \widetilde{\mathcal{E}}_{\sigma_1,1,-,-}(t,k)=\frac{\imm}{(2 \pi)^3}\sum_{\ell \in \Z^3} \widehat{\psi_1}(t,k-\ell)\ell\cdot \widehat{\psi_2}(t,\ell)+ \text{oscillatory terms},
        \]
        for $ k \in \boldsymbol{1}=\{ \ell \in \Z^3 : \ell_i \in \{\pm1 \}, i \in \{1,2,3\}\}$.

         If $\sigma_1=2$ we have
        \[
        \widetilde{\mathcal{E}}_{2,1,\sigma_3,\sigma_4}(t,k)=\frac{\imm}{(2 \pi)^3} \sum_{\ell \in \Z^3}\exp\left(-\frac{\imm t}{\varepsilon}[\sqrt{1+|k|^2}+\sigma_3\sqrt{1+|\ell|^2}+\sigma_41] \right) \widehat{\psi_1}(t,\ell)(k-\ell) \cdot\widehat{\psi_2}(t,k-\ell)
        \]
        with $\sigma_3,\sigma_4 \in \{\pm\}$. We notice that the phase of the oscillatory integrand $ \widetilde{\mathcal{E}}_{2,1,\sigma_3,\sigma_4}$ can be zero if $\sigma_3=-$. Thus, we have
        \begin{align*}
        \widetilde{\mathcal{E}}_{2,1,-,\sigma_4}(t,k) = \frac{\imm}{(2 \pi)^3} \sum_{\ell \in \Omega^{(2)}_{\sigma_4,-\sigma_4}(k)} \widehat{\psi_1}(t,\ell) (k-\ell) \cdot \widehat{\psi_2}(t,k-\ell) + \text{oscillatory terms},
    \end{align*}
        where $\Omega^{(2)}_{\eta_1,\eta_2}(k)$ is defined by \eqref{sec4:eqsigma}.
        Therefore, 
        \begin{align}
        \label{sec4:limit_sol_RHS_E}
            & - \lim_{\varepsilon\to 0}  \left \langle 
            T^\varepsilon_{2,+} \nabla_x \wedge \left(\varepsilon 
           E^\varepsilon \nabla_x \cdot(\varepsilon E_{\text{irr}}^\varepsilon)
            \right), \phi \right \rangle 
            = \frac{1}{(2 \pi)^3} \sum_{k \in \boldsymbol{1}} \int_{-\infty}^{+\infty} k \wedge \left(  d_{0,+} (t) k \cdot \widehat{d_{1,+}} (t,k) \right) \,\overline{\widehat{\phi}}(t,k) dt \nonumber \\
             & \qquad + \frac{1}{(2 \pi)^6} \sum_{k \in \boldsymbol{1}} \int_{-\infty}^{+\infty} k \wedge \left(    \sum_{\ell \in \Z^3} \widehat{d_{1,+}} (t,k-\ell) \ell \cdot \widehat{d_{1,+}} (t,\ell) \right) \,\overline{\widehat{\phi}}(t,k) dt \nonumber \\
            & \qquad  + \frac{1}{(2 \pi)^6} \sum_{\sigma \in \{\pm\}}\sum_{k \in \Z^3} \int_{-\infty}^{+\infty} k \wedge \left(    \sum_{\ell \in \Omega^{(2)}_{\sigma,-\sigma}(k)} \widehat{d_{2,+}} (t,\ell) (k-\ell) \cdot \widehat{d_{1,-\sigma}} (t,k-\ell) \right) \,\overline{\widehat{\phi}}(t,k) dt.
        \end{align}

        \vskip 0.5 cm
        \noindent\underline{\emph{Limit of the $\nabla_x \wedge \mathcal{B}^\varepsilon$ term}}: We now study the third term in \eqref{sec4:differenterms_2}, i.e., we consider
        \begin{equation*}
            \nabla_x \wedge \mathcal{B}^\varepsilon = \nabla_x \wedge \mathcal{B}_1^\varepsilon + \nabla_x \wedge \mathcal{B}_2^\varepsilon,
        \end{equation*}
        where
        $$
        \mathcal{B}^\varepsilon_1:= \int_{M}\rho^\varepsilon_\Theta \xi^\varepsilon_\Theta \mu(d\Theta)\wedge B^\varepsilon
        $$
        and 
        $$
        \mathcal{B}_2^\varepsilon(t,x):=\int_{M}\rho^\varepsilon_\Theta [v(\xi^\varepsilon_\Theta)-\xi^\varepsilon_\Theta] \mu(d\Theta)\wedge B^\varepsilon.
        $$
     The treatment of $\mathcal{B}_2^\varepsilon$ is done using Lemma \ref{sec2:lemma_sobolev} on the difference between the relativistic and non relativistic velocity. Therefore, $\mathcal{B}_2^\varepsilon$ is a remainder term of order $\varepsilon^2$ which is strongly converging to zero.
        Therefore, we only focus on the $\mathcal{B}_1^\varepsilon$ term. By recalling the two formulas in \eqref{sec4:smallb}, we have
        \[
        \mathcal{B}_1^\varepsilon=\int_M \rho^\varepsilon_\Theta(w^\varepsilon_\Theta+W^\varepsilon)\mu(d\Theta)
        \wedge \left( b^\varepsilon(t,x) - \nabla_x  \wedge W^\varepsilon \right)
        =:\bar{\mathcal{A}}_1+\bar{\mathcal{A}}_2+\bar{\mathcal{A}}_3 + \bar{\mathcal{A}}_4,
        \]
        where
        \[
        \bar{\mathcal{A}}_1:=\int_M \rho^\varepsilon_\Theta w^\varepsilon_\Theta \mu(d\Theta)\wedge b^\varepsilon,\quad \bar{\mathcal{A}}_2:=\int_M\rho^\varepsilon_\Theta \mu(d\Theta)W^\varepsilon\wedge b^\varepsilon,
        \]
        \[
        \bar{\mathcal{A}}_3:=-\int_M\rho^\varepsilon_\Theta w^\varepsilon_\Theta \mu(d\Theta)\wedge (\nabla_x  \wedge W^\varepsilon),\quad 
        \bar{\mathcal{A}}_4:=-\int_M\rho^\varepsilon_\Theta \mu(d\Theta)W^\varepsilon\wedge (\nabla_x  \wedge  W^\varepsilon).
        \]
        For $\bar{\mathcal{A}}_1$, since $\int_M \rho^\varepsilon_\Theta w^\varepsilon_\Theta \mu(d\Theta) \wedge b^\varepsilon \to \int_M \rho_\Theta w_\Theta \mu(d\Theta) \wedge B$ in $L^2$ strongly by Proposition \ref{sec4:prop2}, we obtain, using part \eqref{sec4:lemmapt2} of Lemma \ref{sec4:lemma1},
        \[
         T^\varepsilon_{2,+}\bar{\mathcal{A}}_1=T^\varepsilon_{2,+} \left(\int_M \rho^\varepsilon_\Theta w^\varepsilon_\Theta  \mu(d\Theta) \wedge b^\varepsilon\right) \rightharpoonup0 \quad \text{in} \quad L^2_{t,x}.
        \]
        Concerning $\bar{\mathcal{A}}_2$, we know $\int_M \rho^\varepsilon_\Theta  \mu(d\Theta) \to 1$ and $b^\varepsilon \to B$ in $L^2$ strongly by Proposition \ref{sec4:prop2}. Therefore, by adding and subtracting these two limits to $\bar{\mathcal{A}}_2$, we get
        \begin{align*}
            \lim_{\varepsilon\to 0} T^\varepsilon_{2,+} \bar{\mathcal{A}}_2 
            = \lim_{\varepsilon\to 0} T^\varepsilon_{2,+} \left(\int_M\rho^\varepsilon_\Theta \mu(d\Theta)  W^\varepsilon \wedge b^\varepsilon\right)  =
         \lim_{\varepsilon\to 0} T^\varepsilon_{2,+}\left(W^\varepsilon \wedge B \right).
        \end{align*}
        By part \eqref{prop4.4:pt2} of Proposition \ref{sec4.2:prop4.4} and similarly as in \eqref{sec4:splitting_W}, we have
        \begin{align*}
            \lim_{\varepsilon \to 0} T^\varepsilon_{2,+}\bar{\mathcal{A}}_2=
        \lim_{\varepsilon \to 0} T^\varepsilon_{2,+}\left( \left( T^\varepsilon_{1,+} (\widetilde{d}_{0,-}) + T^\varepsilon_{1,-} (\widetilde{d}_{0,+}) +T^\varepsilon_{1,+} (\widetilde{d}_{1,-}) + T^\varepsilon_{1,-} (\widetilde{d}_{1,+})+  T^\varepsilon_{2,+} (\widetilde{d}_{2,-}) + T^\varepsilon_{2,-} (\widetilde{d}_{2,+}) \right) \wedge B\right).
        \end{align*}
        However, by part \eqref{sec4:lemmapt3} of Lemma \ref{sec4:lemma1}, 
        \[
        \lim_{\varepsilon \to 0} T^\varepsilon_{2,+} \left[ \left( T^\varepsilon_{1,+} (\widetilde{d}_{0,-}) + T^\varepsilon_{1,+} (\widetilde{d}_{1,-}) + T^\varepsilon_{2,+} (\widetilde{d}_{2,-}) \right) \wedge B\right]=0.
        \]
        We conclude that the following identity holds:
        \begin{align}
            \label{sec4:lim_T_A2_mean_sol}
        \lim_{\varepsilon \to 0} T^\varepsilon_{2,+}\bar{\mathcal{A}}_2=
        \lim_{\varepsilon \to 0} T^\varepsilon_{2,+}\left[ \left( T^\varepsilon_{1,-} (\widetilde{d}_{0,+}) +T^\varepsilon_{1,-} (\widetilde{d}_{1,+})+T^\varepsilon_{2,-} (\widetilde{d}_{2,+})\right)\wedge B\right].
        \end{align}
        Taking  $\phi\in C^\infty_c((0,T) \times \mathbb{T}^3_x)$, we compute the  weak limit of the first term in \eqref{sec4:lim_T_A2_mean_sol}. By using Plancherel identity,  the expressions for $T_{1,\pm}^\varepsilon$ and $T_{2,+}^\varepsilon$ given by \eqref{sec4:def_T} and the one of $\widetilde{d}_{0,+}$ given by \eqref{sec4:dd}, we obtain
        \begin{align*}
             \lim_{\varepsilon \to 0}\left \langle 
            T^\varepsilon_{2,+}\left( 
             T^\varepsilon_{1,-} (\widetilde{d}_{0,+})\wedge B
            \right), \phi \right \rangle 
            = - \frac{\imm}{(2\pi)^3} \lim_{\varepsilon \to 0} \sum_{k \in \Z^3} \int_{-\infty}^{+\infty}  \exp\left(-\frac{\imm t}{\varepsilon}\left(\sqrt{1+|k|^2}-1\right) \right)
               d_{0,+}(t) \wedge \widehat{B}(t,k)   \overline{\widehat{\phi}}(t,k)dt.
        \end{align*}
        Hence, the last integral is always of oscillatory type except when $k=0$, and since we take the curl this term vanishes. That is,
        \begin{equation*}
            \lim_{\varepsilon\to 0} \left \langle 
            T^\varepsilon_{2,+} \nabla_x \wedge\left( 
             T^\varepsilon_{1,-} (\widetilde{d}_{0,+})\wedge B
            \right), \phi \right \rangle = 0.
        \end{equation*}
        Then, we compute the  weak limit of the second term in \eqref{sec4:lim_T_A2_mean_sol}. By using Plancherel identity, the expressions for $T_{1,-}^\varepsilon$ and $T_{2,+}^\varepsilon$ given by \eqref{sec4:def_T} and the one of $\widetilde{d}_{1,+}$ given by \eqref{sec4:dd}, we obtain
        \begin{align*}
             &\lim_{\varepsilon \to 0}\left \langle 
            T^\varepsilon_{2,+}\left( 
             T^\varepsilon_{1,-} (\widetilde{d}_{1,+})\wedge B
            \right), \phi \right \rangle 
            \\
            &\qquad = - \imm  \lim_{\varepsilon \to 0} \frac{1}{(2\pi)^6} \sum_{k \in \Z^3}\sum_{\ell \in \Z^3}\int_{-\infty}^{+\infty}  \exp\left(-\frac{\imm t}{\varepsilon}\left(\sqrt{1+|k|^2}-1\right) \right)
               \widehat{d_{1,+}}(t,\ell) \wedge \widehat{B}(t,k-\ell)   \overline{\widehat{\phi}}(t,k)dt.
            \end{align*}
        Hence, the last integral is always of oscillatory type except when $k=0$, and since we take the curl this term vanishes. That is,
        \begin{equation*}
            \lim_{\varepsilon\to 0} \left \langle 
            T^\varepsilon_{2,+} \nabla_x \wedge\left( 
             T^\varepsilon_{1,-} (\widetilde{d}_{1,+})\wedge B
            \right), \phi \right \rangle = 0.
        \end{equation*}
        We now look at the weak limit of the third term in \eqref{sec4:lim_T_A2_mean_sol}, i.e.,
        \begin{align}
        \label{sec4:corrector_sol_B_term_A_2}
             &\left \langle 
            T^\varepsilon_{2,+}\left( 
             T^\varepsilon_{2,-} (\widetilde{d}_{2,+})\wedge B
            \right), \phi \right \rangle
            =
            - \frac{\imm}{(2 \pi)^6} \int_{-\infty}^{+\infty} \sum_{k\in \Z^3}
            \exp{\left(-\frac{\imm \sqrt{1+|k|^2} t}{\varepsilon} \right)}   \nonumber \\
            &
            \quad \times \sum_{\ell \in \Z^3}
            \exp{\left(\frac{\imm \sqrt{1+|\ell|^2} t}{\varepsilon} \right)} \widehat{d_{2,+}}(t,\ell)\left(1+\abs{\ell}^2\right)^{-\frac12} 
            \wedge \widehat{B}(t,k-\ell)
            \overline{\widehat{\phi}}(t,k) dt \nonumber \\
            & = -  \frac{\imm}{(2 \pi)^6} \sum_{k,\ell \in \Z^3} \int_{-\infty}^{+\infty} \exp{\left(-\frac{\imm  t}{\varepsilon} (\sqrt{1+|k|^2} - \sqrt{1+|\ell|^2}) \right)}
            \widehat{d_{2,+}}(t,\ell)\left(1+\abs{\ell}^2\right)^{-\frac12} 
            \wedge \widehat{B}(t,k-\ell)
            \overline{\widehat{\phi}}(t,k) dt.
        \end{align}
         Therefore, the last integral is always of oscillatory type except when $\abs{\ell} =  \abs{k} $.
        We conclude that the only contribution for $\bar{\mathcal{A}}_2$ is given by \eqref{sec4:corrector_sol_B_term_A_2}. Hence
        \begin{align}
        \label{sec4:limit_sol_RHS_3}
           - \lim_{\varepsilon\to 0} &\left \langle T_{2,+}^\varepsilon \nabla_x \wedge \bar{\mathcal{A}}_2,\phi\right \rangle 
            = - \frac{1}{(2 \pi)^6} \int_{-\infty}^{+\infty} \sum_{k \in \Z^3} \sum_{\substack{\ell \in \Z^3 \\ |\ell| = |k|}}  k \wedge \left(  \widehat{d_{2,+}}(t,\ell) \left(1+\abs{\ell}^2\right)^{-\frac12}\wedge \widehat{B}(t,k-\ell) \right) \overline{\widehat{\phi}}(t,k)dt.
        \end{align}
    Next, we study $\bar{\mathcal{A}}_3$.  Since $\int_M\rho^\varepsilon_\Theta w^\varepsilon_\Theta \mu(d\Theta) \to \int_M\rho_\Theta w_\Theta \mu(d\Theta)$, by adding and subtracting the limit, we get
        \begin{align*}
        \lim_{\varepsilon \to 0} T^\varepsilon_{2,+}\bar{\mathcal{A}}_3 =
            - \lim_{\varepsilon \to 0} T^\varepsilon_{2,+} \left(
            \int_M\rho_\Theta w_\Theta \mu(d\Theta)\wedge (\nabla_x  \wedge W^\varepsilon) \right).
        \end{align*}
        Next, by part \eqref{prop4.4:pt2} of Proposition \ref{sec4.2:prop4.4} and similarly as in \eqref{sec4:splitting_W}, we have
        \begin{align*}
        \lim_{\varepsilon \to 0} T^\varepsilon_{2,+}\bar{\mathcal{A}}_3 &=
            - \lim_{\varepsilon \to 0} T^\varepsilon_{2,+} \left(
            \int_M\rho_\Theta w_\Theta \mu(d\Theta) \right. \\
            & \left. \wedge \left(\nabla_x  \wedge \left( T^\varepsilon_{1,+} (\widetilde{d}_{0,-}) + T^\varepsilon_{1,-} (\widetilde{d}_{0,+}) + T^\varepsilon_{1,+} (\widetilde{d}_{1,-}) + T^\varepsilon_{1,-} (\widetilde{d}_{1,+})+  T^\varepsilon_{2,+} (\widetilde{d}_{2,-}) + T^\varepsilon_{2,-} (\widetilde{d}_{2,+}) \right) \right) \right).
        \end{align*}
        Observe that $\nabla_x  \wedge \left(T^\varepsilon_{1,+} (\widetilde{d}_{0,-}) + T^\varepsilon_{1,-} (\widetilde{d}_{0,+}) + T^\varepsilon_{1,+} (\widetilde{d}_{1,-}) + T^\varepsilon_{1,-} (\widetilde{d}_{1,+})\right)  = 0$ since $\widetilde{d}_{0,\pm}$ are spatially independent and $\widetilde{d}_{1,\pm}$ are irrotational. Moreover, by  by part \eqref{sec4:lemmapt3} of Lemma \ref{sec4:lemma1} and by commuting the curl with $T_{2,+}^\varepsilon$, we have
        \begin{align*}
             \lim_{\varepsilon \to 0} T^\varepsilon_{2,+} \left(
            \int_M\rho_\Theta w_\Theta \mu(d\Theta)\wedge \left( T^\varepsilon_{2,+} (\nabla_x  \wedge \widetilde{d}_{2,-})  \right) \right) = 0.
        \end{align*}
        Hence, the following identity holds:
        \[
        \lim_{\varepsilon \to 0} T^\varepsilon_{2,+}\bar{\mathcal{A}}_3=-
        \lim_{\varepsilon \to 0} T^\varepsilon_{2,+}\left[\left( \int_M \rho_\Theta w_\Theta\mu(d\Theta)\right) \wedge T^\varepsilon_{2,-} (\nabla_x  \wedge \widetilde{d}_{2,+})\right].
        \]
        Taking  $\phi\in C^\infty_c((0,T) \times \mathbb{T}^3_x)$ and  $\psi(t,x):=  \int_M \rho_\Theta w_\Theta\mu(d\Theta) \in L^\infty_tH^{s-2}_x$, we compute the  weak limit of the last equation. By using Plancherel identity, the expression \eqref{sec4:def_T} for $T_{2,\pm}^\varepsilon$ and the one of $\widetilde{d}_{2,+}$ given by \eqref{sec4:dd}, we obtain
       {\small\begin{align*}
            &\left \langle 
            T^\varepsilon_{2,+}\left( 
            \psi \wedge T^\varepsilon_{2,-} (\nabla_x  \wedge \widetilde{d}_{2,+})
            \right), \phi \right \rangle 
            = \frac{1}{(2 \pi)^6}
            \int_{-\infty}^{+\infty}  \sum_{k \in \Z^3}
            \exp\left(-\frac{\imm\sqrt{1+ |k|^2} t}{\varepsilon}\right) \\
            &\qquad \times \sum_{\ell \in \Z^3} \widehat{\psi}(t,k-\ell)\wedge \left[\ell\wedge\widehat{d_{2,+}}(t,\ell)\left(1+|\ell|^2\right)^{-\frac{1}{2}}\right] \exp\left( \imm \frac{\sqrt{1+|\ell|^2}t}{\varepsilon} \right) \overline{\widehat{\phi}}(t,k) dt\\
            &= \frac{1}{(2 \pi)^6} \sum_{k \in \Z^3}\sum_{\ell \in \Z^3}\int_{-\infty}^{+\infty}  \exp\left(\frac{\imm}{\varepsilon}\left(\sqrt{1+|\ell|^2}-\sqrt{1+|k|^2}\right)t \right)
           \widehat{\psi}(t,k-\ell)\wedge \left[\ell\wedge\widehat{d_{2,+}}(t,\ell)\left(1+|\ell|^2\right)^{-\frac{1}{2}}\right] \overline{\widehat{\phi}}(t,k) dt.
            \end{align*}}
        Hence, the limit is vanishing except when $|\ell|=|k|$, and this gives
        {\small\begin{align}
        \label{sec4:othercontribution}
            &-\lim_{\varepsilon \to 0}\left \langle 
            T^\varepsilon_{2,+}\left( \nabla_x\wedge
           \bar{\mathcal{A}}_3
            \right), \phi \right \rangle \nonumber\\
            &=\frac{\imm}{(2 \pi)^6}\sum_{k \in \Z^3} 
        \sum_{\substack{\ell \in \Z^3 \\ |\ell| = |k|}}  \int_{-\infty}^{+\infty}
           k \wedge \Biggl(\Bigl(\int_M \widehat{\rho_\Theta w_\Theta}(t,k-\ell) d\mu(\Theta) \Bigr)\wedge \left[\ell\wedge\widehat{d_{2,+}}(t,\ell)\left(1+|\ell|^2\right)^{-\frac{1}{2}}\right] \Biggr) \overline{\widehat{\phi}}(t,k) dt.
            \end{align}}
        We finally study $\bar{\mathcal{A}}_4$. We know that $\int_M \rho^\varepsilon_\Theta  \mu(d\Theta) \to 1$ in $L^2$ strongly by Proposition \ref{sec4:prop2}. Therefore, by adding and subtracting the limit as in \eqref{sec4:J_limit_A_2}, we obtain
        \begin{align*}
            \lim_{\varepsilon \to 0} T^\varepsilon_{2,+}\bar{\mathcal{A}}_4 
            = -\lim_{\varepsilon \to 0} T^\varepsilon_{2,+} 
            \left(\int_M\rho^\varepsilon_\Theta \mu(d\Theta)W^\varepsilon\wedge (\nabla_x  \wedge  W^\varepsilon) \right) 
            = -\lim_{\varepsilon \to 0} T^\varepsilon_{2,+} 
            \left(W^\varepsilon\wedge (\nabla_x  \wedge  W^\varepsilon) \right).
        \end{align*}
        By part \eqref{prop4.4:pt2} of Proposition \ref{sec4.2:prop4.4} and proceeding as in \eqref{sec4:splitting_W}, we have
        \begin{align*}
            \lim_{\varepsilon\to 0}\left \langle 
            T^\varepsilon_{2,+}\left( 
            W^\varepsilon\wedge (\nabla_x  \wedge  W^\varepsilon)
            \right), \phi \right \rangle
            & = \sum_{\substack{\sigma_1 \in\{0,1,2\}, \\\sigma_3,\sigma_4 \in \{\pm\}}}\lim_{\varepsilon\to 0}
            \left \langle 
            T^\varepsilon_{2,+}\left( 
            T^\varepsilon_{\sigma_1,\sigma_3}(\widetilde{d}_{\sigma_1,-\sigma_3})\wedge
            (\nabla_x  \wedge T^\varepsilon_{2,\sigma_4}(\widetilde{d}_{2,-\sigma_4}))
            \right), \phi \right \rangle,
        \end{align*}
         where we introduced the notation $T^\varepsilon_{0,+}:=T_{1,+}^\varepsilon$.
        We note that only $T^\varepsilon_{2,\pm}({\widetilde{d}_{2,\mp}})$ appears with the curl operator since $\nabla_x \wedge T^\varepsilon_{1,\pm}({\widetilde{d}_{0,\mp}}) = \nabla_x \wedge T^\varepsilon_{1,\pm}({\widetilde{d}_{1,\mp}}) = 0$, as $\widetilde{d}_{0,\mp}$ are spatially independent and  $\widetilde{d}_{1,\mp}$ are irrotational.
        
        We study, for $\sigma_1 \in\{0,1,2\}, \sigma_3,\sigma_4 \in\{\pm\}$ and $\psi_1:=\widetilde{d}_{\sigma_1,-\sigma_3}, \psi_2:=\widetilde{d}_{2,-\sigma_4} \in L^\infty_t H^{s-2}_x$, which terms in 
        \[
        \widetilde{\mathcal{B}}_{\sigma_1,2, \sigma_3, \sigma_4}(t,k)=\mathcal{F}\left(T^\varepsilon_{2,+}\left( 
            T^\varepsilon_{\sigma_1,\sigma_3}(\psi_1)\wedge
            T^\varepsilon_{2,\sigma_4}(\nabla_x  \wedge \psi_2)
            \right)
            \right)
        \]
        give a non vanishing contribution as $\varepsilon$ goes to zero. 
        If $\sigma_1\in \{0,1\}$ we have
        \[
        \widetilde{\mathcal{B}}_{\sigma_1,2,\sigma_3,\sigma_4}(t,k)=\frac{\imm}{(2 \pi)^3}\sum_{\ell \in \Z^3}\exp\left(-\frac{\imm t}{\varepsilon}[\sqrt{1+|k|^2}+\sigma_31+\sigma_4\sqrt{1+|\ell|^2}] \right) \widehat{\psi_1}(t,k-\ell)\wedge[\ell\wedge \widehat{\psi_2}(t,\ell)]
        \]
        with $\sigma_3,\sigma_4 \in \{\pm\}$. We notice that, for $\sigma_1\in \{0,1\}$,  the phase of the oscillatory integrand $\widetilde{\mathcal{B}}_{\sigma_1,2,\sigma_3,\sigma_4}$ is non-zero, except when $\sigma_4=-$.
        Therefore, we have 
    \begin{align*}
        \widetilde{\mathcal{B}}_{\sigma_1,2,\sigma_3,-}(t,k) 
        = \frac{\imm}{(2 \pi)^3} \sum_{\ell \in \Omega^{(2)}_{\sigma_3,-\sigma_3}(k)} \widehat{\psi_1}(t,k-\ell)\wedge[\ell\wedge \widehat{\psi_2}(t,\ell)] + \text{oscillatory terms},
    \end{align*}
        where $\Omega^{(2)}_{\eta_1,\eta_2}(k)$ is defined by \eqref{sec4:eqsigma}. Observe for $\sigma_1= 0$, $\widetilde{\mathcal{B}}_{0,2,\sigma_3,-}$ gives a zero contribution.
         If $\sigma_1=2$ we have
        {\small\[
        \widetilde{\mathcal{B}}_{2,2,\sigma_3,\sigma_4}(t,k)= \frac{\imm}{(2 \pi)^3} \sum_{\ell \in \Z^3}\exp\left(-\frac{\imm t}{\varepsilon}\left[\sqrt{1+|k|^2}+\sigma_3\sqrt{1+|k-\ell|^2}+\sigma_4\sqrt{1+|\ell|^2}\right] \right) \widehat{\psi_1}(t,k-\ell)\wedge[\ell \wedge \widehat{\psi_2}(t,\ell)]
        \]}
        with $\sigma_3,\sigma_4 \in \{\pm\}$. We observe that the phase of the oscillatory integrand $\widetilde{\mathcal{B}}_{2,2,\sigma_3,\sigma_4}$ is always non-zero.
          We conclude that 
        \begin{align}
        \label{sec4:limit_sol_RHS_3_bis}
            &-\lim_{\varepsilon\to 0}  \left \langle T_{2,+}^\varepsilon \nabla_x \wedge \bar{\mathcal{A}}_4,\phi\right \rangle =\nonumber\\
           &  \frac{1}{(2 \pi)^6} \sum_{\sigma \in \{\pm\}} \sum_{k \in \Z^3} \sigma \int_{-\infty}^{+\infty} k \wedge \left(\sum_{\ell \in \Omega^{(2)}_{\sigma, -\sigma}(k)} \widehat{d_{1,-\sigma}}(t,k-\ell)\wedge\left[\ell\wedge \widehat{d_{2,+}}(t,\ell)\left(1+|\ell|^2\right)^{-\frac{1}{2}}\right] \right)\overline{\widehat{\phi}}(t,k)dt. 
        \end{align}
         The study of the fourth term in \eqref{sec4:differenterms_2}, which corresponds to the remainder, 
         follows by using the Sobolev version of inequality \eqref{Lemma_remainder} in Lemma \ref{sec1:lemmarem} and is converging to zero. 
    \vskip 0.5 cm
    \noindent
    \underline{\emph{Conclusion}}: Finally, by collecting \eqref{sec4:limit_sol_LHS}, \eqref{sec4:limit_sol_RHS_A2}, \eqref{sec4:limit_sol_RHS_A3}, \eqref{sec4:limit_sol_RHS_A4}, \eqref{sec4:limit_sol_RHS_E}, \eqref{sec4:limit_sol_RHS_3}, \eqref{sec4:othercontribution} and \eqref{sec4:limit_sol_RHS_3_bis}  we get the equation satisfied by the corrector $d_{2,+}$ given by \eqref{sec4:eqcorrector2}. 
       \end{proof}

\appendix

\section{Proofs of Lemma \ref{sec1:lemmarem} and Lemma \ref{sec2:lemma_sobolev}}
\label{appA}
\begin{proof}[Proof of Lemma \ref{sec1:lemmarem}]
We refer to Lemma 5.6 in \cite{BHK22} for a proof of \eqref{vxiteta} where the following inequality is proved:
\begin{align}
\label{vxiteta_simple_norm}
    \abs{v(\xi^\varepsilon_\Theta)}_\delta \leq C \abs{\xi^\varepsilon_\Theta}_\delta,
\end{align}
which is useful for our proof.
We now show \eqref{Lemma_remainder}. For this, we remind that $\lambda(\xi):=\nabla_\xi (v(\xi)-\xi)$ for $\xi \in \R^3$ and we explicitly compute this derivative in $\xi$ (recall that $\nabla_\xi$ is a vector gradient):
    \begin{align*}
        \lambda(\xi) & = \nabla_\xi \left( \frac{\xi}{\sqrt{1+\varepsilon^2 \abs{\xi}^2}} - \xi \right) 
        = \text{\rm Id}_{3 \times 3} \left( \left( 1+\varepsilon^2 \abs{\xi^\varepsilon_\Theta}^2 \right)^{-\frac12} - 1 \right)
        + \xi  \otimes\nabla_\xi \left(1+\varepsilon^2 \abs{\xi}^2 \right)^{-\frac12} \nonumber \\
        & = \text{\rm Id}_{3 \times 3} \left( \left( 1+\varepsilon^2 \abs{\xi^\varepsilon_\Theta}^2 \right)^{-\frac12} - 1 \right)
        -  \varepsilon^2  \frac{ \xi \otimes \xi }{\left(1+\varepsilon^2 \abs{\xi}^2\right)^{\frac{3}{2}}} ,
    \end{align*}
   where $(\xi \otimes \xi)_{i,j}=\xi_i \xi_j$ for $i,j \in \{1, 2, 3\}$. 
   It follows that
    \begin{align}
        \label{lambda_explicit}
        \lambda(\xi^\varepsilon_\Theta)  = \text{\rm Id}_{3 \times 3} \left( \left( 1+\varepsilon^2 \abs{\xi^\varepsilon_\Theta}^2 \right)^{-\frac12} - 1 \right) 
        -  \varepsilon^2   \frac{ \xi^\varepsilon_\Theta \otimes \xi^\varepsilon_\Theta }{\left(1+\varepsilon^2 \abs{\xi^\varepsilon_\Theta}^2\right)^{\frac{3}{2}}}.
    \end{align}
We estimate both parts separately, for the first term, by the Taylor series
\begin{equation}
\label{app:taylor}
    (1+x)^{-\frac12}=\sum_{n=0}^\infty \frac{(-1)^n (2n)!}{4^n (n!)^2}x^n,
\end{equation}
we have
\begin{align*}
    \abs{\text{\rm Id}_{3 \times 3} \left( \left( 1+\varepsilon^2 \abs{\xi^\varepsilon_\Theta}^2 \right)^{-\frac12} - 1 \right) }_{\delta}
    & = \abs{\sum_{n=1}^{\infty} \frac{(-1)^n(2n)!}{4^n (n!)^2} \varepsilon^{2n} \abs{\xi^\varepsilon_\Theta}^{2n}}_{\delta} \\
    & = \abs{ \varepsilon^2 \abs{\xi^\varepsilon_\Theta}^{2} \sum_{n=1}^{\infty} \frac{2n(2n-1)}{4 n^2} \frac{ (-1)^{(n-1)} (2(n-1))!}{4^{(n-1)} ((n-1)!)^2} \varepsilon^{2(n-1)} \abs{\xi^\varepsilon_\Theta}^{2(n-1)}}_{\delta}.
\end{align*}
Therefore, by the algebra property \eqref{app:algprop} and the bound $\abs{\frac{2n(2n-1)}{4 n^2}} \leq 1$, we get
\begin{align}
\label{aux_ineq_3}
    \abs{\text{\rm Id}_{3 \times 3} \left( \left( 1+\varepsilon^2 \abs{\xi^\varepsilon_\Theta}^2 \right)^{-\frac12} - 1 \right) }_{\delta} 
    & \leq  \varepsilon^2 \abs{\xi^\varepsilon_\Theta}_{\delta}^{2} \sum_{n=0}^{\infty} \frac{ (2n)!}{4^{n} (n!)^2} \varepsilon^{2n} \abs{\xi^\varepsilon_\Theta}_{\delta}^{2n}=  \frac{\varepsilon^2 \abs{\xi^\varepsilon_\Theta}_{\delta}^{2}}{\left( 1 - \varepsilon^{2} \abs{\xi^\varepsilon_\Theta}_{\delta}^{2} \right)^{\frac{1}{2}}} \leq C \varepsilon^2\abs{\xi^\varepsilon_\Theta}_{\delta}^{2},
\end{align}
where in the last inequality we used the assumption in \eqref{lemmarem_assumption}.
For the second term in \eqref{lambda_explicit}, we use again the algebra property \eqref{app:algprop}:
\begin{align}
\label{aux_ineq_1}
    \abs{\varepsilon^2   \frac{ \xi^\varepsilon_\Theta \otimes \xi^\varepsilon_\Theta }{\left(1+\varepsilon^2 \abs{\xi^\varepsilon_\Theta}^2\right)^{\frac{3}{2}}} }_\delta 
    & \leq \varepsilon^2 \abs{\frac{\xi^\varepsilon_\Theta}{\left(1+\varepsilon^2 \abs{\xi^\varepsilon_\Theta}^2\right)^{\frac{1}{2}}}}_\delta 
    \abs{\frac{\xi^\varepsilon_\Theta}{\left(1+\varepsilon^2 \abs{\xi^\varepsilon_\Theta}^2\right)}}_\delta = \varepsilon^2  \abs{v(\xi^\varepsilon_\Theta)}_\delta
    \abs{\frac{\xi^\varepsilon_\Theta}{\left(1+\varepsilon^2 \abs{\xi^\varepsilon_\Theta}^2\right)}}_\delta \nonumber \\
    & \leq C \varepsilon^2 \abs{\xi^\varepsilon_\Theta}_\delta
    \abs{\frac{\xi^\varepsilon_\Theta}{\left(1+\varepsilon^2 \abs{\xi^\varepsilon_\Theta}^2\right)}}_\delta,
\end{align}
    where we used \eqref{vxiteta_simple_norm} to bound $\abs{v(\xi^\varepsilon_\Theta)}_\delta$. 
    Then, for the last term, we have by the Taylor expansion and the algebra property \eqref{app:algprop}
    \begin{align}
    \label{aux_ineq_2}
        \abs{\frac{\xi^\varepsilon_\Theta}{\left(1+\varepsilon^2 \abs{\xi^\varepsilon_\Theta}^2\right)}}_\delta 
        & = \abs{\xi^\varepsilon_\Theta \sum_{n=0}^\infty (-1)^n \varepsilon^{2n} \abs{\xi^\varepsilon_\Theta}^{2n}}_{\delta} 
         \leq \abs{\xi^\varepsilon_\Theta}_{\delta}  \sum_{n=0}^\infty  \varepsilon^{2n} \abs{\xi^\varepsilon_\Theta}^{2n}_{\delta}  \leq  \frac{\abs{\xi^\varepsilon_\Theta}_{\delta}}{1 - \varepsilon^2 \abs{\xi^\varepsilon_\Theta}^{2}_{\delta}} \leq 2 \abs{\xi^\varepsilon_\Theta}_{\delta},
    \end{align}
    by assumption \eqref{lemmarem_assumption}.
Therefore, going back to \eqref{aux_ineq_1} using \eqref{aux_ineq_2}, we get
\begin{align}
\label{aux_ineq_4}
    \abs{\varepsilon^2   \frac{ \xi^\varepsilon_\Theta \otimes \xi^\varepsilon_\Theta }{\left(1+\varepsilon^2 \abs{\xi^\varepsilon_\Theta}^2\right)^{\frac{3}{2}}} }_\delta 
    & \leq C \varepsilon^2 \abs{\xi^\varepsilon_\Theta}_\delta^2.
\end{align}
Finally, with \eqref{aux_ineq_3} and \eqref{aux_ineq_4} we bound \eqref{lambda_explicit}:
\begin{align*}
        \abs{\lambda (\xi^\varepsilon_\Theta)}_\delta 
        & \leq \abs{\text{\rm Id}_{3 \times 3} \left( \left( 1+\varepsilon^2 \abs{\xi^\varepsilon_\Theta}^2 \right)^{-\frac12}- 1 \right) }_\delta 
        +  \abs{\varepsilon^2  \frac{ \xi^\varepsilon_\Theta \otimes \xi^\varepsilon_\Theta }{\left(1+\varepsilon^2 \abs{\xi^\varepsilon_\Theta}^2\right)^{\frac{3}{2}}} }_{\delta} \leq C \varepsilon^2 \abs{\xi^\varepsilon_\Theta}_\delta^2 \leq C \varepsilon^2 \norm{\xi}_{\delta_0}^2.
\end{align*}
Then, for $\ell \in \{ 1,2,3 \}$, we compute $\partial_{x_\ell} \lambda( \xi^\varepsilon_\Theta)$ using \eqref{lambda_explicit}:
    {\small\begin{align*}
        \partial_{x_\ell} \left[\lambda( \xi^\varepsilon_\Theta)\right]= - \text{\rm Id}_{3\times 3} \left(  \frac{\varepsilon^2 \xi^\varepsilon_\Theta \cdot \partial_{x_\ell} \xi^\varepsilon_\Theta }{\left(1+\varepsilon^2 \abs{\xi^\varepsilon_\Theta}^2\right)^{\frac{3}{2}}} \right) 
        -  \varepsilon^2 \left(  \frac{2 \partial_{x_\ell}  \xi^\varepsilon_\Theta \otimes \xi^\varepsilon_\Theta }{\left(1+\varepsilon^2 \abs{\xi^\varepsilon_\Theta}^2\right)^{\frac{3}{2}}} \right) 
        +  \varepsilon^4 (\xi^\varepsilon_\Theta \otimes \xi^\varepsilon_\Theta) \left(  \frac{3 \xi^\varepsilon_\Theta \cdot \partial_{x_\ell} \xi^\varepsilon_\Theta   }{\left(1+\varepsilon^2 \abs{\xi^\varepsilon_\Theta}^2\right)^{\frac{5}{2}}} \right).
    \end{align*}}
Doing the same type of estimates as before on the three terms, we get
\begin{align*}
        \left( \delta_0 - \delta - \frac{t}{\eta} \right)^\beta \abs{\partial_{x_\ell} \lambda( \xi^\varepsilon_\Theta)}_\delta &\leq C \varepsilon^2 \abs{\xi^\varepsilon_\Theta}_\delta \left( \delta_0 - \delta - \frac{t}{\eta} \right)^\beta \abs{\partial_{x_\ell}\xi^\varepsilon_\Theta}_\delta + C \varepsilon^4 \abs{\xi^\varepsilon_\Theta}_\delta^3 \left( \delta_0 - \delta - \frac{t}{\eta} \right)^\beta\abs{\partial_{x_\ell}\xi^\varepsilon_\Theta}_\delta   \\ 
        & \leq C \varepsilon^2 \norm{\xi^\varepsilon_\Theta}_{\delta_0}^2 + C \varepsilon^4  \norm{\xi^\varepsilon_\Theta}_{\delta_0}^4.
    \end{align*}
Thus, we obtain
\begin{align*}
    \norm{\lambda(\xi^\varepsilon_\Theta)}_{\delta_0} 
        \leq
        C \varepsilon^2  \norm{\xi^\varepsilon_\Theta}_{\delta_0}^2 + C \varepsilon^4  \norm{\xi^\varepsilon_\Theta}_{\delta_0}^4 
        \leq C \varepsilon^2  \norm{\xi^\varepsilon_\Theta}_{\delta_0}^2,
\end{align*}
for $\varepsilon$ small.
Finally, we prove \eqref{Lemma_remainder_2}. We can rewrite $v(\xi^{\varepsilon,(1)}_\Theta)-v(\xi^{\varepsilon,(2)}_\Theta)$ as follow
{\small
\begin{align*}
    & v\left(\xi^{\varepsilon,(1)}_\Theta\right)-v\left(\xi^{\varepsilon,(2)}_\Theta\right) =   \frac{\xi^{\varepsilon,(1)}_\Theta}{\sqrt{1+\varepsilon^2 \abs{\xi^{\varepsilon,(1)}_\Theta}^2}} -  \frac{\xi^{\varepsilon,(2)}_\Theta}{\sqrt{1+\varepsilon^2 \abs{\xi^{\varepsilon,(2)}_\Theta}^2}}  \\
    & \quad = \frac{\xi^{\varepsilon,(1)}_\Theta \sqrt{1 + \varepsilon^2 \abs{\xi^{\varepsilon,(2)}_\Theta}^2} - \xi^{\varepsilon,(1)}_\Theta \sqrt{1 + \varepsilon^2 \abs{\xi^{\varepsilon,(1)}_\Theta}^2} + \xi^{\varepsilon,(1)}_\Theta \sqrt{1 + \varepsilon^2 \abs{\xi^{\varepsilon,(1)}_\Theta}^2} - \xi^{\varepsilon,(2)}_\Theta \sqrt{1 + \varepsilon^2 \abs{\xi^{\varepsilon,(1)}_\Theta}^2}}{ \sqrt{1 + \varepsilon^2 \abs{\xi^{\varepsilon,(1)}_\Theta}^2}  \sqrt{1 + \varepsilon^2 \abs{\xi^{\varepsilon,(2)}_\Theta}^2}} \\
    & \quad = (\xi^{\varepsilon,(2)}_\Theta - \xi^{\varepsilon,(1)}_\Theta ) \\
    & \qquad  \times \left(\frac{\varepsilon^2 \xi^{\varepsilon,(1)}_\Theta \left(\xi^{\varepsilon,(1)}_\Theta + \xi^{\varepsilon,(2)}_\Theta\right)}{\sqrt{1 + \varepsilon^2 \abs{\xi^{\varepsilon,(1)}_\Theta}^2} \sqrt{1 + \varepsilon^2 \abs{\xi^{\varepsilon,(2)}_\Theta}^2} \left( \sqrt{1 + \varepsilon^2 \abs{\xi^{\varepsilon,(1)}_\Theta}^2} +\sqrt{1 + \varepsilon^2 \abs{\xi^{\varepsilon,(2)}_\Theta}^2} \right)} -\frac{1}{\sqrt{1 + \varepsilon^2 \abs{\xi^{\varepsilon,(2)}_\Theta}^2}} \right).
\end{align*}}
Then, using Taylor expansion as before, we can bound all denominators as follows
\begin{align*}
    \abs{  \frac{1}{\sqrt{1+\varepsilon^2 \abs{\xi^\varepsilon_\Theta}^2}} }_{\delta} 
    & \leq  \sum_{n=0}^{\infty} \frac{ (2n)!}{4^{n} (n!)^2} \varepsilon^{2n} \abs{\xi^\varepsilon_\Theta}_{\delta}^{2n}=  \frac{1}{\left( 1 - \varepsilon^{2} \abs{\xi^\varepsilon_\Theta}_{\delta}^{2} \right)^{\frac{1}{2}}} \leq C,
\end{align*}
where we used assumption \eqref{lemmarem_assumption} on $\xi^{\varepsilon,(1)}_\Theta$ and $\xi^{\varepsilon,(2)}_\Theta$. Hence
\begin{align*}
    \abs{v(\xi^{(1)})-v(\xi^{(2)})}_\delta & \leq 
         C\abs{\xi^{\varepsilon,(1)}_\Theta-\xi^{\varepsilon,(2)}_\Theta}_{\delta} \left( \varepsilon^2\abs{\xi^{\varepsilon,(1)}_\Theta}_{\delta} \left( \abs{\xi^{\varepsilon,(1)}_\Theta}_{\delta} + \abs{\xi^{\varepsilon,(2)}_\Theta}_{\delta} \right) + 1 \right) \\
         & \leq C \norm{\xi^{\varepsilon,(1)}_\Theta-\xi^{\varepsilon,(2)}_\Theta}_{\delta_0},
\end{align*}
where we used assumption \eqref{lemmarem_assumption_2} for the last inequality.
The estimates for $\partial_{x_\ell} \left(v(\xi^{\varepsilon,(1)}_\Theta)-v(\xi^{\varepsilon,(2)}_\Theta)\right)$ for $\ell \in \{ 1,2,3 \}$ are done similarly.
Thus, we get our result:
\begin{align*}
     \norm{ v(\xi^{\varepsilon,(1)}_\Theta)-v(\xi^{\varepsilon,(2)}_\Theta)}_{\delta_0} 
         & \leq C \norm{\xi^{\varepsilon,(1)}_\Theta-\xi^{\varepsilon,(2)}_\Theta}_{\delta_0}.
\end{align*}
Concerning inequality \eqref{Lemma_remainder_3}, we have
\begin{align*}
    &\lambda\left(\xi^{\varepsilon,(1)}_\Theta\right)-\lambda\left(\xi^{\varepsilon,(2)}_\Theta\right)
    =\text{Id}_{3 \times 3}\Bigg( \left(1+\varepsilon^2\abs{\xi^{\varepsilon,(1)}_\Theta}^2 \right)^{-\frac12}  - \left(1+\varepsilon^2\abs{\xi^{\varepsilon,(2)}_\Theta}^2 \right)^{-\frac12} \Bigg)\\
    &\quad-\varepsilon^2
    \Bigg( \left(\xi^{\varepsilon,(1)}_\Theta \otimes \xi^{\varepsilon,(1)}_\Theta \right) \left(1+\varepsilon^2\abs{\xi^{\varepsilon,(1)}_\Theta}^2\right)^{-\frac32}- \left(\xi^{\varepsilon,(2)}_\Theta \otimes \xi^{\varepsilon,(2)}_\Theta \right) \left(1+\varepsilon^2\abs{\xi^{\varepsilon,(2)}_\Theta}^2\right)^{-\frac32} \Bigg).
\end{align*}
Since 
\begin{align*}
\frac{1}{\sqrt{1+\varepsilon^2\abs{\xi^{\varepsilon,(1)}_\Theta}^2}} &-\frac{1}{\sqrt{1+\varepsilon^2\abs{\xi^{\varepsilon,(2)}_\Theta}^2}}=\\
& \frac{\varepsilon^2 \left(\xi^{\varepsilon,(1)}_\Theta-\xi^{\varepsilon,(2)}_\Theta\right) \cdot\left(\xi^{\varepsilon,(1)}_\Theta + \xi^{\varepsilon,(2)}_\Theta\right)}{\sqrt{1 + \varepsilon^2 \abs{\xi^{\varepsilon,(1)}_\Theta}^2} \sqrt{1 + \varepsilon^2 \abs{\xi^{\varepsilon,(2)}_\Theta}^2} \left( \sqrt{1 + \varepsilon^2 \abs{\xi^{\varepsilon,(1)}_\Theta}^2} +\sqrt{1 + \varepsilon^2 \abs{\xi^{\varepsilon,(2)}_\Theta}^2} \right)},
\end{align*}
and
\begin{align*}
    \frac{\xi^{\varepsilon,(1)}_\Theta \otimes \xi^{\varepsilon,(1)}_\Theta}{\left(1+\varepsilon^2\abs{\xi^{\varepsilon,(1)}_\Theta}^2\right)^{\frac32}} - \frac{\xi^{\varepsilon,(2)}_\Theta \otimes \xi^{\varepsilon,(2)}_\Theta}{\left(1+\varepsilon^2\abs{\xi^{\varepsilon,(2)}_\Theta}^2\right)^{\frac32}} =
    \frac{\left(\xi^{\varepsilon,(1)}_\Theta -\xi^{\varepsilon,(2)}_\Theta\right)\otimes \xi^{\varepsilon,(1)}_\Theta+\xi^{\varepsilon,(2)}_\Theta\otimes\left(\xi^{\varepsilon,(1)}_\Theta -\xi^{\varepsilon,(2)}_\Theta\right) }{\left(1+\varepsilon^2\abs{\xi^{\varepsilon,(1)}_\Theta}^2\right)^{\frac32}}\\
    + \xi_\Theta^{\varepsilon,(2)}\otimes \xi_\Theta^{\varepsilon,(2)}\Bigg(\left(1+\varepsilon^2\abs{\xi^{\varepsilon,(1)}_\Theta}^2\right)^{-\frac{3}{2}} -\left(1+\varepsilon^2\abs{\xi^{\varepsilon,(2)}_\Theta}^2\right)^{-\frac{3}{2}} \Bigg),
\end{align*}
we can proceed as in \eqref{Lemma_remainder_2} to get the result.
\end{proof}

\begin{proof}
    [Proof of Lemma \ref{sec2:lemma_sobolev}]
    By recalling the Taylor expression \eqref{app:taylor}, we can write 
    \begin{align*}
        v(\xi^\varepsilon_\Theta) - \xi^\varepsilon_\Theta &= \frac{\xi^\varepsilon_\Theta}{\sqrt{1+\varepsilon^2 |\xi^\varepsilon_\Theta|^2}}-\xi^\varepsilon_\Theta 
        = \xi^\varepsilon_\Theta \left( \sum_{n=1}^\infty \frac{(-1)^n (2n)!}{4^n (n!)^2} \varepsilon^{2n} |\xi^\varepsilon_\Theta|^{2n}  \right) \\
        & = \varepsilon^2 \xi^\varepsilon_\Theta \abs{\xi^\varepsilon_\Theta}^{2} \sum_{n=1}^{\infty} \frac{2n(2n-1)}{4 n^2} \frac{ (-1)^{(n-1)} (2(n-1))!}{4^{(n-1)} ((n-1)!)^2} \varepsilon^{2(n-1)} \abs{\xi^\varepsilon_\Theta}^{2(n-1)}.
    \end{align*}
    Therefore, by using that $\abs{\frac{2n(2n-1)}{4 n^2}} \leq 1$ and the algebra property for the Sobolev space $H^s_x$ with $s>\frac32$, we obtain
    \begin{align*}
    \norm{v(\xi^\varepsilon_\Theta) - \xi^\varepsilon_\Theta}_{L_t^\infty H^s_x} 
    & \leq \sup_t\norm{ \varepsilon^2  \abs{\xi^\varepsilon_\Theta(t)}^{3} \sum_{n=0}^{\infty} \frac{ (2n)!}{4^n (n!)^2} \varepsilon^{2n} \abs{\xi^\varepsilon_\Theta(t)}^{2n}}_{ H^s_x} \\
    & \leq \sup_t \left( \varepsilon^2 \norm{  \xi^\varepsilon_\Theta(t)}^{3}_{ H^s_x} \sum_{n=0}^{\infty} \frac{ (2n)!}{4^n (n!)^2} \varepsilon^{2n} \norm{\xi^\varepsilon_\Theta(t)}^{2n}_{ H^s_x} \right) \\
    & \leq \sup_t \Bigg( \varepsilon^2 \norm{  \xi^\varepsilon_\Theta(t)}^{3}_{ H^s_x} \frac{1}{\big(1 - \varepsilon^{2} \norm{\xi^\varepsilon_\Theta(t)}^{2}_{ H^s_x} \big)^{\frac{1}{2}}} \Bigg) \leq C \varepsilon^2 \norm{  \xi^\varepsilon_\Theta}^{3}_{L_t^\infty H^s_x},
    \end{align*}
    where we used the Taylor expansion \eqref{app:taylor} for the penultimate inequality and assumption \eqref{sec2:assumption_sobolev} for the last one.

\end{proof}

\section*{Acknowledgments}
The authors are grateful to Daniel Han-Kwan and Fr\'ed\'eric Rousset for insightful comments on an early version of this manuscript.
M. Iacobelli acknowledges the support of the National Science Foundation through the grant No. DMS-1926.

\bibliographystyle{plain}
\bibliography{quasineutral_VM}

\end{document}